\documentclass[11pt]{article}
\usepackage{amsmath}
\usepackage{amsfonts}
\usepackage{amssymb}
\usepackage{amsthm}
\usepackage{graphicx}
\usepackage{empheq}
\usepackage{cases}
\usepackage{indentfirst}
\usepackage{cite}
\usepackage{mathrsfs}
\usepackage{xcolor}
\usepackage{bm}
\usepackage{makeidx}
\usepackage[T1]{fontenc}
\usepackage{times}
\usepackage{enumitem}

\usepackage{tocloft}
\usepackage[colorlinks=true,linkcolor=blue,citecolor=blue, urlcolor=blue,hypertexnames=false]{hyperref}

\newtheorem{theorem}{\textbf{Theorem}}[section]
\newtheorem{lemma}{\textbf{Lemma}}[section]
\newtheorem{proposition}{\textbf{Proposition}}[section]

\newtheorem{remark}{\textbf{Remark}}[section]
\newtheorem{definition}{\textbf{Definition}}[section]

\makeatletter
\renewenvironment{proof}[1][\proofname]{%
	\par\pushQED{\qed}\normalfont%
	\topsep6\p@\@plus6\p@\relax
	\trivlist\item[\hskip\labelsep\bfseries#1\@addpunct{.}]%
	\ignorespaces
}{%
	\popQED\endtrivlist\@endpefalse
}
\makeatother
\allowdisplaybreaks[4]

\begin{document}

\title{
{Global Quasi-strong Solutions to the Voigt Regularization
of a Diffuse Interface Model for Incompressible Two-phase Flows \\
with Bulk-surface Interaction}
}

\author{
Patrik Knopf \footnotemark[1] \hspace{0pt} \footnotemark[2] \ ,
\and
Maoyin Lv\footnotemark[3] \ ,
\and
Hao Wu\footnotemark[3] \hspace{0pt} \footnotemark[4]
}

\date{}

\maketitle

\renewcommand{\thefootnote}{\fnsymbol{footnote}}

\footnotetext[1]{
Faculty for Mathematics, University of Regensburg, D-93053 Regensburg, Germany. Email: \texttt{\href{mailto:patrik.knopf@ur.de}{patrik.knopf@ur.de}}
}

\footnotetext[2]{
Institute for Analysis, Department of Mathematics, Karlsruhe Institute of Technology (KIT), D-76128 Karlsruhe, Germany.
}

\footnotetext[3]{
School of Mathematical Sciences, Fudan University, Shanghai
200433, P.R. China. \\
$\phantom{xxx}$
Email: \texttt{\href{mailto:mylv22@m.fudan.edu.cn}{mylv22@m.fudan.edu.cn}},\,
\texttt{\href{mailto:haowufd@fudan.edu.cn}{haowufd@fudan.edu.cn}}
}

\footnotetext[4]{
Corresponding author.
}

%
%

\begin{abstract}
\noindent We analyze a thermodynamically consistent diffuse interface model for incompressible two-phase viscous flows
with unmatched densities in a smooth bounded domain $\Omega\subset\mathbb{R}^d$ $(d=2,3)$.
This model consists of the Navier--Stokes--Voigt equations for the velocity field and a convective Cahn--Hilliard equation with a singular potential for the phase-field variable.
The resulting hydrodynamic system is subject to a generalized Navier slip boundary condition for the velocity field and Cahn--Hilliard-type dynamic boundary conditions for the phase-field variable and the chemical potential.
With the aid of the Voigt regularization, we establish the existence of global quasi-strong solutions that satisfy an energy equality. This is achieved through a combination of delicate approximation schemes and a semi-Galerkin method.
\noindent

\medskip \noindent
\textit{Keywords}: Two-phase flow,
Navier--Stokes--Voigt system,
Cahn--Hilliard equation, generalized
Navier boundary condition, dynamic boundary condition, global quasi-strong solution.

\smallskip \noindent
\textit{MSC 2020}: 35A01, 35K55, 35Q35, 76D05.
\end{abstract}

\begin{footnotesize}
\setcounter{tocdepth}{2}
\hypersetup{linkcolor=black}
\tableofcontents
\end{footnotesize}

\newpage


\section{Introduction}
\setcounter{equation}{0}
{The mathematical description of multi-phase flows is a central topic in modern continuum fluid dynamics. Such flows arise in a wide range of natural and industrial processes, including biological systems, chemical engineering, and materials science, and involve the evolution of complex fluid interfaces.
A widely used diffuse-interface model for incompressible binary mixtures of immiscible fluids with unmatched densities}
is the so-called AGG model, which was proposed by Abels, Garcke, and Gr\"un in \cite{AGG}:
\begin{subequations}
\label{AGG}
\begin{alignat}{2}
&\partial_t( \rho(\varphi) \mathbf{v})+ \mathrm{div}(\mathbf{v}\otimes(\rho(\varphi)\mathbf{v}+{\mathbf{J}}))
-\mathrm{div}(2\nu(\varphi) \mathbb{D}\mathbf{v})+\nabla p
=-\mathrm{div}( \nabla \varphi \otimes\nabla \varphi  )
&&\quad\text{in }Q_T,
\\
&\mathrm{div}\,\mathbf{v} =0
&&\quad\text{in }Q_T,
\\
&\partial_t\varphi+{  \mathbf{v}}\cdot\nabla \varphi=\Delta\mu
&&\quad\text{in }Q_T,
\\
&\mu=-\Delta\varphi+F'(\varphi)
&&\quad\text{in }Q_T.
\end{alignat}
\end{subequations}
Here, $T\in(0,+\infty)$ is a given final time, $\Omega\subset\mathbb{R}^d$ $(d=2,3)$ is a bounded domain with smooth boundary $\Gamma \coloneqq \partial\Omega$, and we write $Q_T \coloneqq \Omega\times(0,T)$.
The letter $\mathbf{n}$ denotes the outward unit normal vector on $\partial\Omega$,
and $\partial_\mathbf{n}$ denotes the associated normal derivative on $\partial\Omega$.
The unknowns $\mathbf{v} :\Omega\times(0,T)\to\mathbb{R}^d$,
$\varphi:\Omega\times(0,T)\to[-1,1]$, and $p:\Omega\times(0,T)\to\mathbb{R}$
represent the volume-averaged velocity field,
the order parameter (i.e., the difference $\varphi=\varphi_2-\varphi_1$ of the volume fractions
$\varphi_1$ and $\varphi_2$ of the two fluids),
and the pressure of the fluid mixture, respectively.
The flux term $\mathbf{J}$, the density $\rho$, and the viscosity $\nu$ are given by
\begin{equation}
    \label{DEF:JRN}
    {\mathbf{J}}=-\frac{\rho_1-\rho_2}{2}\nabla\mu,
    \quad\rho (\varphi) =\rho_1\frac{1+\varphi}{2}+\rho_2\frac{1-\varphi}{2},
    \quad\nu (\varphi) =\nu_1\frac{1+\varphi}{2}+\nu_2\frac{1-\varphi}{2},
\end{equation}
where the positive constants $\rho_1$, $\rho_2$ and $\nu_1$, $\nu_2$ represent the homogeneous densities and viscosities of the two fluids, respectively.
Without loss of generality, we assume that $\rho_1>\rho_2>0$.
As $\rho$ is linear in $\varphi$, we write
\[\rho':=\rho'(\varphi)=\frac{\rho_1-\rho_2}{2}\]
for simplicity. Moreover, the symbol $\mathbb{D}\mathbf{v}=\tfrac 12 (\nabla\mathbf{v}+(\nabla\mathbf{v})^{\top})$ denotes the symmetric gradient of $\mathbf{v}$, where the superscript $\top$ denotes the transposition of the matrix.
The nonlinear function $F$ is the homogeneous free energy density, which is assumed to exhibit a double-well structure, and $F'$ is its derivative.
For the description of binary fluids, the physically most relevant choice is the Flory--Huggins potential given by
\begin{align*}
    F_{\text{log}}(r)=\frac{\Theta}{2}[(1+r)\text{ln}(1+r)+(1-r)\text{ln}(1-r)]
    -\frac{\Theta_0}{2}r^2, \quad\text{for all $r\in(-1,1)$}.
\end{align*}
Here, the constant parameters $\Theta$ and $\Theta_0$ are assumed to fulfill the condition $0<\Theta<\Theta_0$, which ensures that $F_{\text{log}}$ is double-well shaped.

In the case of matched densities $\rho_1=\rho_2$, the AGG model \eqref{AGG} reduces to the well-known ``Model H'', which was originally proposed in \cite{HH} and later rigorously derived in \cite{GPV}.
Since then, it has been investigated in numerous works such as \cite{Ab2,GMT} and the references therein.
For the development of thermodynamically consistent diffuse interface models for multi-phase flows with unmatched densities, we refer to \cite{AGG,ADGK,LT,SYW,SSBv,tvS}.
They are based on different choices of the mean velocity of the mixture as well as further constitutive assumptions.
Further discussions on the existing Navier--Stokes--Cahn--Hilliard models can be found in \cite{tvAS}, where a unified framework was developed.

In the literature, the AGG model has mostly been supplemented with the classical no-slip boundary condition for $\mathbf{v}$ and homogeneous Neumann boundary conditions for $\varphi$ and $\mu$:
\begin{subequations}
\begin{align}
    \mathbf{v}=\boldsymbol{0}&\qquad\text{on }\Sigma_T,\label{no-slip}\\
    \partial_\mathbf{n}\varphi=0&\qquad\text{on }\Sigma_T,\label{homo-phi}\\
    \partial_\mathbf{n}\mu=0&\qquad\text{on }\Sigma_T.\label{homo-mu}
\end{align}
\end{subequations}
The AGG model \eqref{AGG}, subject to \eqref{no-slip}--\eqref{homo-mu} and suitable initial conditions, has been extensively studied; we refer to \cite{ADG,ADG1,AW,Gior21,Gior22,AGiG,GarckeLvWu} and the references therein.
Regarding progress in its nonlocal variant and other generalizations, we refer to
\cite{AT,AGP,AGP2,Fri16,Fri21,GGGP,HKP}.
Although the AGG model \eqref{AGG} is a useful model for the study of two-phase flows with unmatched densities,
it inherits certain limitations resulting from the no-slip boundary condition on the velocity field
and the homogeneous Neumann boundary conditions imposed on the underlying convective Cahn--Hilliard subsystem:
\begin{itemize}
     \item [(L1)] As discussed in \cite{QWS}, the no-slip boundary condition \eqref{no-slip} is often inadequate for the description of moving contact line phenomena. This is mainly because it neglects convective effects at the boundary and in its close vicinity.
    \item [(L2)]  The boundary condition \eqref{homo-phi} implies that the diffuse interface intersects the boundary $\Gamma$ at a perfect angle of ninety degrees.
    However, in the context of two-phase flows, the contact angle is expected to deviate from ninety degrees and even evolve dynamically due to the motion of the mixture \cite{QWS}.
    \item [(L3)] The boundary condition \eqref{homo-mu} can be regarded as a no-flux boundary condition,
    as it implies that the normal mass flux $\mathbf{J}\cdot\mathbf{n}$ vanishes at the boundary.
    Therefore, the system \eqref{AGG} can only describe the situation in which the mass of both fluids in the bulk is conserved. However, any transfer of material between bulk and surface (which could, for instance, be caused by absorption or desorption processes or chemical reactions that take place at the boundary) cannot be described (see, e.g., \cite{KLLM}).
\end{itemize}

To overcome the aforementioned limitations {(L1) and (L2), a new class of boundary conditions based on Onsager's principle of maximal energy dissipation was proposed in \cite{QWS}:}
 \begin{subequations}
 \label{BD-QWS}
 \begin{alignat}{2}
    \label{BD-QWS:1}
   &\mathbf{v}\cdot\mathbf{n}=0, \quad\partial_\mathbf{n}\mu=0,
   \quad \psi =\varphi
   &&\quad\text{on }\Sigma_T,
   \\[1mm]
    \label{BD-QWS:2}
   &2\nu(\psi)[\mathbb{D}\mathbf{v}\,\mathbf{n}]_{\boldsymbol{\tau}}+{\beta}\mathbf{v}_{\boldsymbol{\tau}}
   =\mathcal{G}(\psi)\nabla_{\!\boldsymbol{\tau}\,}\psi
   &&\quad\text{on }\Sigma_T,
   \\[1mm]
    \label{BD-QWS:3}
   &\partial_t\psi+\mathbf{v}_{\boldsymbol{\tau}} \cdot{\nabla_{\!\boldsymbol{\tau}\,}}\psi = - {\gamma} \mathcal{G}(\psi)
   &&\quad\text{on }\Sigma_T,
   \\[1mm]
    \label{BD-QWS:4}
   &\mathcal{G}(\psi)=-\kappa\Delta_{\boldsymbol{\tau}}\psi+\varepsilon\partial_\mathbf{n}\varphi+G'(\psi)
   &&\quad\text{on }\Sigma_T.
\end{alignat}
\end{subequations}
It involves a generalized Navier slip boundary condition on the tangential part of the velocity field (see \eqref{BD-QWS:2}) and a second-order (Allen--Cahn type) dynamic boundary condition on the phase field (see \eqref{BD-QWS:3}--\eqref{BD-QWS:4}).
Here, the subscript $\boldsymbol{\tau}$ denotes the tangential component of {any vector field} $\mathbf{w}$,
i.e., $\mathbf{w}_{\boldsymbol{\tau}}=\mathbf{w}-(\mathbf{w}\cdot\mathbf{n})\mathbf{n}$,
$\nabla_{\!\boldsymbol{\tau}\,}$ is the tangential gradient on $\Gamma$
and $\Delta_{\boldsymbol{\tau}}$ denotes the Laplace--Beltrami operator on $\Gamma$.
The phenomenological parameters $\beta$, $\kappa$ and $\gamma$ are positive,
and the function $G$ represents a surface double-well potential.
The ``Model~H'' endowed with the boundary conditions \eqref{BD-QWS} was first studied in \cite{GGM},
where the authors proved the existence of global weak solutions for the cases where both
$F$ and $G$ are polynomial-like potentials, or $F$ is a singular potential and $G$ is a polynomial.
Later, the existence of global weak solutions to \eqref{AGG} under \eqref{BD-QWS} was established in \cite{GGW19}.

To overcome the limitation (L3) of the AGG model \eqref{AGG},
a new class of boundary conditions was derived in \cite{GK23}.
It consists of Cahn--Hilliard-type dynamic boundary conditions for the phase-field variable $\varphi$ and the chemical potential $\mu$, and a generalized Navier slip boundary condition for the velocity field $\mathbf{v}$:
 \begin{subequations}
 \label{BD}
 \begin{alignat}{2}
    \label{BD:1}
    &\mathbf{v}\cdot \mathbf{n}=0
    &&\quad  \mbox{on }\Sigma _{T },
    \\[1mm]
    \label{BD:2}
    &\Big[ 2\nu ( \psi) \mathbb{D}\mathbf{v}\,\mathbf{n}+\beta(\psi)\mathbf{v}\Big]_{\bm{\tau}}
    = \Big[-\psi\nabla_{\!\bm{\tau}\,}\mathcal{L}+\dfrac{1}{2}(\mathbf{J}\cdot\mathbf{n})\mathbf{v}\Big]_{\bm{\tau}}
    &&\quad \mbox{on}\,\Sigma _{T },
    \\[1mm]
    \label{BD:3}
    &\begin{cases}
        L \partial_{\mathbf{n}}\mu =\mathcal{L}-\mu, & \text{if}\ L\in[0,+\infty),\\
        \partial_{\mathbf{n}}\mu=0,& \text{if}\ L=+\infty,
    \end{cases}
    \quad
    \quad \psi =\varphi \quad
    &&\quad\mbox{on }\Sigma _{T },
    \\[1mm]
    \label{BD:4}
    &\partial_{t}\psi +\mathrm{div}_\Gamma(\psi\mathbf{v}_{\bm{\tau}})
    =\Delta_{\bm{\tau}} \mathcal{L}-\partial_{\mathbf{n}}\mu
    &&\quad \mbox{on }\,\Sigma_{T },
    \\[1mm]
    \label{BD:5}
    &\mathcal{L}=- \Delta _{\bm{\tau}}\psi +\partial _{\mathbf{n}}\varphi+G'(\psi )
    &&\quad\mbox{on }\,\Sigma_{T }.
\end{alignat}
\end{subequations}
Here the positive slip parameter $\beta$ may also depend on the surface phase-field $\psi$.
The derivation of system \eqref{AGG} subject to the boundary conditions \eqref{BD} was carried out in \cite{GK23} by means of local mass balance laws, local energy dissipation laws, and the Lagrange multiplier approach.
This model allows for surface diffusion, a dynamic contact angle between the bulk diffuse interface and the solid boundary, and possible mass transfer between bulk and surface.
Given the initial conditions
\begin{equation}
\mathbf{v}|_{t=0}=\mathbf{v}_{0},\quad \varphi |_{t=0}=\varphi _{0}\quad \text{in }\Omega,\qquad
\psi |_{t=0}=\psi _{0}=\mathrm{tr}(\varphi _{0})\quad \text{on}\ \Gamma,
\label{initial}
\end{equation}
the system \eqref{AGG} subject to \eqref{BD}--\eqref{initial} has been studied in \cite{GK23,GLW}.
In \cite{GK23}, the authors considered the case $L=0$ with regular potentials and matched densities, i.e., $\rho_1=\rho_2$.
They showed the existence of global weak solutions in two and three dimensions, as well as the uniqueness of the weak solution in two dimensions.
Later, in \cite{GLW}, the existence of global weak solutions in two and three dimensions was established for singular potentials and unmatched densities when $L\in(0,+\infty]$ or matched densities when $L=0$.
For prescribed velocity fields in the bulk and on the boundary, the associated convective bulk-surface Cahn--Hilliard subsystem was analyzed in \cite{KS2024,KSJDE,GKS,KPSY}.
The non-convective version (i.e., $\mathbf{v}\equiv\mathbf{0}$) has been the subject of extensive investigation in the literature. We refer to \cite{CFW20,FW,GK20,GMS,LW,LvWuIFB,LvWuAA,LvWuJEE,Stange-CHDBC,KLLM,GKY} for a selection of representative contributions.
For a comprehensive overview of the Cahn--Hilliard equation with classical or dynamic boundary conditions,
we refer to the book \cite{Mi19} and the recent review paper \cite{Wu22}.
Finally, we mention a new diffuse interface model
for incompressible, viscous fluid mixtures with bulk-surface interaction, consisting of a Navier--Stokes--Cahn--Hilliard model in the bulk that is coupled to a surface Navier--Stokes--Cahn--Hilliard model on the boundary, which has been proposed and analyzed in \cite{Stange,KS,Stange-NSCH}.

Regarding the regularity of solutions to \eqref{AGG} subject to \eqref{BD}--\eqref{initial},
it seems difficult to go beyond the mere existence of global weak solutions in the case of unmatched densities.
This is due to the low regularity of the trace of the velocity field on the boundary and the fact that $\mathrm{div}_{\Gamma}\mathbf{v}_{\boldsymbol{\tau}}\neq0$ on $\Sigma_T$, which make the nonlinear coupling term $\mathrm{div}_\Gamma(\psi\mathbf{v}_{\boldsymbol{\tau}})$ difficult to handle. Instead, we shall consider an alternative new class of boundary conditions,
which {was recently derived} in \cite{KL2025}
based on {local mass balance laws}, local energy dissipation laws, and the Lagrange multiplier approach:
\begin{subequations}
\label{BD-KL}
\begin{alignat}{2}
    \label{BD-KL:1}
    &\mathbf{v}\cdot \mathbf{n}=0,\quad\mathrm{div}_\Gamma\,\mathbf{v}_{\boldsymbol{\tau}}=0
    &&\quad  \mbox{on }\Sigma _{T },
    \\[1mm]
    \label{BD-KL:2}
    &\begin{cases}
    K \partial_{\mathbf{n}}\varphi =\psi-\varphi, & \text{if}\ K\in[0,+\infty)\\
    \partial_{\mathbf{n}}\varphi=0,& \text{if}\ K=+\infty
    \end{cases}
    &&\quad\mbox{on }\Sigma _{T },
    \\
    \label{BD-KL:3}
    &\Big[ 2\nu ( \psi) \mathbb{D}\mathbf{v}\,\mathbf{n}+\beta ( \psi ) \mathbf{v}\Big]_{\bm{\tau}}
    + \nabla_{\!\boldsymbol{\tau}\,} q
    = \Big[-\psi\nabla_{\!\bm{\tau}\,}\mathcal{L}
    +\dfrac{1}{2}(\mathbf{J}\cdot\mathbf{n})\mathbf{v}\Big]_{\bm{\tau}}
    &&\quad \mbox{on}\,\Sigma _{T },
    \\[1mm]
    \label{BD-KL:4}
    &\begin{cases}
    L \partial_{\mathbf{n}}\mu =\mathcal{L}-\mu, & \text{if}\ L\in[0,+\infty)\\
    \partial_{\mathbf{n}}\mu=0,& \text{if}\ L=+\infty
    \end{cases}
    &&\quad\mbox{on }\Sigma _{T },
    \\[1mm]
    \label{BD-KL:5}
    &\partial _{t}\psi + \mathbf{v}_{\bm{\tau}} \cdot \nabla_{\!\bm{\tau}\,}\psi
    =\Delta_{\bm{\tau}} \mathcal{L}
    -\partial_{\mathbf{n}}\mu
    &&\quad \mbox{on }\,\Sigma_{T },
    \\[1mm]
    \label{BD-KL:6}
    &\mathcal{L}=- \Delta _{\bm{\tau}}\psi +\partial _{\mathbf{n}}\varphi
    +G'(\psi )
    &&\quad\mbox{on }\,\Sigma_{T }.
\end{alignat}
\end{subequations}
In particular, compared to \eqref{BD:2}, the boundary condition \eqref{BD-KL:3} involves an additional surface pressure $q$, which acts as a Lagrange multiplier for the condition $\mathrm{div}_\Gamma\,\mathbf{v}_{\boldsymbol{\tau}}=0$ on $\Sigma_T$.
To our knowledge, no theoretical results are available for the AGG model \eqref{AGG} subject to the initial conditions \eqref{initial} and boundary conditions \eqref{BD-KL}.
However, a strategy similar to that employed in \cite{GLW} can be used to establish the existence of a global weak solution.

In the case of unmatched densities, it remains difficult to obtain results beyond the existence of weak solutions.
Nevertheless, inspired by \cite{GGP}, we will show that the existence of more regular solutions can at least be established for a \textit{Voigt regularization} of the system \eqref{AGG}, subject to \eqref{initial}--\eqref{BD-KL}. From a physical point of view, the Voigt regularization (sometimes also referred to as the Kelvin--Voigt regularization) is mainly used to describe viscoelastic fluids or highly turbulent flows. For mathematical purposes, it has also been extensively employed as a regularization of various incompressible flow models for analytical or numerical purposes.
The Voigt-regularized stress tensor is given by
\begin{align*}
    \mathbf{T}(\mathbf{v})=2\nu(\varphi)\mathbb{D}\mathbf{v}+2\alpha\mathbb{D}\partial_t\mathbf{v}.
\end{align*}
The additional term $2\alpha\mathbb{D}\partial_t\mathbf{v}$ with $\alpha>0$ can be interpreted as a time relaxation (see, e.g., \cite{BS,Layton,Mohan} and references therein).

In summary, for any $\alpha>0$, we consider the following Voigt-regularized variant of the Navier--Stokes--Cahn--Hilliard model \eqref{AGG} subject to \eqref{initial}--\eqref{BD-KL}:
\begin{subequations}
\label{eqmain0new}
\begin{alignat}{2}
&\partial _{t}(\rho (\varphi)\mathbf{v})+\mathrm{div}\big(\mathbf{v}\otimes (\rho (\varphi)\mathbf{v}+\mathbf{J})\big)
-\mathrm{div}\big(2\nu (\varphi )\mathbb{D}\mathbf{v}+{2\alpha \mathbb{D}\partial_t\mathbf{v}}\big)+\nabla p
=\mu \nabla\varphi
&&\quad \mbox{in }\,Q_{T },
\\[1mm]
&\mathrm{div}\,\mathbf{v}=0
&&\quad \mbox{in }\,Q_{T },
\\[1mm]
&\partial _{t}\varphi +\mathbf{v}\cdot \nabla \varphi=\Delta \mu
&&\quad \mbox{in }\,Q_{T },
\\[1mm]
&\mu =-\Delta \varphi +F'(\varphi )
&&\quad \mbox{in }\,Q_{T },
\\[1mm]
&\rho (\varphi )=\displaystyle{\frac{\rho_{1}-\rho_{2}}{2}\varphi +\frac{\rho_{1}+\rho_{2}}{2} },\qquad
\mathbf{J}=-\rho'\nabla \mu
&&\quad \mbox{in }\,Q_{T },%
\end{alignat}
\end{subequations}
subject to the boundary conditions
\begin{subequations}
\label{BD-1}
\begin{alignat}{2}
&\mathbf{v}\cdot \mathbf{n}=0,\quad\mathrm{div}_\Gamma\,\mathbf{v}_{\boldsymbol{\tau}}=0
&&\quad  \mbox{on }\Sigma _{T },
\\[1mm]
&\begin{cases}
K \partial_{\mathbf{n}}\varphi =\psi-\varphi, & \text{if}\ K\in[0,+\infty)\\
\partial_{\mathbf{n}}\varphi=0,& \text{if}\ K=+\infty
\end{cases}
&&\quad {\mbox{on }\Sigma _{T },}
\\[1mm]
&\Big[ \big(2\nu ( \varphi) \mathbb{D}\mathbf{v}+{2\alpha \mathbb{D}\partial_t\mathbf{v}}\big)\mathbf{n}\Big]_{\bm{\tau}}
+\beta ( \psi ) \mathbf{v}_{\bm{\tau}} + \nabla_{\!\boldsymbol{\tau}\,} q = \Big[-\psi\nabla_{\!\bm{\tau}\,}\mathcal{L}
+\dfrac{1}{2}(\mathbf{J}\cdot\mathbf{n})\mathbf{v}\Big]_{\bm{\tau}}
&&\quad \mbox{on}\,\Sigma _{T },
\\[1mm]
&\begin{cases}
L \partial_{\mathbf{n}}\mu =\mathcal{L}-\mu, & \text{if}\ L\in[0,+\infty)\\
\partial_{\mathbf{n}}\mu=0,& \text{if}\ L=+\infty
\end{cases}
&&\quad\mbox{on }\Sigma _{T },
\\[1mm]
&\partial _{t}\psi + \mathbf{v}_{\bm{\tau}}  \cdot \nabla_{\!\bm{\tau}\,}\psi
=\Delta_{\bm{\tau}} \mathcal{L}-\partial_{\mathbf{n}}\mu
&&\quad \mbox{on }\,\Sigma_{T },
\\[1mm]
&\mathcal{L}=- \Delta _{\bm{\tau}}\psi +\partial _{\mathbf{n}}\varphi+G'(\psi )
&&\quad \mbox{on }\,\Sigma_{T },
\end{alignat}
\end{subequations}
and the initial conditions
\begin{equation}
\mathbf{v}|_{t=0}=\mathbf{v}_{0},\quad \varphi |_{t=0}=\varphi _{0}\quad\text{in }\Omega,\qquad
\psi |_{t=0}=\psi _{0}\quad \text{on}\ \Gamma.
\label{icnew}
\end{equation}
The total energy of the system \eqref{eqmain0new}--\eqref{BD-1} is given by
\begin{align}
    E_{\mathrm{tot}}(\mathbf{v},\boldsymbol{\varphi}) \coloneqq
    E_{\text{kin}}(\mathbf{v},\boldsymbol{\varphi})
    +E_{\text{free}}(\boldsymbol{\varphi})\notag
\end{align}
with
\begin{align*}
    E_{\text{kin}}(\mathbf{v},\boldsymbol{\varphi})
    &\coloneqq \frac{1}{2}\int_\Omega \rho(\varphi)|\mathbf{v}|^2\,\mathrm{d}x
    +\alpha\int_\Omega |\mathbb{D}\mathbf{v}|^2\,\mathrm{d}x,
    \\
    E_{\text{free}}(\boldsymbol{\varphi})
    &\coloneqq
    \underbrace{\int_\Omega\Big(\frac{1}{2}|\nabla\varphi|^2+F(\varphi)\Big)\,\mathrm{d}x}_{\text{bulk free energy}}
    +\underbrace{\int_\Gamma \Big(\frac{1}{2}|\nabla_{\!\boldsymbol{\tau}\,}\psi|^2+G(\psi)\Big)\,\mathrm{d}S}_{\text{surface free energy}}
    \,+\, \frac{\chi(K)}{2} \int_\Gamma|\psi-\varphi|^2\,\mathrm{d}S
\end{align*}
where the third term in $E_{\text{free}}$ denotes a penalization of the mismatch between $\phi$ and $\psi$ on $\Gamma$, and
\begin{align}
\label{DEF:CHI}
\chi(r)=\left\{
    \begin{aligned}
        &\frac 1r ,
        &&\quad\text{if }r\in(0,+\infty),\\
        &0,
        &&\quad\text{if }r=0\ \text{or}\ +\infty.
    \end{aligned}
    \right.
\end{align}
Moreover, sufficiently regular solutions to problem \eqref{eqmain0new}--\eqref{icnew} satisfy the mass conservation property
\begin{align*}
    \left\{
    \begin{aligned}
        &\displaystyle\int_\Omega \varphi(t)\,\mathrm{d}x+\int_\Gamma \psi(t)\,\mathrm{d}S
        =\int_\Omega\varphi_0\,\mathrm{d}x+\int_\Gamma \psi_0\,\mathrm{d}S,
        &&\quad\text{for all}\,t\in[0,T]\ \text{and} \ L\in[0,+\infty),
        \\[2mm]
        &\displaystyle\int_\Omega \varphi(t)\,\mathrm{d}x=\int_\Omega\varphi_0\,\mathrm{d}x,
        \quad \int_\Gamma \psi(t)\,\mathrm{d}S=\int_\Gamma \psi_0\,\mathrm{d}S,
        &&\quad\text{for all}\,t\in[0,T]\ \text{and}\ L=+\infty,
    \end{aligned}
    \right.
\end{align*}
and satisfy the energy identity
\begin{align}
    \begin{split}
    \frac{\mathrm{d}}{\mathrm{d}t} E_{\mathrm{tot}}(\mathbf{v}(t),\boldsymbol{\varphi}(t))
    +\int_\Omega 2\nu(\varphi(t))|\mathbb{D}\mathbf{v}(t)|^2\,\mathrm{d}x
    +\int_\Gamma\beta(\psi(t))|\mathbf{v}_{\boldsymbol{\tau}}(t)|^2\,\mathrm{d}S
    &{}
    \\
    \quad+\int_\Omega |\nabla\mu(t)|^2\,\mathrm{d}x
    +\int_\Gamma |\nabla_{\!\boldsymbol{\tau}\,}\mathcal{L}(t)|^2\,\mathrm{d}S
    +\chi(L)\int_\Gamma|\mu(t)-\mathcal{L}(t)|^2\,\mathrm{d}S
    & = 0
    \label{energy-identity}
    \end{split}
\end{align}
for all $t\in(0,T)$, where $\chi$ is defined as in \eqref{DEF:CHI}.

\medskip

\noindent\textbf{Goals and main ideas.}
The goal of this work is to establish the global existence of a more regular (with respect to time) solution
to the system \eqref{eqmain0new}--\eqref{icnew} (see Theorem~\ref{quasi-strong}).
This {so-called quasi-strong solution} satisfies the energy equality in both two and three dimensions.
The proof is based on an appropriate semi-Galerkin approximation scheme. In the following, we outline some features of the problem and key ideas of the proof.
\begin{itemize}
    \item [(1)] \emph{Voigt regularization.}
The parameter $\alpha>0$ plays a significant role in the analysis, as the Voigt term $\mathrm{div}(2\alpha\mathbb{D}\partial_t\mathbf{v})$
provides estimates for $\partial_t\mathbf{v}$ in $\mathbf{H}^1(\Omega)$.
This enables us to improve the regularity of $\partial_t\mathbf{v}$ and to show the existence of a quasi-strong solution. More precisely, to obtain the $\mathbf{H}^1(\Omega)$-regularity of $\partial_t\mathbf{v}$,
we need to take $\mathbf{w}=\partial_t\mathbf{v}$ {as a test function for the velocity equation (cf.~\eqref{quasi-1'})}.
Then, there will {emerge three integrals over the boundary $\Gamma$} involving $\partial_t\mathbf{v}$, that is,
\begin{align*}
    \int_\Gamma \beta(\psi)\mathbf{v}_{\boldsymbol{\tau}}\cdot\partial_t\mathbf{v}_{\boldsymbol{\tau}}\,\mathrm{d}S
    +\int_\Gamma \psi\nabla_{\!\boldsymbol{\tau}\,}\mathcal{L}\cdot\partial_t\mathbf{v}_{\boldsymbol{\tau}}\,\mathrm{d}S
    -\frac{1}{2}\int_\Gamma (\mathbf{J}\cdot\mathbf{n})\mathbf{v}_{\boldsymbol{\tau}}\cdot\partial_t\mathbf{v}_{\boldsymbol{\tau}}\,\mathrm{d}S.
\end{align*}
Since $\alpha>0$, the term $\|\partial_t\mathbf{v}\|_{\mathbf{L}^4(\Gamma)}$ can be estimated using the Sobolev embedding theorem
and eventually be controlled by
$$
\int_\Omega \rho(\varphi)|\partial_t\mathbf{v}|^2\,\mathrm{d}x+2\alpha\int_\Omega |\mathbb{D}\partial_t\mathbf{v}|^2\,\mathrm{d}x.
$$
For more details, we refer to the terms $K_2$, $K_4$ and $K_{11}$ in \eqref{0227-20}.

\item [(2)] \emph{Construction of approximate solutions.} We employ a suitable approximation procedure involving a viscous regularization to the bulk-surface convective Cahn--Hilliard subsystem (see Appendix~\ref{appendix}),
together with an appropriate regularization of the initial data (see Proposition~\ref{approximate-initial-datum}).
Such approximations are crucial to rigorously justify higher-order Sobolev estimates for the corresponding approximate solutions. The most crucial step is to obtain the regularity
\begin{align*}
(\partial_t \mu, \partial_t \mathcal{L})\in L^2(0,T;\mathcal{H}^1),
\end{align*}
which allows us to rigorously compute \emph{a priori} estimates.
However, to derive this regularity property, we need some higher-order regularity of \(\partial_\mathbf{n}\varphi\).
When $K\in(0,+\infty]$, this can be derived via
the Robin/Neumann boundary condition $K\partial_\mathbf{n}\varphi=\psi-\varphi$ on $\Sigma_T$.
Therefore, in Appendix~\ref{appendix}, we shall focus on the case
\begin{align*}
(K,L)\in(0,+\infty]\times(0,+\infty),
\end{align*}
and establish the existence of approximate solutions in this case (see Section~\ref{approximating system for L}).
Then, by deriving suitable estimates that are uniform with respect to the approximate parameters,
we can deduce the existence of a global quasi-strong solution to problem \eqref{eqmain0new}--\eqref{icnew}
with $(K,L)\in(0,+\infty]\times(0,+\infty)$ using a compactness argument (see Section~\ref{existence of quasi}).
For the limit models corresponding to $(K,L)\in (0,+\infty]\times\{0,+\infty\}$,
the existence of a global quasi-strong solution can be obtained
by studying the corresponding asymptotic limits as $L\to0$ and $L\to+\infty$ (see Section~\ref{asy}).

\item [(3)] \emph{Motivation for an additional singular potential.}
For our analysis of the viscous approximation to the bulk-surface convective Cahn--Hilliard system in Appendix~\ref{appendix}, we introduce a further singular potential $F_\sigma$ (see \eqref{F_sigma}), which exhibits its singularities at $\pm(1-\sigma)$. It enforces a so-called strict separation property, which means that the values of the phase-field variables $\varphi$ and $\psi$ cannot exceed the interval $(-1+\sigma,1-\sigma)$. As the singularities of the potentials $F$ and $G$ lie at $\pm 1$, these potentials and their derivatives $f=F'$ and $g=G'$ can be interpreted as regular functions on $(-1+\sigma,1-\sigma)$.

In Appendix~\ref{appendix}, we prove the {aforementioned} strict separation property using a comparison argument.
To this end, we rewrite the convective viscous Cahn--Hilliard system \eqref{convective-viscous} as \eqref{re-1}
and consider the ODE system \eqref{ODE}.
Taking the difference between \eqref{re-1} and \eqref{ODE} and testing the resultant with a suitable test function,
we arrive at \eqref{0303-13}. By the monotonicity of $f_\sigma = F_\sigma'$,
the terms in the second line of \eqref{0303-13} are nonnegative.
If we had not introduced the additional function $f_\sigma$,
we would have obtained the following ODE system instead of \eqref{ODE}:
\begin{align}
    \begin{cases}
        \gamma\partial_t U+ f_0(U)= H,\\
        U(0)=1-\delta_0,
    \end{cases}\quad\quad	
    \begin{cases}
    \gamma\partial_t V+f_0(V)= -H,\\
    V(0)=-1+\delta_0.
    \end{cases}\notag
\end{align}
Then, the second line of \eqref{0303-13} would read as
\begin{align*}
    \underbrace{\int_\Omega (f_0(\varphi)-f_0(U))w^+\,\mathrm{d}x}_{\geq0}
    +\underbrace{\int_\Gamma (g_0(\psi)-f_0(U))w_\Gamma^+\,\mathrm{d}S}_{\text{No sign, unless }g_0=f_0}.
\end{align*}
Here, $f_0$ and $g_0$ denote the singular parts of the functions $f=F'$ and $g=G'$, respectively (see assumption~\ref{ASS:A2}).
Thus, the introduction of the singular potential $f_\sigma$ allows us to avoid the additional assumption $f_0=g_0$.
\end{itemize}

\noindent\textbf{Plan of this paper.}
In Section~\ref{pre}, we present some preliminaries for the analysis and state the main result of this paper.
Sections~\ref{approximating system for L} and~\ref{existence of quasi} are devoted to the proof of Theorem~\ref{quasi-strong} for the case $(K,L)\in(0,+\infty]\times(0,+\infty)$.
More precisely, the existence of approximate solutions is established in Section~\ref{approximating system for L},
while the existence of quasi-strong solutions to system \eqref{eqmain0new}--\eqref{icnew} with $(K,L)\in(0,+\infty]\times(0,+\infty)$ is shown in Section~\ref{existence of quasi}.
In Section~\ref{asy}, we complete the proof of Theorem~\ref{quasi-strong} by investigating the asymptotic limits as $L\to0$ and $L\to+\infty$ with $K\in(0,+\infty]$.
Finally, in Appendix~\ref{appendix}, we establish the well-posedness of the bulk-surface convective Cahn--Hilliard subsystem with viscous regularization.


\section{Preliminaries and Main Result}
\label{pre}
\setcounter{equation}{0}

\subsection{Preliminaries}
We write $\mathbf{a}\otimes \mathbf{b}=(a_{i}b_{j})_{i,j=1}^{d}$
for any vectors $\mathbf{a},\mathbf{b}\in \mathbb{R%
}^{d}$ and $A_{\mathrm{sym}}=\frac{1}{2}(A+A^\top)$ for any matrix $A\in
\mathbb{R}^{d\times d}$.
Let $X$ be a real Banach space.
We denote its norm by $\|\cdot\|_X$, its dual space by $X'$ and the duality pairing by
$\langle \cdot,\cdot\rangle_{X',X}$. If $X$ is a Hilbert space, its inner product will be denoted by $(\cdot,\cdot)_X$.
The notation $X\hookrightarrow Y$ means that $X$ is continuously
embedded in $Y$. We further denote by $\mathbf{X}$ the generic space of vectors or matrices,
with each component belonging to $X$. The space $L^q(0,T;X)$ with $1\leq q\leq +\infty$ denotes the Bochner space consisting of all strongly measurable $q$-integrable functions with values in $X$, or, if $q=+\infty$, essentially bounded
functions. Besides, the space $C([ 0,T] ;X)$ denotes the Banach
space of all bounded and continuous functions $u:[0,T]
\rightarrow X$ equipped with the supremum norm. For $q\in [1,+\infty]$, $W^{1,q}(0,T;X)$ denotes the space of functions $f$
such that $f\in L^q(0,T;X)$ with $\partial_t f\in L^q(0,T;X)$,
where $\partial_t$ denotes the vector-valued distributional derivative of $f$ with respect to $t$.
For $q=2$, we set $H^1(0,T;X)=W^{1,2}(0,T;X)$.


Let $\Omega \subset \mathbb{R}^{d}$ ($d\in\{2,3\}$) be a bounded domain with smooth boundary $\Gamma =\partial \Omega$. The Lebesgue spaces in $\Omega$ and on $\Gamma$ are denoted by $L^{q}(\Omega )$,
$L^{q}(\Gamma )$ ($1\leq q\leq+ \infty$) with norms $\|\cdot \|_{L^{q}(\Omega)}$
and $\|\cdot \|_{L^{q}( \Gamma) }$,
respectively. For $s\geq 0$ and $q\in [1,+\infty )$,
we denote by $H^{s,q}(\Omega )$ the Bessel-potential
spaces and by $W^{s,q}(\Omega )$ the Slobodeckij spaces. It holds $H^{s,2}(\Omega )=W^{s,2}(\Omega )$ for all $s$, but for $q\neq 2$ the identity $H^{s,q}(\Omega )=W^{s,q}(\Omega )$ is only true if $s\in \mathbb{N}$. For clarity, we denote by $\mathbb{N}$ the set of natural numbers including zero.
If $s\in \mathbb{N}$, then $H^{s,q}(\Omega )$ and $W^{s,q}(\Omega
)$ coincide with the usual Sobolev spaces. The corresponding function spaces over the boundary $\Gamma$ are defined via local charts. Let $\Upsilon_{i}:U_{i}\subset \mathbb{R}^{d-1}\rightarrow \Gamma $ be a finite family of parametrizations such that $\bigcup_{i}\Upsilon _{i}(U_{i})$ covers $\Gamma $, and let $\{\psi _{i}\}$ be a partition of unity for $%
\Gamma $ subordinate to this cover. Then for $s\geq 0$, we have
\begin{equation*}
H^{s,q}(\Gamma )=\big \{u\in L^{q}(\Gamma )\,:\,(\psi _{i}u)\circ \Upsilon
_{i}\in H^{s,q}(\mathbb{R}^{d-1})\text{ for all $i$}\big\},
\end{equation*}
with an equivalent norm given by $\| u\|_{H^{s,q}(\Gamma
)}=\sum_{i}\Vert (\psi _{i}u)\circ \Upsilon _{i}\Vert _{H^{s,q}(\mathbb{R}%
^{d-1})}.$ The spaces $W^{s,q}(\Gamma )$ are defined similarly.
The properties of the spaces over $\Omega$ described above
can easily be transferred to the spaces over $\Gamma$.


If $q=2$, $H^{s,2}(\Omega)=W^{s,2}(\Omega )$ for all $s$ and these spaces are Hilbert spaces. Throughout this paper, we use the notation $H^s(\Omega)=H^{s,2}(\Omega)=W^{s,2}(\Omega )$ and $H^0(\Omega)$ is identified with $L^2(\Omega)$.
The Lebesgue spaces, Sobolev spaces and Slobodeckij spaces on the boundary $\Gamma$ can be defined analogously,
provided that $\Gamma$ is sufficiently regular.
We write $H^s(\Gamma)=H^{s,2}(\Gamma)=W^{s,2}(\Gamma)$ and identify $H^0(\Gamma)$ with $L^2(\Gamma)$.
Hereafter, the following shortcuts will be applied:
\begin{align*}
  H \coloneqq L^{2}(\Omega),\quad H_{\Gamma} \coloneqq L^{2}(\Gamma),\quad V \coloneqq H^{1}(\Omega),\quad V_{\Gamma} \coloneqq H^{1}(\Gamma).
\end{align*}
For every $y\in V'$, we denote by $\langle y\rangle_\Omega=|\Omega|^{-1}\langle
y,1\rangle_{V',V}$ its generalized mean
value over $\Omega$. If in addition $y\in L^1(\Omega)$, its spatial mean can be expressed as $\langle y\rangle_\Omega=|\Omega|^{-1}\int_\Omega y \,\mathrm{d}x$.
The spatial mean for {any $y_\Gamma \in V_\Gamma'$} on $\Gamma$, denoted by $\langle
y_\Gamma\rangle_\Gamma$, is defined analogously.
Furthermore, we introduce the following subspaces of functions with zero mean:
\begin{alignat*}{2}
&V_{0} \coloneqq \big\{y\in V:\;\langle y\rangle_\Omega =0\big\},
&&\qquad V_{0}^* \coloneqq \big\{y^*\in V':\;\langle y^*\rangle_\Omega =0\big\},
 \\
&V_{\Gamma,0} \coloneqq \big\{y_{\Gamma}\in V_{\Gamma}:\; \langle y_\Gamma\rangle_\Gamma =0\big\},
&&\qquad V_{\Gamma,0}^* \coloneqq \big\{y_\Gamma^* \in V_{\Gamma}':\;\langle y_\Gamma^*\rangle_\Gamma =0\big\}.
\end{alignat*}
We further recall the Poincar\'{e}--Wirtinger inequalities in $\Omega$ and on $\Gamma$
(see, e.g., \cite[Theorem 2.12]{DE13} for the case on $\Gamma$):
\begin{alignat*}{2}
\|u-\langle u\rangle_\Omega\|_{H} &\leq C_\Omega\|\nabla u\|_{\mathbf{L}^2(\Omega)}
&&\quad \text{for all $u\in V$,}
\\
\|u_\Gamma-\langle u_\Gamma\rangle_\Gamma\|_{H_\Gamma}
&\leq C_\Gamma\|{\nabla_{\!\boldsymbol{\tau}\,}} u_\Gamma\|_{\mathbf{L}^2(\Gamma)}
&&\quad\text{for all $u_\Gamma\in V_\Gamma$.}
\end{alignat*}
Here, $C_\Omega$ (resp. $C_\Gamma$) is a positive constant depending only on $\Omega$ (resp. $\Gamma$).
Moreover, the following Poincar\'e-type inequality (see, e.g., \cite[Chapter II, Section 1.4]{Te97}) holds:
\begin{align}
    \|u\|_H\leq \widetilde{C}_\mathrm{P}(\|\nabla u\|_{\mathbf{L}^2(\Omega)}+\|u\|_{H_\Gamma})
    \quad\text{for all $u\in V$}.\label{bsP}
\end{align}
Here, the constant $\widetilde{C}_\mathrm{P}>0$ depends only on $\Omega$.

Let us now consider the Poisson--Neumann problem
\begin{align}
\left\{
\begin{aligned}
    -\Delta u &= y &&\quad \text{in} \ \Omega,\\
    \partial_{\mathbf{n}} u &=0 &&\quad \text{on}\ \Gamma.
\end{aligned}
\right.
\label{P-n}
\end{align}
By the Lax--Milgram theorem, for every $y\in V_0^*$,
problem \eqref{P-n} admits a unique weak solution $u\in V_{0}$. {This means that $u$ satisfies}
\begin{align}
\int_{\Omega}\nabla u\cdot\nabla \zeta\,\mathrm{d}x
= \langle y,\zeta \rangle_{V',V}\quad\text{for all $\zeta\in V$}.
\notag
\end{align}
Thus, we can define the solution operator
$\mathcal{N}_{\Omega}: V_{0}^*\rightarrow V_{0}$ such that $u=\mathcal{N}_{\Omega}y$.
Furthermore, we consider the surface Poisson equation
\begin{align}
    {-\Delta_{\boldsymbol{\tau}}} u_\Gamma = y_\Gamma\quad \text{on}\ \ \Gamma,
\label{P-s}
\end{align}
where $\Delta_{\boldsymbol{\tau}} = \mathrm{div}_{\Gamma} \nabla_{\!\boldsymbol{\tau}\,}$ is the Laplace--Beltrami operator on $\Gamma$.
For every $y_\Gamma\in V_{\Gamma,0}^*$, problem \eqref{P-s} admits a unique weak solution $u_\Gamma\in V_{\Gamma,0}$ satisfying
\begin{align}
\int_{\Gamma}{\nabla_{\!\boldsymbol{\tau}}\,} u_{\Gamma}\cdot{\nabla_{\!\boldsymbol{\tau}\,}} \zeta_{\Gamma}\,\mathrm{d}S
= \langle y_{\Gamma},\zeta_{\Gamma} \rangle_{V_{\Gamma}',V_{\Gamma}}
\quad\text{for all $\zeta_{\Gamma}\in V_{\Gamma}$.}\notag
\end{align}
Hence, we can define the solution operator $\mathcal{N}_{\Gamma}:V_{\Gamma,0}^*\rightarrow V_{\Gamma,0}$
such that $u_{\Gamma}=\mathcal{N}_{\Gamma}y_{\Gamma}$.
By virtue of these definitions, we introduce the following equivalent norms:
\begin{alignat*}{2}
\Vert y\Vert_{V_{0}^*}
&\coloneqq \Big(\int_{\Omega}|\nabla\mathcal{N}_{\Omega}y|^{2}\,\mathrm{d}x\Big)^{1/2}
&&\quad\text{for all $y\in V_{0}^*$},
\\
\Vert y\Vert_{V'}
&\coloneqq \Big(\Vert y - \langle y\rangle_\Omega \Vert_{V_{0}^*}^2+ |\langle y\rangle_\Omega|^2\Big)^{1/2}
&&\quad\text{for all $y\in V'$},
\\
\Vert y_{\Gamma}\Vert_{V_{\Gamma,0}^*}
&\coloneqq \Big(\int_{\Gamma}|{\nabla_{\!\boldsymbol{\tau}\,}}\mathcal{N}_{\Gamma}y_{\Gamma}|^{2}\,\mathrm{d}S\Big)^{1/2}
&&\quad\text{for all $y_{\Gamma}\in V_{\Gamma,0}^*$},
\\
\Vert y_{\Gamma}\Vert_{V_{\Gamma}'}
&\coloneqq \Big(\Vert y_{\Gamma}- \langle y_{\Gamma}\rangle_\Gamma \Vert_{V_{\Gamma,0}^*}^2
+ |\langle y_{\Gamma}\rangle_\Gamma|^2\Big)^{1/2}
&&\quad\text{for all $y_{\Gamma}\in V_{\Gamma}'$}.
\end{alignat*}


Next, for $q\in [1,+\infty]$ and $k\in \mathbb{N}$, we introduce the product spaces
\begin{align*}
    \mathcal{L}^{q} \coloneqq L^{q}(\Omega)\times L^{q}(\Gamma)\quad\mathrm{and}
    \quad\mathcal{W}^{k,q} \coloneqq W^{k,q}(\Omega)\times W^{k,q}(\Gamma).
\end{align*}
In particular, we use the notation
\begin{align*}
    \mathcal{H}^k \coloneqq H^k(\Omega)\times H^k(\Gamma)=\mathcal{W}^{k,2}.
\end{align*}
As before, we identify  $\mathcal{H}^{0}$ with $\mathcal{L}^{2}$.
For any $k\in \mathbb{N}$, $\mathcal{H}^{k}$ is a Hilbert space endowed with the standard inner product
\begin{align*}
  ((y,y_{\Gamma}),(z,z_{\Gamma}))_{\mathcal{H}^{k}} \coloneqq (y,z)_{H^{k}(\Omega)}+(y_{\Gamma},z_{\Gamma})_{H^{k}(\Gamma)}
  \quad\text{for all $(y,y_{\Gamma}), (z,z_{\Gamma})\in\mathcal{H}^{k}$,}
\end{align*}
and the induced norm $\Vert\cdot\Vert_{\mathcal{H}^{k}} \coloneqq (\cdot,\cdot)_{\mathcal{H}^{k}}^{1/2}$.
We further introduce the bilinear form
\begin{align*}
  \langle (y,y_\Gamma),(\zeta, \zeta_\Gamma)\rangle_{(\mathcal{H}^1)',\mathcal{H}^1}
  \coloneqq (y,\zeta)_{L^2(\Omega)}+ (y_\Gamma, \zeta_\Gamma)_{L^2(\Gamma)}
  \quad \text{for all $(y,y_\Gamma)\in \mathcal{L}^2,\ (\zeta, \zeta_\Gamma)\in \mathcal{H}^1$.}
\end{align*}
By the Riesz representation theorem, this product
can be extended to a duality pairing on $(\mathcal{H}^1)'\times \mathcal{H}^1$.


For any $k\in\mathbb{N}^+$, we introduce the Hilbert space
\begin{align*}
  \mathcal{V}^{k} \coloneqq \big\{(y,y_{\Gamma})\in\mathcal{H}^{k}\;:\;\text{tr}(y)=y_{\Gamma}\ \ \text{a.e.~on }\Gamma\big\},
\end{align*}
endowed with the inner product $(\cdot,\cdot)_{\mathcal{V}^{k}} \coloneqq (\cdot,\cdot)_{\mathcal{H}^{k}}$
and the associated norm $\Vert\cdot\Vert_{\mathcal{V}^{k}} \coloneqq \Vert\cdot\Vert_{\mathcal{H}^{k}}$.
Here, $\text{tr}(y)$ stands for the trace of $y\in H^k(\Omega)$ on the boundary $\Gamma$,
which makes sense for $k\in \mathbb{N}^+$.
The duality pairing on $(\mathcal{V}^1)'\times \mathcal{V}^1$ can be defined analogously to the one on $(\mathcal{H}^1)'\times \mathcal{H}^1$.
Besides, for $K\in[0,+\infty]$, we introduce the space
\begin{align}
    \label{DEF:W2K}
    \mathcal{W}_{K}^2 \coloneqq
    \begin{cases}
        \mathcal{V}^2
        &\text{if }K=0,
        \\
        \{(y,y_\Gamma)\in\mathcal{H}^2:K\partial_\mathbf{n}y=y_\Gamma-y\}
        &\text{if }K\in(0,+\infty),
        \\
        \{(y,y_\Gamma)\in\mathcal{H}^2:\partial_\mathbf{n}y=0\}
        &\text{if }K=+\infty.
    \end{cases}
\end{align}


For any given $m\in\mathbb{R}$, we set
\begin{align*}
 \mathcal{L}^{2}_{(m)} \coloneqq \big\{(y,y_{\Gamma})\in\mathcal{L}^{2}\;:\;\overline{m}(y,y_\Gamma)=m\big\},
\end{align*}
where the generalized bulk-surface mean is defined by
\begin{align}
     \overline{m}(y,y_\Gamma) \coloneqq \frac{|\Omega|\langle y\rangle_\Omega
     +|\Gamma|\langle y_\Gamma\rangle_\Gamma}{|\Omega|+|\Gamma|}.\label{generalized-mean}
\end{align}
Furthermore, we define the projection operator
\begin{align}
    \label{DEF:PRO:L0}
    \mathbb{P}:\mathcal{L}^2\to\mathcal{L}_{(0)}^2 \,, \quad
    \mathbb{P}(y,y_\Gamma)=\big( y-\overline{m}(y,y_\Gamma),y_\Gamma-\overline{m}(y,y_\Gamma) \big).
\end{align}
For any $k\in \mathbb{N}^+$, the closed linear subspaces
\begin{align*}
    \mathcal{H}_{(0)}^k
    \coloneqq \mathcal{H}^{k}\cap\mathcal{L}_{(0)}^{2},
    \qquad
    \mathcal{V}_{(0)}^k
    \coloneqq \mathcal{V}^{k}\cap\mathcal{L}_{(0)}^{2},
\end{align*}
both endowed with the inner product $(\cdot,\cdot)_{\mathcal{H}^{k}}$
and the induced norm $\Vert\cdot\Vert_{\mathcal{H}^{k}}$, are Hilbert spaces. For $L\in [0,+\infty)$ and $k\in \mathbb{N}^+$,  we introduce the function spaces
\begin{align*}
  \mathcal{H}^{k}_{L} \coloneqq
  \begin{cases}
  \mathcal{H}^k,& \text{if}\ L\in (0,+\infty),\\
  \mathcal{V}^{k},&  \text{if}\ L=0,
  \end{cases}\qquad
  \mathcal{H}^{k}_{L,0} \coloneqq
  \begin{cases}
  \mathcal{H}_{(0)}^k,& \text{if}\ L\in (0,+\infty),\\
  \mathcal{V}^{k}_{(0)},&  \text{if}\ L=0.
  \end{cases}
\end{align*}
We further consider the bilinear form
\begin{align}
  &a_{L}((y,y_{\Gamma}),(z,z_{\Gamma}))  \coloneqq
  \int_{\Omega}\nabla y\cdot\nabla z \,\mathrm{d}x
  +\int_{\Gamma}{\nabla_{\!\boldsymbol{\tau}\,}}y_{\Gamma}\cdot{\nabla_{\!\boldsymbol{\tau}\,}}z_{\Gamma}\,\mathrm{d}S
  +\chi(L)\int_{\Gamma}(y-y_{\Gamma})(z-z_{\Gamma})\,\mathrm{d}S,\notag
\end{align}
for all $(y,y_{\Gamma}), (z,z_{\Gamma})\in \mathcal{H}^{1}$.
For $L\in [0,+\infty)$, we define the inner product on $\mathcal{H}_{L,0}^1$
by $(\cdot,\cdot)_{\mathcal{H}_{L,0}^1}=a_L(\cdot,\cdot)$,
and for any $(y,y_{\Gamma})\in \mathcal{H}^{1}_{L,0}$, we define the norm
\begin{align}
  \Vert(y,y_{\Gamma})\Vert_{\mathcal{H}^{1}_{L,0}} \coloneqq ((y,y_{\Gamma}),(y,y_{\Gamma}))_{\mathcal{H}^{1}_{L,0}}^{1/2}
  = \big[ a_{L}((y,y_{\Gamma}),(y,y_{\Gamma})) \big]^{1/2}.
  \label{norm-hL}
\end{align}
For $(y,y_{\Gamma})\in \mathcal{V}_{(0)}^1\subseteq\mathcal{H}^{1}_{L,0}$,
$\Vert(y,y_{\Gamma})\Vert_{\mathcal{H}^{1}_{L,0}}$ does not depend on $L$,
since the third term in $a_L$ vanishes.
In addition, the following Poincar\'{e}-type inequality has been proved in \cite[Lemma A.1]{KL}:
\begin{lemma}
\label{generalized-Poin}
There exists a constant $C_\mathrm{P}>0$ depending only on $L\in [0,+\infty)$ and $\Omega$ such that
\begin{align}
  \|(y,y_{\Gamma})\|_{\mathcal{L}^2}
  \leq C_\mathrm{P} \Vert(y,y_{\Gamma})\Vert_{\mathcal{H}^{1}_{L,0}}
  \quad\text{for all $(y,y_{\Gamma})\in \mathcal{H}^{1}_{L,0}$}.
  \notag
\end{align}
\end{lemma}

\noindent
This shows that for every $L\in [0,+\infty)$, $\mathcal{H}^{1}_{L,0}$ is a Hilbert space
with the inner product $(\cdot,\cdot)_{\mathcal{H}^{1}_{L,0}}$.
The induced norm $\Vert\cdot\Vert_{\mathcal{H}^{1}_{L,0}}$ introduced in \eqref{norm-hL} is equivalent to the standard one $\Vert\cdot\Vert_{\mathcal{H}^{1}}$ on $\mathcal{H}^{1}_{L,0}$.

For $L\in[0,+\infty)$, let us now consider the following elliptic boundary value problem
\begin{align}
    \left\{
    \begin{aligned}
        -\Delta u &= y
        &&\quad \text{in }\Omega,\\
        - \Delta_{\boldsymbol{\tau}}u_\Gamma +\partial_{\mathbf{n}}u&= y_{\Gamma}
        &&\quad \text{on }\Gamma,\\
        L\partial_{\mathbf{n}}u &=u_\Gamma-u
        &&\quad \mathrm{on\;}\Gamma.
    \end{aligned}
    \right.\label{2.2}
\end{align}
We define the space
\begin{align*}
     \mathcal{H}_{L,0}^{-1}=
     \begin{cases}
         \mathcal{H}_{(0)}^{-1} \coloneqq \{(y,y_\Gamma)\in(\mathcal{H}^1)':\,\overline{m}(y,y_\Gamma)=0\},&\text{if }L\in(0,+\infty),\\
         \mathcal{V}_{(0)}^{-1} \coloneqq \{(y,y_\Gamma)\in(\mathcal{V}^1)':\,\overline{m}(y,y_\Gamma)=0\},&\text{if }L=0,
     \end{cases}
\end{align*}
where $\overline{m}$ is given by \eqref{generalized-mean} if $L\in(0,+\infty)$ and for $L=0$, it is defined as
\begin{align*}
\overline{m}(y,y_\Gamma)
\coloneqq \frac{\langle(y,y_\Gamma),(1,1)\rangle_{(\mathcal{V}^1)',\mathcal{V}^1}}{|\Omega|+|\Gamma|}.
\end{align*}
In particular, we have the chain of inclusions
\begin{align*}
   \mathcal{H}^{1}_{L,0}
   \subseteq \mathcal{L}_{(0)}^2
   \subseteq \mathcal{H}_{L,0}^{-1}
   \subseteq (\mathcal{H}_L^1)'.
\end{align*}
It has been shown in \cite[Theorem 3.3]{KL} that for every $(y,y_{\Gamma})\in\mathcal{H}_{(0)}^{-1}$,
problem \eqref{2.2} admits a unique weak solution $(u,u_\Gamma)\in\mathcal{H}_{L,0}^{1}$
satisfying the weak formulation
\begin{align}
      a_L\big((u,u_{\Gamma}),(\zeta,\zeta_{\Gamma})\big)
      = \big\langle(y,y_{\Gamma}),(\zeta,\zeta_{\Gamma})\big\rangle_{(\mathcal{H}_L^{1})',\mathcal{H}_L^{1}}
      \quad \text{for all $(\zeta,\zeta_{\Gamma})\in\mathcal{H}_L^{1}$},\notag
\end{align}
and the estimate
\begin{align}
      \|(u,u_{\Gamma})\|_{\mathcal{H}^1}\leq C\|(y,y_{\Gamma})\|_{(\mathcal{H}^{1}_L)'},\notag
\end{align}
for some constant $C>0$ depending only on $L$ and $\Omega$.
Furthermore, if the domain $\Omega$ is of class $C^{k+2}$ and $(y,y_{\Gamma})\in \mathcal{H}_{L,0}^{k}$ for some $k\in \mathbb{N}$, it was further shown in \cite[Theorem 3.3]{KL} that
then $(u,u_\Gamma)\in \mathcal{H}_{L}^{k+2}$ and the following regularity estimate holds
\begin{align}
      \|(u,u_{\Gamma})\|_{\mathcal{H}^{k+2}}\leq C\|(y,y_{\Gamma})\|_{\mathcal{H}^{k}}.
      \label{Hk-regularity}
\end{align}
The above facts enable us to define the solution operator
\begin{align*}
\mathfrak{S}^{L}:\mathcal{H}_{(0)}^{-1}\rightarrow\mathcal{H}_{L,0}^{1},
\quad(y,y_{\Gamma})\mapsto \mathfrak{S}^{L}(y,y_{\Gamma})
= \big(\mathfrak{S}^{L}_{\Omega}(y,y_{\Gamma}),\mathfrak{S}^{L}_{\Gamma}(y,y_{\Gamma})\big)
= (u,u_\Gamma).
\end{align*}
For the case $L=0$, similar results have been established in \cite{CF15}.
A direct calculation yields
\begin{align*}
      \big( (u,u_{\Gamma}), (z,z_\Gamma) \big)_{\mathcal{L}^2}
      =\big( (u,u_{\Gamma}), \mathfrak{S}^{L}(z,z_\Gamma) \big)_{\mathcal{H}^1_{L,0}}
      \quad \text{for all $(u,u_{\Gamma})\in \mathcal{H}_{(0)}^1,\ (z,z_\Gamma)\in \mathcal{L}^2_{(0)}$.}
\end{align*}
Thanks to \cite[Corollary 3.5]{KL}, the bilinear form
\begin{align}
      \big( (y,y_{\Gamma}),(z,z_{\Gamma}) \big)_{\mathcal{H}_{L,0}^{-1}}&
       \coloneqq \big( \mathfrak{S}^{L}(y,y_{\Gamma}),\mathfrak{S}^{L}(z,z_{\Gamma}) \big)_{\mathcal{H}^{1}_{L,0}}
      \quad \text{for all $(y,y_{\Gamma}), (z,z_{\Gamma})\in \mathcal{H}_{L,0}^{-1}$}
      \notag
\end{align}
defines an inner product on $\mathcal{H}_{L,0}^{-1}$.
The associated norm $\Vert(y,y_{\Gamma})\Vert_{\mathcal{H}_{L,0}^{-1}}  \coloneqq \big((y,y_{\Gamma}),(y,y_{\Gamma})\big)_{\mathcal{H}_{L,0}^{-1}}^{1/2}$
is equivalent to the standard dual norm {$\|\cdot\|_{(\mathcal{H}_{L,0}^1)'} = \|\cdot\|_{(\mathcal{H}_{L}^1)'}$} on $\mathcal{H}_{L,0}^{-1}$.
Then it follows that
\begin{align}
      \|(y,y_{\Gamma})\|_{*}& \coloneqq \Big(\Vert(y,y_{\Gamma})-\overline{m}(y,y_{\Gamma}) \mathbf{1}\Vert_{\mathcal{H}_{L,0}^{-1}}^2
      + |\overline{m}(y,y_{\Gamma})|^2\Big)^{1/2}
      \quad \text{for all $(y,y_{\Gamma})\in (\mathcal{H}^{1})'$,}
      \notag
\end{align}
is equivalent to the usual dual norm $\|\cdot\|_{(\mathcal{H}_L^1)'}$ on $(\mathcal{H}^{1})'$.

Concerning the velocity fields, we introduce the space
$\boldsymbol{\mathcal{L}}_{\mathrm{div}}^2=\mathbf{L}_{\mathrm{div}}^2(\Omega)\times\mathbf{L}_{\mathrm{div}}^2(\Gamma)$,
where
\begin{align*}
    &\mathbf{L}_{\mathrm{div}}^2(\Omega) \coloneqq \{\mathbf{v}\in \mathbf{L}^2(\Omega):
    \,\mathrm{div}\,\mathbf{v}=0\ \text{in }\Omega,\ \mathbf{v}\cdot\mathbf{n}=0\ \text{on }\Gamma\},\\
    &\mathbf{L}_{\mathrm{div}}^2(\Gamma) \coloneqq \{\mathbf{w}\in \mathbf{L}^2(\Gamma):
    \, \mathrm{div}_{\Gamma}\,\mathbf{w}=0,\ \mathbf{w}\cdot\mathbf{n}=0\ \text{on }\Gamma\}.
\end{align*}
Then, for $s\geq0$ and $p\in[1,+\infty]$, we define
$\mathbf{W}^{s,p}_{\mathrm{div}}(\Omega)=\mathbf{W}^{s,p}(\Omega)\cap \mathbf{L}_{\mathrm{div}}^2(\Omega)$.
Analogously, we define
$\mathbf{W}^{s,p}_{\mathrm{div}}(\Gamma)=\mathbf{W}^{s,p}(\Gamma)\cap \mathbf{L}_{\mathrm{div}}^2(\Gamma)$,
and then set $\boldsymbol{\mathcal{W}}^{s,p} =\mathbf{W}^{s,p} (\Omega)\times\mathbf{W}^{s,p}(\Gamma)$,
$\boldsymbol{\mathcal{W}}^{s,p}_{\mathrm{div}}=\mathbf{W}^{s,p}_{\mathrm{div}}(\Omega)\times\mathbf{W}^{s,p}_{\mathrm{div}}(\Gamma)$.
Furthermore, {for $s> 1/p$}, we set
\begin{align*}
    & \boldsymbol{\mathcal{W}}^{s,p}_0
    \coloneqq \big\{(\mathbf{v},\mathbf{w})\in \boldsymbol{\mathcal{W}}^{s,p}:
    \, \mathbf{v}\cdot\mathbf{n}=0,\ \mathbf{v}|_\Gamma=\mathbf{w}\ \text{on }\Gamma
    \big\},
    \quad \boldsymbol{\mathcal{W}}^{s,p}_{0,\mathrm{div}}
 \coloneqq \boldsymbol{\mathcal{W}}^{s,p}_0\cap \boldsymbol{\mathcal{W}}^{s,p}_{\mathrm{div}}.
\end{align*}
As before, we use the notation $\mathbf{H}^s_{\mathrm{div}}(\Omega)= \mathbf{W}^{s,2}_{\mathrm{div}}(\Omega)$,
$\boldsymbol{\mathcal{H}}^s = \boldsymbol{\mathcal{W}}^{s,2}$,
$\boldsymbol{\mathcal{H}}^s_{0,\mathrm{div}} =
\boldsymbol{\mathcal{W}}_{0,\mathrm{div}}^{s,2}$.

Next, we consider the Helmholtz projection
\begin{align*}
    \mathbf{P}_{\mathrm{div}}^{\Omega}:\mathbf{L}^2(\Omega)\to\mathbf{L}_{\mathrm{div}}^2(\Omega),
    \quad
    \mathbf{P}_{\mathrm{div}}^\Omega(\mathbf{v})=\mathbf{v}-\nabla p ,
\end{align*}
where $p\in V_0$ is the unique mean-free solution of the Poisson--Neumann problem
\begin{align*}
    \int_\Omega \nabla p\cdot\nabla\zeta\,\mathrm{d}x
    =\int_\Omega \mathbf{v}\cdot \nabla\zeta\,\mathrm{d}x\quad \text{for all}\ \zeta\in V.
\end{align*}
In a similar fashion, we define the Helmholtz projection
\begin{align*}
    \mathbf{P}_{\mathrm{div}}^{\Gamma}:\mathbf{L}^2(\Gamma)\to\mathbf{L}_{\mathrm{div}}^2(\Gamma),
    \quad
    \mathbf{P}_{\mathrm{div}}^\Gamma(\mathbf{w})=\mathbf{w}-\nabla_{\!\boldsymbol{\tau}\,} q
\end{align*}
where $q\in V_{\Gamma,0}$ is the unique mean-free solution of
\begin{align*}
\int_\Gamma \nabla_{\!\boldsymbol{\tau}\,} q\cdot\nabla_{\!\boldsymbol{\tau}\,}\xi\,\mathrm{d}S
=\int_\Gamma \mathbf{w}\cdot\nabla_{\!\boldsymbol{\tau}\,}\xi\,\mathrm{d}S\quad \text{for all}\ \xi\in V_\Gamma.
\end{align*}
We also set
$\mathcal{\boldsymbol{P}}_{\mathrm{div}}:\boldsymbol{\mathcal{L}}^2\to\boldsymbol{\mathcal{L}}^2_{\mathrm{div}}$,
$\mathcal{\boldsymbol{P}}_{\mathrm{div}} \coloneqq (\mathbf{P}_{\mathrm{div}}^{\Omega},\mathbf{P}_{\mathrm{div}}^{\Gamma})$. For more details, see \cite{KS}.

Furthermore, for sufficiently regular vector fields $\mathbf{v}:\Omega \to \mathbb R^d$ and
$\mathbf{w}:\Gamma \to \mathbb R^d$,
we denote by
\begin{align*}
    \mathbb{D}\mathbf{v}=\frac{\nabla\mathbf{v}+(\nabla\mathbf{v})^{\top}}{2},
    \qquad\mathbb{D}_{\Gamma}\mathbf{w}=\frac{{\nabla_{\!\boldsymbol{\tau}\,}}\mathbf{w}
    +({\nabla_{\!\boldsymbol{\tau}\,}}\mathbf{w})^{\top}}{2}
\end{align*}
the symmetric gradient of $\mathbf{v}$ and the symmetric surface gradient of $\mathbf{w}$, respectively. Then we report the following Korn-type inequality, which can be found in \cite[Appendix]{Ab3} (see also \cite{HWW2014} for a variant version):
\begin{align}
    \|\nabla\mathbf{v}\|_{\mathbf{L}^2(\Omega)}
    \leq \widetilde{C}_\mathrm{K}
    \big( \|\mathbb{D}\mathbf{v}\|_{\mathbf{L}^2(\Omega)}
        + \|\mathbf{v}\|_{\mathbf{L}^2(\Gamma)} \big)
    \quad\text{for all }\mathbf{v}\in \mathbf{H}_{\mathrm{div}}^1(\Omega),\label{Korn}
\end{align}
where the constant $\widetilde{C}_\mathrm{K} >0$ depends only on $\Omega$.
The following lemma is a corollary of \eqref{Korn} and will be frequently used in the subsequent analysis.
\begin{lemma}
\label{Korn-lemma}
    The exists a positive constant $\widetilde{C}_0$ depending only on $\Omega$ such that for all $\mathbf{v}\in \mathbf{H}^1_{\mathrm{div}}(\Omega)$, we have     \begin{align}
        &\|\mathbf{v}\|_{\mathbf{L}^2(\Gamma)}\leq \widetilde{C}_0(\|\mathbb{D}\mathbf{v}\|_{\mathbf{L}^2(\Omega)}+\|\mathbf{v}\|_{\mathbf{L}^2(\Omega)}),
        \label{v-trace}\\
        &\|\mathbf{v}\|_{\mathbf{H}^1(\Omega)}\leq (1+\widetilde{C}_{\mathrm{K}}+\widetilde{C}_{\mathrm{K}}\widetilde{C}_0) \big( \|\mathbb{D}\mathbf{v}\|_{\mathbf{L}^2(\Omega)}
        + \|\mathbf{v}\|_{\mathbf{L}^2(\Omega)} \big).\label{Korn-cor}
    \end{align}
\end{lemma}
\begin{proof}
Invoking a trace estimate presented in \cite[Theorem~II.4.1]{Galdi}, we have
\begin{align}
    \|u\|_{H_\Gamma}\leq C_{\mathrm{in}}\|u\|_{H}^\frac{1}{2} \|u\|_{V}^{\frac{1}{2}}\quad\text{for all}\  u\in V,\notag
\end{align}
for some constant $C_{\mathrm{in}}>0$ depending only on $\Omega$. Using this estimate together with \eqref{Korn}, we infer that
\begin{align}
    \|\mathbf{v}\|_{\mathbf{L}^2(\Gamma)}^2&\leq C_{\mathrm{in}}^2
    \|\mathbf{v}\|_{\mathbf{L}^2(\Omega)} \|\mathbf{v}\|_{\mathbf{H}^1(\Omega)}\notag\\
    &\leq  C_{\mathrm{in}}^2\|\mathbf{v}\|_{\mathbf{L}^2(\Omega)}^2 +C_{\mathrm{in}}^2
    \|\mathbf{v}\|_{\mathbf{L}^2(\Omega)} \|\nabla\mathbf{v}\|_{\mathbf{L}^2(\Omega)}  \notag\\
    &\leq C_{\mathrm{in}}^2\|\mathbf{v}\|_{\mathbf{L}^2(\Omega)}^2+C_{\mathrm{in}}^2\widetilde{C}_{\mathrm{K}}
    \|\mathbf{v}\|_{\mathbf{L}^2(\Omega)}(\|\mathbb{D}\mathbf{v}\|_{\mathbf{L}^2(\Omega)}+ \|\mathbf{v}\|_{\mathbf{L}^2(\Gamma)}) \notag\\
    &\leq \frac{1}{2} \|\mathbf{v}\|_{\mathbf{L}^2(\Gamma)}^2+ \frac{1}{2}C_{\mathrm{in}}^2\widetilde{C}_{\mathrm{K}}\|\mathbb{D}\mathbf{v}\|_{\mathbf{L}^2(\Omega)}^2+\Big(C_{\mathrm{in}}^2+\frac{1}{2}C_{\mathrm{in}}^2\widetilde{C}_{\mathrm{K}}+\frac{1}{2}C_{\mathrm{in}}^4\widetilde{C}_{\mathrm{K}}^2\Big)\|\mathbf{v}\|_{\mathbf{L}^2(\Omega)}^2.\notag
    \end{align}
    As a consequence, it holds
    \begin{align}
     \|\mathbf{v}\|_{\mathbf{L}^2(\Gamma)}^2&\leq C_{\mathrm{in}}^2\widetilde{C}_{\mathrm{K}}\|\mathbb{D}\mathbf{v}\|_{\mathbf{L}^2(\Omega)}^2+\Big(2C_{\mathrm{in}}^2+C_{\mathrm{in}}^2\widetilde{C}_{\mathrm{K}}+C_{\mathrm{in}}^4\widetilde{C}_{\mathrm{K}}^2\Big)\|\mathbf{v}\|_{\mathbf{L}^2(\Omega)}^2\notag\\
     &\leq \widetilde{C}_0^2\big( \|\mathbb{D}\mathbf{v}\|_{\mathbf{L}^2(\Omega)}
        + \|\mathbf{v}\|_{\mathbf{L}^2(\Omega)}\big)^2, \notag
     \end{align}
     where $\widetilde{C}_0^2:=2C_{\mathrm{in}}^2+2C_{\mathrm{in}}^2\widetilde{C}_{\mathrm{K}}+C_{\mathrm{in}}^4\widetilde{C}_{\mathrm{K}}^2$ depends only on $\Omega$. This verifies \eqref{v-trace}.
     Using Korn's inequality \eqref{Korn} once more, we find that
     \begin{align}
         \|\mathbf{v}\|_{\mathbf{H}^1(\Omega)}&\leq \|\mathbf{v}\|_{\mathbf{L}^2(\Omega)}+\|\nabla\mathbf{v}\|_{\mathbf{L}^2(\Omega)}\notag\\
         &\leq  \|\mathbf{v}\|_{\mathbf{L}^2(\Omega)}+\widetilde{C}_\mathrm{K}
    \big( \|\mathbb{D}\mathbf{v}\|_{\mathbf{L}^2(\Omega)}
        + \|\mathbf{v}\|_{\mathbf{L}^2(\Gamma)} \big)
        \notag\\
        &\leq (1+\widetilde{C}_{\mathrm{K}}+\widetilde{C}_{\mathrm{K}}\widetilde{C}_0) \big( \|\mathbb{D}\mathbf{v}\|_{\mathbf{L}^2(\Omega)}
        + \|\mathbf{v}\|_{\mathbf{L}^2(\Omega)} \big),\notag
     \end{align}
     which implies \eqref{Korn-cor} and completes the proof of Lemma \ref{Korn-lemma}.
     \end{proof}
Finally, we report a variant of Gronwall's lemma (see \cite[Chapter IV, Lemma 4.1]{R.E.S}),
which will be useful in the subsequent analysis.
\begin{lemma}\label{LEM:SGW}
    Let $T>0$, $0\le \alpha<1$, and $a,b\in L^1(0,T)$ with $b\ge 0$ a.e.~in $(0,T)$. Moreover, suppose that $v:[0,T] \to [0,+\infty)$ is absolutely continuous and satisfies
    \begin{align*}
        (1-\alpha)\, v'(t) \le a(t)\, v(t) + b(t)\, v(t)^\alpha
        \quad\text{for almost all $t\in [0,T]$.}
    \end{align*}
    Then it holds that
    \begin{align*}
        v^{1-\alpha}(t)
        \le v^{1-\alpha}(0)\, \exp\left(\int_0^t a(\tau) \,\mathrm d\tau \right)
            + \int_0^t \exp\left(\int_s^t a(\tau) \,\mathrm d\tau \right)\, b(s) \,\mathrm ds\quad \text{for all}\ t\in [0,T].
    \end{align*}
\end{lemma}

\subsection{Main result}

Let us introduce some assumptions that are used in the subsequent analysis.
\begin{enumerate}[label=\textnormal{\bfseries(A\arabic*)}, leftmargin=*, topsep=0.5ex]

\item \label{ASS:A1}
The coefficients satisfy $\nu,\, \beta \in C^{0,1}_{\mathrm{loc}}(\mathbb{R})$ and
\begin{equation*}
0<\nu_\ast\leq \nu(r)\leq \nu^\ast,
\quad\, 0<\beta_\ast\leq \beta(r) \leq \beta^\ast
\quad \text{for all $r\in \mathbb{R}$.}
\end{equation*}

\item \label{ASS:A2}
The bulk free energy density is given by
\begin{equation*}
    F =F_0 +F_1,
\end{equation*}
where $F_0\in C([ -1,1] )\cap C^{2}( -1,1)$ with $F_0( 0) =0$.
For simplicity, we extend $F_0$ by defining $F_0(r) = +\infty$ for all $r\in \mathbb R\setminus [-1,1]$.
For the derivative $f_0 \coloneqq F_0'$, we assume that $f_0(0)=0$, $f_0'(r)\geq \kappa>0$ for all $r\in (-1,1)$, and
\begin{equation*}
\lim_{r\rightarrow \pm 1}f_0\left( r\right) =\pm \infty ,\quad \quad
\lim_{r\rightarrow \pm 1}f_0^{\prime }\left( r\right) =+\infty .
\end{equation*}
In particular, this entails that $F_0$ is strongly convex on $[-1,1]$.
Moreover, $F_1\in {C^2}(\mathbb{R})$ and $f_1 \coloneqq F_1'$
is Lipschitz continuous with Lipschitz constant $C_{\text{Lip}}$.
Furthermore, for the surface singular potential $G$, we {assume an analogous decomposition $G = G_0 + G_1$,}
where $G_0$ and $G_1$ have the same {properties} as $F_0$ and $F_1$, respectively.
We also write $g_0 \coloneqq G'_0$ and $g_1 \coloneqq G_1'$.

\item \label{ASS:A3}
There exist constants $\varrho>0$ and $c_{0}>0$ such that
\begin{align*}
|f_0(r)|\leq\varrho|g_{0}(r)|+c_{0} \quad \text{for all $r\in (-1,1)$.}
\end{align*}

\end{enumerate}

We proceed by introducing the notion of a quasi-strong solution.

\begin{definition}\rm
    \label{quasi}
    Let $\Omega\subset\mathbb{R}^d$ with $d\in\{2,3\}$ be a bounded domain
    with smooth boundary $\Gamma=\partial\Omega$ and let $T>0$ be a prescribed final time. Fix $K\in(0,+\infty]$ and $L\in[0,+\infty]$.
    Assume that \ref{ASS:A1}--\ref{ASS:A3} hold
    and let $(\mathbf{v}_0,\mathbf{v}_{0,\boldsymbol{\tau}})\in \boldsymbol{\mathcal{H}}_{0,\mathrm{div}}^1$ and
    $\boldsymbol{\varphi}_0=(\varphi_0,\psi_0)\in\mathcal{W}_K^2$ be prescribed initial data
    with
    $\|\boldsymbol{\varphi}_0\|_{\mathcal{L}^\infty}\leq1$ and
    \begin{align*}
        \begin{cases}
            |\overline{m}(\boldsymbol{\varphi}_0)|<1,&\text{if }L\in[0,+\infty),\\
            |\langle \varphi_0\rangle_\Omega|<1,\ \ |\langle\psi_0\rangle_\Gamma|<1,&\text{if }L=+\infty.
        \end{cases}
    \end{align*}
    If $L=0$, we further require $\rho_1=\rho_2=\mathrm{const.}>0$.
    A triplet $(\mathbf{v},\boldsymbol{\varphi},\boldsymbol{\mu})$ is called a \emph{quasi-strong solution} to problem \eqref{eqmain0new}--\eqref{icnew}
    if the following properties hold:
    \begin{enumerate}[label=\textnormal{(\roman*)}, leftmargin=*, topsep=0.5ex]
        \item $\mathbf{v}:Q_T \to \mathbb{R}^d$ is a vector field and $\boldsymbol{\varphi} = (\varphi,\psi)$, $\boldsymbol{\mu} = (\mu,\mathcal{L})$ are pairs of scalar functions $\varphi,\mu: Q_T \to \mathbb{R}$ and $\psi,\mathcal{L}:\Sigma_T\to\mathbb{R}$, which have the regularities
        \begin{align*}
                &\mathbf{v}\in L^{\infty}(0,T;\mathbf{H}_{\mathrm{div}}^1(\Omega)),
                \quad\partial_t\mathbf{v}\in L^{2}(0,T;\mathbf{H}_{\mathrm{div}}^1(\Omega)),\\
                &\boldsymbol{\varphi}\in L^\infty(0,T;\mathcal{H}^2),
                \quad\partial_t\boldsymbol{\varphi}\in L^\infty(0,T;(\mathcal{H}_{L,0}^1)')\cap L^2(0,T;\mathcal{H}^1),\\
                &\varphi\in L^\infty(Q_T)\;\; \text{ with }
                \;\; |\varphi(x,t)|<1 \;\; \text{ a.e.~in } Q_T,\\
                &\psi\in L^\infty(\Sigma_T)\;\; \text{ with }
                \;\; |\psi(x,t)|<1 \;\; \text{ a.e.~on } \Sigma_T,\\
                &\boldsymbol{\mu}\in L^\infty(0,T;\mathcal{H}_L^1)\cap L^2(0,T;\mathcal{H}^2),
                \quad \big(f_0(\varphi),\,g_0(\psi)\big)\in L^\infty(0,T;\mathcal{L}^2).
        \end{align*}

        \item The weak formulation
        \begin{align}
        	& \int_\Omega \rho(\varphi)\partial_{t}\mathbf{v}\cdot\mathbf{w}\,\mathrm{d}x
            +\int_\Omega2\alpha\, \mathbb{D}\partial_t\mathbf{v}:\mathbb{D}\mathbf{w}\,\mathrm{d}x
            +\int_\Omega 2\nu(\varphi )\mathbb{D}\mathbf{v}:\mathbb{D}\mathbf{w}\,\mathrm{d}x
        	+\int_\Gamma \beta(\psi)\mathbf{v}_{\bm{\tau}}\cdot\mathbf{w}_{\bm{\tau}}\,\mathrm{d}S
            \notag\\[1mm]
        	&\begin{aligned}
            &\quad =\int_\Omega \mu \nabla \varphi\cdot\mathbf{w}\,\mathrm{d}x
        	-\int_\Gamma \psi\nabla_{\!\bm{\tau}\,}\mathcal{L}\cdot \mathbf{w}_{\bm{\tau}}\,\mathrm{d}S
            -\int_\Omega \rho(\varphi)(\mathbf{v}\cdot\nabla)\mathbf{v}\cdot\mathbf{w}\,\mathrm{d}x
            \\[1mm]
            &\qquad+\rho'\int_\Omega (\nabla\mu\cdot\nabla)\mathbf{v}\cdot\mathbf{w}\,\mathrm{d}x
            -\frac{\rho'}{2}\int_\Gamma (\nabla\mu\cdot\mathbf{n})\mathbf{v}_{\bm{\tau}}\cdot\mathbf{w}_{\bm{\tau}}\,\mathrm{d}S,
            \end{aligned}
        	\label{quasi-1'}
        \end{align}
        is satisfied for all $(\mathbf{w},\mathbf{w}_{\boldsymbol{\tau}})\in \boldsymbol{\mathcal{H}}^2_{0,\mathrm{div}}$ and almost all $t\in[0,T]$, and it further holds that
        \begin{subequations}
        \label{quasi-2'}
        \begin{alignat}{2}
        	&\partial_{t}\varphi+\mathbf{v} \cdot\nabla\varphi=\Delta\mu
            &&\quad \mbox{a.e.~in } Q_{T},
        	\label{quasi-1} \\
        	& \mu =-\Delta \varphi +f(\varphi )
            &&\quad \mbox{a.e.~in } Q_{T},
        	\label{quasi-2} \\
            & {\begin{cases}
        	K\partial_{\mathbf{n}}\varphi=\psi-\varphi,&\text{if } K\in(0,+\infty) \\
        	\partial_{\mathbf{n}}\varphi=0,&\text{if }K=+\infty
        	\end{cases}}
            &&\quad {\mbox{a.e.~on }
        	\Sigma _{T},}
        	\label{quasi-3'} \\[1ex]
        	&\begin{cases}
        	L\partial_{\mathbf{n}}\mu=\mathcal{L}-\mu,&\text{if }L\in[0,+\infty)\\
        	\partial_{\mathbf{n}}\mu=0,&\text{if }L=+\infty
        	\end{cases}
            &&\quad\mbox{a.e.~on }
        	\Sigma _{T},
        	\label{quasi-3} \\
        	&\partial_t\psi +\mathbf{v}_{\bm{\tau}} \cdot\nabla_{\!\bm{\tau}\,}\psi
            =\Delta_{\bm{\tau}}\mathcal{L}-\partial_{\mathbf{n}}\mu
            &&\quad \mbox{a.e.~on }
        	\Sigma _{T},
        	\label{quasi-4} \\
        	& \mathcal{L}  = - \Delta _{\bm{\tau}}\psi +\partial _{%
        		\mathbf{n}}\varphi+ g(\psi)
            &&\quad \mbox{a.e.~on }
        	\Sigma _{T}.
        	\label{quasi-5}
        \end{alignat}
        \end{subequations}

        \item The functions $\mathbf{v}$ and $\boldsymbol{\varphi}$ satisfy the initial conditions
        \begin{align}
        	\mathbf{v}|_{t=0}=\mathbf{v}_0\quad \text{a.e. in }\Omega,\quad(\varphi,\psi)|_{t=0}=(\varphi_{0},\psi_0)\quad\text{in }\Omega\times\Gamma.\label{0821-initial-condition}
        \end{align}

    \end{enumerate}
\end{definition}

\begin{remark}\rm
   The regularity $\boldsymbol{\mu}\in L^2(0,T;\mathcal{H}^2)$ ensures that $\nabla\mu\cdot\mathbf{n}\in L^2(0,T;H^{\frac{1}{2}}(\Gamma))$, and thus the last term on the right-hand side of \eqref{quasi-1'} is well-defined. For $L\in (0,+\infty)$, this term can be equivalently written as $-\frac{\rho'}{2L}\int_\Gamma (\mathcal{L}- \mu)\mathbf{v}_{\bm{\tau}}\cdot\mathbf{w}_{\bm{\tau}}\,\mathrm{d}S$ (cf. \cite{GLW,KS}), while for $L=0$ (with $\rho_1=\rho_2$) or $L=+\infty$, this term simply vanish. Moreover, it follows from the regularities in Definition \ref{quasi} that
   $$\mathbf{v}\in H^1(0,T;\mathbf{H}_{\mathrm{div}}^{1}(\Omega))\hookrightarrow C([0,T];\mathbf{H}_{\mathrm{div}}^1(\Omega)).$$
   As $d\in\{2,3\}$, the Aubin--Lions--Simon lemma, together with the Sobolev embedding theorem, gives
   $$\boldsymbol{\varphi}\in C([0,T];\mathcal{H}^{2-\epsilon})\hookrightarrow C([0,T];C(\overline{\Omega})\times C(\Gamma))$$
   for every $\epsilon\in(0,1/2)$. Therefore, the initial conditions in \eqref{0821-initial-condition} make sense.
\end{remark}

We are now in a position to state the main result of this work.

\begin{theorem}[Existence of global quasi-strong solutions]
\label{quasi-strong}
Let $\Omega\subset\mathbb{R}^d$ with $d\in\{2,3\}$ be a bounded domain
with smooth boundary $\Gamma=\partial\Omega$ and let $T>0$ be a prescribed final time. Fix $K\in(0,+\infty]$ and $L\in[0,+\infty]$. Assume that \ref{ASS:A1}--\ref{ASS:A3} hold
and let $(\mathbf{v}_0, \mathbf{v}_{0,\boldsymbol{\tau}})\in \boldsymbol{\mathcal{H}}_{0,\mathrm{div}}^1$ and
$\boldsymbol{\varphi}_0=(\varphi_0,\psi_0)\in\mathcal{W}_K^2$ be prescribed initial data with
$\|\boldsymbol{\varphi}_0\|_{\mathcal{L}^\infty}\leq1$ and \begin{align}
        \begin{cases}
            |\overline{m}(\boldsymbol{\varphi}_0)|<1,&\text{if }L\in[0,+\infty),\\
            |\langle \varphi_0\rangle_\Omega|<1,\ \ |\langle\psi_0\rangle_\Gamma|<1,&\text{if }L=+\infty.
        \end{cases}\label{initial-mean-value}
    \end{align}
Moreover, assume that the following additional assumption holds:
\begin{enumerate}[label=\textnormal{\bfseries(A\arabic*)}, leftmargin=*, topsep=0.5ex, start=4]
    \item \label{ASS:A4}
    $
    (-\Delta\varphi_0+f_0(\varphi_0),
    -\Delta_{\boldsymbol{\tau}}\psi_0+\partial_{\mathbf{n}}\varphi_0+g_0(\psi_0))
    \in \mathcal{H}^1.
    $
\end{enumerate}
In the case $L=0$, we further require $\rho_1=\rho_2=\mathrm{const.}>0$ and assume the compatibility condition:
\begin{enumerate}[label=\textnormal{\bfseries(C)}, leftmargin=*, topsep=0.5ex]
    \item \label{ASS:C}
    $
    -\Delta\varphi_0+f(\varphi_0)
    =-\Delta_{\boldsymbol{\tau}}\psi_0+\partial_{\mathbf{n}}\varphi_0+g(\psi_0)
    \quad\text{a.e.~on}\ \Gamma
    $.
\end{enumerate}
Then problem \eqref{eqmain0new}--\eqref{icnew} admits a quasi-strong solution $(\mathbf{v},\boldsymbol{\varphi},{\boldsymbol{\mu}})$
on $[0,T]$ in the sense of Definition~\ref{quasi}.
Furthermore, the energy {identity}
\begin{align}
    &E_{\mathrm{tot}}(\mathbf{v}(t),\boldsymbol{\varphi}(t))
    +\int_0^t\int_\Omega 2\nu(\varphi(s))|\mathbb{D}\mathbf{v}(s)|^2\,\mathrm{d}x\,\mathrm{d}s
    +\int_0^t\int_\Gamma\beta(\psi(s))|\mathbf{v}_{\boldsymbol{\tau}}(s)|^2\,\mathrm{d}S\,\mathrm{d}s\notag\\
    &\qquad+\int_0^t\int_\Omega |\nabla\mu(s)|^2\,\mathrm{d}x\,\mathrm{d}s
    +\int_0^t\int_\Gamma |\nabla_{\!\boldsymbol{\tau}\,}\mathcal{L}(s)|^2\,\mathrm{d}S\,\mathrm{d}s
    +\chi(L)\int_0^t\int_\Gamma|\mu(s)-\mathcal{L}(s)|^2\,\mathrm{d}S\,\mathrm{d}s
    \notag\\[1ex]
    &\quad= E_{\mathrm{tot}}(\mathbf{v}_0,\boldsymbol{\varphi}_0)\label{energy-equality}
\end{align}
holds for almost all $t\in[0,T]$.
\end{theorem}

\begin{remark}\rm
For $(K,L)\in(0,+\infty]\times(0,+\infty)$, the assertion of the theorem will be established in Section~\ref{existence of quasi}. The proof is based on an approximation scheme presented in Section~\ref{approximating system for L}. In the limiting cases $(K,L)\in(0,+\infty]\times\{0\}$ and $(K,L)\in(0,+\infty]\times\{+\infty\}$, the conclusion will be verified in Section~\ref{asy}. We note that the case
$(K,L)\in\{0\}\times[0,+\infty]$ is not covered in the present work. One possible approach to establish the existence of global quasi-strong solutions for $(K,L)\in\{0\}\times[0,+\infty]$ is to pass to
the limit as $K\to0$ with $L$ fixed. In this case, however, the initial datum have to satisfy $\boldsymbol{\varphi}_0\in\mathcal{W}_0^2=\mathcal{V}^2$.
At present, it is unclear how to approximate such an initial datum by a sequence satisfying the assumptions of Theorem~\ref{A1}. Hence, the approximation scheme developed in Section~\ref{approximating system for L} cannot be directly applied to construct approximate solutions for the case $K=0$.
\end{remark}

\begin{remark}\rm
The matched-density assumption, i.e., $\rho_1=\rho_2>0$, is required in
our approach to the limit as $L\to0$. In the unmatched-density case, the lack of
a uniform bound for
$\|\boldsymbol{\mu}^L\|_{L^2(0,T;\mathcal{H}^2)}$ with respect to $L\in(0,1]$
(see Remark~\ref{no-bound}) prevents us from passing to the limit in the
boundary term
$\rho'\int_\Gamma
(\nabla\mu^L\cdot\mathbf n)
\mathbf v_{\boldsymbol\tau}^L\cdot
\mathbf w_{\boldsymbol\tau}\,\mathrm dS$. By contrast, under the matched-density assumption, the density is constant
and hence $\rho'\equiv0$. Therefore, the above boundary term vanishes identically, thereby avoiding this difficulty.
In contrast to the case $L\to0$, no matched-density assumption is
required when passing to the limit as $L\to+\infty$. Indeed, the
uniform $L^\infty(0,T;\mathcal{H}^1)$-estimate for
$\boldsymbol{\mu}^L$ with respect to $L\geq L_0$ (cf. \eqref{uniform-0318-6}), together with the trace theorem and the boundary condition, yields
\[
\|\partial_{\mathbf n}\mu^L\|_{L^\infty(0,T;H_\Gamma)}
=\frac{1}{L}
\|\mathcal{L}^L-\mu^L\|_{L^\infty(0,T;H_\Gamma)}
\leq\frac{C}{L}\to 0,\quad\text{as }L\to+\infty.
\]
This, together with the uniform $L^\infty(0,T;\mathbf{H}_{\mathrm{div}}^1(\Omega))$-estimate for $\mathbf{v}^L$ (cf. \eqref{m-uni-7}), allows us to pass to the limit in the boundary term.
\end{remark}

\begin{remark}\rm
    The assumption $(\mathbf{v}_0,\mathbf{v}_{0,\boldsymbol{\tau}})\in \boldsymbol{\mathcal{H}}_{0,\mathrm{div}}^1$ is a technical requirement that arises from our proof and is needed to derive uniform estimates of approximate solutions. More precisely, we introduce an approximate system whose solvability is established via a semi-Galerkin approximation scheme combined with a fixed-point argument inspired by \cite{Gior22}. In the semi-Galerkin scheme, the initial datum for the approximate velocity field $\mathbf{v}_m$ is chosen as $\mathbf{v}_m(0)=\mathbb{P}_m^\Omega(\mathbf{v}_0,\mathbf{v}_{0,\boldsymbol{\tau}})$; see Section \ref{SEC:3.2} for the definition of the projection operator $\mathbb{P}_m$.
    To derive estimates that are uniform with respect to the Galerkin parameter $m\in \mathbb{N}$, we need to control the term $\alpha\int_\Omega |\mathbb{D}\mathbb{P}_m^\Omega(\mathbf{v}_0,\mathbf{v}_{0,\boldsymbol{\tau}})|^2\,\mathrm{d}x$ appearing in \eqref{0328-energy-id}. By the assumption $(\mathbf{v}_0,\mathbf{v}_{0,\boldsymbol{\tau}})\in \boldsymbol{\mathcal{H}}_{0,\mathrm{div}}^1$ and the stability property \eqref{Pm-stability} of $\mathbb{P}_m$, we derive the estimate
    \begin{align}
        &\alpha\int_\Omega |\mathbb{D}\mathbb{P}^\Omega_m(\mathbf{v}_0,\mathbf{v}_{0,\boldsymbol{\tau}})|^2\,\mathrm{d}x\notag\\
        &\quad\leq \alpha\int_\Omega |\mathbb{D}\mathbb{P}^\Omega_m(\mathbf{v}_0,\mathbf{v}_{0,\boldsymbol{\tau}})|^2\,\mathrm{d}x+\alpha\int_\Gamma |\mathbb{D}_{\boldsymbol{\tau}}\mathbb{P}^\Gamma_m(\mathbf{v}_0,\mathbf{v}_{0,\boldsymbol{\tau}})|^2\,\mathrm{d}S+\alpha\int_\Gamma |\mathbb{P}_m^\Gamma(\mathbf{v}_0,\mathbf{v}_{0,\boldsymbol{\tau}})|^2\,\mathrm{d}S
        \notag\\
        &\quad=\alpha\|\mathbb{P}_m(\mathbf{v}_0, \mathbf{v}_{0,\boldsymbol{\tau}})\|_{\boldsymbol{\mathcal{H}}_{0,\mathrm{div}}^1}^2
        \notag\\
        &\quad\leq \alpha\|(\mathbf{v}_0,\mathbf{v}_{0,\boldsymbol{\tau}})\|_{\boldsymbol{\mathcal{H}}_{0,\mathrm{div}}^1}^2,
        \notag
        \end{align}
        which is independent of $m$. This explains why the assumption $(\mathbf{v}_0,\mathbf{v}_{0,\boldsymbol{\tau}})\in \boldsymbol{\mathcal{H}}_{0,\mathrm{div}}^1$ is required in our argument.
\end{remark}


\section{An Approximate Problem in the Case \texorpdfstring{$(K,L)\in(0,+\infty]\times(0,+\infty)$}
{(K,L) ∈ (0,+∞] × (0,+∞)}}
\label{approximating system for L}
\setcounter{equation}{0}
%


\subsection{Approximation of the initial data}
For any $\sigma\in \big[ 0,1/2\big)$, we first introduce an auxiliary singular potential
\begin{align}
    f_\sigma(r)=f_0\Big(\frac{r}{1-\sigma}\Big)
    \quad\text{for all $r\in(-1+\sigma,1-\sigma)$}.
    \label{f_sigma}
\end{align}
Due to \ref{ASS:A2}, we have $f_\sigma\in C^1(-1+\sigma,1-\sigma)$, $f_\sigma(0)=0$ and
\begin{align}
    &f_\sigma'(r)\geq \frac{\kappa}{1-\sigma}>0\qquad\text{for all $r\in(-1+\sigma,1-\sigma)$},\label{f_lambda_p1}
    \\
    &\lim_{r\to\pm(1-\sigma)}f_\sigma(r)=\pm\infty,\quad \lim_{r\to\pm(1-\sigma)}f_\sigma'(r)=+\infty,\label{f_lambda_p2}
    \\
    &f_\sigma(r)\,f_0(r)\geq0,\quad f_\sigma(r)g_0(r)\geq0
    \qquad\text{for all $r\in(-1+\sigma,1-\sigma)$}.\label{f_lambda_p3}
\end{align}
It is easy to see that $f_\sigma$ is the derivative of $F_\sigma$, where
\begin{align}
    \label{F_sigma}
    F_\sigma(r)=(1-\sigma)F_0\Big(\frac{r}{1-\sigma}\Big)
    \quad\text{for all $r\in[-1+\sigma,1-\sigma]$.}
\end{align}
Let $\boldsymbol{\varphi}_0=(\varphi_0,\psi_0)\in\mathcal{W}_K^2$ with
$\|\boldsymbol{\varphi}_0\|_{\mathcal{L}^\infty}\leq1$ and $|\overline{m}(\boldsymbol{\varphi}_0)|<1$ be any prescribed initial data. In addition, we denote
\begin{align*}
    \widetilde{\mu}_0=-\Delta\varphi_0+f_0(\varphi_0)
    \quad\text{and}\quad
    \widetilde{\mathcal{L}}_0=-\Delta_{\boldsymbol{\tau}}\psi_0+\partial_{\mathbf{n}}\varphi_0+g_0(\psi_0).
\end{align*}
For $k\in\mathbb{N}^+$, let $h_k$ be a globally Lipschitz continuous truncation function
\begin{align}
    h_k:\mathbb{R}\to\mathbb{R}, \quad
    h_k(r)=
    \begin{cases}
        -k,&\text{if }r<-k,\\
        r,&\text{if }-k\leq r\leq k,\\
        k,&\text{if }r>k.
    \end{cases}\label{h_k}
\end{align}
Then we define
$$
\widetilde{\mu}_{0,k}=h_k\circ\widetilde{\mu}_0,\quad  \widetilde{\mathcal{L}}_{0,k}=h_k\circ\widetilde{\mathcal{L}}_{0}.
$$
The assumption \ref{ASS:A4} implies that
$(\widetilde{\mu}_0,\widetilde{\mathcal{L}}_0) \in \mathcal{H}^1$, and thus $\widetilde{\mu}_{0,k}\in V\cap L^\infty(\Omega)$,
$\widetilde{\mathcal{L}}_{0,k}\in V_\Gamma\cap L^\infty(\Gamma)$
satisfying
\begin{align}
    \|\widetilde{\mu}_{0,k}\|_{V}
    \leq \|\widetilde{\mu}_0\|_{V}
    \quad\text{and}\quad
    \|\widetilde{\mathcal{L}}_{0,k}\|_{V_\Gamma}
    \leq \|\widetilde{\mathcal{L}}_0\|_{V_\Gamma}.
    \label{superposition}
\end{align}
Now, for $\sigma\in[0,1/2)$ and $k\in\mathbb{N}^+$, we approximate the initial data $\boldsymbol{\varphi}_0$ using the following bulk-surface elliptic problem
\begin{subequations}
\label{initial-elliptic}
\begin{empheq}[left=(S_{\sigma,k})\left\{, right=\right.]{align}
    &-\Delta\varphi_{0,\sigma,k}+f_0(\varphi_{0,\sigma,k})
    +k^{-1}f_\sigma(\varphi_{0,\sigma,k})
    =\widetilde{\mu}_{0,k}
    &&\text{in }\Omega,\label{initial-elliptic-1}\\
    &
    \begin{cases}
        K\partial_\mathbf{n}\varphi_{0,\sigma,k} =\psi_{0,\sigma,k}-\varphi_{0,\sigma,k}, &{K\in(0,+\infty)} \\
        \partial_\mathbf{n}\varphi_{0,\sigma,k}=0, &K=+\infty
    \end{cases}
    &&{\text{on }\Gamma,}\label{initial-elliptic-2}
    \\
    &-\Delta_{\boldsymbol{\tau}}\psi_{0,\sigma,k}+\partial_{\mathbf{n}}\varphi_{0,\sigma,k}
    +g_0(\psi_{0,\sigma,k}) +k^{-1}f_\sigma(\psi_{0,\sigma,k})=\widetilde{\mathcal{L}}_{0,k}
    &&\text{on }\Gamma. \label{initial-elliptic-3}
\end{empheq}
\end{subequations}
Below we show that the problem $(S_{\sigma,k})$
admits a unique strong solution that satisfies the strict separation property.
\begin{proposition}
\label{approximate-initial-datum}
For any $\sigma\in [ 0,1/2)$ and $k\in\mathbb{N}^+$,
the bulk-surface elliptic problem $(S_{\sigma,k})$ admits a unique strong solution
$\boldsymbol{\varphi}_{0,\sigma,k}=(\varphi_{0,\sigma,k},\psi_{0,\sigma,k})\in{\mathcal{W}_K^2}$ that satisfies
\begin{align}
       \|\boldsymbol{\varphi}_{0,\sigma,k}\|_{\mathcal{H}^2}
       +\|f_{0}(\varphi_{0,\sigma,k})\|_{H}
       +\|g_{0}(\psi_{0,\sigma,k})\|_{H_\Gamma}
       &\leq  C(1+\|\widetilde{\mu}_0\|_{V}+\|\widetilde{\mathcal{L}}_0\|_{V_\Gamma}),\label{0315-0}\\
       k^{-1}\|f_{\sigma}(\varphi_{0,\sigma,k})\|_{H}
       +k^{-1}\|f_{\sigma}(\psi_{0,\sigma,k})\|_{H_\Gamma}
       &\leq  C(1+\|\widetilde{\mu}_0\|_{V}+\|\widetilde{\mathcal{L}}_0\|_{V_\Gamma}),\label{0315-0'}\\
      \|-\Delta\varphi_{0,\sigma,k}+f_0(\varphi_{0,\sigma,k})
      +k^{-1}f_\sigma(\varphi_{0,\sigma,k})\|_{V}
      &\leq \|\widetilde{\mu}_0\|_{V},\label{0315-1}\\
      \|-\Delta_{\boldsymbol{\tau}}\psi_{0,\sigma,k}
      +\partial_{\mathbf{n}}\varphi_{0,\sigma,k}
      +g_0(\psi_{0,\sigma,k})+k^{-1}f_\sigma(\psi_{0,\sigma,k}))\|_{V_\Gamma}
      &\leq \|\widetilde{\mathcal{L}}_0\|_{V_\Gamma},\label{0315-2}
\end{align}
for a constant $C>0$ independent of $\sigma$ and $k$,
and there exists $\delta(k)\in(0,1)$ such that
\begin{align}
      \|\varphi_{0,\sigma,k}\|_{L^\infty(\Omega)}\leq (1-\sigma)\big(1-\delta(k)\big),
      \quad\|\psi_{0,\sigma,k}\|_{L^\infty(\Gamma)}\leq (1-\sigma)\big(1-\delta(k)\big).
      \label{initial-separation}
\end{align}
Furthermore, we have
\begin{align}
      &\boldsymbol{\varphi}_{0,\sigma,k} \to\boldsymbol{\varphi}_{0,k}\quad\text{in }{\mathcal{H}^1}\ \text{as }\sigma\to0
      \quad\text{and}\quad\boldsymbol{\varphi}_{0,k} \to\boldsymbol{\varphi}_0
      \quad\text{in }{\mathcal{H}^1}\ \text{as }k\to+\infty,
      \label{initial-convergence}
\end{align}
where $\boldsymbol{\varphi}_{0,k} \coloneqq \boldsymbol{\varphi}_{0,0,k}$ denotes the unique strong solution to problem $(S_{0,k})$.
\end{proposition}
\begin{proof}
Since
$\widetilde{\mu}_{0,k}\in V\cap L^\infty(\Omega)$ and
$\widetilde{\mathcal{L}}_{0,k}\in V_\Gamma\cap L^\infty(\Gamma)$,
the existence and uniqueness of a strong solution $\boldsymbol{\varphi}_{0,\sigma,k}=(\varphi_{0,\sigma,k},\psi_{0,\sigma,k})\in\mathcal{H}^2$ to problem $(S_{\sigma,k})$ as well as the estimates \eqref{0315-0} and \eqref{0315-0'}
are a direct consequence of \cite[Propositions 5.1 and 5.2]{GKS}. Moreover, combining \eqref{superposition} and \eqref{initial-elliptic}, we easily obtain \eqref{0315-1} and \eqref{0315-2}. In what follows, we prove the strict separation property and the convergence results.
\smallskip

\noindent\textbf{\itshape Separation property.}
The strict separation property \eqref{initial-separation} can be established
in a manner similar to that in \cite[Proposition 5.3]{GKS}.
For completeness, we provide a brief sketch of the proof.
To begin with, for any $l \in \mathbb{N}$ with $l \ge 2$,
we introduce the globally Lipschitz continuous truncation
\begin{align}
    j_l:\mathbb{R}\to\mathbb{R}, \quad
    j_l(r)=
    \begin{cases}
        -1+\sigma+\frac{1}{l},&\text{if }r< -1+\sigma+\frac{1}{l},\\
        r,&\text{if }r\in \big[-1+\sigma+\frac{1}{l},1-\sigma-\frac{1}{l} \big],\\
        1-\sigma-\frac{1}{l},&\text{if }r>1-\sigma-\frac{1}{l},
    \end{cases}\notag
\end{align}
and define
$\varphi_{0,\sigma,k}^l=j_l\circ\varphi_{0,\sigma,k}$ and
$\psi_{0,\sigma,k}^l=j_l\circ\psi_{0,\sigma,k}$.
It easily follows that $(\varphi_{0,\sigma,k}^l,\psi_{0,\sigma,k}^l)\in\mathcal{H}^1$
with
\begin{align*}
  \nabla\varphi_{0,\sigma,k}^l
  =\chi_{[-1+\sigma+\frac{1}{l},1-\sigma-\frac{1}{l}]}\nabla\varphi_{0,\sigma,k}
  \quad\text{and}\quad
  \nabla_{\!\boldsymbol{\tau}\,}\psi_{0,\sigma,k}^l
  =\chi_{[-1+\sigma+\frac{1}{l},1-\sigma-\frac{1}{l}]}\nabla_{\!\boldsymbol{\tau}\,}\psi_{0,\sigma,k}.
\end{align*}
For any $p\geq2$, testing \eqref{initial-elliptic-1} by
$|f_\sigma(\varphi_{0,\sigma,k}^l)|^{p-2}f_\sigma(\varphi_{0,\sigma,k}^l)$
and \eqref{initial-elliptic-3} by
$|f_\sigma(\psi_{0,\sigma,k}^l)|^{p-2}f_\sigma(\psi_{0,\sigma,k}^l)$,
we find
\begin{align}
    &k^{-1}\int_\Omega|f_\sigma(\varphi_{0,\sigma,k}^l)|^{p-2}f_\sigma(\varphi_{0,\sigma,k}^l)f_\sigma(\varphi_{0,\sigma,k})\,\mathrm{d}x
    +k^{-1}\int_\Gamma|f_\sigma(\psi_{0,\sigma,k}^l)|^{p-2}f_\sigma(\psi_{0,\sigma,k}^l)f_{\sigma}(\psi_{0,\sigma,k})\,\mathrm{d}S\notag\\
    &\qquad+(p-1)\int_\Omega|f_\sigma(\varphi_{0,\sigma,k}^l)|^{p-2}f_\sigma'(\varphi_{0,\sigma,k}^l)|\nabla\varphi_{0,\sigma,k}^l|^2\,\mathrm{d}x\notag\\
    &\qquad+(p-1)\int_\Gamma|f_\sigma(\psi_{0,\sigma,k}^l)|^{p-2}f_\sigma'(\psi_{0,\sigma,k}^l)|\nabla_{\!\boldsymbol{\tau}\,}\psi_{0,\sigma,k}^l|^2\,\mathrm{d}S\notag\\
    &\qquad+\int_\Omega|f_\sigma(\varphi_{0,\sigma,k}^l)|^{p-2}f_\sigma(\varphi_{0,\sigma,k}^l)f_0(\varphi_{0,\sigma,k})\,\mathrm{d}x
    +\int_\Gamma|f_\sigma(\psi_{0,\sigma,k}^l)|^{p-2}f_\sigma(\psi_{0,\sigma,k}^l)g_0(\psi_{0,\sigma,k})\,\mathrm{d}S \notag\\
    &\quad=\int_\Omega \widetilde{\mu}_{0,k}|f_\sigma(\varphi_{0,\sigma,k}^l)|^{p-2}f_\sigma(\varphi_{0,\sigma,k}^l)\,\mathrm{d}x
    +\int_\Gamma \widetilde{\mathcal{L}}_{0,k}|f_\sigma(\psi_{0,\sigma,k}^l)|^{p-2}f_\sigma(\psi_{0,\sigma,k}^l)\,\mathrm{d}S\notag\\
     &\qquad+\int_\Gamma\partial_\mathbf{n}\varphi_{0,\sigma,k}(|f_\sigma(\varphi_{0,\sigma,k}^l)|^{p-2}f_\sigma(\varphi_{0,\sigma,k}^l)
    -|f_\sigma(\psi_{0,\sigma,k}^l)|^{p-2}f_\sigma(\psi_{0,\sigma,k}^l))\,\mathrm{d}S\notag\\
    &\quad\leq \frac{k^{-1}}{2}\int_\Omega|f_\sigma(\varphi_{0,\sigma,k}^l)|^p\,\mathrm{d}x
    +\frac{k^{-1}}{2}\int_\Gamma|f_\sigma(\psi_{0,\sigma,k}^l)|^p\,\mathrm{d}S
    +C_k\Big(\int_\Omega|\widetilde{\mu}_{0,k}|^p \,\mathrm{d}x+\int_\Gamma |\widetilde{\mathcal{L}}_{0,k}|^p\,\mathrm{d}S\Big).
    \label{20260301-1}
\end{align}
Here, we have used the estimate
\begin{align*}
    &\int_\Gamma\partial_\mathbf{n}\varphi_{0,\sigma,k}(|f_\sigma(\varphi_{0,\sigma,k}^l)|^{p-2}f_\sigma(\varphi_{0,\sigma,k}^l)
    -|f_\sigma(\psi_{0,\sigma,k}^l)|^{p-2}f_\sigma(\psi_{0,\sigma,k}^l))\,\mathrm{d}S
    \\
    &\quad=-\chi(K)\int_\Gamma(\varphi_{0,\sigma,k}-\psi_{0,\sigma,k})(|f_\sigma(\varphi_{0,\sigma,k}^l)|^{p-2}f_\sigma(\varphi_{0,\sigma,k}^l)
    -|f_\sigma(\psi_{0,\sigma,k}^l)|^{p-2}f_\sigma(\psi_{0,\sigma,k}^l))\,\mathrm{d}S\leq0,
\end{align*}
which was shown, based on \eqref{initial-elliptic-2}, in \cite[Proof of Lemma 3.4]{LvWuAA}.

Thanks to \ref{ASS:A2}, we know that the expressions
$f_0(r)$, $g_0(r)$, $f_\sigma(r)$, and $f_\sigma(j_l(r))$
have the same sign for all
$r\in{(-1+\sigma,1-\sigma)}$.
Consequently, due to the monotonicity of $f_\sigma$, we have
\begin{alignat*}{2}
    f_\sigma(\varphi^l_{0,\sigma,k}) f_\sigma(\varphi_{0,\sigma,k})
    &\geq |f_\sigma(\varphi_{0,\sigma,k}^l)|^2
    &&\quad\text{a.e.~in }\Omega,
    \\
    \quad f_\sigma(\psi^l_{0,\sigma,k})f_\sigma(\psi_{0,\sigma,k})
    &\geq |f_\sigma(\psi_{0,\sigma,k}^l)|^2
    &&\quad\text{a.e.~on }\Gamma,
    \\
    f_\sigma(\varphi_{0,\sigma,k}^l)f_0(\varphi_{0,\sigma,k})
    &\geq0
    &&\quad\text{a.e.~in }\Omega,
    \\
    f_\sigma(\psi_{0,\sigma,k}^l)g_0(\psi_{0,\sigma,k})
    &\geq0
    &&\quad\text{a.e.~on }\Gamma.
\end{alignat*}
Thus, we infer from \eqref{20260301-1} that
\begin{align}
    \|f_\sigma(\varphi_{0,\sigma,k}^l)\|_{L^p(\Omega)}
    +\|f_\sigma(\psi_{0,\sigma,k}^l)\|_{L^p(\Gamma)}
    \leq C_k(\|\widetilde{\mu}_{0,k}\|_{L^p(\Omega)}
    +\|\widetilde{\mathcal{L}}_{0,k}\|_{L^p(\Gamma)})
    \leq C(k)
    \label{0317-2}
\end{align}
for all $p\ge 2$. Since the constant $C(k)$ in \eqref{0317-2} is independent of $p$,
we get
\begin{align}
    \|f_\sigma(\varphi_{0,\sigma,k}^l)\|_{L^\infty(\Omega)}
    +\|f_\sigma(\psi_{0,\sigma,k}^l)\|_{L^\infty(\Gamma)}
    \leq C(k).
    \label{0317-3}
\end{align}
According to Fatou's lemma, we can pass to the limit as $ l\to+\infty$ in \eqref{0317-3}
and conclude that
\begin{align}
    \|f_\sigma(\varphi_{0,\sigma,k})\|_{L^\infty(\Omega)}
    +\|f_\sigma(\psi_{0,\sigma,k})\|_{L^\infty(\Gamma)}
    \leq C(k).\notag
\end{align}
Choosing $\delta(k)=1-\max\{f_0^{-1}(C(k)),-f_0^{-1}(-C(k))\}\in(0,1)$,
this allows us to obtain
\begin{align*}
    \|\varphi_{0,\sigma,k}\|_{L^\infty(\Omega)}\leq (1-\sigma)(1-\delta(k)),
    \quad\|\psi_{0,\sigma,k}\|_{L^\infty(\Gamma)}\leq (1-\sigma)(1-\delta(k)).
\end{align*}

\noindent\textbf{\itshape Convergence results.}
To investigate the convergence
$\boldsymbol{\varphi}_{0,\sigma,k}\to\boldsymbol{\varphi}_{0,k}$ in $\mathcal{H}^1$,
we introduce the notation
\begin{align*}
    \boldsymbol{\varphi}^\sharp_{0,\sigma,k}
    = (\varphi^\sharp_{0,\sigma,k} , \psi^\sharp_{0,\sigma,k})
    \coloneqq \boldsymbol{\varphi}_{0,\sigma,k}-\boldsymbol{\varphi}_{0,k}
    \in \mathcal{H}^2.
\end{align*}
In view of \eqref{initial-elliptic}, we have
\begin{subequations}
    \begin{align}
          &-\Delta\varphi^\sharp_{0,\sigma,k}
    +f_0(\varphi_{0,\sigma,k})-f_0(\varphi_{0,k})
    +k^{-1}f_\sigma(\varphi_{0,\sigma,k})-k^{-1}f_0(\varphi_{0,k})
    =0&&\text{a.e.~in }\Omega, \label{diff-20260301-2-1}\\
    &
        \begin{cases}
            K\partial_\mathbf{n}\varphi_{0,\sigma,k}^\sharp =\psi_{0,\sigma,k}^\sharp-\varphi_{0,\sigma,k}^\sharp, & K\in(0,+\infty)\\
            \partial_\mathbf{n}\varphi_{0,\sigma,k}^\sharp=0, &K=+\infty
        \end{cases}&& \text{a.e.~on }\Gamma, \label{diff-20260301-2-2}
        \\
    &-\Delta_{\boldsymbol{\tau}}\psi^\sharp_{0,\sigma,k}
    +\partial_{\mathbf{n}}\varphi_{0,\sigma,k}^\sharp +g_0(\psi_{0,\sigma,k})-g_{0}(\psi_{0,k})
    +k^{-1}f_\sigma(\psi_{0,\sigma,k})-k^{-1}f_0(\psi_{0,k})=0 &&\text{a.e.~on }\Gamma.
    \label{diff-20260301-2-3}
    \end{align}
\end{subequations}
Testing \eqref{diff-20260301-2-1} by $ \frac{\varphi_{0,\sigma,k}}{1-\sigma}-\varphi_{0,k}$,
\eqref{diff-20260301-2-3} by $\frac{\psi_{0,\sigma,k}}{1-\sigma}-\psi_{0,k}$, respectively, using \eqref{diff-20260301-2-2}, we obtain
\begin{align}
    0&= \int_\Omega\nabla\varphi_{0,\sigma,k}^\sharp\cdot\nabla\Big(\frac{\varphi_{0,\sigma,k}}{1-\sigma}-\varphi_{0,k}\Big)\,\mathrm{d}x
    +\int_\Gamma\nabla_{\!\boldsymbol{\tau}\,}\psi_{0,\sigma,k}^\sharp
    \cdot\nabla_{\!\boldsymbol{\tau}\,} \Big(\frac{\psi_{0,\sigma,k}}{1-\sigma}-\psi_{0,k}\Big)\,\mathrm{d}S
    \notag\\
    &\quad+\int_\Omega\Big(f_0(\varphi_{0,\sigma,k})-f_0(\varphi_{0,k})\Big)\Big(\frac{\varphi_{0,\sigma,k}}{1-\sigma}-\varphi_{0,k}\Big)\,\mathrm{d}x
    +\int_\Gamma\Big(g_0(\psi_{0,\sigma,k})-g_0(\psi_{0,k})\Big)\Big(\frac{\psi_{0,\sigma,k}}{1-\sigma}-\psi_{0,k}\Big)\,\mathrm{d}S
    \notag\\
    &\quad+k^{-1}\int_\Omega \Big(f_0\Big(\frac{\varphi_{0,\sigma,k}}{1-\sigma}\Big)-f_0(\varphi_{0,k})\Big)\Big(\frac{\varphi_{0,\sigma,k}}{1-\sigma}-\varphi_{0,k}\Big)\,\mathrm{d}x
    \notag\\
    &\quad+k^{-1}\int_\Gamma \Big(f_0\Big(\frac{\psi_{0,\sigma,k}}{1-\sigma}\Big)-f_0(\psi_{0,k})\Big)\Big(\frac{\psi_{0,\sigma,k}}{1-\sigma}-\psi_{0,k}\Big)\,\mathrm{d}S
    \notag\\
    &\quad {-\int_\Gamma \partial_\mathbf{n}\varphi_{0,\sigma,k}^\sharp\Big(\frac{\varphi_{0,\sigma,k}-\psi_{0,\sigma,k}}{1-\sigma}-\varphi_{0,k}+\psi_{0,k}\Big)\,\mathrm{d}S}
    \notag\\
    &\geq\int_\Omega |\nabla\varphi_{0,\sigma,k}^\sharp|^2\,\mathrm{d}x
    +\int_\Gamma|\nabla_{\!\boldsymbol{\tau}\,} \psi_{0,\sigma,k}^\sharp|^2\,\mathrm{d}S {+\chi(K)\int_\Gamma|\psi_{0,\sigma,k}^\sharp-\varphi_{0,\sigma,k}^\sharp|^2\,\mathrm{d}S}\notag\\
    &\quad+\kappa\|\boldsymbol{\varphi}^\sharp_{0,\sigma,k} \|_{\mathcal{L}^2}^2
    +k^{-1}\kappa\Big\|\frac{\boldsymbol{\varphi}_{0,\sigma,k}}{1-\sigma}-\boldsymbol{\varphi}_{0,k}\Big\|_{\mathcal{L}^2}^2
    {+\chi(K)\frac{\sigma}{1-\sigma}\int_\Gamma(\varphi_{0,\sigma,k}^\sharp-\psi_{0,\sigma,k}^\sharp)(\varphi_{0,\sigma,k}-\psi_{0,\sigma,k})\,\mathrm{d}S}
    \notag\\
    &\quad+\frac{\sigma}{1-\sigma}\int_\Omega\nabla\varphi_{0,\sigma,k}\cdot\nabla(\varphi_{0,\sigma,k}-\varphi_{0,k})\,\mathrm{d}x
    +\frac{\sigma}{1-\sigma}\int_\Gamma\nabla_{\!\boldsymbol{\tau}\,}\psi_{0,\sigma,k}\cdot\nabla_{\!\boldsymbol{\tau}\,}(\psi_{0,\sigma,k}-\psi_{0,k})\,\mathrm{d}S
    \notag\\
    &\quad+\frac{\sigma}{1-\sigma}\int_\Omega (f_0(\varphi_{0,\sigma,k})-f_0(\varphi_{0,k}))\varphi_{0,\sigma,k}\,\mathrm{d}x
    +\frac{\sigma}{1-\sigma}\int_\Gamma (g_0(\psi_{0,\sigma,k})-g_0(\psi_{0,k}))\psi_{0,\sigma,k}\,\mathrm{d}S.\label{20260301-2}
\end{align}
Here, $\kappa>0$ is the constant from \ref{ASS:A2}.
Using the estimates \eqref{0315-0} and \eqref{0315-0'}, we infer from \eqref{20260301-2} that
\begin{align}
    &\kappa \|\boldsymbol{\varphi}^\sharp_{0,\sigma,k}\|_{\mathcal{L}^2}^2
        + \|\nabla \varphi^\sharp_{0,\sigma,k}\|_{\mathbf{L}^2(\Omega)}^2
        + \|\nabla_{\!\boldsymbol{\tau}\,} \psi^\sharp_{0,\sigma,k}\|_{\mathbf{L}^2(\Gamma)}^2+\chi(K)\|\varphi^\sharp_{0,\sigma,k}-\psi^\sharp_{0,\sigma,k}\|_{H_\Gamma}^2\notag\\
     &\quad\leq  C\sigma\big(\|\boldsymbol{\varphi}_{0,\sigma,k}\|_{\mathcal{H}^1}^2+\|\boldsymbol{\varphi}_{0,k}\|_{\mathcal{H}^1}^2+\|f_0(\varphi_{0,\sigma,k})\|_{H}^2
     \notag\\
     &\qquad\qquad
     +\|f_{0}(\varphi_{0,k})\|_H^2+\|g_0(\psi_{0,\sigma,k})\|_{H_\Gamma}^2+\|g_{0}(\psi_{0,k})\|_{H_\Gamma}^2\big)\notag \\
    &\quad\leq  C_*\sigma(1+\|\widetilde{\mu}_0\|_{V}^2+\|\widetilde{\mathcal{L}}_0\|_{V_\Gamma}^2)\label{error}
\end{align}
for some constant $C_*>0$ independent of $\sigma$ and $k$. Hence, passing to the limit $\sigma\to0$ in \eqref{error},
we conclude the first convergence in \eqref{initial-convergence}.

To establish the convergence $\boldsymbol{\varphi}_{0,k}\to\boldsymbol{\varphi}_{0}$ in $\mathcal{H}^1$,
we use the notation
$\boldsymbol{\varphi}^\sharp_{0,k} = (\varphi^\sharp_{0,k},\psi^\sharp_{0,k})
\coloneqq \boldsymbol{\varphi}_{0,k}-\boldsymbol{\varphi}_0
\in \mathcal{H}^2$.
Recalling the definition of $\widetilde{\mu}_{0,k}$ and $\widetilde{\mathcal{L}}_{0,k}$,
we use \eqref{initial-elliptic} to deduce that
\begin{subequations}
    \begin{align}
        &-\Delta\varphi_{0,k}^\sharp+(1+k^{-1})(f_0(\varphi_{0,k})-f_0(\varphi_{0}))
    =h_k\circ\widetilde{\mu}_0-\widetilde{\mu}_0 -k^{-1}f_0(\varphi_0)
    &&\text{a.e.~in }\Omega,\label{diff-0317-2-1}\\
     & {
        \begin{cases}
            K\partial_\mathbf{n}\varphi_{0,k}^\sharp =\psi_{0,k}^\sharp-\varphi_{0,k}^\sharp, & K\in(0,+\infty)\\
            \partial_\mathbf{n}\varphi_{0,k}^\sharp=0,&K=+\infty
        \end{cases}} && {\text{a.e.~on }\Gamma,}\label{diff-0317-2-2}
        \\
    &-\Delta_{\boldsymbol{\tau}}\psi^\sharp_{0,k} +\partial_{\mathbf{n}}\varphi^\sharp_{0,k}
    +g_0(\psi_{0,k})-g_0(\psi_0)+k^{-1}(f_0(\psi_{0,k})-f_0(\psi_0))&&\\
    &\qquad=h_k\circ\widetilde{\mathcal{L}}_{0}-\widetilde{\mathcal{L}}_0 -k^{-1}f_0(\psi_0)
    &&\text{a.e.~on }\Gamma. \label{diff-0317-2-3}
    \end{align}
\end{subequations}
Testing \eqref{diff-0317-2-1} by $\varphi_{0,k}^\sharp$
and \eqref{diff-0317-2-3} by $\psi_{0,k}^\sharp$, using \eqref{diff-0317-2-2}, we have
\begin{align}
    &\int_\Omega|\nabla\varphi_{0,k}^\sharp|^2\,\mathrm{d}x
    +\int_\Gamma|\nabla_{\!\boldsymbol{\tau}\,}\psi^\sharp_{0,k}|^2\,\mathrm{d}S {+\chi(K)\int_\Gamma|\varphi_{0,k}^\sharp-\psi_{0,k}^\sharp|^2\,\mathrm{d}S}\notag\\
    &\qquad+(1+k^{-1})\kappa\Big(\int_\Omega|\varphi_{0,k}^\sharp|^2\,\mathrm{d}x
    +\int_\Gamma|\psi_{0,k}^\sharp|^2\,\mathrm{d}S\Big)\notag\\
    &\quad\leq \int_\Omega(h_k\circ\widetilde{\mu}_0-\widetilde{\mu}_0)\varphi_{0,k}^\sharp\,\mathrm{d}x
    +\int_\Gamma(h_k\circ\widetilde{\mathcal{L}}_{0}-\widetilde{\mathcal{L}}_0)\psi^\sharp_{0,k}\,\mathrm{d}S\notag\\
    &\qquad-k^{-1}\int_\Omega f_0(\varphi_0)\varphi_{0,k}^\sharp\,\mathrm{d}x-k^{-1}\int_\Gamma f_0(\psi_0)\psi_{0,k}^\sharp\,\mathrm{d}S\notag\\
    &\quad\leq \frac{\kappa}{2}\Big(\int_\Omega|\varphi_{0,k}^\sharp|^2\,\mathrm{d}x
    +\int_\Gamma|\psi_{0,k}^\sharp|^2\,\mathrm{d}S\Big)
    +\frac{1}{\kappa}\int_\Omega|h_k\circ\widetilde{\mu}_0-\widetilde{\mu}_0|^2\,\mathrm{d}x
    +\frac{1}{\kappa}\int_\Gamma|h_k\circ\widetilde{\mathcal{L}}_{0}-\widetilde{\mathcal{L}}_0|^2\,\mathrm{d}S\notag\\
    &\qquad+\frac{k^{-2}}{\kappa}\int_\Omega|f_0(\varphi_0)|^2\,\mathrm{d}x+\frac{k^{-2}}{\kappa}\int_\Gamma|f_0(\psi_0)|^2\,\mathrm{d}S.\label{0317-1}
\end{align}
By \ref{ASS:A3}, \ref{ASS:A4} and the Lebesgue dominated convergence theorem, we obtain
\begin{align*}
        &\frac{1}{2\kappa}\int_\Omega|h_k\circ\widetilde{\mu}_0-\widetilde{\mu}_0|^2\,\mathrm{d}x
        +\frac{1}{2\kappa}\int_\Gamma|h_k\circ\widetilde{\mathcal{L}}_{0}-\widetilde{\mathcal{L}}_0|^2\,\mathrm{d}S\\
        &\qquad+\frac{k^{-2}}{\kappa}\int_\Omega|f_0(\varphi_0)|^2\,\mathrm{d}x+\frac{k^{-2}}{\kappa}\int_\Gamma|f_0(\psi_0)|^2\,\mathrm{d}S
        \to0\quad\text{as }k\to+\infty.
\end{align*}
The above convergence result together with \eqref{0317-1} allows us to pass to the limit $k\to+\infty$ in
the right-hand side of \eqref{0317-1}. Consequently, we can conclude that
$\boldsymbol{\varphi}_{0,k}\to\boldsymbol{\varphi}_{0}$ in $\mathcal{H}^1$ as $k\to+\infty$.
This verifies the second convergence in \eqref{initial-convergence} and thus the proof of Proposition~\ref{approximate-initial-datum} is complete.
\end{proof}

\subsection{The approximate problem}
\label{SEC:3.2}
Let us consider the following bulk-surface Stokes operator, which was first introduced and analyzed in \cite[Section~5]{KS}:
\begin{align*}
    \boldsymbol{\mathcal{{A}}}:
    \,D(\boldsymbol{\mathcal{A}})\subset \boldsymbol{\mathcal{{L}}}_{\mathrm{div}}^2
    &\to\boldsymbol{ \mathcal{{L}}}_{\mathrm{div}}^2,\\
    (\mathbf{v},\mathbf{w})
    &\mapsto\big(\mathbf{P}^\Omega_{\mathrm{div}}(-\Delta\mathbf{v}),
    \mathbf{P}_{\mathrm{div}}^\Gamma( -\Delta_{\boldsymbol{\tau}}\mathbf{w}
    +2 [\mathbb{D}\mathbf{v}\,\mathbf{n}]_{\boldsymbol{\tau}}+\mathbf{w})\big),
\end{align*}
where
$D(\boldsymbol{\mathcal{{A}}})=  \boldsymbol{\mathcal{{H}}}_{0,\mathrm{div}}^2$. The operator $\boldsymbol{\mathcal{A}}$ is a {positive
self-adjoint} operator on $\boldsymbol{\mathcal{L}}_{\mathrm{div}}^2$.
We note that $\boldsymbol{\mathcal{A}}$ is bounded from above as
\begin{align*}
(\mathbf{v},\mathbf{w})\mapsto \|(\mathbf{v},\mathbf{w})\|_{\boldsymbol{\mathcal{H}}_{0,\mathrm{div}}^1}^2:=
\int_\Omega|\mathbb{D}\mathbf{v}|^2\,\mathrm{d}x
+\int_\Gamma |\mathbb{D}_{\boldsymbol{\tau}}\mathbf{w}|^2\,\mathrm{d}S
+\int_\Gamma|\mathbf{w}|^2\,\mathrm{d}S
\end{align*}
defines a norm on $D(\boldsymbol{\mathcal{A}})$
that is equivalent to the $\boldsymbol{\mathcal{H}}^1$-norm (see \cite[Lemma~5.1]{KS}).
In \cite[Theorem 5.9]{KS}, it was shown that
there exists a sequence of countable many positive eigenvalues $\{\lambda_k\}_{k\in\mathbb{N}}$
and the corresponding eigenfunctions
$
\{(\widetilde{\mathbf{v}}_k, \widetilde{\mathbf{w}}_k)\}_{k\in\mathbb{N}}\subset \boldsymbol{\mathcal{H}}_{0,\mathrm{div}}^2$, which satisfy
\begin{align*}
    &\int_\Omega \mathbb{D}\widetilde{\mathbf{v}}_k:\mathbb{D}\widehat{\mathbf{v}}\,\mathrm{d}x
    +\int_\Gamma \mathbb{D}_{\boldsymbol{\tau}}\widetilde{\mathbf{w}}_k:\mathbb{D}_{\boldsymbol{\tau}}\widehat{\mathbf{w}}\,\mathrm{d}S
    +\int_\Gamma \widetilde{\mathbf{w}}_k\cdot\widehat{\mathbf{w}}\,\mathrm{d}S\\
    &\quad=\lambda_k\int_\Omega \widetilde{\mathbf{v}}_k\cdot\widehat{\mathbf{v}}\,\mathrm{d}x
    +\lambda_k\int_\Gamma \widetilde{\mathbf{w}}_k\cdot\widehat{\mathbf{w}}\,\mathrm{d}S
    \qquad\text{for all $(\widehat{\mathbf{v}},\widehat{\mathbf{w}})\in \boldsymbol{\mathcal{H}}_{0,\mathrm{div}}^1\,$,}
\end{align*}
can be chosen in such a way that they form an orthonormal basis of $\boldsymbol{\mathcal{L}}_{\mathrm{div}}^2$.
We also note that
$(\widetilde{\mathbf{v}}_k,\widetilde{\mathbf{w}}_k)\in \boldsymbol{\mathcal{W}}^{2,r}_{0,\mathrm{div}}$
for all $r\in(1,+\infty)$ (see \cite[Theorem 5.7]{KS}) and the functions $\{(\widetilde{\mathbf{v}}_k, \widetilde{\mathbf{w}}_k)\}_{k\in\mathbb{N}}$ form a complete orthogonal system of $\boldsymbol{\mathcal{H}}^2_{0,\mathrm{div}}$ (see \cite[Corollary 5.11]{KS}). By definition, it holds $\widetilde{\mathbf{v}}_k =\widetilde{\mathbf{v}}_{k,\boldsymbol{\tau}}= \widetilde{\mathbf{w}}_k$ almost everywhere on $\Gamma$. For any $m\in\mathbb{N}^+$, we define the finite-dimensional subspace
\begin{align*}
\mathcal{\boldsymbol{V}}_m
=\text{span}\big\{(\widetilde{\mathbf{v}}_k,\widetilde{\mathbf{w}}_k)\big\}_{k=1}^m
\subset\boldsymbol{\mathcal{H}}_{0,\mathrm{div}}^1
\end{align*}
and the $\boldsymbol{\mathcal{L}}_{\mathrm{div}}^2$-orthogonal projection
onto $\mathcal{\boldsymbol{V}}_m$ is denoted by $\mathbb{P}_m=(\mathbb{P}_m^\Omega,\mathbb{P}_m^\Gamma)$, which is uniformly bounded on $\boldsymbol{\mathcal{L}}^2_{\mathrm{div}}$ and $\boldsymbol{\mathcal{H}}^1_{0,\mathrm{div}}$ in the following way:
\begin{align}
    \left\{\;
    \begin{aligned}
        \|\mathbb{P}_m(\mathbf{v},\mathbf{w})\|_{\boldsymbol{{\mathcal{L}}}^2}
        &\leq \|(\mathbf{v},\mathbf{w})\|_{\boldsymbol{{\mathcal{L}}}^2}
        &&\quad\text{for all $(\mathbf{v},\mathbf{w})\in \boldsymbol{\mathcal{L}}^2_{\mathrm{div}}$,}
        \\
        \|\mathbb{P}_m(\mathbf{v},\mathbf{w})\|_{\boldsymbol{\mathcal{H}}^1_{0,\mathrm{div}}}
        &\leq \|(\mathbf{v},\mathbf{w})\|_{\boldsymbol{\mathcal{H}}^1_{0,\mathrm{div}}}
        &&\quad\text{for all $(\mathbf{v},\mathbf{w})\in \boldsymbol{\mathcal{H}}^1_{0,\mathrm{div}}$.}
    \end{aligned}
    \right.\label{Pm-stability}
\end{align}
Moreover, as a direct consequence of the regularity estimate in \cite[Theorem 5.7]{KS},
the following inverse Sobolev embedding inequalities
hold for all $( \mathbf{v},\mathbf{w})\in \mathcal{V}_m$:
\begin{align}
    {\|( \mathbf{v},\mathbf{w} )\|_{\boldsymbol{\mathcal{H}}^1}
    \leq C_m \|( \mathbf{v},\mathbf{w} )\|_{\boldsymbol{\mathcal{L}}^2},
    \quad \|( \mathbf{v},\mathbf{w} )\|_{\boldsymbol{\mathcal{H}}^2}
    \leq C_m \|( \mathbf{v},\mathbf{w} )\|_{\boldsymbol{\mathcal{L}}^2}.}\label{inverse-So}
\end{align}


Now, let us fix $T>0$. For any $m,\, k\in\mathbb{N}^+$, $\gamma>0$ and $\sigma \in (0,1/2)$,
we intend to construct an approximate solution $(\mathbf{v}_{m},\mathbf{v}_{m,\boldsymbol{\tau}}, \boldsymbol{\varphi}_{m},\boldsymbol{\mu}_m)$
to the system \eqref{eqmain0new}--\eqref{icnew} with the regularities
\begin{align}
    \label{REG:APPROX}
    \left\{\,
    \begin{aligned}
	&( \mathbf{v}_m,\mathbf{v}_{m,\boldsymbol{\tau}} )\in C^{1}([0,T];\mathcal{V}_{m}),
    \\
	&\bm{\varphi}_{m}=(\varphi_m,\psi_m)\in L^{\infty}(0,T;\mathcal{H}^{3}),
    \\
    &\partial_{t}\bm{\varphi}_{m}\in L^{\infty}(0,T;\mathcal{H}^{1})\cap L^{2}(0,T;\mathcal{H}^2) \cap H^1(0,T;\mathcal{L}^2),
    \\
	&\varphi_{m}\in L^{\infty}(Q_T)
    \;\; \text{ with }
    \;\;\,|\varphi_{m}(x,t)|\leq (1-\sigma)(1-\delta )\ \
    \text{ a.e.~in }Q_T,
    \\
	&\psi_{m}\in L^{\infty}(\Sigma_T)
    \;\; \text{ with }
    \;\;\,|\psi_{m}(x,t)|\leq (1-\sigma)(1-\delta) \ \
    \text{ a.e.~on }\Sigma_T,
    \\
	&\bm{\mu}_m=(\mu_m,\mathcal{L}_m)\in L^\infty(0,T;\mathcal{H}^2) \cap H^1(0,T;\mathcal{H}^1),
    \end{aligned}
    \right.
\end{align}
for some $\delta= \delta(\sigma,k,\gamma)\in(0,1)$,
which satisfies the following weak formulation
\begin{align}
	& \int_\Omega \rho(\varphi_m)\partial_{t}\mathbf{v}_{m}\cdot\mathbf{w}\,\mathrm{d}x
    +\int_\Omega 2\alpha \mathbb{D}\partial_t\mathbf{v}_m:\mathbb{D}\mathbf{w}\,\mathrm{d}x
    \notag\\
    &\qquad +\int_\Omega 2\nu(\varphi_m )\mathbb{D}\mathbf{v}_m:\mathbb{D}\mathbf{w}\,\mathrm{d}x
	+\int_\Gamma \beta(\psi_m )\mathbf{v}_{m,\bm{\tau}}\cdot\mathbf{w}_{\bm{\tau}}\,\mathrm{d}S
	\notag \\[1mm]
	& \quad =\int_\Omega \mu_m \nabla \varphi_m\cdot\mathbf{w}\,\mathrm{d}x
	-\int_\Gamma  \psi_m\nabla_{\!\bm{\tau}\,}\mathcal{L}_m\cdot\mathbf{w}_{\bm{\tau}}\,\mathrm{d}S
    -\int_\Omega \rho(\varphi_m)(\mathbf{v}_m\cdot\nabla)\mathbf{v}_m\cdot\mathbf{w}\,\mathrm{d}x
    \notag\\
    &\qquad+\rho'\int_\Omega (\nabla\mu_m\cdot\nabla)\mathbf{v}_m\cdot\mathbf{w} \,\mathrm{d}x
    -\frac{\rho'}{2}\int_\Gamma (\nabla\mu_m\cdot\mathbf{n})\mathbf{v}_{m,\bm{\tau}}\cdot\mathbf{w}_{\bm{\tau}}\,\mathrm{d}S,
	\label{eqd1}
\end{align}
almost everywhere in $[0,T]$ for all $(\mathbf{w},\mathbf{w}_{\boldsymbol{\tau}})\in \mathcal{V}_{m}$, the equations
\begin{subequations}
    \label{approximating-solution}
    \begin{align}
	&\partial_{t}\varphi_m+\mathbf{v}_m \cdot\nabla\varphi_m=\Delta\mu_m & & \mbox{a.e.~in } Q_{T},
	\label{eqd2} \\
	& \mu_m =\gamma\partial_t \varphi_m-\Delta \varphi _m+f(\varphi _m)+k^{-1} f_\sigma(\varphi_m) & & \mbox{a.e.~in } Q_{T},
	\label{eqd4} \\
    &
    \begin{cases}
        K\partial_\mathbf{n}\varphi_m=\psi_m-\varphi_m,&K\in(0,+\infty)\\
        \partial_\mathbf{n}\varphi_m=0,&K=+\infty
    \end{cases}
    && {\mbox{a.e.~on }
	\Sigma _{T},}
	\label{eqd3''} \\
	&
	L\partial_{\mathbf{n}}\mu_m=\mathcal{L}_m-\mu_m,\quad L\in(0,+\infty)
    &&\mbox{a.e.~on }
	\Sigma _{T},
	\label{eqd3'} \\
	&\partial_t\psi_m+{\mathbf{v}_{m,\bm{\tau}}\cdot\nabla_{\!\bm{\tau}\,}\psi_m}
    =\Delta_{\bm{\tau}}\mathcal{L}_m-\partial_{\mathbf{n}}\mu_m && \mbox{a.e.~on }
	\Sigma _{T},
	\label{eqd3} \\
	& \mathcal{L} _m = \gamma\partial_t \psi_m- \Delta _{\bm{\tau}}\psi_m +\partial _{\mathbf{n}}\varphi_m + g(\psi_m)+k^{-1} f_\sigma(\psi_m)
    && \mbox{a.e.~on }
	\Sigma _{T},
	\label{eqd5}
\end{align}
\end{subequations}
and the initial conditions
\begin{align}
	\mathbf{v}_m|_{t=0}= \mathbb{P}^\Omega_m(\mathbf{v}_0,\mathbf{v}_{0,\boldsymbol{\tau}}) \quad\text{a.e.~in }\Omega,
    \quad(\varphi_m,\psi_m)|_{t=0}=(\varphi_{0,\sigma,k},\psi_{0,\sigma,k})\quad\text{a.e.~in }\Omega\times\Gamma.
    \label{approinitial}
\end{align}

\subsection{Existence of approximate solutions} \label{SECT:EXAP}
We establish the existence of approximate solutions to problem \eqref{eqd1}--\eqref{approinitial} on $[0,T]$ by a fixed-point argument (cf. \cite{Gior22}). To this aim, we fix $( \mathbf{v},\mathbf{v}_{\boldsymbol{\tau}} )\in H^1(0,T;\mathcal{V}_m)$
and consider the following convective viscous Cahn--Hilliard equation with dynamic boundary conditions
\begin{subequations}
\label{0321-eq-1}
    \begin{align}
    	&\partial_{t}\varphi_m+\mathbf{v} \cdot\nabla\varphi_m=\Delta\mu_m
        && \mbox{a.e.~in } Q_{T},\label{0321-eq-1-1}
	 \\
	 &\mu_m =\gamma\partial_t\varphi_m -\Delta \varphi _m+f(\varphi _m)+k^{-1} f_\sigma(\varphi_m)
     && \mbox{a.e.~in } Q_{T},\label{0321-eq-1-2}
	 \\
        &
    \begin{cases}
        K\partial_\mathbf{n}\varphi_m=\psi_m-\varphi_m,&K\in(0,+\infty) \\
        \partial_\mathbf{n}\varphi_m=0,&K=+\infty
    \end{cases}
    && \mbox{a.e.~on } \Sigma _{T},
	\label{0321-eq-1-3} \\
    &	L\partial_{\mathbf{n}}\mu_m=\mathcal{L}_m-\mu_m,\quad L\in(0,+\infty)
    &&\mbox{a.e.~on }
	\Sigma _{T},\label{0321-eq-1-4}
	 \\
	&\partial_t\psi_m+{\mathbf{v}_{\bm{\tau}}\cdot\nabla_{\!\bm{\tau}\,}\psi_m}=\Delta_{\bm{\tau}}\mathcal{L}_m-\partial_{\mathbf{n}}\mu_m
    && \mbox{a.e.~on }
	\Sigma _{T},\label{0321-eq-1-5}
	 \\
	 &\mathcal{L} _m = \gamma\partial_t\psi_m- \Delta _{\bm{\tau}}\psi_m +\partial _{%
		\mathbf{n}}\varphi_m+g(\psi_m)+k^{-1} f_\sigma(\psi_m)
        && \mbox{a.e.~on }
	\Sigma _{T},\label{0321-eq-1-6}
    \end{align}
\end{subequations}
subject to the initial conditions
\begin{align}
	(\varphi_m,\psi_m)|_{t=0}=(\varphi_{0,\sigma,k},\psi_{0,\sigma,k})\quad \text{ in }\Omega\times\Gamma.
    \label{appini}
\end{align}
According to Theorem~\ref{A1}, there exists a unique solution $(\boldsymbol{\varphi}_m,\boldsymbol{\mu}_m)$ to \eqref{0321-eq-1}--\eqref{appini}
satisfying
\begin{align}
	\begin{cases}
       \boldsymbol{\varphi}_m = (\varphi_m,\psi_m)\in L^{\infty}(0,T;\mathcal{H}^{3})\cap \big(C(\overline{Q_{T}})\times C(\Sigma_{T})\big),\\[1mm]
       \partial_t\boldsymbol{\varphi}_m \in L^{\infty}(0,T;\mathcal{H}^{1})\cap L^{2}(0,T;\mathcal{H}^2)\cap H^1(0,T;\mathcal{L}^2),\\[1mm]
       \varphi_m\in L^\infty(Q_T)\;\;\text{with}\;\;
       |\varphi_m(x,t)|\leq (1-\sigma)(1-\widetilde{\delta})\ \ \text{a.e.~in }Q_T,\\[1mm]
       \psi_m\in L^\infty(\Sigma_T)\;\;\text{with}\;\;
       |\psi_m(x,t)|\leq (1-\sigma)(1-\widetilde{\delta})\ \ \text{a.e.~on }\Sigma_T,\\[1mm]
       \boldsymbol{\mu}_m=(\mu_m,\mathcal{L}_m) \in L^{\infty}(0,T;\mathcal{H}^2)\cap H^{1}(0,T;\mathcal{H}^1),
	\end{cases}\label{strong-regu}
\end{align}
for some constant $\widetilde{\delta}= \widetilde{\delta}(\sigma,k,\gamma) \in (0,1)$.
It is straightforward to check that this weak solution also satisfies the energy estimate
\begin{align}
	&\sup_{s\in[0,t]}E^{\sigma,k}_{\text{free}}(\boldsymbol{\varphi}_m(s))
    +\frac{1}{2}\int_0^t \int_\Omega |\nabla\mu_m(s)|^2\,\mathrm{d}x\,\mathrm{d}s
    +\frac{1}{2}\int_0^t \int_\Gamma |\nabla_{\!\bm{\tau}\,}\mathcal{L}_m(s)|^2\,\mathrm{d}S \,\mathrm{d}s
    \notag\\
	&\qquad+\frac{1}{L}\int_{0}^t \int_\Gamma (\mathcal{L}_m(s)-\mu_m(s))^2\,\mathrm{d}S\,\mathrm{d}s
    +\gamma\int_0^t\int_\Omega|\partial_t\varphi_m(s)|^2 \,\mathrm{d}x\,\mathrm{d}s
    +\gamma\int_0^t\int_\Gamma|\partial_t\psi_m(s)|^2 \,\mathrm{d}S\,\mathrm{d}s
    \notag\\
	&\quad\leq E_{\text{free}}^{\sigma,k}(\boldsymbol{\varphi}_{0,\sigma,k})
    +\frac{1}{2}\int_0^T\|\mathbf{v}(s)\|_{\mathbf{L}^2(\Omega)}^2\,\mathrm{d}s
    +\frac{1}{2}\int_0^T\|\mathbf{v}_{\bm{\tau}}(s)\|_{\mathbf{L}^2(\Gamma)}^2\,\mathrm{d}s,
    \label{approenergy}
\end{align}
for all $t\in[0,T]$. Here, the free energy functional $E_{\text{free}}^{\sigma,k}$
associated with system \eqref{0321-eq-1}--\eqref{appini} is given by
\begin{align*}
E_{\text{free}}^{\sigma,k}(\boldsymbol{\varphi}_m)=
E_{\text{free}}(\boldsymbol{\varphi}_m)
+k^{-1}\int_\Omega F_\sigma(\varphi_m)\,\mathrm{d}x
+k^{-1}\int_\Gamma F_\sigma(\psi_m)\,\mathrm{d}S.
\end{align*}
The assumption \ref{ASS:A2} implies that there exists a positive constant $C_\mathrm{b}$
such that $E_{\text{free}}^{\sigma,k}(\boldsymbol{\varphi}_m)\geq -C_\text{b}$.

Next, we look for a spatially discrete solution
$$
\mathbf{v}_m (x,t)=\sum_{i=1}^{m}a_{i}^{m}(t)\widetilde{\mathbf{v}}_i (x),
\quad\text{with}\ \ \ \mathbf{v}_{m,\boldsymbol{\tau}} (x,t)=\sum_{i=1}^{m}a_{i}^{m}(t)\widetilde{\mathbf{w}}_i (x)
=\sum_{i=1}^{m}a_{i}^{m}(t)\widetilde{\mathbf{v}}_{i,\boldsymbol{\tau}} (x)
$$
to the following weak formulation
\begin{align}
	& \int_\Omega \rho(\varphi_m)\partial_{t}\mathbf{v}_{m}\cdot\widetilde{\mathbf{v}}_j\,\mathrm{d}x
    +\int_\Omega 2\alpha \mathbb{D}\partial_t\mathbf{v}_m:\mathbb{D}\widetilde{\mathbf{v}}_j\,\mathrm{d}x\notag\\
    &\qquad+\int_\Omega 2\nu(\varphi_m )\mathbb{D}\mathbf{v}_m:\mathbb{D}\widetilde{\mathbf{v}}_j\,\mathrm{d}x
	+\int_\Gamma \beta(\psi_m )\mathbf{v}_{m,\bm{\tau}}\cdot\widetilde{\mathbf{v}}_{j,\boldsymbol{\tau}}\,\mathrm{d}S
	\notag \\[1mm]
	& \quad =\int_\Omega \mu_m \nabla \varphi_m\cdot\widetilde{\mathbf{v}}_j\,\mathrm{d}x
	-\int_\Gamma  \psi_m\nabla_{\!\bm{\tau}\,}\mathcal{L}_m\cdot\widetilde{\mathbf{v}}_{j,\boldsymbol{\tau}}\,\mathrm{d}S
    -\int_\Omega \rho(\varphi_m)(\mathbf{v}\cdot\nabla)\mathbf{v}_m\cdot\widetilde{\mathbf{v}}_j\,\mathrm{d}x
    \notag\\
    &\qquad+\rho'\int_\Omega (\nabla\mu_m\cdot\nabla)\mathbf{v}_m\cdot\widetilde{\mathbf{v}}_j\,\mathrm{d}x
     -\frac{\rho'}{2}\int_\Gamma (\nabla\mu_m\cdot\mathbf{n})\mathbf{v}_{m,\bm{\tau}}\cdot\widetilde{\mathbf{v}}_{j,\boldsymbol{\tau}}\,\mathrm{d}S
	\label{galerkineq1}
\end{align}
for $j\in\{1,...,m\}$. Moreover, $\mathbf{v}_m$ satisfies the initial condition
$$
\mathbf{v}_m|_{t=0}= \mathbb{P}_m^\Omega(\mathbf{v}_{0}, \mathbf{v}_{0,\boldsymbol{\tau}})\quad\text{a.e. in }\Omega.
$$
Defining $\mathbf{A}^m (t)=(a_{1}^m (t),...,a_m^m (t))^{\top}$, we can rewrite \eqref{galerkineq1} into the system of ordinary differential equations
\begin{align}
     \mathbf{M}^m (t)\frac{\mathrm{d}}{\mathrm{d}t}\mathbf{A}^m(t)
     =\mathbf{L}^m (t)\mathbf{A}^m (t)+\mathbf{G}^m (t),
     \label{ode}
\end{align}
where the matrices $\mathbf{M}^m (t)$, $\mathbf{L}^m (t)$ and the vector $\mathbf{G}^m (t)$ are given by
\begin{align*}
	(\mathbf{M}^m (t))_{i,j}&=\int_\Omega\rho(\varphi_m) \widetilde{\mathbf{v}}_i \cdot \widetilde{\mathbf{v}}_j\,\mathrm{d}x
    +2\alpha\int_\Omega \mathbb{D}\widetilde{\mathbf{v}}_i: \mathbb{D}\widetilde{\mathbf{v}}_j\,\mathrm{d}x,
    \\
	(\mathbf{L}^m (t))_{i,j}&=-\int_\Gamma\beta(\psi_m)\widetilde{\mathbf{v}}_{i,\bm{\tau}}\cdot\widetilde{\mathbf{v}}_{j,\bm{\tau}}\,\mathrm{d}S
    -2\int_\Omega \nu(\varphi_m)\mathbb{D}\widetilde{\mathbf{v}}_i:\mathbb{D}\widetilde{\mathbf{v}}_j\,\mathrm{d}x
    \\
	&\quad-\int_\Omega (\rho(\varphi_m)\mathbf{v}\cdot\nabla)\widetilde{\mathbf{v}}_i\cdot\widetilde{\mathbf{v}}_j\,\mathrm{d}x
    +\rho'\int_\Omega(\nabla\mu_m\cdot\nabla)\widetilde{\mathbf{v}}_i\cdot\widetilde{\mathbf{v}}_j\,\mathrm{d}x
    \\
    &\quad -\frac{\rho'}{2}\int_\Gamma (\nabla\mu_m\cdot\mathbf{n})\widetilde{\mathbf{v}}_{i,\bm{\tau}}\cdot\widetilde{\mathbf{v}}_{j,\bm{\tau}}\,\mathrm{d}S,
    \\
	(\mathbf{G}^{m}(t))_i&=\int_{\Omega}\mu_{m}\nabla\varphi_m\cdot\widetilde{\mathbf{v}}_i\,\mathrm{d}x
    -\int_\Gamma \psi_m\nabla_{\!\bm{\tau}\,}\mathcal{L}_m \cdot\widetilde{\mathbf{v}}_{i,\bm{\tau}}\,\mathrm{d}S,
\end{align*}
and the initial condition reduces to
\begin{align}
\mathbf{A}^m(0)= \big( (\mathbb{P}_m (\mathbf{v}_0,\mathbf{v}_{0,\boldsymbol{\tau}}),(\widetilde{\mathbf{v}}_1,\widetilde{\mathbf{w}}_1) )_{\boldsymbol{\mathcal{{L}}}_{\mathrm{div}}^2},
..., ( (\mathbb{P}_m (\mathbf{v}_0,\mathbf{v}_{0,\boldsymbol{\tau}}),(\widetilde{\mathbf{v}}_m,\widetilde{\mathbf{w}}_m))_{\boldsymbol{\mathcal{{L}}}_{\mathrm{div}}^2}\big)^\top.
\label{ode-ini}
\end{align}
Thanks to \eqref{strong-regu}, it follows that
$\varphi_m\in C([0,T];W^{1,4}(\Omega))$, $\psi_m\in  C([0,T];W^{1,4}(\Gamma))$.
This, in turn, implies that
$$
\rho(\varphi_m),\,\nu(\varphi_m)\in C(\overline{\Omega}\times[0,T]),
\quad \beta(\psi_m)\in C(\Gamma\times[0,T]).
$$
Using \eqref{strong-regu}, we find  $\boldsymbol{\mu}_m \in C([0,T];\mathcal{H}^1)$. Then from the definition of $\mathbf J$ (see \eqref{DEF:JRN}) and \eqref{0321-eq-1-4}, we infer that $\mathbf{J}\cdot\mathbf{n} \in C([0,T];H_\Gamma)$.
We also recall the fact that
$(\mathbf{v},\mathbf{v}_{\boldsymbol{\tau}})\in C([0,T];\mathcal{V}_m)$. Consequently, it follows that $\mathbf{M}^m$ and $\mathbf{L}^m$ belong to $C([0,T];\mathbb{R}^{m\times m})$
and $\mathbf{G}^m \in C([0,T];\mathbb{R}^m)$.
Furthermore, like in \cite[Appendix A]{GT},
we can show that the matrix $\mathbf{M}^m(\cdot)$ is positive definite on $[0,T]$. Indeed, for any fixed $t\in[0,T]$ and a vector $\mathbf{a}\in\mathbb{R}^m\setminus\{\boldsymbol{0}\}$, from \eqref{v-trace} we infer that
 \begin{align*}	\mathbf{a}^\top\mathbf{M}^m(t)\mathbf{a} &=\sum_{i=1}^m\sum_{j=1}^m \Big(\int_\Omega\rho(\varphi_m)a_i\widetilde{\mathbf{v}}_i\cdot a_j\widetilde{\mathbf{v}}_j\,\mathrm{d}x
 +\alpha\int_\Omega a_i \mathbb{D}\widetilde{\mathbf{v}}_i: a_j \mathbb{D}\widetilde{\mathbf{v}}_j\,\mathrm{d}x\Big)
 \\	&=\int_\Omega\rho(\varphi_m(t))|\boldsymbol{\xi}|^2\,\mathrm{d}x+\alpha\int_\Omega|\mathbb{D}\boldsymbol{\xi}|^2\,\mathrm{d}x
 \\	&\geq\rho_2\|\boldsymbol{\xi}\|_{\mathbf{L}^2(\Omega)}^2 +\alpha\|\mathbb{D}\boldsymbol{\xi}\|_{\mathbf{L}^2(\Omega)}^2
 \\
 &\geq C(\rho_2,\alpha,\Omega) \|( \boldsymbol{\xi},\boldsymbol{\xi}_{\boldsymbol{\tau}})\|_{\boldsymbol{\mathcal{L}}^2_{\mathrm{div}}}^2= C(\rho_2,\alpha,\Omega)|\mathbf{a}|_{\mathbb{R}^m}^2>0,
\end{align*}
where $\boldsymbol{\xi}=\sum_{i=1}^m a_i\widetilde{\mathbf{v}}_i$ with $\boldsymbol{\xi}_{\boldsymbol{\tau}}=\sum_{i=1}^m a_i\widetilde{\mathbf{v}}_{i,\boldsymbol{\tau}}$.
As a consequence, $\mathbf{M}^m(t)$ is positive definite on $[0,T]$. By the continuity of $\rho(\varphi_m)$, the inverse matrix further satisfies $(\mathbf{M}^{m})^{-1}\in C([0,T];\mathbb{R}^{m\times m})$.
Hence, by the classical existence and uniqueness theorem for systems of linear ODEs, there exists a unique vector $\mathbf{A}^m \in C^{1}([0,T];\mathbb{R}^m)$ that solves the initial value problem \eqref{ode}--\eqref{ode-ini} on $[0,T]$.
This yields that problem \eqref{galerkineq1} admits a unique solution $(\mathbf{v}_m,\mathbf{v}_{m,\boldsymbol{\tau}}) \in C^1([0,T];\mathcal{V}_m)$.

Multiplying \eqref{galerkineq1} by $a_{j}^m$ and summing over $j$, we find
\begin{align}
&\frac{1}{2}\frac{\mathrm{d}}{\mathrm{d}t}\Big(\int_\Omega \rho(\varphi_m)|\mathbf{v}_m|^2\,\mathrm{d}x
+2\alpha\int_\Omega|\mathbb{D}\mathbf{v}_m|^2\,\mathrm{d}x\Big)
+\int_\Omega 2\nu(\varphi_m)|\mathbb{D}\mathbf{v}_m|^2\,\mathrm{d}x
+\int_\Gamma\beta(\psi_m)|\mathbf{v}_{m,\bm{\tau}}|^2\,\mathrm{d}S
\notag\\
&\quad=-\int_\Omega\varphi_m\nabla\mu_m\cdot\mathbf{v}_m\,\mathrm{d}x
-\int_\Gamma \psi_m\nabla_{\!\bm{\tau}\,}\mathcal{L}_m\cdot\mathbf{v}_{m,\bm{\tau}}\,\mathrm{d}S
-\int_\Omega (\rho(\varphi_m)\mathbf{v}\cdot\nabla)\mathbf{v}_m\cdot\mathbf{v}_m\,\mathrm{d}x
\notag\\
&\qquad+\rho'\int_\Omega(\nabla\mu_m\cdot\nabla)\mathbf{v}_m\cdot\mathbf{v}_m\,\mathrm{d}x
 -\frac{\rho'}{2}\int_\Gamma (\nabla\mu_m\cdot\mathbf{n})|\mathbf{v}_{m,\bm{\tau}}|^2\,\mathrm{d}S
 +\frac{\rho'}{2}\int_\Omega \partial_t\varphi_m|\mathbf{v}_m|^2 \,\mathrm{d}x.
 \notag
\end{align}
Similar to \cite[Formula~(6.18)]{KS}, we further obtain the identity
\begin{align}
&\frac{1}{2}\frac{\mathrm{d}}{\mathrm{d}t}\Big(\int_\Omega \rho(\varphi_m)|\mathbf{v}_m|^2\,\mathrm{d}x
+2\alpha\int_\Omega|\mathbb{D}\mathbf{v}_m|^2\,\mathrm{d}x\Big)
+\int_\Omega 2\nu(\varphi_m)|\mathbb{D}\mathbf{v}_m|^2\,\mathrm{d}x
+\int_\Gamma\beta(\psi_m)|\mathbf{v}_{m,\bm{\tau}}|^2\,\mathrm{d}S
\notag\\
&\quad=-\int_\Omega\varphi_m\nabla\mu_m\cdot\mathbf{v}_m\,\mathrm{d}x
-\int_\Gamma \psi_m\nabla_{\!\bm{\tau}\,}\mathcal{L}_m\cdot\mathbf{v}_{m,\bm{\tau}}\,\mathrm{d}S.
\label{uni-ga-1}
\end{align}
Due to \ref{ASS:A1}, it holds
\begin{align}
	\int_\Omega 2\nu(\varphi_m)|\mathbb{D}\mathbf{v}_m|^2\,\mathrm{d}x
    +\int_\Gamma\beta(\psi_m) |\mathbf{v}_{m,\bm{\tau}}|^2\,\mathrm{d}S
    &\geq 2\nu_\ast \|\mathbb{D}\mathbf{v}_m\|_{\mathbf{L}^2(\Omega)}^2
    +\beta_\ast\|\mathbf{v}_{m,\bm{\tau}} \|_{\mathbf{L}^2(\Gamma)}^{2}.
    \label{uni-ga-4}
\end{align}
For terms on the right-hand side of \eqref{uni-ga-1},
we use the Poincar\'e inequality \eqref{bsP}, Korn's inequality \eqref{Korn}, and Young's inequality to derive the following estimate
\begin{align}
    &-\int_\Omega \varphi_m\nabla\mu_m\cdot\mathbf{v}_m\,\mathrm{d}x
    -\int_\Gamma \psi_m\nabla_{\!\bm{\tau}\,}\mathcal{L}_m\cdot\mathbf{v}_{m,\bm{\tau}}\,\mathrm{d}S
    \notag\\
    &\quad\leq \|\varphi_m\|_{L^\infty(\Omega)}\|\nabla\mu_m\|_{\mathbf{L}^2(\Omega)}
    \|\mathbf{v}_m\|_{\mathbf{L}^2(\Omega)}
    +\|\psi_m\|_{L^\infty(\Gamma)} \|\nabla_{\!\bm{\tau}\,}\mathcal{L}_m\|_{\mathbf{L}^2(\Gamma)}
    \|\mathbf{v}_{m,\bm{\tau}}\|_{\mathbf{L}^2(\Gamma)}
    \notag\\[1mm]
    &\quad\leq \widetilde{C}_\text{P}\|\nabla\mu_m\|_{\mathbf{L}^2(\Omega)}
    (\|\nabla\mathbf{v}_m\|_{\mathbf{L}^2(\Omega)}
    +\|\mathbf{v}_{m,\bm{\tau}}\|_{\mathbf{L}^2(\Gamma)})
    +\|\nabla_{\!\bm{\tau}\,} \mathcal{L}_m\|_{\mathbf{L}^2(\Gamma)}
    \|\mathbf{v}_{m,\bm{\tau}}\|_{\mathbf{L}^2(\Gamma)}
    \notag\\
    &\quad\leq \widetilde{C}_\text{P}\widetilde{C}_\text{K}
        \|\nabla\mu_m\|_{\mathbf{L}^2(\Omega)}
        \|\mathbb{D}\mathbf{v}_m\|_{\mathbf{L}^2(\Omega)}
    \notag\\
    &\qquad
    + \Big(\widetilde{C}_\text{P}\widetilde{C}_\text{K}
        \|\nabla\mu_m\|_{\mathbf{L}^2(\Omega)}
        + \widetilde{C}_\text{P}\|\nabla\mu_m\|_{\mathbf{L}^2(\Omega)}
    +\|\nabla_{\!\bm{\tau}\,} \mathcal{L}_m\|_{\mathbf{L}^2(\Gamma)}
    \Big)
    \|\mathbf{v}_{m,\bm{\tau}}\|_{\mathbf{L}^2(\Gamma)}
    \notag\\
    &\quad\leq \frac{\min\{2\nu_\ast,\beta_\ast\}}{2} \big( \|\mathbb{D}\mathbf{v}_m\|_{\mathbf{L}^2(\Omega)}
    + \|\mathbf{v}_{m,\bm{\tau}}\|_{\mathbf{L}^2(\Gamma)} \big)^2
    + C_* \Big(\|\nabla\mu_m\|_{\mathbf{L}^2(\Omega)}^2
        +\|\nabla_{\!\bm{\tau}\,} \mathcal{L}_m\|_{\mathbf{L}^2(\Gamma)}^2
    \Big)
    \notag\\
    &\quad\leq \nu_* \|\mathbb{D}\mathbf{v}_m\|_{\mathbf{L}^2(\Omega)}^2
    + \frac{\beta_*}{2}\|\mathbf{v}_{m,\bm{\tau}}\|_{\mathbf{L}^2(\Gamma)}^2
    + C_* \Big(\|\nabla\mu_m\|_{\mathbf{L}^2(\Omega)}^2
        +\|\nabla_{\!\bm{\tau}\,} \mathcal{L}_m\|_{\mathbf{L}^2(\Gamma)}^2
    \Big),
    \label{0227-4}
\end{align}
where the constant $C_*>0$ depends only on $\widetilde{C}_\text{P}$, $\widetilde{C}_\text{K}$,
$\nu_*$ and $\beta_*$, but is independent of $m$.
Combining \eqref{uni-ga-1}, \eqref{uni-ga-4} and \eqref{0227-4}, we obtain
\begin{align}
    &\frac{1}{2}\frac{\mathrm{d}}{\mathrm{d}t}\Big(\int_\Omega\rho(\varphi_m)|\mathbf{v}_m|^2\,\mathrm{d}x
    +2\alpha\int_\Omega|\mathbb{D}\mathbf{v}_m|^2 \,\mathrm{d}x\Big)
    +\nu_\ast\|\mathbb{D}\mathbf{v}_m\|_{\mathbf{L}^2(\Omega)}^2
    +\frac{\beta_\ast}{2}\|\mathbf{v}_{m,\bm{\tau}}\|_{\mathbf{L}^2(\Gamma)}^{2}
    \notag\\[1mm]
    &\quad\leq C_* \Big(\|\nabla\mu_m\|_{\mathbf{L}^2(\Omega)}^2
        +\|\nabla_{\!\bm{\tau}\,} \mathcal{L}_m\|_{\mathbf{L}^2(\Gamma)}^2
    \Big).\label{0227-5}
\end{align}
From \eqref{approenergy} and \eqref{0227-5}, we infer that
\begin{align}
    &\sup_{s\in[0,t]}\Big(\int_\Omega \rho(\varphi_m(s))|\mathbf{v}_m(s)|^2\,\mathrm{d}x
    +2\alpha\int_\Omega|\mathbb{D} \mathbf{v}_m(s)|^2\,\mathrm{d}x\Big)
    \notag\\
    &\qquad+2\nu_\ast\int_0^t \|\mathbb{D}\mathbf{v}_m(s)\|_{\mathbf{L}^2(\Omega)}^2 \,\mathrm{d}s
    +\beta_\ast\int_0^t \|\mathbf{v}_{m,\bm{\tau}}(s)\|_{\mathbf{L}^2(\Gamma)}^2\,\mathrm{d}s
    \notag\\
    &\quad\leq \int_\Omega \rho(\varphi_{0,\sigma,k})|\mathbb{P}_m^\Omega(\mathbf{v}_0,\mathbf{v}_{0,\boldsymbol{\tau}})|^2\,\mathrm{d}x
    +2\alpha\int_\Omega|\mathbb{D} \mathbb{P}_m^\Omega(\mathbf{v}_0,\mathbf{v}_{0,\boldsymbol{\tau}})|^2\,\mathrm{d}x
    \notag\\
    &\qquad+4C_*
    \left(C_{\text{b}}+E_{\text{free}}^{\sigma,k}(\boldsymbol{\varphi}_{0,\sigma,k})
    +\int_0^t\|(\mathbf{v}(s),\mathbf{v}_{\boldsymbol{\tau}}(s)\|_{\boldsymbol{\mathcal{L}}^2}^2\,\mathrm{d}s
    \right),\label{uni-ga-5'}
\end{align}
which, together with \eqref{Pm-stability}, implies that
\begin{align}
	\|\mathbf{v}_m(t)\|_{\mathbf{L}^2(\Omega)}^2
    &\leq \frac{\rho_1}{\rho_2}\|(\mathbf{v}_0,\mathbf{v}_{0,\boldsymbol{\tau}})\|_{\boldsymbol{\mathcal{L}}^2}^2
    +\frac{2\alpha}{\rho_2}\|(\mathbf{v}_0,\mathbf{v}_{0,\boldsymbol{\tau}})\|_{\boldsymbol{\mathcal{H}}^1_{0,\mathrm{div}}}^2
    \notag\\
    &\quad+\frac{4C_*}{\rho_2}\Big(C_\text{b}+E_{\text{free}}^{\sigma,k}
    (\boldsymbol{\varphi}_{0,\sigma,k})
    +\|(\mathbf{v},\mathbf{v}_{\boldsymbol{\tau}})\|_{L^2(0,t;\bm{\mathcal{L}}^2_\mathrm{div})}^2\Big),
    \label{0305-1}\\
    \|\mathbb{D}\mathbf{v}_m(t)\|_{\mathbf{L}^2(\Omega)}^2
    &\leq  \frac{\rho_1}{2\alpha}\|(\mathbf{v}_0,\mathbf{v}_{0,\boldsymbol{\tau}})\|_{\boldsymbol{\mathcal{L}}^2}^2
    +\|(\mathbf{v}_0,\mathbf{v}_{0,\boldsymbol{\tau}})\|_{\boldsymbol{\mathcal{H}}^1_{0,\mathrm{div}}}^2
    \notag\\
    &\quad+\frac{2C_*}{\alpha}
    \Big(C_\text{b}+E_{\text{free}}^{\sigma,k}(\boldsymbol{\varphi}_{0,\sigma,k})
    +\|(\mathbf{v}, \mathbf{v}_{\boldsymbol{\tau}})\|_{L^2(0,t;\bm{\mathcal{L}}^2_\mathrm{div})}^2\Big).
    \label{0305-2}
\end{align}
Recalling \eqref{v-trace}, we infer from \eqref{0305-1} and \eqref{0305-2} that
    \begin{align}
     \|\mathbf{v}_{m,\bm{\tau}}(t)\|_{\mathbf{L}^2(\Gamma)}^2&\leq 2\widetilde{C}_0^2(\|\mathbb{D}\mathbf{v}_m(t)\|_{\mathbf{L}^2(\Omega)}^2+\|\mathbf{v}_m(t)\|_{\mathbf{L}^2(\Omega)}^2)
     \notag\\
    &\leq  \widetilde{C}_1\Big(\|(\mathbf{v}_0,\mathbf{v}_{0,\boldsymbol{\tau}})\|_{\boldsymbol{\mathcal{L}}^2}^2
    +\|(\mathbf{v}_0,\mathbf{v}_{0,\boldsymbol{\tau}})\|_{{\boldsymbol{\mathcal{H}}^1_{0,\mathrm{div}}}}^2\Big)
    \notag\\
    &\quad+ \widetilde{C}_2\Big(C_\text{b}+E_{\text{free}}^{\sigma,k}
    (\boldsymbol{\varphi}_{0,\sigma,k})+\|(\mathbf{v},\mathbf{v}_{\boldsymbol{\tau}}) \|_{L^2(0,t;\bm{\mathcal{L}}^2_\mathrm{div})}^2\Big),
    \label{0305-3}
\end{align}
where
\begin{align*}
    &\widetilde{C}_1 \coloneqq 2(\widetilde{C}_0^2+1) \max\Big\{\frac{\rho_1}{\rho_2}+\frac{\rho_1}{2\alpha},\frac{2\alpha}{\rho_2}+1\Big\},\qquad\widetilde{C}_2 \coloneqq 4(\widetilde{C}_0^2+1)C_\ast\Big(\frac{2}{\rho_2}+\frac{1}{\alpha}\Big).
\end{align*}
At this point, we set
\begin{align*}
&C_1=2\widetilde{C}_1\big(\|(\mathbf{v}_0,\mathbf{v}_{0,\boldsymbol{\tau}})\|_{\boldsymbol{\mathcal{L}}^2}^2
    +\|(\mathbf{v}_0,\mathbf{v}_{0,\boldsymbol{\tau}})\|_{{\boldsymbol{\mathcal{H}}^1_{0,\mathrm{div}}}}^2\big)
+2\widetilde{C}_2 \big(C_\text{b}+E_{\text{free}}^{\sigma,k}(\boldsymbol{\varphi}_{0,\sigma,k})\big),
\quad C_2=2\widetilde{C}_2,
\end{align*}
and assuming
\begin{align}
    \|(\mathbf{v},\mathbf{v}_{\boldsymbol{\tau}})\|_{L^2(0,t;\bm{\mathcal{L}}^2_\mathrm{div})}^2\leq C_3 e^{C_2t},\quad\text{for all}\, t\in[0,T],\label{assume1}
\end{align}
where $C_3=C_1 T$. We deduce from \eqref{0305-1}, \eqref{0305-3} and \eqref{assume1} that
\begin{align}
    \int_0^t \|(\mathbf{v}_{m}(s),\mathbf{v}_{m,\boldsymbol{\tau}}(s))\|_{\bm{\mathcal{L}}^2_{\mathrm{div}}}^2\,\mathrm{d}s
    \leq C_3+C_2\int_0^t\int_0^s\|(\mathbf{v},\mathbf{v}_{\boldsymbol{\tau}})\|_{\bm{\mathcal{L}}^2_{\mathrm{div}}}^2\,\mathrm{d}\tau\,\mathrm{d}s
    \leq C_3 e^{C_2t}\label{0305-4}
\end{align}
for all $t\in[0,T]$.
Furthermore, thanks to \eqref{0305-1} and \eqref{0305-2}, we also infer that
\begin{align}
    \sup_{t\in[0,T]}(\|\mathbf{v}_m(t)\|_{\mathbf{L}^2(\Omega)}^2
    +\|\mathbb{D}\mathbf{v}_m(t)\|_{\mathbf{L}^2(\Omega)}^2)
    \leq 2C_1+C_2C_3e^{C_2T} \eqqcolon K_0.
    \label{0305-5}
\end{align}
Next, we estimate the time derivative of $\mathbf{v}_m$.
To this end, multiplying \eqref{galerkineq1} by $\frac{\mathrm{d}}{\mathrm{d}t}a_{j}^m$ and summing over $j$,
we obtain
\begin{align}
&\int_\Omega\rho(\varphi_m)|\partial_t\mathbf{v}_m|^2\,\mathrm{d}x
+2\alpha\int_\Omega|\mathbb{D}\partial_t\mathbf{v}_m|^2\,\mathrm{d}x
\notag\\
&\quad=-2\int_\Omega\nu(\varphi_m)\mathbb{D}\mathbf{v}_m: \mathbb{D}\partial_t \mathbf{v}_m\,\mathrm{d}x
-\int_\Omega \varphi_m\nabla\mu_m \cdot\partial_t\mathbf{v}_m\,\mathrm{d}x-\int_\Omega (\rho(\varphi_m)\mathbf{v}\cdot\nabla)\mathbf{v}_m\cdot\partial_t \mathbf{v}_m\,\mathrm{d}x
\notag\\
&\qquad+\rho'\int_\Omega(\nabla\mu_m\cdot\nabla)\mathbf{v}_m\cdot\partial_t\mathbf{v}_m\,\mathrm{d}x -\frac{\rho'}{2}\int_\Gamma (\nabla\mu_m\cdot\mathbf{n})\mathbf{v}_{m,\bm{\tau}}\cdot\partial_t\mathbf{v}_{m,\bm{\tau}}\,\mathrm{d}S
\notag\\
&\qquad-\int_\Gamma \beta(\psi_m)\mathbf{v}_{m,\bm{\tau}}\cdot\partial_t \mathbf{v}_{m,\bm{\tau}}\,\mathrm{d}S-\int_\Gamma\psi_m\nabla_{\!\bm{\tau}\,}\mathcal{L}_m\cdot\partial_t\mathbf{v}_{m,\bm{\tau}}\,\mathrm{d}S
\notag\\
&\quad \eqqcolon \sum_{j=1}^7 I_j.
\label{uni-ti-1}
\end{align}
By H\"older's inequality and \eqref{inverse-So},
the terms $I_j$ $(1\leq j\leq 5)$ can be estimated as follows:
\begin{align}
I_1
&\leq\frac{\alpha}{4}\int_\Omega|\mathbb{D}\partial_t\mathbf{v}_m|^2\,\mathrm{d}x
+\frac{4(\nu^\ast)^2}{\alpha}\int_\Omega|\mathbb{D}\mathbf{v}_m|^2\,\mathrm{d}x,\notag\\[1ex]
I_2
&\leq \| \varphi_m\|_{L^\infty(\Omega)}\|\nabla\mu_m\|_{\mathbf{L}^2(\Omega)} \|\partial_t\mathbf{v}_m\|_{\mathbf{L}^2(\Omega)}\notag\\
&\leq \frac{\rho_2}{8}\int_\Omega|\partial_t\mathbf{v}_m|^2\,\mathrm{d}x+\frac{2}{\rho_2}\int_\Omega|\nabla\mu_m|^2\,\mathrm{d}x,\notag\\[1ex]
I_3
&\leq \|\rho(\varphi_m)\|_{L^\infty(\Omega)}\|\mathbf{v}\|_{\mathbf{L}^2(\Omega)}
\|\nabla\mathbf{v}_m\|_{\mathbf{L}^6(\Omega)}\|\partial_t\mathbf{v}_m\|_{\mathbf{L}^3(\Omega)}\notag\\
&\leq C_m\|\mathbf{v}\|_{\mathbf{L}^2(\Omega)}\|\mathbf{v}_m\|_{\mathbf{L}^2(\Omega)}
\|\partial_t\mathbf{v}_m\|_{\mathbf{L}^2(\Omega)}\notag\\
&\leq\frac{\rho_2}{8}\int_\Omega|\partial_t\mathbf{v}_m|^2\,\mathrm{d}x
+\frac{2C_m^2}{\rho_2}\|\mathbf{v}\|_{\mathbf{L}^2(\Omega)}^2\|\mathbf{v}_m\|_{\mathbf{L}^2(\Omega)}^2,\notag\\[1ex]
I_4
&\leq\rho'\|\nabla\mu_m\|_{\mathbf{L}^2(\Omega)}
\|\nabla\mathbf{v}_m\|_{\mathbf{L}^6(\Omega)}\|\partial_t\mathbf{v}_m\|_{\mathbf{L}^3(\Omega)}\notag\\
&\leq C_m\|\nabla\mu_m\|_{\mathbf{L}^2(\Omega)}\|\mathbf{v}_m\|_{\mathbf{L}^2(\Omega)}\|\partial_t\mathbf{v}_m\|_{\mathbf{L}^2(\Omega)}\notag\\
&\leq\frac{\rho_2}{8}\int_\Omega|\partial_t\mathbf{v}_m|^2\,\mathrm{d}x
+\frac{2C_m^2}{\rho_2}\|\nabla\mu_m\|_{\mathbf{L}^2(\Omega)}^2\|\mathbf{v}_m\|_{\mathbf{L}^2(\Omega)}^2,
\notag\\[1ex]
I_5
&\leq \frac{\rho'}{2L}\|\mu_m-\mathcal{L}_m\|_{H_\Gamma}\|\mathbf{v}_{m,\bm{\tau}}\|_{\mathbf{L}^2(\Gamma)}
\|\partial_t\mathbf{v}_{m,\bm{\tau}}\|_{\mathbf{L}^2(\Gamma)}\notag\\
&\leq \frac{C_m}{L}\|\mu_m-\mathcal{L}_m\|_{H_\Gamma}\|\mathbf{v}_{m}\|_{\mathbf{L}^2(\Omega)}
\|\partial_t\mathbf{v}_{m}\|_{\mathbf{L}^2(\Omega)}\notag\\
&\leq\frac{\rho_2}{8}\int_\Omega|\partial_t\mathbf{v}_m|^2\,\mathrm{d}x
+\frac{2C_m^2}{\rho_2 L}\|\mu_m-\mathcal{L}_m\|_{H_\Gamma}^2\|\mathbf{v}_{m}\|_{\mathbf{L}^2(\Omega)}^2.\notag
\end{align}
Moreover, using \eqref{v-trace}, the terms $I_6$ and $I_7$ can be estimated as
\begin{align*}
   I_6
&\leq
\beta^\ast\|\mathbf{v}_{m,\bm{\tau}}\|_{\mathbf{L}^2(\Gamma)}
\|\partial_t \mathbf{v}_{m,\bm{\tau}}\|_{\mathbf{L}^2(\Gamma)}\notag\\
&\leq \beta^\ast \widetilde{C}_0\|\mathbf{v}_{m,\bm{\tau}}\|_{\mathbf{L}^2(\Gamma)}
(\|\partial_t\mathbf{v}_{m}\|_{\mathbf{L}^2(\Omega)}+\|\mathbb{D}\partial_t\mathbf{v}_{m}\|_{\mathbf{L}^2(\Omega)})\notag\\
&\leq\frac{\rho_2}{8}\int_\Omega|\partial_t\mathbf{v}_m|^2\,\mathrm{d}x
+\frac{\alpha}{4}\int_\Omega|\mathbb{D}\partial_t\mathbf{v}_m|^2\,\mathrm{d}x
+\frac{(\beta^\ast)^2 \widetilde{C}_0^2(2\alpha+\rho_2)}{\alpha\rho_2}\int_\Gamma|\mathbf{v}_{m,\bm{\tau}}|^2\,\mathrm{d}S,\notag\\[1ex]
I_7
&\leq\|\psi_m\|_{L^\infty(\Gamma)}\|\nabla_{\!\bm{\tau}\,}\mathcal{L}_m\|_{\mathbf{L}^2(\Gamma)}
\|\partial_t\mathbf{v}_{m,\bm{\tau}}\|_{\mathbf{L}^2(\Gamma)}\notag\\
&\leq \widetilde{C}_0\|\psi_m\|_{L^\infty(\Gamma)}\|\nabla_{\!\bm{\tau}\,}\mathcal{L}_m\|_{\mathbf{L}^2(\Gamma)}
(\|\partial_t\mathbf{v}_m\|_{\mathbf{L}^2(\Omega)}+\|\mathbb{D}\partial_t\mathbf{v}_m\|_{\mathbf{L}^2(\Omega)})\notag\\
&\leq \frac{\rho_2}{8}\int_\Omega|\partial_t\mathbf{v}_m|^2\,\mathrm{d}x
+\frac{\alpha}{4}\int_\Omega|\mathbb{D}\partial_t\mathbf{v}_m|^2\,\mathrm{d}x
+\frac{\widetilde{C}_0^2(2\alpha+\rho_2)}{\alpha\rho_2}\int_\Gamma|\nabla_{\!\bm{\tau}\,}\mathcal{L}_m|^2\,\mathrm{d}S.\notag
\end{align*}
Therefore, combining \eqref{uni-ti-1} with the estimates for $I_j$ $(1\leq j\leq7)$, we conclude that
\begin{align}
&\frac{\rho_2}{4}\int_\Omega|\partial_t\mathbf{v}_m|^2\,\mathrm{d}x
+\alpha\int_\Omega|\mathbb{D}\partial_t\mathbf{v}_m|^2\,\mathrm{d}x
\notag\\
&\quad\leq \frac{4(\nu^\ast)^2}{\alpha}\int_\Omega|\mathbb{D}\mathbf{v}_m|^2 \,\mathrm{d}x
+\frac{2}{\rho_2}\int_\Omega|\nabla\mu_m|^2\,\mathrm{d}x
+\frac{2C_m^2}{\rho_2}\|\mathbf{v}\|_{\mathbf{L}^2(\Omega)}^2\|\mathbf{v}_m\|_{\mathbf{L}^2(\Omega)}^2
\notag\\
&\qquad+\frac{2C_m^2}{\rho_2}\|\nabla\mu_m\|_{\mathbf{L}^2(\Omega)}^2\|\mathbf{v}_m\|_{\mathbf{L}^2(\Omega)}^2
+\frac{2C_m^2}{\rho_2 L}\|\mu_m-\mathcal{L}_m\|_{H_\Gamma}^2\|\mathbf{v}_{m}\|_{\mathbf{L}^2(\Omega)}^2
\notag\\
&\qquad+\frac{(\beta^\ast)^2\widetilde{C}_0^2(2\alpha+\rho_2)}{\alpha\rho_2}\int_\Gamma|\mathbf{v}_{m,\bm{\tau}}|^2\,\mathrm{d}S+\frac{\widetilde{C}_0^2(2\alpha+\rho_2)}{\alpha\rho_2}\int_\Gamma|\nabla_{\!\bm{\tau}\,}\mathcal{L}_m|^2\,\mathrm{d}S.
\label{0227-6}
\end{align}
Next, we set $\gamma_0=\min\{\rho_2/4,\alpha\}$.
Then, by \eqref{v-trace}, \eqref{approenergy}, \eqref{uni-ga-5'} and \eqref{0305-5}, we infer from \eqref{0227-6} that
\begin{align}
&\int_0^T(\|\partial_t\mathbf{v}_m(t)\|_{\mathbf{L}^2(\Omega)}^2
+\|\partial_t\mathbf{v}_{m,\bm{\tau}}(t)\|_{\mathbf{L}^2(\Gamma)}^2)\,\mathrm{d}t
\notag\\
&\quad\leq 2(1+\widetilde{C}_0^2)\int_0^T(\|\partial_t\mathbf{v}_m(t)\|_{\mathbf{L}^2(\Omega)}^2
+\|\mathbb{D}\partial_t\mathbf{v}_{m}(t)\|_{\mathbf{L}^2(\Omega)}^2)\,\mathrm{d}t
\notag\\
&\quad\leq  \frac{8(\nu^\ast)^2(1+\widetilde{C}_0^2)}{\alpha\gamma_0}\int_0^T\int_\Omega|\mathbb{D}\mathbf{v}_m(t)|^2\,\mathrm{d}x\,\mathrm{d}t
+\frac{2(\beta^\ast)^2\widetilde{C}_0^2(2\alpha+\rho_2)(1+\widetilde{C}_0^2)}{\alpha\rho_2\gamma_0}
\int_0^T\int_\Gamma|\mathbf{v}_{m,\bm{\tau}}(t)|^2\,\mathrm{d}S\,\mathrm{d}t
\notag\\
&\qquad+\frac{4(1+\widetilde{C}_0^2)}{\rho_2\gamma_0}\int_0^T\int_\Omega|\nabla\mu_m(t)|^2\,\mathrm{d}x\,\mathrm{d}t
+\frac{4\widetilde{C}_0^2(\alpha+\rho_2)(1+\widetilde{C}_0^2)}{\alpha\rho_2\gamma_0}
\int_0^T\int_\Gamma|\nabla_{\!\bm{\tau}\,}\mathcal{L}_m(t)|^2\,\mathrm{d}S\,\mathrm{d}t
\notag\\
&\qquad+\frac{4C_m^2(1+\widetilde{C}_0^2)}{\rho_2\gamma_0}
\int_0^T\|\mathbf{v}(t)\|_{\mathbf{L}^2(\Omega)}^2\|\mathbf{v}_m(t)\|_{\mathbf{L}^2(\Omega)}^2\,\mathrm{d}t
\notag\\
&\qquad+\frac{4C_m^2(1+\widetilde{C}_0^2)}{\rho_2\gamma_0}
\int_0^T\|\nabla\mu_m(t)\|_{\mathbf{L}^2(\Omega)}^2\|\mathbf{v}_m(t)\|_{\mathbf{L}^2(\Omega)}^2\,\mathrm{d}t
\notag\\
&\qquad+\frac{4C_m^2(1+\widetilde{C}_0^2)\chi(L)}{\rho_2\gamma_0}
\int_0^T\|\mu_m(t)-\mathcal{L}_m(t)\|_{H_\Gamma}^2\|\mathbf{v}_{m}(t)\|_{\mathbf{L}^2(\Omega)}^2\,\mathrm{d}t
\notag\\
&\quad\leq  \frac{8(\nu^\ast)^2(1+\widetilde{C}_0^2)TK_0}{\alpha\gamma_0}+\frac{2(\beta^\ast)^2\widetilde{C}_0^4(2\alpha+\rho_2)(1+\widetilde{C}_0^2)T K_0}{\alpha\rho_2\gamma_0}
\notag\\
&\qquad+\Big(\frac{8(1+\widetilde{C}_0^2)}{\rho_2\gamma_0}+\frac{8\widetilde{C}_0^2(2\alpha+\rho_2)(1+\widetilde{C}_0^2)}{\alpha\rho_2\gamma_0}\Big)
\Big(E_{\text{free}}^{\sigma,k}(\bm{\varphi}_{0,\sigma,k})+C_\mathrm{b}+\frac{1}{2}C_3 e^{C_2 T}\Big)
\notag\\
&\qquad+\frac{4C_m^2(1+\widetilde{C}_0^2)}{\rho_2\gamma_0}K_0 C_3 e^{C_2T}
+\frac{8C_m^2(1+\widetilde{C}_0^2)K_0}{\rho_2\gamma_0}\Big(E_{\text{free}}^{\sigma,k}(\bm{\varphi}_{0,\sigma,k})
+C_\mathrm{b}+\frac{1}{2}C_3 e^{C_2 T}\Big)
\notag\\
&\qquad+\frac{4C_m^2(1+\widetilde{C}_0^2)K_0}{\rho_2\gamma_0}
\Big(E_{\text{free}}^{\sigma,k}(\bm{\varphi}_{0,\sigma,k})+C_\mathrm{b}+\frac{1}{2}C_3 e^{C_2 T}\Big)
\notag\\
&\quad \eqqcolon  K_1^2.
\label{0305-6}
\end{align}

We are now in a position to manage the setting for the fixed-point argument.
Define
\begin{align*}
    S=\left\{(\mathbf{v},\mathbf{v}_{\boldsymbol{\tau}})\in H^1(0,T;\mathcal{V}_m)
    \,\middle|\,
    \begin{aligned}
    &\|(\mathbf{v},\mathbf{v}_{\boldsymbol{\tau}})\|_{L^{2}(0,t;\bm{\mathcal{L}}_{\mathrm{div}}^2)}^2\leq C_3 e^{C_2t}
    \;\;\text{for all $t\in[0,T]$},
    \\
    &\|(\partial_t\mathbf{v},\partial_t\mathbf{v}_{\boldsymbol{\tau}})\|_{L^{2}(0,T;\bm{\mathcal{L}}_{\mathrm{div}}^2)}\leq K_{1}
    \end{aligned}
    \right\}.
\end{align*}
Then $S$ is a closed, convex subset of $L^2(0,T;\mathcal{V}_m)$. In addition, since $\mathcal{V}_m$ is finite-dimensional, $S$ is a compact subset of  $L^{2}(0,T;\mathcal{V}_m)$. Next, we define the operator
$$
\Lambda:S\to L^{2}(0,T;\mathcal{V}_m),\quad \Lambda(\mathbf{v},\mathbf{v}_{\boldsymbol{\tau}})=(\mathbf{v}_m,\mathbf{v}_{m,\boldsymbol{\tau}}),
$$
where, for any $(\mathbf{v},\mathbf{v}_{\boldsymbol{\tau}})\in S$, $(\mathbf{v}_m,\mathbf{v}_{m,\boldsymbol{\tau}})$ is the solution to the system \eqref{galerkineq1} constructed above. In light of the estimates \eqref{0305-4} and \eqref{0305-6}, we find that $\Lambda(S) \subseteq S$.

In order to apply Schauder's fixed point theorem,
it remains to prove that $\Lambda$ is continuous.
To this end, we consider a sequence $\{(\mathbf{v}^n, \mathbf{v}^n_{\boldsymbol{\tau}})\}_{n\geq1}\subset S$
such that $(\mathbf{v}^n,\mathbf{v}^n_{\boldsymbol{\tau}})\to(\widetilde{\mathbf{v}},\widetilde{\mathbf{v}}_{\boldsymbol{\tau}})$ in $L^2(0,T;\mathcal{V}_m)$ as $n\to+\infty$.
Arguing as above, we find a sequence $\{(\boldsymbol{\varphi}_m^n,\boldsymbol{\mu}_m^n) = (\varphi^n_m,\psi_{m}^n,\mu^n_m,\mathcal{L}^n_{m})\}_{n\geq1}$
and a pair $(\widetilde{\boldsymbol{\varphi}}_m,\widetilde{\boldsymbol{\mu}}_m)
= (\widetilde{\varphi}_m,\widetilde{\psi}_m, \widetilde{\mu}_m,\widetilde{\mathcal{L}}_m)$,
which solve the convective Cahn--Hilliard system \eqref{0321-eq-1}--\eqref{appini},
where $(\mathbf{v},\mathbf{v}_{\boldsymbol{\tau}})$ is replaced by $(\mathbf{v}^n,\mathbf{v}^n_{\boldsymbol{\tau}}) $ and $(\widetilde{\mathbf{v}}, \widetilde{\mathbf{v}}_{\boldsymbol{\tau}})$, respectively.
By repeating the uniqueness argument developed in the proof of Theorem~\ref{A1}, we have
\begin{align}
    &\frac{1}{2}\frac{\mathrm{d}}{\mathrm{d}t} \Big(\|\bm{\varphi}^n_m-\widetilde{\boldsymbol{\varphi}}_m\|_{\mathcal{H}_{L,0}^{-1}}^2
    +\gamma \|\bm{\varphi}^n_m -\widetilde{\boldsymbol{\varphi}}_m\|_{\mathcal{L}^2}^2\Big)
    \notag\\
    &\qquad+\|\nabla(\varphi^n_m -\widetilde{\varphi}_m)\|_{\mathbf{L}^2(\Omega)}^2
    +\|\nabla_{\!\bm{\tau}\,}(\psi^n_{m}-\widetilde{\psi}_m)\|_{\mathbf{L}^2(\Gamma)}^2
    +\chi(K)\int_\Gamma|\varphi^n_m-\widetilde{\varphi}_m -(\psi^n_{m}-\widetilde{\psi}_m)|^2\,\mathrm{d}S
    \notag\\
    &\quad\leq C_{\mathrm{Lip}} \|\bm{\varphi}^n_m -\widetilde{\boldsymbol{\varphi}}_m\|_{\mathcal{L}^2}^2
    +\int_\Omega \varphi_m^n (\mathbf{v}^n-\widetilde{\mathbf{v}})\cdot\nabla \mathfrak{S}^L_\Omega(\bm{\varphi}^n_m
    -\widetilde{\bm{\varphi}}_m)\,\mathrm{d}x
    +\int_\Omega(\varphi^n_m-\widetilde{\varphi}_m)\widetilde{\mathbf{v}} \cdot\nabla \mathfrak{S}^L_\Omega(\bm{\varphi}^n_m
    -\widetilde{\bm{\varphi}}_m)\,\mathrm{d}x
    \notag\\
    &\qquad +\int_\Gamma \psi^n_{m}(\mathbf{v}^n_{\bm{\tau}}-\widetilde{\mathbf{v}}_{\bm{\tau}})
    \cdot\nabla_{\!\bm{\tau}\,} \mathfrak{S}_\Gamma^L (\bm{\varphi}^n_m -\widetilde{\bm{\varphi}}_m) \,\mathrm{d}S
    +\int_\Gamma (\psi^n_{m}-\widetilde{\psi}_m) \widetilde{\mathbf{v}}_{\bm{\tau}}
    \cdot\nabla_{\!\bm{\tau}\,} \mathfrak{S}_\Gamma^L (\bm{\varphi}^n_m-\widetilde{\bm{\varphi}}_m) \,\mathrm{d}S.
    \notag
\end{align}
Since $(\mathbf{v}^n,\mathbf{v}^n_{\boldsymbol{\tau}})$, $(\widetilde{\mathbf{v}},\widetilde{\mathbf{v}}_{\boldsymbol{\tau}}) \in S$, we can deduce that
\begin{align*}
    &\frac{1}{2}\frac{\mathrm{d}}{\mathrm{d}t}\mathcal{Q}
    +\frac{1}{2}\|\nabla(\varphi^n_m
    -\widetilde{\varphi}_m)\|_{\mathbf{L}^2(\Omega)}^2
    +\frac{1}{2}\|\nabla_{\!\bm{\tau}\,}(\psi^n_{m}
    -\widetilde{\psi}_m)\|_{\mathbf{L}^2(\Gamma)}^2
    +\frac{\chi(K)}{2}\int_\Gamma|\varphi^n_m
    -\widetilde{\varphi}_m -(\psi^n_{m}
    -\widetilde{\psi}_m)|^2 \,\mathrm{d}S
     \\
    &\quad\leq C\mathcal{Q}(t)
    +\|\mathbf{v}^n
    -\widetilde{\mathbf{v}}\|_{\mathbf{L}^2(\Omega)}^2
    +\|\mathbf{v}^n_{\bm{\tau}}-\widetilde{\mathbf{v}}_{\bm{\tau}}\|_{\mathbf{L}^2(\Gamma)}^2\quad \text{a.e.~in}\ (0,T),
\end{align*}
 where
$
\mathcal{Q}(t) \coloneqq
\|\bm{\varphi}^n_m(t) -\widetilde{\boldsymbol{\varphi}}_m(t)\|_{\mathcal{H}_{L,0}^{-1}}^2
+\gamma\|\bm{\varphi}^n_m(t)-\widetilde{\boldsymbol{\varphi}}_m(t)\|_{\mathcal{L}^2}^2
$. Observing that $\bm{\varphi}^n_m(0)=\widetilde{\bm{\varphi}}_m(0)$, applying Gronwall's lemma, we obtain
\begin{align}
    &\|\bm{\varphi}^n_m-\widetilde{\boldsymbol{\varphi}}_m\|_{L^\infty(0,T;\mathcal{L}^2)\cap L^2(0,T;\mathcal{H}^1)}
    \notag\\
    &\quad\leq e^{CT}\int_0^T(\|\mathbf{v}^n(t)-\widetilde{\mathbf{v}}(t)\|_{\mathbf{L}^2(\Omega)}^2
    +\|\mathbf{v}^n_{\bm{\tau}}(t)-\widetilde{\mathbf{v}}_{\bm{\tau}}(t)\|_{\mathbf{L}^2(\Gamma)}^2)\,\mathrm{d}t \to 0
    \quad \text{as}\ n\to+\infty.
    \label{0305-11}
\end{align}
Furthermore,
we can use the continuous embedding
$H^1(0,T;\mathcal{V}_m)\hookrightarrow C([0,T];\mathcal{V}_m)$,
the properties of $\boldsymbol{\varphi}_{0,\sigma,k}\,$ presented in Proposition~\ref{approximate-initial-datum}, and Theorem~\ref{A1} to deduce that
\begin{subequations}
\begin{align}
	& \|\partial_t\boldsymbol{\varphi}^n_m\|_{L^{\infty}(0,T;\mathcal{H}^{1})}
    +\|\partial_t\widetilde{\boldsymbol{\varphi}}_m \|_{L^{\infty}(0,T;\mathcal{H}^{1})}
    \leq C,\label{continuous2}
    \\
	& \|\partial_t\boldsymbol{\varphi}^n_m\|_{L^{2}(0,T;\mathcal{H}^{2})}
    +\|\partial_t\widetilde{\boldsymbol{\varphi}}_m \|_{L^{2}(0,T;\mathcal{H}^{2})}
    \leq C,\label{continuous3}
    \\
	& \|\boldsymbol{\mu}^n_m\|_{L^\infty (0,T;\mathcal{H}^2)}
    +\|\widetilde{\boldsymbol{\mu}}_m\|_{L^\infty (0,T;\mathcal{H}^2)}
    \leq C,\label{0305-12}
    \\
	& \|\boldsymbol{\varphi}^n_m\|_{L^{\infty}(0,T;\mathcal{H}^{3})}
    +\|\widetilde{\boldsymbol{\varphi}}_m\|_{L^{\infty}(0,T;\mathcal{H}^{3})}
    \leq C,\label{0305-13}
    \\
	& \|\partial_t\bm{\mu}^n_m\|_{L^2(0,T;\mathcal{H}^1)}
    +\|\partial_t\widetilde{\bm{\mu}}_m \|_{L^2(0,T;\mathcal{H}^1)}
    \leq C ,\label{0305-14}
    \\
    & \|\bm{\varphi}^n_m\|_{\mathcal{L}^\infty}
    \leq (1-\sigma)(1-\delta^\ast),
    \label{0305-15}
    \\
    & \|\widetilde{\bm{\varphi}}_m\|_{\mathcal{L}^\infty}
    \leq (1-\sigma)(1-\delta^\ast),
    \label{0305-16}
\end{align}
\end{subequations}
for some generic constant $C>0$ and $\delta^\ast\in(0,1)$. Taking a closer look at the proofs of Proposition~\ref{approximate-initial-datum} and Theorem~\ref{A1}, we find that $C$ and $\delta^*$ can be chosen independent of $n$.
By interpolation, we deduce from \eqref{0305-11} and \eqref{0305-13} that
\begin{align}
		\|\boldsymbol{\varphi}^n_m-\widetilde{\boldsymbol{\varphi}}_m\|_{L^{\infty}(0,T;\mathcal{H}^2)}
        \to0
        \quad\text{as }n\to+\infty.
        \label{conti11}
\end{align}
Next, by \eqref{0305-12}, \eqref{0305-14} and the Aubin--Lions--Simon lemma,
there exists a subsequence $\{\boldsymbol{\mu}^{n_j}_m\}_{j\geq1}$
such that
$\boldsymbol{\mu}^{n_j}_m-\widetilde{\boldsymbol{\mu}}_m\to \boldsymbol{\mu}_m^\ast$ in $L^\infty(0,T;\mathcal{H}^1)$.
Consider the equations
\begin{align}
    \label{SYS:THEPSI}
    \left\{
    \begin{aligned}
        \mu_m^n-\widetilde{\mu}_m
        &=
        \gamma\partial_t(\varphi^n_m-\widetilde{\varphi}_m)
        -\Delta(\varphi^n_m-\widetilde{\varphi}_m)
        \\
        &\quad+f(\varphi^n_m)-f(\widetilde{\varphi}_m)
        +k^{-1} f_\sigma(\varphi^n_m)-k^{-1} f_\sigma(\widetilde{\varphi}_m)
        &&\quad\text{a.e.~in }Q_T,
        \\[1ex]
        \mathcal{L}^n_{m}-\widetilde{\mathcal{L}}_m
        &=\gamma\partial_t(\psi^n_{m}-\widetilde{\psi}_m)
        -\Delta_{\boldsymbol{\tau}}(\psi^n_{m}-\widetilde{\psi}_m)
        +\partial_\mathbf{n}(\varphi^n_m-\widetilde{\varphi}_m)
        \\
        &\quad+g(\psi^n_{m})-g(\widetilde{\psi}_m)
        +k^{-1} f_\sigma(\psi^n_{m})-k^{-1} f_\sigma(\widetilde{\psi}_m)
        &&\quad\text{a.e.~on }\Sigma_T.
    \end{aligned}
    \right.
\end{align}
By \eqref{0305-15}, \eqref{0305-16} and \eqref{conti11}, we obtain
\begin{alignat*}{2}
    \|f(\varphi^n_m)-f(\widetilde{\varphi}_m)\|_{L^\infty(0,T;H)}
    +\|g(\psi^n_{m})-g(\widetilde{\psi}_m)\|_{L^\infty(0,T;H_\Gamma)}
    &\to 0
    &&\quad\text{as }\ n\to+\infty,\\
    k^{-1}\|f_\sigma(\varphi^n_m)-f_\sigma(\widetilde{\varphi}_m)\|_{L^\infty(0,T;H)}
    +k^{-1}\|f_\sigma(\psi^n_{m})-f_\sigma(\widetilde{\psi}_m)\|_{L^\infty(0,T;H_\Gamma)}
    &\to0
    &&\quad\text{as }\ n\to+\infty.
\end{alignat*}
Moreover, it follows from \eqref{0305-11}, \eqref{continuous2} and \eqref{continuous3}
that (up to a subsequence)
$\partial_t(\boldsymbol{\varphi}^{n_j}_m-\widetilde{\boldsymbol{\varphi}}_m)\to0$
weakly in $L^2(0,T;\mathcal{H}^2)$.
Combining these convergence results, we can pass to the limit as $n\to+\infty$ on the right-hand side of \eqref{SYS:THEPSI} to deduce that $\boldsymbol{\mu}_m^\ast=(0,0)$.
Finally, due to the uniqueness of the (weak) limit, we arrive at
\begin{align}	
\|\boldsymbol{\mu}^n_m-\widetilde{\boldsymbol{\mu}}_m\|_{L^{\infty}(0,T;\mathcal{H}^1)}
\to0\quad\text{as }n\to+\infty.\label{conti12}
\end{align}


Define  $(\mathbf{v}^n_m,\mathbf{v}^n_{m,\boldsymbol{\tau}})\coloneqq \Lambda(\mathbf{v}^n,\mathbf{v}^n_{\boldsymbol{\tau}})\in S$ for any $n\in\mathbb{N}$
and $(\widetilde{\mathbf{v}}_m,\widetilde{\mathbf{v}}_{m,\boldsymbol{\tau}})\coloneqq \Lambda(\widetilde{\mathbf{v}},\widetilde{\mathbf{v}}_{\boldsymbol{\tau}})\in S$.
Consider the differences
$$\widehat{\mathbf{v}} \coloneqq \mathbf{v}^n-\widetilde{\mathbf{v}},\ \ \widehat{\mathbf{v}}_m\coloneqq \mathbf{v}^n_m-\widetilde{\mathbf{v}}_m, \ \
\widehat{\boldsymbol{\varphi}}_m=(\widehat{\varphi}_m,\widehat{\psi}_m) \coloneqq \boldsymbol{\varphi}^n_m-\widetilde{\boldsymbol{\varphi}}_m,\ \
\widehat{\boldsymbol{\mu}}_m = (\widehat{\mu}_m,\widehat{\mathcal{L}}_m) \coloneqq \boldsymbol{\mu}^n_m-\widetilde{\boldsymbol{\mu}}_m.
$$
We find that
\begin{align}
	&\int_\Omega \rho(\varphi^n_m)\partial_t \widehat{\mathbf{v}}_m\cdot\mathbf{w}\,\mathrm{d}x
    +\int_\Omega 2\alpha\mathbb{D}\partial_t\widehat{\mathbf{v}}_m:\mathbb{D}\mathbf{w} \,\mathrm{d}x
    +\int_\Omega 2\nu(\varphi^n_m)\mathbb{D}\widehat{\mathbf{v}}_m:\mathbb{D}\mathbf{w} \,\mathrm{d}x
    +\int_\Gamma \beta(\psi^n_{m})\widehat{\mathbf{v}}_{m,\bm{\tau}}\cdot\mathbf{w}_{\bm{\tau}} \,\mathrm{d}S
    \notag\\
    &\quad=-\int_\Omega 2(\nu(\varphi^n_m)-\nu(\widetilde{\varphi}_m))\mathbb{D} \widetilde{\mathbf{v}}_m:\mathbb{D}\mathbf{w} \,\mathrm{d}x
    -\int_\Gamma (\beta(\psi^n_{m})-\beta(\widetilde{\psi}_m))\widetilde{\mathbf{v}}_{m,\bm{\tau}}\cdot\mathbf{w}_{\bm{\tau}} \,\mathrm{d}S
    \notag\\
    &\qquad-\int_\Omega (\rho(\varphi^n_m)-\rho(\widetilde{\varphi}_m)) \partial_t\widetilde{\mathbf{v}}_m \cdot\mathbf{w} \,\mathrm{d}x
    -\int_\Omega \widehat{\varphi}_m \nabla\mu^n_m\cdot \mathbf{w} \,\mathrm{d}x
    -\int_\Omega \widetilde{\varphi}_m \nabla\widehat{\mu}_m \cdot\mathbf{w} \,\mathrm{d}x
    \notag\\
    &\qquad-\int_\Gamma \widehat{\psi}_m \nabla_{\!\bm{\tau}\,}\mathcal{L}^n_{m}\cdot\mathbf{w}_{\bm{\tau}} \,\mathrm{d}S
    -\int_\Gamma \widetilde{\psi}_m \nabla_{\!\bm{\tau}\,} \widehat{\mathcal{L}}_m\cdot\mathbf{w}_{\bm{\tau}} \,\mathrm{d}S
    \notag\\
	&\qquad-\int_\Omega \rho(\varphi^n_m)(\mathbf{v}^n\cdot\nabla)\mathbf{v}^n_m\cdot\mathbf{w} \,\mathrm{d}x
    +\int_\Omega \rho(\widetilde{\varphi}_m)(\widetilde{\mathbf{v}}\cdot\nabla) \widetilde{\mathbf{v}}_m\cdot\mathbf{w} \,\mathrm{d}x
    \notag\\
    &\qquad+\rho'\int_\Omega (\nabla\widehat{\mu}_m \cdot\nabla)\mathbf{v}^n_m \cdot\mathbf{w} \,\mathrm{d}x
    +\rho'\int_\Omega (\nabla\widetilde{\mu}_m\cdot\nabla)\widehat{\mathbf{v}}_m\cdot\mathbf{w} \,\mathrm{d}x
    \notag\\
    &\qquad-\frac{\rho'}{2}\int_\Gamma (\nabla\widehat{\mu}_m \cdot\mathbf{n})\mathbf{v}^n_{m,\bm{\tau}}\cdot \mathbf{w}_{\bm{\tau}} \,\mathrm{d}S
    -\frac{\rho'}{2}\int_\Gamma (\nabla\widetilde{\mu}_m\cdot\mathbf{n})\widehat{\mathbf{v}}_{m,\bm{\tau}}\cdot\mathbf{w}_{\bm{\tau}} \,\mathrm{d}S,
    \label{d-1}
\end{align}
for all $(\mathbf{w},\mathbf{w}_{\boldsymbol{\tau}})\in \mathcal{V}_m$ and for all $t\in[0,T]$.
Taking $(\mathbf{w},\mathbf{w}_{\boldsymbol{\tau}})=(\widehat{\mathbf{v}}_m,\widehat{\mathbf{v}}_{m,\boldsymbol{\tau}})$ in \eqref{d-1}, we obtain
\begin{align}
	&\frac{1}{2}\frac{\mathrm{d}}{\mathrm{d}t} \Big(\int_\Omega \rho(\varphi^n_m)|\widehat{\mathbf{v}}_m|^2 \,\mathrm{d}x
    +2\alpha\int_\Omega |\mathbb{D}\widehat{\mathbf{v}}_m|^2\,\mathrm{d}x\Big)
    +2\int_\Omega \nu(\varphi^n_m)|\mathbb{D}\widehat{\mathbf{v}}_m|^2\,\mathrm{d}x
    +\int_\Gamma \beta(\psi^n_{m})|\widehat{\mathbf{v}}_{m,\bm{\tau}}|^2\,\mathrm{d}S\notag\\
	&\quad =-2\int_\Omega(\nu(\varphi^n_m)-\nu(\widetilde{\varphi}_m))\mathbb{D}\widetilde{\mathbf{v}}_m:\mathbb{D}\widehat{\mathbf{v}}_m\,\mathrm{d}x
    -\int_\Gamma (\beta(\psi^n_{m})-\beta(\widetilde{\psi}_m))\widetilde{\mathbf{v}}_{m,\bm{\tau}}\cdot\widehat{\mathbf{v}}_{m,\bm{\tau}}\,\mathrm{d}S
    \notag\\
	&\qquad-\int_\Omega\widehat{\varphi}_m \nabla\mu_m^n\cdot\widehat{\mathbf{v}}_m \,\mathrm{d}x
    -\int_\Omega\widetilde{\varphi}_m\nabla\widehat{\mu}_m \cdot\widehat{\mathbf{v}}_m \,\mathrm{d}x
    -\int_\Gamma \widehat{\psi}_m \nabla_{\!\bm{\tau}\,}\mathcal{L}^n_{m}\cdot \widehat{\mathbf{v}}_{m,\bm{\tau}}\,\mathrm{d}S
    -\int_\Gamma \widetilde{\psi}_m \nabla_{\!\bm{\tau}\,}\widehat{\mathcal{L}}_m \cdot\widehat{\mathbf{v}}_{m,\bm{\tau}} \,\mathrm{d}S
    \notag\\
	&\qquad
    +\rho' \int_\Omega(\nabla\widehat{\mu}_m \cdot\nabla)\mathbf{v}^n_m\cdot \widehat{\mathbf{v}}_m \,\mathrm{d}x
    +\rho'\int_\Omega(\nabla\widetilde{\mu}_m \cdot\nabla)\widehat{\mathbf{v}}_m \cdot\widehat{\mathbf{v}}_m \,\mathrm{d}x
    \notag \\
    &\qquad
    -\frac{\rho'}{2} \int_\Gamma\partial_{\mathbf{n}}\widehat{\mu}_m\,\mathbf{v}^n_{m,\bm{\tau}} \cdot\widehat{\mathbf{v}}_{m,\bm{\tau}} \,\mathrm{d}S
    -\frac{\rho'}{2}\int_\Gamma\partial_{\mathbf{n}}\widetilde{\mu}_m\,|\widehat{\mathbf{v}}_{m,\bm{\tau}}|^2 \,\mathrm{d}S
    \notag\\
    &\qquad
    -\rho'\int_\Omega \widehat{\varphi}_m \partial_t\widetilde{\mathbf{v}}_m \cdot\widehat{\mathbf{v}}_m \,\mathrm{d}x
    +\frac{\rho'}{2}\int_\Omega \partial_t\varphi^n_m|\widehat{\mathbf{v}}_m|^2 \,\mathrm{d}x
    \notag\\
    &\qquad-\int_\Omega(\rho(\varphi^n_m)\mathbf{v}^n\cdot\nabla)\mathbf{v}^n_m\cdot\widehat{\mathbf{v}}_m\,\mathrm{d}x
    +\int_\Omega (\rho(\widetilde{\varphi}_m) \widetilde{\mathbf{v}} \cdot\nabla)\widetilde{\mathbf{v}}_m \cdot\widehat{\mathbf{v}}_m\,\mathrm{d}x
    \notag\\
    &\quad \eqqcolon \sum_{k=1}^{14}J_k.
    \label{difference}
\end{align}
Now, we estimate the terms on the right-hand side of \eqref{difference}. Recalling that $\rho(s)\geq \rho_2$ for all $s\in[-1,1]$ and $\alpha>0$, we use H\"older's inequality, \eqref{v-trace} and \eqref{Korn-cor}, to deduce that
\begin{align}
	J_1
    &\leq \frac{\nu_\ast}{2}\|\mathbb{D}\widehat{\mathbf{v}}_m\|_{\mathbf{L}^2(\Omega)}^2
    +C\|\widehat{\varphi}_m\|_{L^\infty(\Omega)}^2 \|\mathbb{D}\widetilde{\mathbf{v}}_m\|_{\mathbf{L}^2(\Omega)}^2,\notag\\
	J_2
    &\leq \frac{\beta_\ast}{8}\|\widehat{\mathbf{v}}_{m,\bm{\tau}}\|_{\mathbf{L}^2(\Gamma)}^2
    +C\|\widehat{\psi}_m\|_{L^\infty(\Gamma)}^2 \|\widetilde{\mathbf{v}}_{m,\bm{\tau}}\|_{\mathbf{L}^2(\Gamma)}^2,
    \notag\\
	J_3
    &\leq\|\sqrt{\rho(\varphi^n_m)} \,\widehat{\mathbf{v}}_m \|_{\mathbf{L}^2(\Omega)}^2
    +\frac{1}{4\rho_2} \|\widehat{\varphi}_m\|_{L^\infty(\Omega)}^2 \|\nabla\mu^n_m\|_{\mathbf{L}^2(\Omega)}^2,
    \notag\\
	J_4
    &=\int_\Omega \widehat{\mu}_m \nabla\widetilde{\varphi}_m\cdot\widehat{\mathbf{v}}_m \,\mathrm{d}x
    \leq\|\sqrt{\rho(\varphi^n_m)} \,\widehat{\mathbf{v}}_m \|_{\mathbf{L}^2(\Omega)}^2
    +\frac{1}{4\rho_2}\|\widehat{\mu}_m\|_{H}^2\|\widetilde{\varphi}_m\|_{H^3(\Omega)}^2,
    \notag\\
	J_5
    &\leq \frac{\beta_\ast}{8}\|\widehat{\mathbf{v}}_{m,\bm{\tau}}\|_{\mathbf{L}^2(\Gamma)}^2
    +\frac{2}{\beta_\ast}\|\widehat{\psi}_m\|_{L^\infty(\Gamma)}^2\|\nabla_{\!\bm{\tau}\,}\mathcal{L}^n_{m}\|_{\mathbf{L}^2(\Gamma)}^2,
    \notag\\
    J_6
    &=\int_\Gamma \widehat{\mathcal{L}}_m \nabla_{\!\bm{\tau}\,}\widetilde{\psi}_m\cdot \widehat{\mathbf{v}}_{m,\bm{\tau}} \,\mathrm{d}S
    \leq \frac{\beta_\ast}{8}\|\widehat{\mathbf{v}}_{m,\bm{\tau}}\|_{\mathbf{L}^2(\Gamma)}^2
    +\frac{2}{\beta_\ast}\|\widehat{\mathcal{L}}_m\|_{H_\Gamma}^2\|\widetilde{\psi}_m\|_{H^3(\Gamma)}^2,\notag\\
    J_7
    &\leq C\|\nabla\widehat{\mu}_m\|_{\mathbf{L}^2(\Omega)} \|\nabla\mathbf{v}^n_m\|_{\mathbf{L}^6(\Omega)} \|\widehat{\mathbf{v}}_m\|_{\mathbf{L}^3(\Omega)} \notag\\
    &\leq C_m\|\widehat{\mu}_m\|_{V} \|\widehat{\mathbf{v}}_m\|_{\mathbf{H}^1(\Omega)} \notag\\
    &\leq C_m(\|\sqrt{\rho(\varphi^n_m)}\widehat{\mathbf{v}}_m\|_{\mathbf{L}^2(\Omega)}^2
    +2\alpha\|\mathbb{D}\widehat{\mathbf{v}}_m \|_{\mathbf{L}^2(\Omega)}^2) +C_m\|\widehat{\mu}_m\|_{V}^2,
    \notag\\
    J_8
    &\leq C\|\nabla \widetilde{\mu}_m\|_{\mathbf{L}^6(\Omega)} \|\nabla\widehat{\mathbf{v}}_m\|_{\mathbf{L}^2(\Omega)}\|\widehat{\mathbf{v}}_m\|_{\mathbf{L}^3(\Omega)} \leq C\|\widehat{\mathbf{v}}_m\|_{\mathbf{H}^1(\Omega)}^2\notag\\
    &\leq C(\|\sqrt{\rho(\varphi^n_m)}\widehat{\mathbf{v}}_m\|_{\mathbf{L}^2(\Omega)}^2
    +2\alpha\|\mathbb{D}\widehat{\mathbf{v}}_m \|_{\mathbf{L}^2(\Omega)}^2),
    \notag\\
    J_9
    &\leq\frac{\beta_\ast}{8}\|\widehat{\mathbf{v}}_{m,\bm{\tau}}\|_{\mathbf{L}^2(\Gamma)}^2
    +\frac{C_m}{L}\|\widehat{\mu}_m-\widehat{\mathcal{L}}_m\|_{H_\Gamma}^2,
    \notag\\
    J_{10}&{=-\frac{\rho'}{2L}\int_\Gamma(\widetilde{\mathcal{L}}_m-\widetilde{\mu}_m)|\widehat{\mathbf{v}}_{m,\boldsymbol{\tau}}|^2\,\mathrm{d}S}
    \notag \\
    & \leq C(\|\widetilde{\mathcal{L}}_m\|_{H_\Gamma} + \|\widetilde{\mu}_m\|_{H_\Gamma)})   \|\widehat{\mathbf{v}}_{m,\boldsymbol{\tau}}\|_{\mathbf{L}^4(\Gamma)}^2
    \notag\\
    & \leq C(\|\sqrt{\rho(\varphi^n_m)}\widehat{\mathbf{v}}_m\|_{\mathbf{L}^2(\Omega)}^2
    +2\alpha\|\mathbb{D}\widehat{\mathbf{v}}_m \|_{\mathbf{L}^2(\Omega)}^2),
    \notag\\
    J_{11}
    &\leq C\|\sqrt{\rho(\varphi^n_m)}\widehat{\mathbf{v}}_m\|_{\mathbf{L}^2(\Omega)}^2
    +C\|\partial_t\widetilde{\mathbf{v}}_m \|_{\mathbf{L}^2(\Omega)}^2 \|\widehat{\varphi}_m\|_{H^2(\Omega)}^2,
    \notag\\
    J_{12}
    &\leq C\|\partial_t\varphi^n_m\|_{L^6(\Omega)}\|\widehat{\mathbf{v}}_m \|_{\mathbf{L}^2(\Omega)} \|\widehat{\mathbf{v}}_m \|_{\mathbf{L}^3(\Omega)}
    \leq C\|\widehat{\mathbf{v}}_m \|_{\mathbf{H}^1(\Omega)}^2
    \notag\\
   &\leq C(\|\sqrt{\rho(\varphi^n_m)}\widehat{\mathbf{v}}_m\|_{\mathbf{L}^2(\Omega)}^2
   +2\alpha\|\mathbb{D}\widehat{\mathbf{v}}_m \|_{\mathbf{L}^2(\Omega)}^2).
   \notag
\end{align}
Noticing that $\mathbf{v}^n$, $\widetilde{\mathbf{v}}$, $\mathbf{v}^n_m$, $\widetilde{\mathbf{v}}_m\in S$, we find that
\begin{align*}
    J_{13}+J_{14}
    &=-\frac{\rho_1-\rho_2}{2}\int_\Omega\widehat{\varphi}_m((\mathbf{v}^n\cdot\nabla) \mathbf{v}^n_m) \cdot \widehat{\mathbf{v}}_m \,\mathrm{d}x
    -\int_\Omega\rho(\widetilde{\varphi}_m)((\widehat{\mathbf{v}}\cdot\nabla)\mathbf{v}^n_m) \cdot\widehat{\mathbf{v}}_m \,\mathrm{d}x
    \\
    &\quad-\int_\Omega\rho(\widetilde{\varphi}_m)((\widetilde{\mathbf{v}}\cdot\nabla)\widehat{\mathbf{v}}_m)\cdot\widehat{\mathbf{v}}_m
    \,\mathrm{d}x
    \\
    &\leq C\|\widehat{\varphi}_m\|_{L^\infty(\Omega)} \|\mathbf{v}^n\|_{\mathbf{L}^\infty(\Omega)} \|\nabla\mathbf{v}^n\|_{\mathbf{L}^2(\Omega)} \|\widehat{\mathbf{v}}_m\|_{\mathbf{L}^2(\Omega)}
    \notag \\
    &\quad +C\|\widehat{\mathbf{v}}\|_{\mathbf{L}^2(\Omega)} \|\nabla\mathbf{v}^n_m\|_{\mathbf{L}^3(\Omega)} \|\widehat{\mathbf{v}}_m\|_{\mathbf{L}^6(\Omega)}
    \\[1mm]
    &\quad+C\|\widetilde{\mathbf{v}} \|_{\mathbf{L}^\infty(\Omega)} \|\nabla\widehat{\mathbf{v}}_m\|_{\mathbf{L}^2(\Omega)}\|\widehat{\mathbf{v}}_m\|_{\mathbf{L}^2(\Omega)}\\[1mm]
    &\leq C_m\|\widehat{\mathbf{v}}_m\|_{\mathbf{H}^1(\Omega)}^2
    +C_m(\|\widehat{\varphi}_m\|_{H^2(\Omega)}^2 +\|\widehat{\mathbf{v}}\|_{\mathbf{L}^2(\Omega)}^2)
    \\
     &\leq C_m(\|\sqrt{\rho(\varphi^n_m)}\widehat{\mathbf{v}}_m\|_{\mathbf{L}^2(\Omega)}^2
     +2\alpha\|\mathbb{D}\widehat{\mathbf{v}}_m \|_{\mathbf{L}^2(\Omega)}^2)
     +C_m(\|\widehat{\varphi}_m\|_{H^2(\Omega)}^2 +\|\widehat{\mathbf{v}} \|_{\mathbf{L}^2(\Omega)}^2).
\end{align*}
Collecting the above estimates, we obtain
\begin{align}
	& \frac{\mathrm{d}}{\mathrm{d}t} \Big(\int_\Omega \rho(\varphi^n_m)|\widehat{\mathbf{v}}_m|^2\,\mathrm{d}x
    +2\alpha\int_\Omega |\mathbb{D}\widehat{\mathbf{v}}_m|^2\,\mathrm{d}x\Big)
    \notag \\
    &\quad
    \leq C_m\Big(\int_\Omega \rho(\varphi^n_m)|\widehat{\mathbf{v}}_m|^2\,\mathrm{d}x
    +2\alpha\int_\Omega |\mathbb{D}\widehat{\mathbf{v}}_m|^2\,\mathrm{d}x\Big)+h(t),\label{diff2}
\end{align}
where
\begin{align*}
	h(t) \coloneqq C_m\Big(\|\widehat{\bm{\varphi}}_m(t)\|_{\mathcal{H}^2}^2 +\|\partial_t\widetilde{\mathbf{v}}_m(t)\|_{\mathbf{L}^2(\Omega)}^2 \|\widehat{\bm{\varphi}}_m(t)\|_{\mathcal{H}^2}^2 +\|\widehat{\bm{\mu}}_m(t)\|_{\mathcal{H}^1}^2 +\|\widehat{\mathbf{v}}(t)\|_{\mathbf{L}^2(\Omega)}^2\Big).
\end{align*}
Applying Gronwall's lemma to \eqref{diff2} and using \eqref{conti11} and \eqref{conti12}, we infer that
\begin{align*}
	\sup_{t\in[0,T]} \|\widehat{\mathbf{v}}_m(t)\|_{\mathbf{H}^{1}(\Omega)}^2
    \leq e^{C_mT}\int_0^T h(t)\,\mathrm{d}t\to0,\quad\text{as }n\to+\infty.
\end{align*}
This shows that the map $\Lambda$ is continuous on $S$.

Hence, applying Schauder's fixed point theorem, we conclude that the map $\Lambda$ admits a fixed point in $S$. This implies the existence of an approximate solution $(\mathbf{v}_m,\boldsymbol{\varphi}_m,\boldsymbol{\mu}_m)$ on $[0,T]$ for any $m\in\mathbb{N}^+$, which has the regularities \eqref{REG:APPROX} and satisfies \eqref{eqd1}--\eqref{approinitial}. This completes the proof.


\section{Existence of Quasi-strong Solutions in the Case \texorpdfstring{$(K,L)\in(0,+\infty]\times(0,+\infty)$}
{(K,L) ∈ (0,+∞] × (0,+∞)}}
\label{existence of quasi}
\setcounter{equation}{0}

In this section, we establish our main result Theorem~\ref{quasi-strong} in the case $(K,L)\in(0,+\infty]\times(0,+\infty)$.
To this end, we assume that the premises of Theorem~\ref{quasi-strong} hold and we take
\begin{align}
    \label{assump-0323-1}
    m,\,k\in\mathbb{N}^+,\qquad 0<\sigma\leq \frac{c_1}{2(1+c_1)} < \frac 12,\qquad0<\gamma<1
\end{align}
for a constant $c_1 > 0$ (independent of $m$, $k$, $\sigma$ and $\gamma$), which will be specified later in \eqref{MZ-ieq-1}.

We first derive uniform \emph{a priori} estimates
for the approximate solutions
$(\mathbf{v}_m,\boldsymbol{\varphi}_m,\boldsymbol{\mu}_m)$ obtained in Section~\ref{approximating system for L}. Then, by means of compactness arguments, we pass to the limits as
$m \to +\infty$, $\sigma \to 0$, $\gamma=k^{-1} \to 0$ to conclude the existence of a global quasi-strong solution $[0,T]$.


\subsection{Uniform estimates for the approximate solutions}
First of all, we have the following mass conservation law.
\begin{lemma}
\label{mass-conservation-law}
Let $L\in(0,+\infty)$. Then for all $t\in[0,T]$, it holds that
\begin{align}
	\int_\Omega\varphi_m(t)\,\mathrm{d}x+\int_\Gamma \psi_m(t)\,\mathrm{d}S
    = \int_\Omega\varphi_{0,\sigma,k}\,\mathrm{d}x+\int_\Gamma \psi_{0,\sigma,k}\,\mathrm{d}S.
    \label{0321-mass-L}
\end{align}
\end{lemma}

\begin{proof}
Integrating \eqref{eqd2} over $\Omega$ and \eqref{eqd3} over $\Gamma$,
adding the resultants together, and using integration by parts, we obtain
\begin{align}
        \frac{\mathrm{d}}{\mathrm{d}t}\Big(\int_\Omega \varphi_m(t)\,\mathrm{d}x+\int_\Gamma \psi_m(t)\,\mathrm{d}S\Big)
        &=-\int_\Omega \mathbf{v}_m\cdot\nabla\varphi_m\,\mathrm{d}x+\int_\Omega\Delta\mu_m\,\mathrm{d}x
        \notag\\
        &\quad-\int_\Gamma\mathbf{v}_{m,\boldsymbol{\tau}}\cdot\nabla_{\!\boldsymbol{\tau}\,}\psi_m\,\mathrm{d}S
        +\int_\Gamma \Delta_{\boldsymbol{\tau}}\mathcal{L}_m\,\mathrm{d}S
        -\int_\Gamma\partial_{\mathbf{n}}\mu_m\,\mathrm{d}S
        \notag\\
        &=-\int_\Omega\mathrm{div}(\varphi_m\mathbf{v}_m)\,\mathrm{d}x
        +\int_\Omega\varphi_m \,\mathrm{div}\,\mathbf{v}_m\,\mathrm{d}x
        +\int_\Gamma\partial_{\mathbf{n}}\mu_m \,\mathrm{d}S
        \notag\\
        &\quad-\int_\Gamma  \mathrm{div}_{\Gamma}(\psi_m\mathbf{v}_{m,\boldsymbol{\tau}})\,\mathrm{d}S
        +\int_\Gamma \psi_m \, \mathrm{div}_{\Gamma} \mathbf{v}_{m,\boldsymbol{\tau}}\,\mathrm{d}S
        -\int_\Gamma\partial_{\mathbf{n}}\mu_m\,\mathrm{d}S
        \notag\\
        &=0,
        \label{0321-mass-1}
\end{align}
where we have used the identities
$\mathrm{div}\,\mathbf{v}_m=0$ in $\Omega$,
and $\mathrm{div}_{\Gamma}\mathbf{v}_{m,\boldsymbol{\tau}}=0$, $\mathbf{v}_m\cdot\mathbf{n}=0$ on $\Gamma$. Then, integrating \eqref{0321-mass-1} from $0$ to $t$, we obtain \eqref{0321-mass-L}. Hence, we complete the proof of Lemma~\ref{mass-conservation-law}.
\end{proof}

The following lemma provides uniform \textit{a priori} bounds on the approximate solutions.

\begin{lemma}
    \label{energy-estimates}
There exists a positive constant $C$, independent of $m$, $\gamma$, $\sigma$, $k$ {and $L$},
such that
\begin{align}
	\|\mathbf{v}_m\|_{L^\infty(0,T;\mathbf{H}^1_{\mathrm{div}}(\Omega))}
    +\|\boldsymbol{\varphi}_m\|_{L^{\infty}(0,T;\mathcal{H}^1)}
    &\leq C,
    \label{m-uni-7}
    \\
	\|\nabla\mu_m\|_{L^2(0,T;\mathbf{L}^2(\Omega))}
    +\|\nabla_{\!\bm{\tau}\,}\mathcal{L}_m\|_{L^2(0,T;\mathbf{L}^2(\Gamma))}
    +\frac{1}{\sqrt{L}}\|\mathcal{L}_m-\mu_m\|_{L^2(0,T;H_\Gamma)}
    &\leq C,\label{m-uni-9}
    \\
    \sqrt{\gamma}\|\partial_t\bm{\varphi}_m\|_{L^2(0,T;\mathcal{L}^2)}
    &\leq C.\label{m-uni-10}
\end{align}
\end{lemma}

\begin{proof}
Multiplying \eqref{eqd2} by $\mu_m$ and integrating by parts,
taking \eqref{eqd4} into account, we obtain
\begin{align}
	&\frac{\mathrm{d}}{\mathrm{d}t}\Big(\int_\Omega\frac{1}{2}|\nabla\varphi_m|^2\,\mathrm{d}x
    +\int_\Omega F(\varphi_m)\,\mathrm{d}x+k^{-1}\int_\Omega F_\sigma(\varphi_m)\,\mathrm{d}x\Big)\notag\\
    &\qquad+\int_\Omega|\nabla\mu_m|^2\,\mathrm{d}x+\gamma\int_\Omega|\partial_t\varphi_m|^2\,\mathrm{d}x\notag\\
    &\quad=-\int_\Omega \mu_m\nabla\varphi_m\cdot\mathbf{v}_m\,\mathrm{d}x
    +\int_\Gamma \partial_t\varphi_m\partial_\mathbf{n}\varphi_m\,\mathrm{d}S
    +\int_\Gamma \partial_{\mathbf{n}}\mu_m\,\mu_m\,\mathrm{d}S.\label{0227-7}
\end{align}
Similarly, testing \eqref{eqd3} by $\mathcal{L}_m$ and using the boundary condition \eqref{eqd3'},
we obtain
\begin{align}
	&\frac{\mathrm{d}}{\mathrm{d}t}\Big(\int_\Gamma\frac{1}{2}|\nabla_{\!\bm{\tau}\,}\psi_m|^2\,\mathrm{d}S
    +\int_\Gamma G(\psi_m)\,\mathrm{d}S+k^{-1}\int_\Gamma F_\sigma(\psi_m)\,\mathrm{d}S\Big)
    \notag\\
    &\qquad
    +\int_\Gamma|\nabla_{\!\bm{\tau}\,}\mathcal{L}_m|^2\,\mathrm{d}S
    +\gamma\int_\Gamma|\partial_t\psi_m|^2\,\mathrm{d}S\notag\\
    &\quad=\int_\Gamma\psi_m\nabla_{\!\bm{\tau}\,}\mathcal{L}_m\cdot\mathbf{v}_{m,\bm{\tau}}\,\mathrm{d}S
    -\int_\Gamma \partial_t\psi_m\partial_\mathbf{n}\varphi_m\,\mathrm{d}S
    -\int_\Gamma\partial_{\mathbf{n}}\mu_m\,\mathcal{L}_m\,\mathrm{d}S.\label{0227-8}
\end{align}
Combining \eqref{uni-ga-1}, \eqref{0227-7} and \eqref{0227-8}, and
using the boundary condition \eqref{eqd3'}, we deduce that
\begin{align}
&\frac{\mathrm{d}}{\mathrm{d}t}E_{\text{tot}}^{\sigma,k}(\mathbf{v}_m,\boldsymbol{\varphi}_m)
+\int_\Omega2\nu(\varphi_m)|\mathbb{D}\mathbf{v}_m|^2\,\mathrm{d}x
+\int_\Gamma\beta(\psi_m) |\mathbf{v}_{m,\bm{\tau}}|^2\,\mathrm{d}S\notag\\
&\qquad+\gamma\int_\Omega|\partial_t\varphi_m|^2\,\mathrm{d}x
+\gamma\int_\Gamma|\partial_t\psi_m|^2\,\mathrm{d}S\notag\\
&\qquad+\int_\Omega|\nabla\mu_m|^2\,\mathrm{d}x
+\int_\Gamma|\nabla_{\!\bm{\tau}\,}\mathcal{L}_m|^2\,\mathrm{d}S
+\frac{1}{L}\int_\Gamma|\mathcal{L}_m-\mu_m|^2\,\mathrm{d}S=0,\label{m-uni-4}
\end{align}
where the approximate total energy functional $E_{\text{tot}}^{\sigma,k}$ is given by
\begin{align*}
E_{\text{tot}}^{\sigma,k}(\mathbf{v}_m,\boldsymbol{\varphi}_m)
\coloneqq E_{\text{kin}}(\mathbf{v}_m,\boldsymbol{\varphi}_m)
+E_{\text{free}}^{\sigma,k}(\boldsymbol{\varphi}_m).
\end{align*}
Integrating \eqref{m-uni-4} from $0$ to $t$ and using \eqref{Pm-stability}, we find that
\begin{align}
	&E_{\text{tot}}^{\sigma,k}(\mathbf{v}_m(t),\boldsymbol{\varphi}_m(t))
    +2\nu_\ast\int_0^t \int_\Omega|\mathbb{D}\mathbf{v}_m(s)|^2\,\mathrm{d}x\,\mathrm{d}s
    +\beta_\ast\int_0^t\int_\Gamma |\mathbf{v}_{m,\bm{\tau}}(s)|^2\,\mathrm{d}S\,\mathrm{d}s
    \notag\\
    &\qquad+\gamma\int_0^t \int_\Omega|\partial_t\varphi_m(s)|^2\,\mathrm{d}x\,\mathrm{d}s
    +\gamma\int_0^t\int_\Gamma |\partial_t\psi_m(s)|^2\,\mathrm{d}S\,\mathrm{d}s\notag\\
	&\qquad+\int_0^t\int_\Omega |\nabla\mu_m(s)|^2\,\mathrm{d}x\,\mathrm{d}s
    +\int_0^t\int_\Gamma |\nabla_{\!\bm{\tau}\,}\mathcal{L}_m(s)|^2\,\mathrm{d}S\,\mathrm{d}s
    +\frac{1}{L}\int_0^t\int_\Gamma|\mathcal{L}_m(s)-\mu_m(s)|^2\,\mathrm{d}S\,\mathrm{d}s
    \notag\\
	&\quad\leq E_{\text{free}}^{\sigma,k}\big(\boldsymbol{\varphi}_{0,\sigma,k}\big)
      +\frac{1}{2}\int_{\Omega}\rho(\varphi_{0,\sigma,k})|\mathbb{P}_m^\Omega(\mathbf{v}_0,\mathbf{v}_{0,\boldsymbol{\tau}})|^2\,\mathrm{d}x
   +\alpha\int_\Omega |\mathbb{D}\mathbb{P}_m^\Omega(\mathbf{v}_0,\mathbf{v}_{0,\boldsymbol{\tau}})|^2\,\mathrm{d}x
   \notag\\
    &\quad\leq E_{\text{free}}^{\sigma,k}\big(\boldsymbol{\varphi}_{0,\sigma,k}\big)
     +\frac{\rho_1}{2}\|(\mathbf{v}_0,\mathbf{v}_{0,\boldsymbol{\tau}})\|_{\boldsymbol{\mathcal{L}}^2}^2+\alpha\|(\mathbf{v}_0,\mathbf{v}_{0,\boldsymbol{\tau}})\|_{\boldsymbol{\mathcal{H}}^1_{0,\mathrm{div}}}^2.
    \label{0328-energy-id}
\end{align}
This implies that \eqref{m-uni-7}, \eqref{m-uni-9} and \eqref{m-uni-10}
hold with a positive constant $C$ that is independent of $m$, $\gamma$, $\sigma$, and $k$.
\end{proof}

Next, we derive a uniform bound on $\boldsymbol{\mu}_m$.

\begin{lemma}
    \label{mu-H1-phi-H2}
There exists a positive constant $C$, {independent of $m$, $\gamma$, $\sigma$ and $k$,} such that
\begin{align*}
        \|\bm{\mu}_m\|_{L^2(0,T;\mathcal{H}^1)}\leq C.
\end{align*}
\end{lemma}

\begin{proof}
We first recall the well-known Miranville--Zelik inequality (cf. \cite{MZ04}), which yields:
\begin{align}
	c_1|f_0(r)|
    &\leq  f_0(r)(r-\overline{m}_{0,\sigma,k})+c_2,
    \quad\text{for all $r\in(-1,1)$},
    \label{MZ-ieq-1}\\
    c_1|g_0(r)|
    &\leq g_0(r)(r-\overline{m}_{0,\sigma,k})+c_2,
    \quad\text{for all $r\in(-1,1)$},
    \label{MZ-ieq-2}
\end{align}
where $\overline{m}_{0,\sigma,k} \coloneqq \overline{m}(\boldsymbol{\varphi}_{0,\sigma,k})$
and the positive constants $c_1$, $c_2$ depend only on $\overline{m}_{0,\sigma,k}$.
According to \eqref{initial-convergence}, we have
\[\lim_{k\to+\infty}\overline{m}(\boldsymbol{\varphi}_{0,k})=\overline{m}(\boldsymbol{\varphi}_{0})\in(-1,1),\]
which, together with the fact $\overline{m}(\boldsymbol{\varphi}_{0,k})\in(-1,1)$ for all $k\in\mathbb{N}^+$, implies that there exists a subinterval $[a,b]\subset(-1,1)$ such that $\overline{m}(\boldsymbol{\varphi}_{0,k})\in[a,b]$ for all $k\in\mathbb{N}^+$.
Moreover, we infer from \eqref{error} that
\begin{align}
    |\overline{m}(\boldsymbol{\varphi}_{0,\sigma,k})|&\leq |\overline{m}(\boldsymbol{\varphi}_{0,\sigma,k})-\overline{m}(\boldsymbol{\varphi}_{0,k})|+|\overline{m}(\boldsymbol{\varphi}_{0,k})|
    \notag\\
    &\leq C\|\boldsymbol{\varphi}_{0,\sigma,k}-\boldsymbol{\varphi}_{0,k}\|_{\mathcal{H}^1}+|\overline{m}(\boldsymbol{\varphi}_{0,k})|
    \notag\\
    &\leq C\sigma^{\frac{1}{2}}+|\overline{m}(\boldsymbol{\varphi}_{0,k})|
    \notag,
\end{align}
where the constant $C$ is independent of $\sigma$ and $k$, which, together with the fact $\overline{m}_{0,\sigma,k}\in(-1,1)$ for all $\sigma\in(0,1/2)$, $k\in\mathbb{N}^+$, implies that there exists a subinterval $[a',b']\subset(-1,1)$ such that $\overline{m}_{0,\sigma,k}\in[a',b']$.
Therefore, we can select $c_1$ and $c_2$ that are independent of $\sigma$ and $k$.

Testing \eqref{eqd4} by $\varphi_m-\overline{m}_{0,\sigma,k}$
and \eqref{eqd5} by $\psi_m-\overline{m}_{0,\sigma,k}$,
we have
\begin{align}
    &\int_\Omega f_0(\varphi_m)(\varphi_m-\overline{m}_{0,\sigma,k})\,\mathrm{d}x
    +k^{-1}\int_\Omega f_\sigma(\varphi_m)(\varphi_m-\overline{m}_{0,\sigma,k})\,\mathrm{d}x\notag\\
    &\qquad+\int_\Gamma g_0(\psi_m)(\psi_m-\overline{m}_{0,\sigma,k})\,\mathrm{d}S
    +k^{-1}\int_\Gamma f_\sigma(\psi_m)(\psi_m-\overline{m}_{0,\sigma,k})\,\mathrm{d}S\notag\\
    &=\int_\Omega(\mu_m-\gamma\partial_t\varphi_m-f_1(\varphi_m))(\varphi_m-\overline{m}_{0,\sigma,k})\,\mathrm{d}x
    \notag\\
    &\qquad+\int_\Gamma( \mathcal{L}_m-\gamma\partial_t\psi_m-g_1(\psi_m))(\psi_m-\overline{m}_{0,\sigma,k})\,\mathrm{d}S
    \notag\\
    &\qquad-\int_\Omega |\nabla\varphi_m|^2\,\mathrm{d}x-\int_\Gamma|\nabla_{\!\bm{\tau}\,}\psi_m|^2\,\mathrm{d}S
    -\chi(K)\int_\Gamma |\varphi_m-\psi_m|^2\,\mathrm{d}S.
    \label{uni-0323-1}
\end{align}
For the terms on the left-hand side of \eqref{uni-0323-1},
using \eqref{MZ-ieq-1} and \eqref{MZ-ieq-2},
we obtain
\begin{align}
    &\int_\Omega f_0(\varphi_m)(\varphi_m-\overline{m}_{0,\sigma,k})\,\mathrm{d}x
    +\int_\Gamma g_0(\psi_m)(\psi_m-\overline{m}_{0,\sigma,k})\,\mathrm{d}S\notag\\
    &\quad\geq c_1\int_\Omega |f_0(\varphi_m)|\,\mathrm{d}x
    +c_1\int_\Gamma |g_0(\psi_m)|\,\mathrm{d}S-c_2(|\Omega|+|\Gamma|).    \label{uni-0323-2}
\end{align}
Moreover, due to \eqref{assump-0323-1}, it holds
\begin{align}
    &k^{-1}\int_\Omega f_\sigma(\varphi_m)(\varphi_m-\overline{m}_{0,\sigma,k})\,\mathrm{d}x
    +k^{-1}\int_\Gamma f_\sigma(\psi_m)(\psi_m-\overline{m}_{0,\sigma,k})\,\mathrm{d}S
    \notag\\
    &\quad=k^{-1}(1-\sigma)\int_\Omega f_0\Big(\frac{\varphi_m}{1-\sigma}\Big)\Big(\frac{\varphi_m}{1-\sigma}-\overline{m}_{0,\sigma,k}\Big)\,\mathrm{d}x
    \notag\\
    &\qquad +k^{-1}(1-\sigma)\int_\Gamma f_0\Big(\frac{\psi_m}{1-\sigma}\Big)\Big(\frac{\psi_m}{1-\sigma}-\overline{m}_{0,\sigma,k}\Big)\,\mathrm{d}S
    \notag\\
    &\qquad-k^{-1}\sigma\, \overline{m}_{0,\sigma,k}\int_\Omega f_0\Big(\frac{\varphi_m}{1-\sigma}\Big)\,\mathrm{d}x
    -k^{-1}\sigma \overline{m}_{0,\sigma,k}\int_\Gamma f_0\Big(\frac{\psi_m}{1-\sigma}\Big)\,\mathrm{d}S
    \notag\\
    &\quad\geq k^{-1}(1-\sigma)c_1\Big(\int_\Omega\Big|f_0\Big(\frac{\varphi_m}{1-\sigma}\Big)\Big|\,\mathrm{d}x
    +\int_\Gamma\Big|f_0\Big(\frac{\psi_m}{1-\sigma}\Big)\Big|\,\mathrm{d}S\Big) -k^{-1}(1-\sigma)c_2(|\Omega|+|\Gamma|)\notag\\
    &\qquad-k^{-1}\sigma\int_\Omega \Big|f_0\Big(\frac{\varphi_m}{1-\sigma}\Big)\Big|\,\mathrm{d}x
    -k^{-1}\sigma \int_\Gamma \Big|f_0\Big(\frac{\psi_m}{1-\sigma}\Big)\Big|\,\mathrm{d}S
    \notag\\
    &\quad=k^{-1}(c_1(1-\sigma)-\sigma)\Big(\int_\Omega\Big|f_0\Big(\frac{\varphi_m}{1-\sigma}\Big)\Big|\,\mathrm{d}x
    +\int_\Gamma\Big|f_0\Big(\frac{\psi_m}{1-\sigma}\Big)\Big|\,\mathrm{d}S\Big) -k^{-1}(1-\sigma)c_2(|\Omega|+|\Gamma|)
    \notag\\
    &\quad\geq \frac{c_1}{2}\Big(k^{-1}\int_\Omega\Big|f_0\Big(\frac{\varphi_m}{1-\sigma}\Big)\Big|\,\mathrm{d}x
    +k^{-1}\int_\Gamma\Big|f_0\Big(\frac{\psi_m}{1-\sigma}\Big)\Big|\,\mathrm{d}S\Big)-k^{-1}(1-\sigma)c_2(|\Omega|+|\Gamma|)
    \notag\\
    &\quad=\frac{c_1}{2}\Big(\int_\Omega k^{-1}|f_\sigma(\varphi_m)|\,\mathrm{d}x
    +\int_\Gamma k^{-1} |f_\sigma(\psi_m)|\,\mathrm{d}S\Big)-k^{-1}(1-\sigma)c_2(|\Omega|+|\Gamma|).\label{uni-0323-3}
\end{align}
For the terms on the right-hand side of \eqref{uni-0323-1},
we apply  H\"older's inequality to get
\begin{align*}
&\int_\Omega(\mu_m-\gamma\partial_t\varphi_m-f_1(\varphi_m))(\varphi_m-\overline{m}_{0,\sigma,k})\,\mathrm{d}x\notag\\
&\qquad+\int_\Gamma( \mathcal{L}_m-\gamma\partial_t\psi_m-g_1(\psi_m))(\psi_m-\overline{m}_{0,\sigma,k})\,\mathrm{d}S\notag\\
 &\qquad-\int_\Omega |\nabla\varphi_m|^2\,\mathrm{d}x-\int_\Gamma|\nabla_{\!\bm{\tau}\,}\psi_m|^2\,\mathrm{d}S
 {-\chi(K)\int_\Gamma |\varphi_m-\psi_m|^2\,\mathrm{d}S}\\
&\quad\leq  \int_\Omega(\mu_m-\overline{m}(\boldsymbol{\mu}_m))\varphi_m\,\mathrm{d}x
+\int_\Gamma(\mathcal{L}_m-\overline{m}(\boldsymbol{\mu}_m))\psi_m\,\mathrm{d}S\notag\\
&\qquad-\int_\Omega\gamma\partial_t\varphi_m(\varphi_m-\overline{m}_{0,\sigma,k})\,\mathrm{d}x
-\int_\Gamma\gamma\partial_t\psi_m(\psi_m-\overline{m}_{0,\sigma,k})\,\mathrm{d}S\notag\\
&\qquad-\int_\Omega f_1(\varphi_m)(\varphi_m-\overline{m}_{0,\sigma,k})\,\mathrm{d}x
-\int_\Gamma g_1(\psi_m)(\psi_m-\overline{m}_{0,\sigma,k})\,\mathrm{d}S\notag\\
&\quad\leq C\Big(1+\|\boldsymbol{\mu}_m-\overline{m}(\boldsymbol{\mu}_m)\boldsymbol{1}\|_{\mathcal{L}^2}
+\gamma\|\partial_t\boldsymbol{\varphi}_m\|_{\mathcal{L}^2}\Big)\notag\\
&\quad\leq C\Big(1+\|\nabla\mu_m\|_{\mathbf{L}^2(\Omega)}
+\|\nabla_{\!\bm{\tau}\,}\mathcal{L}_m\|_{\mathbf{L}^2(\Gamma)}
+\frac{1}{\sqrt{L}}\|\mathcal{L}_m-\mu_m\|_{H_\Gamma}
+\gamma\|\partial_t\boldsymbol{\varphi}_m\|_{\mathcal{L}^2}\Big).
\end{align*}
Collecting \eqref{uni-0323-1}, \eqref{uni-0323-2} and \eqref{uni-0323-3}, we obtain
\begin{align}
    &\|f_0(\varphi_m)\|_{L^1(\Omega)} +\|g_0(\psi_m)\|_{L^1(\Gamma)}+k^{-1}\|f_\sigma(\varphi_m)\|_{L^1(\Omega)}+k^{-1}\|f_\sigma(\psi_m)\|_{L^1(\Gamma)}
    \notag\\
    &\quad\leq C(1+\|\mathbb{P}\boldsymbol{\mu}_m\|_{\mathcal{H}_{L,0}^1}+\gamma\|\partial_t\boldsymbol{\varphi}_m\|_{\mathcal{L}^2}),
    \label{260807-1}
\end{align}
which together with \eqref{m-uni-9} and \eqref{m-uni-10} implies that
\begin{align}
\|f_0(\varphi_m)\|_{L^2(0,T;L^1(\Omega))}+	\|g_0(\psi_m)\|_{L^2(0,T;L^1(\Gamma))}
&\leq C,
\label{mu-uni-1}\\
k^{-1}\|f_\sigma(\varphi_m)\|_{L^2(0,T;L^1(\Omega))}+ k^{-1}\|f_\sigma(\psi_m)\|_{L^2(0,T;L^1(\Gamma))}
&\leq C.
\label{uni-0323-5}
\end{align}
Integrating \eqref{eqd4} over $\Omega$ and \eqref{eqd5} over $\Gamma$, adding the resultants together, we have
\begin{align}
&\Big|\int_\Omega \mu_m\,\mathrm{d}x+\int_\Gamma \mathcal{L}_m\,\mathrm{d}S\Big|
\notag\\
&\quad\leq \Big|\int_\Omega f_0(\varphi_m)\,\mathrm{d}x
+k^{-1}\int_\Omega f_\sigma(\varphi_m)\,\mathrm{d}x+\int_\Omega f_1(\varphi_m)\,\mathrm{d}x\Big|
\notag\\
&\qquad+\Big|\int_\Gamma g_0(\psi_m)\,\mathrm{d}S
+k^{-1}\int_\Gamma f_\sigma(\psi_m)\,\mathrm{d}S+\int_\Gamma g_1(\psi_m)\,\mathrm{d}S\Big|
\notag\\
&\quad\leq C\Big(1+\|f_0(\varphi_m)\|_{L^1(\Omega)}+\|g_0(\psi_m)\|_{L^1(\Gamma)}
+k^{-1}\|f_\sigma(\varphi_m)\|_{L^1(\Omega)}+k^{-1}\|f_\sigma(\psi_m)\|_{L^1(\Gamma)}\Big).
\label{260807-2}
\end{align}
This combined with \eqref{mu-uni-1} and \eqref{uni-0323-5} implies that
\begin{align}
	\|\overline{m}(\boldsymbol{\mu}_m)\|_{L^2(0,T)}\leq C.\label{mu-uni-2}
\end{align}
Hence, using Lemma~\ref{energy-estimates} and the generalized Poincar\'e inequality, we can conclude
\begin{align}
	\|\boldsymbol{\mu}_m \|_{L^2(0,T;\mathcal{H}^1)}\leq C,\label{mu-uni-3}
\end{align}
for some positive constant $C$ independent of $m$, $\gamma$, $\sigma$ and $k$.
\end{proof}

The following lemma provides uniform bounds on the singular nonlinearities.
\begin{lemma}
\label{f-L2L2}
There exists a positive constant $C$, {independent of $m$, $\gamma$, $\sigma$ and $k$}, such that
\begin{align}
\|f_0(\varphi_m)\|_{L^2(0,T;H)}+\|f_0(\psi_m)\|_{L^2(0,T;H_\Gamma)}
&\leq C,\label{uni-0323-f_0-L2-L2}
\\
{k^{-1}}\|f_\sigma(\varphi_m)\|_{L^2(0,T;H)}+{k^{-1}}\|f_\sigma(\psi_m)\|_{L^2(0,T;H_\Gamma)}
&\leq C.\label{uni-0323-f_lambda-L2}
\end{align}
\end{lemma}

\begin{proof}
Multiplying \eqref{eqd4} by $f_0(\varphi_m)$ and integrating over $\Omega$,
multiplying \eqref{eqd5} by $f_0(\psi_m)$ and integrating over $\Gamma$, adding the resultants together,
we find that
\begin{align}
&\int_\Omega|f_0(\varphi_m)|^2\,\mathrm{d}x+\int_\Gamma g_0(\psi_m)f_0(\psi_m)\,\mathrm{d}S
+k^{-1}\int_\Omega f_\sigma(\varphi_m)f_0(\varphi_m)\,\mathrm{d}x
+k^{-1}\int_\Gamma f_\sigma(\psi_m) f_0(\psi_m)\,\mathrm{d}S
\notag\\
&\quad=\int_\Omega(\mu_m-\gamma\partial_t\varphi_m-f_1(\varphi_m))f_0(\varphi_m)\,\mathrm{d}x
+\int_\Gamma(\mathcal{L}_m-\gamma\partial_t\psi_m-g_1(\psi_m))f_0(\psi_m)\,\mathrm{d}S
\notag\\
&\qquad-\int_\Omega f_0'(\varphi_m)|\nabla\varphi_m|^2\,\mathrm{d}x
-\int_\Gamma g_0'(\psi_m)|\nabla_{\!\bm{\tau}\,}\psi_m|^2\,\mathrm{d}S
 {+\int_\Gamma \partial_{\mathbf{n}}\varphi_m (f_0(\varphi_m)-f_0(\psi_m))\,\mathrm{d}S}
 \notag\\
&\quad\leq \int_\Omega(\mu_m-\gamma\partial_t\varphi_m-f_1(\varphi_m))f_0(\varphi_m)\,\mathrm{d}x
+\int_\Gamma(\mathcal{L}_m-\gamma\partial_t\psi_m-g_1(\psi_m))f_0(\psi_m)\,\mathrm{d}S.
\label{0306-1}
\end{align}
Due to \eqref{f_lambda_p3}, it holds
\begin{align}
    k^{-1}\int_\Omega f_\sigma(\varphi_m)f_0(\varphi_m)\,\mathrm{d}x
    +k^{-1}\int_\Gamma f_\sigma(\psi_m) f_0(\psi_m)\,\mathrm{d}S\geq0.\label{uni-0323-sign-1}
\end{align}
Using the compatibility condition \ref{ASS:A3} and noticing that $f_0(r)$, $g_0(r)$ have the same sign for $r\in(-1,1)$,
we get
\begin{align}
	\int_\Gamma f_0(\psi_m) g_0(\psi_m)\,\mathrm{d}S
    \geq\frac{1}{2\varrho}\int_\Gamma|f_0(\psi_m)|^2\,\mathrm{d}S-C.\label{0306-1'}
\end{align}
The right-hand side of \eqref{0306-1} can be estimated as follows:
\begin{align}
	&\int_\Omega(\mu_m-\gamma\partial_t\varphi_m-f_1(\varphi_m))f_0(\varphi_m)\,\mathrm{d}x
    +\int_\Gamma(\mathcal{L}_m-\gamma\partial_t\psi_m-g_1(\psi_m))f_0(\psi_m)\,\mathrm{d}S
    \notag\\
	&\quad \leq\frac{1}{2}\int_\Omega|f_0(\varphi_m)|^2\,\mathrm{d}x
    +\frac{1}{4\varrho}\int_\Gamma|f_0(\psi_m)|^2\,\mathrm{d}S\notag\\
	&\qquad +\frac{3}{2}\Big(\int_\Omega|\mu_m|^2\,\mathrm{d}x
    +\gamma^2\int_\Omega |\partial_t\varphi_m|^2\,\mathrm{d}x
    +\int_\Omega|f_1(\varphi_m)|^2 \,\mathrm{d}x\Big)
    \notag\\
    &\qquad+3\varrho\Big(\int_\Gamma |\mathcal{L}_m|^2\,\mathrm{d}S
    +\gamma^2\int_\Gamma|\partial_t\psi_m|^2 \,\mathrm{d}S
    +\int_\Gamma|g_1(\psi_m)|^2\,\mathrm{d}S\Big). \label{0306-2}
\end{align}
Combining \eqref{m-uni-10}, \eqref{mu-uni-3} and \eqref{0306-1}--\eqref{0306-2},
we obtain \eqref{uni-0323-f_0-L2-L2}.
Next, multiplying \eqref{eqd4} by $k^{-1} f_\sigma(\varphi_m)$ and integrating over $\Omega$,
multiplying \eqref{eqd5} by $k^{-1} f_\sigma(\psi_m)$ and integrating over $\Gamma$,
a similar line of argument leads to \eqref{uni-0323-f_lambda-L2}.
\end{proof}

In the following three lemmas, we derive some uniform higher-order estimates.
\begin{lemma}
\label{H2}
There exists a positive constant $C$, {independent of $m$, $\gamma$, $\sigma$ and $k$}, such that
\begin{align}
\|g_0(\psi_m)\|_{L^2(0,T;H_\Gamma)}+\|\bm{\varphi}_m\|_{L^2(0,T;\mathcal{H}^2)}\leq C.\label{uni-0323-phi-H2}
\end{align}
\end{lemma}

\begin{proof}
Multiplying \eqref{eqd5} by $g_0(\psi_m)$ and integrating over $\Gamma$, we get
\begin{align}
&\int_\Omega|g_0(\psi_m)|^2\,\mathrm{d}S
+k^{-1}\int_\Gamma f_\sigma(\psi_m)g_0(\psi_m)\,\mathrm{d}S
\notag\\
&\quad=\int_\Omega \big(\mathcal{L}_m-\gamma\partial_t\psi_m-g_1(\psi_m)-\partial_{\mathbf{n}}\varphi_m \big)g_0(\psi_m)\,\mathrm{d}S
-\int_\Gamma g_0'(\psi_m)|\nabla_{\!\bm{\tau}\,}\psi_m|^2\,\mathrm{d}S
\notag\\
&\quad\leq \frac{1}{2}\int_\Omega|g_0(\psi_m)|^2\,\mathrm{d}S
+2\Big(\int_\Gamma|\mathcal{L}_m|^2\,\mathrm{d}S+\gamma^2\int_\Gamma|\partial_t\psi_m|^2\,\mathrm{d}S
+\int_\Gamma|g_1(\psi_m)|^2\,\mathrm{d}S
+\int_\Gamma|\partial_{\mathbf{n}}\varphi_m|^2\,\mathrm{d}S\Big).
\notag
\end{align}
This together with \eqref{eqd3''}, \eqref{f_lambda_p3}, \eqref{m-uni-10} and \eqref{mu-uni-3} implies that
\begin{align}
	\|g_0(\psi_m)\|_{L^2(0,T;H_\Gamma)}\leq C.\label{0306-4}
\end{align}
Let us rewrite the equations for $\mu_m$ and $\mathcal{L}_m$ as a bulk-surface elliptic system of $(\varphi_m,\psi_m)$:
\begin{align*}
    &-\Delta\varphi_m=\mu_m-\gamma\partial_t\varphi_m-f(\varphi_m)-k^{-1}f_\sigma(\varphi_m)
    &&\text{a.e.~in }Q_T,
    \\
    &
    \begin{cases}
        K\partial_\mathbf{n}\varphi_m=\psi_m-\varphi_m,& K\in(0,+\infty)\\
        \partial_\mathbf{n}\varphi_m=0,&K=+\infty
    \end{cases}
    && {\text{a.e.~on }\Sigma_T,}
    \\
    &-\Delta_{\bm{\tau}}\psi_m + \partial_{\mathbf{n}}\varphi_m
    =\mathcal{L}_m - \gamma\partial_t\psi_m-g(\psi_m)-k^{-1} f_\sigma(\psi_m)
    &&\text{a.e.~on }\Sigma_T.
\end{align*}
Applying the elliptic regularity theory for bulk-surface elliptic systems (see, e.g., \cite[Proposition~A.1]{KSJDE})
together with \eqref{m-uni-10}, \eqref{mu-uni-3}, \eqref{uni-0323-f_0-L2-L2},
\eqref{uni-0323-f_lambda-L2} and \eqref{0306-4}, we deduce that
\begin{align}
	\|\boldsymbol{\varphi}_m\|_{\mathcal{H}^2}
    &\leq C\Big(1+\|\boldsymbol{\mu}_m\|_{\mathcal{L}^2}
    +\gamma\|\partial_t \boldsymbol{\varphi}_m\|_{\mathcal{L}^2}
    +\|f_0(\varphi_m)\|_H+\|g_0(\psi_m)\|_{H_\Gamma}
    \notag\\
    &\qquad\quad+k^{-1}\|f_\sigma(\varphi_m)\|_H
    +k^{-1}\|f_\sigma(\psi_m) \|_{H_\Gamma}\Big).
    \label{260806-1}
\end{align}
This completes the proof of Lemma~\ref{H2}.
\end{proof}

\begin{lemma}    \label{higher}
There exists a positive constant $C$, independent of $m$, $\gamma$, $\sigma$, and $k$, such that
\begin{align}
\sqrt{\gamma}\|\partial_t\bm{\varphi}_m\|_{L^\infty(0,T;\mathcal{L}^2)}
+\|\bm{\mu}_m\|_{L^\infty(0,T;\mathcal{H}^1)}+\|\partial_t\bm{\varphi}_m\|_{L^2(0,T;\mathcal{H}^1)}
+\|\partial_t\mathbf{v}_m\|_{L^2(0,T;\mathbf{H}^1(\Omega))}\leq C.
\label{0227-23}
\end{align}
\end{lemma}
\begin{proof}
Testing \eqref{eqd2} by $\partial_t\mu_m$ and \eqref{eqd3} by $\partial_t\mathcal{L}_m$,
we see that
\begin{align}
	&\frac{1}{2}\frac{\mathrm{d}}{\mathrm{d}t}\Big(\int_\Omega|\nabla\mu_m|^2\,\mathrm{d}x
    +\int_\Gamma |\nabla_{\!\bm{\tau}\,}\mathcal{L}_m|^2 \,\mathrm{d}S\Big)
    +\int_\Omega\partial_t \mu_m\partial_t\varphi_m\,\mathrm{d}x
    +\int_\Gamma\partial_t \mathcal{L}_m\partial_t\psi_m\,\mathrm{d}S
    \notag\\
	&\quad=\int_\Gamma \partial_{\mathbf{n}}\mu_m(\partial_t\mu_m-\partial_t\mathcal{L}_m)\,\mathrm{d}S
    -\int_\Omega\mathbf{v}_m \cdot\nabla\varphi_m\partial_t\mu_m\,\mathrm{d}x{
    -\int_\Gamma \mathbf{v}_{m,\bm{\tau}}\cdot\nabla_{\!\bm{\tau}\,}\psi_m\,\partial_t\mathcal{L}_m\,\mathrm{d}S}.
    \label{0227-16}
\end{align}
The last two terms on the right-hand side of \eqref{0227-16} can be reformulated as
\begin{align}
	&-\int_\Omega\mathbf{v}_m\cdot\nabla\varphi_m\partial_t\mu_m\,\mathrm{d}x
    \notag\\
	&\quad=-\frac{\mathrm{d}}{\mathrm{d}t}\int_\Omega\mu_m\nabla\varphi_m\cdot\mathbf{v}_m\,\mathrm{d}x
    +\int_\Omega\mu_m \nabla\partial_t\varphi_m\cdot\mathbf{v}_m\,\mathrm{d}x
    +\int_\Omega\mu_m \nabla\varphi_m\cdot\partial_t\mathbf{v}_m\,\mathrm{d}x,
    \label{0227-17}
    \\
	&-\int_\Gamma\mathbf{v}_{m,\bm{\tau}}\cdot\nabla_{\!\bm{\tau}\,}\psi_m \,\partial_t\mathcal{L}_m\,\mathrm{d}S
    \notag\\
	&\quad=-\frac{\mathrm{d}}{\mathrm{d}t}\int_\Gamma\mathcal{L}_m \nabla_{\!\bm{\tau}\,}\psi_m
    \cdot\mathbf{v}_{m,\bm{\tau}}\,\mathrm{d}S
    +\int_\Gamma\mathcal{L}_m \nabla_{\!\bm{\tau}\,}\partial_t\psi_m
    \cdot\mathbf{v}_{m,\bm{\tau}}\,\mathrm{d}S
    +\int_\Gamma\mathcal{L}_m \nabla_{\!\bm{\tau}\,}\psi_m
    \cdot\partial_t \mathbf{v}_{m,\bm{\tau}}\,\mathrm{d}S.
    \label{0227-18}
\end{align}	
Recalling the definition of $\mu_m$ and $\mathcal{L}_m$, we find that
\begin{align}
\int_\Omega\partial_t\mu_m\partial_t\varphi_m\,\mathrm{d}x
&=\frac{\gamma}{2}\frac{\mathrm{d}}{\mathrm{d}t}\|\partial_t\varphi_m\|_{H}^2
+\int_\Omega|\nabla\partial_t\varphi_m|^2\,\mathrm{d}x
+\int_\Omega f_0'(\varphi_m)|\partial_t\varphi_m|^2\,\mathrm{d}x
\notag\\
&\quad+k^{-1}\int_\Omega f_\sigma'(\varphi_m)|\partial_t\varphi_m|^2\,\mathrm{d}x
+\int_\Omega f_1'(\varphi_m)|\partial_t\varphi_m|^2\,\mathrm{d}x
-\int_\Gamma\partial_{\mathbf{n}}\partial_t\varphi_m\,\partial_t\varphi_m\,\mathrm{d}S,
\notag\\
\int_\Gamma\partial_t\mathcal{L}_m\partial_t\psi_m\,\mathrm{d}S
&=\frac{\gamma}{2}\frac{\mathrm{d}}{\mathrm{d}t}\|\partial_t\psi_m\|_{H_\Gamma}^2
+	\int_\Gamma|\nabla_{\!\bm{\tau}\,}\partial_t\psi_m|^2\,\mathrm{d}S
+\int_\Gamma g_0'(\psi_m)|\partial_t\psi_m|^2\,\mathrm{d}S
\notag\\
&\quad+k^{-1}\int_\Gamma f_\sigma'(\psi_m)|\partial_t\psi_m|^2\,\mathrm{d}S
+\int_\Gamma g_1'(\psi_m)|\partial_t\psi_m|^2\,\mathrm{d}S
+\int_\Gamma\partial_{\mathbf{n}}\partial_t\varphi_m\,\partial_t\psi_m\,\mathrm{d}S,
\notag
\end{align}
and
\begin{align}
	\int_\Gamma\partial_{\mathbf{n}} \mu_m(\partial_t\mu_m-\partial_t\mathcal{L}_m)\,\mathrm{d}S
    =-\frac{1}{2L}\frac{\mathrm{d}}{\mathrm{d}t}\int_\Gamma|\mu_m-\mathcal{L}_m|^2\,\mathrm{d}S.
    \notag
\end{align}
Thus, combining \eqref{0227-16}, \eqref{0227-17}, \eqref{0227-18},
and the three identities above, we obtain
\begin{align}
    & \frac{\mathrm{d}}{\mathrm{d}t}\mathcal{G}_m
     +\int_\Omega |\nabla\partial_t\varphi_m|^2\,\mathrm{d}x
    +\int_\Gamma|\nabla_{\!\bm{\tau}\,} \partial_t\psi_m|^2\,\mathrm{d}S +\chi(K)\int_\Gamma|\partial_t\psi_m-\partial_t\varphi_m|^2\,\mathrm{d}S
    \notag\\
    &\quad\leq -\int_\Omega f_1'(\varphi_m)|\partial_t\varphi_m|^2 \,\mathrm{d}x
    -\int_\Gamma g_1'(\psi_m)|\partial_t\psi_m|^2\,\mathrm{d}S
    \notag\\
    &\qquad+\int_\Omega\mu_m \nabla\partial_t\varphi_m\cdot\mathbf{v}_m\,\mathrm{d}x
    +\int_\Omega\mu_m\nabla\varphi_m \cdot\partial_t\mathbf{v}_m\,\mathrm{d}x
    \notag\\
    &\qquad+\int_\Gamma\mathcal{L}_m \nabla_{\!\bm{\tau}\,}\partial_t\psi_m
    \cdot\mathbf{v}_{m,\bm{\tau}}\,\mathrm{d}S
    +\int_\Gamma\mathcal{L}_m \nabla_{\!\bm{\tau}\,}\psi_m
    \cdot\partial_t\mathbf{v}_{m,\bm{\tau}} \,\mathrm{d}S,
    \label{0227-19}
\end{align}
where
\begin{align*}
	\mathcal{G}_m&:=\frac{\gamma}{2}\int_\Omega|\partial_t\varphi_m|^2\,\mathrm{d}x
    +\frac{\gamma}{2}\int_\Gamma|\partial_t\psi_m|^2\,\mathrm{d}S
    +\frac{1}{2}\int_\Omega|\nabla\mu_m|^2\,\mathrm{d}x
    +\frac{1}{2}\int_\Gamma|\nabla_{\!\bm{\tau}\,}\mathcal{L}_m|^2\,\mathrm{d}S
    \\
    &\quad+\frac{1}{2L}\int_\Gamma|\mu_m-\mathcal{L}_m|^2\,\mathrm{d}S
    +\int_\Omega\mu_m\nabla \varphi_m\cdot\mathbf{v}_m\,\mathrm{d}x
    +\int_\Gamma\mathcal{L}_m \nabla_{\!\bm{\tau}\,}\psi_m\cdot\mathbf{v}_{m,\bm{\tau}}\,\mathrm{d}S.
\end{align*}
Combining \eqref{uni-ti-1} (replacing $\mathbf{v}$ therein by $\mathbf{v}_m$) and \eqref{0227-19}, we deduce that
\begin{align}
	&\frac{\mathrm{d}}{\mathrm{d}t}\mathcal{G}_m
    +\int_\Omega|\nabla\partial_t\varphi_m|^2\,\mathrm{d}x
    +\int_\Gamma|\nabla_{\!\bm{\tau}\,} \partial_t\psi_m|^2\,\mathrm{d}S +\chi(K)\int_\Gamma|\partial_t\psi_m-\partial_t\varphi_m|^2\,\mathrm{d}S
    \notag\\
    &\qquad+\int_\Omega \rho(\varphi_m)|\partial_t \mathbf{v}_m|^2\,\mathrm{d}x
    +2\alpha\int_\Omega|\mathbb{D}\partial_t \mathbf{v}_m|^2\,\mathrm{d}x
    \notag\\
	&\quad\leq -2\int_\Omega\varphi_m\nabla\mu_m\cdot\partial_t\mathbf{v}_m\,\mathrm{d}x
    -2\int_\Gamma\psi_m\nabla_{\!\bm{\tau}\,} \mathcal{L}_m\cdot\partial_t\mathbf{v}_{m,\bm{\tau}}\,\mathrm{d}S
    -\int_\Omega2\nu(\varphi_m)\mathbb{D}\mathbf{v}_m:\mathbb{D}\partial_t\mathbf{v}_m\,\mathrm{d}x
    \notag\\
	&\qquad-\int_\Gamma\beta(\psi_m)\mathbf{v}_{m,\bm{\tau}}\cdot\partial_t\mathbf{v}_{m,\bm{\tau}}\,\mathrm{d}S
    -\int_\Omega(\rho(\varphi_m)\mathbf{v}_m\cdot\nabla)\mathbf{v}_m\cdot\partial_t\mathbf{v}_m\,\mathrm{d}x
    \notag\\
	&\qquad+\int_\Omega\mu_m\nabla\partial_t \varphi_m\cdot\mathbf{v}_m\,\mathrm{d}x
  +\int_\Gamma\mathcal{L}_m\nabla_{\!\bm{\tau}\,} \partial_t\psi_m\cdot\mathbf{v}_{m,\bm{\tau}}\,\mathrm{d}S
    \notag\\
    &\qquad-\int_\Omega f_1'(\varphi_m)|\partial_t\varphi_m|^2\,\mathrm{d}x
    -\int_\Gamma g_1'(\psi_m)|\partial_t\psi_m|^2\,\mathrm{d}S\notag\\
    &\qquad+\rho'\int_\Omega (\nabla\mu_m\cdot\nabla)\mathbf{v}_m\cdot\partial_t\mathbf{v}_m\,\mathrm{d}x
    -\frac{\rho'}{2}\int_\Gamma(\nabla\mu_m\cdot\mathbf{n})\mathbf{v}_{m,\bm{\tau}}\cdot\partial_t\mathbf{v}_{m,\bm{\tau}}\,\mathrm{d}S\notag\\
    &\quad {\eqqcolon} \sum_{j=1}^{11}K_j .
    \label{0227-20}
\end{align}
Furthermore, by \eqref{m-uni-7}, we obtain
\begin{align*}
    &\int_\Omega\mu_m\nabla \varphi_m\cdot\mathbf{v}_m\,\mathrm{d}x
    +\int_\Gamma\mathcal{L}_m \nabla_{\!\bm{\tau}\,}\psi_m\cdot\mathbf{v}_{m,\bm{\tau}}\,\mathrm{d}S
    \\
    &\quad=-\int_\Omega\varphi_m\mathbf{v}_m\cdot\nabla\mu_m\,\mathrm{d}x
    -\int_\Gamma \psi_m{\mathbf{v}}_{m,\bm{\tau}}\cdot\nabla_{\!\bm{\tau}\,}\mathcal{L}_m\,\mathrm{d}S
    \\
	&\quad\leq \|\nabla\mu_m\|_{\mathbf{L}^2(\Omega)}\|\varphi_m\|_{L^\infty(\Omega)}\|\mathbf{v}_m\|_{\mathbf{L}^2(\Omega)}
    +\|\nabla_{\!\bm{\tau}\,} \mathcal{L}_m\|_{\mathbf{L}^2(\Gamma)}\|\psi_m\|_{L^\infty(\Gamma)}\|\mathbf{v}_{m,\bm{\tau}}\|_{\mathbf{L}^2(\Gamma)}
    \\[1mm]
	&\quad\leq C(\|\nabla\mu_m\|_{\mathbf{L}^2(\Omega)}
    +\|\nabla_{\!\bm{\tau}\,} \mathcal{L}_m\|_{\mathbf{L}^2(\Gamma)})\|\bm{\varphi}_m\|_{\mathcal{L}^\infty}\|\mathbf{v}_m\|_{\mathbf{H}^1(\Omega)}
    \\
	&\quad\leq \frac{1}{4}\Big(\int_\Omega|\nabla\mu_m|^2\,\mathrm{d}x
    +\int_\Gamma|\nabla_{\!\bm{\tau}\,} \mathcal{L}_m|^2\,\mathrm{d}S
    +\frac{1}{L}\int_\Gamma|\mu_m-\mathcal{L}_m|^2\,\mathrm{d}S\Big)+C.
\end{align*}
Hence, it follows that
\begin{align}
	\mathcal{G}_m&\geq\frac{\gamma}{2}\int_\Omega|\partial_t\varphi_m|^2\,\mathrm{d}x
    +\frac{\gamma}{2} \int_\Gamma |\partial_t\psi_m|^2\,\mathrm{d}S
    \notag\\
    &\quad+\frac{1}{4}\Big(\int_\Omega|\nabla\mu_m|^2\,\mathrm{d}x
    +\int_\Gamma|\nabla_{\!\bm{\tau}\,} \mathcal{L}_m|^2\,\mathrm{d}S
    +\frac{1}{L}\int_\Gamma|\mu_m-\mathcal{L}_m|^2\,\mathrm{d}S\Big)-C.
    \label{0227-21}
\end{align}
Now we proceed to estimate $K_j$ for all $1\leq j\leq 11$.
First, arguing similarly to the estimates for $I_{1}$, $I_2$, $I_6$ and $I_7$ in Section~\ref{SECT:EXAP}, we have
\begin{align*}
    \sum_{j=1}^4 K_j&\leq \frac{\rho_2}{4}\|\partial_t\mathbf{v}_m\|_{\mathbf{L}^2(\Omega)}^2
    +\frac{3\alpha}{4}\|\mathbb{D}\partial_t\mathbf{v}_m\|_{\mathbf{L}^2(\Omega)}^2
    \\
    &\qquad+C\Big(\|\nabla\mu_m \|_{\mathbf{L}^2(\Omega)}^2
    +\|\nabla_{\!\boldsymbol{\tau}\,} \mathcal{L}_m\|_{\mathbf{L}^2(\Gamma)}^2
    +\|\mathbf{v}_{m,\boldsymbol{\tau}} \|_{\mathbf{L}^2(\Gamma)}^2
    +\|\mathbb{D}\mathbf{v}_m \|_{\mathbf{L}^2(\Omega)}^2\Big).
\end{align*}
For $K_j$ with $5\leq j\leq 11$, applying \eqref{Korn-cor}, we can deduce that
\begin{align*}
	K_5&\leq C\|\mathbf{v}_m\|_{\mathbf{L}^4(\Omega)}\|\nabla\mathbf{v}_m\|_{\mathbf{L}^2(\Omega)}\|\partial_t\mathbf{v}_m\|_{\mathbf{L}^4(\Omega)}
    \\
	&\leq C\|\mathbf{v}_m\|_{\mathbf{H}^1(\Omega)}^2\|\partial_t\mathbf{v}_m\|_{\mathbf{H}^1(\Omega)}
    \\
    &\leq C\|\mathbf{v}_m\|_{\mathbf{H}^1(\Omega)}^2(\|\partial_t\mathbf{v}_m\|_{\mathbf{L}^2(\Omega)}+\|\mathbb{D}\partial_t\mathbf{v}_m\|_{\mathbf{L}^2(\Omega)})
    \\
	&\leq \frac{\rho_2}{12}\|\partial_t\mathbf{v}_m\|_{\mathbf{L}^2(\Omega)}^2
    +\frac{\alpha}{4}\|\mathbb{D}\partial_t\mathbf{v}_m\|_{\mathbf{L}^2(\Omega)}^2
    +C\|\mathbf{v}_m\|_{\mathbf{H}^1(\Omega)}^4,
    \\[1ex]
	K_6&=-\int_\Omega\partial_t\varphi_m\nabla\mu_m\cdot\mathbf{v}_m\,\mathrm{d}x
    \\
    &\leq \|\partial_t\varphi_m\|_{L^4(\Omega)}\|\nabla\mu_m\|_{\mathbf{L}^2(\Omega)}\|\mathbf{v}_m\|_{\mathbf{L}^4(\Omega)}
    \\[1mm]
	&\leq C\|\partial_t\varphi_m\|_{V}\|\nabla\mu_m\|_{\mathbf{L}^2(\Omega)}\|\mathbf{v}_m\|_{\mathbf{H}^1(\Omega)}
    \\
	&\leq \frac{1}{4}(\|\nabla\partial_t\varphi_m\|_{\mathbf{L}^2(\Omega)}^2
    +\|\nabla_{\!\bm{\tau}\,} \partial_t\psi_m\|_{\mathbf{L}^2(\Gamma)}^2+\chi(K)\|\partial_t\varphi_m-\partial_t\psi_m\|_{H_\Gamma}^2)
    \\
   &\quad +C\|\nabla\mu_m\|_{\mathbf{L}^2(\Omega)}^2 \|\mathbf{v}_m\|_{\mathbf{H}^1(\Omega)}^2,
   \\[1ex]
	K_7&=-\int_\Gamma\partial_t\psi_m \nabla_{\!\bm{\tau}\,}\mathcal{L}_m\cdot\mathbf{v}_{m,\bm{\tau}}\,\mathrm{d}S
    \\
	&\leq \|\partial_t\psi_m\|_{L^4(\Gamma)}\|\nabla_{\!\bm{\tau}\,}\mathcal{L}_m\|_{\mathbf{L}^2(\Gamma)}\|\mathbf{v}_{m}\|_{\mathbf{L}^4(\Gamma)}
    \\
	&\leq\frac{1}{4}(\|\nabla\partial_t\varphi_m\|_{\mathbf{L}^2(\Omega)}^2
    +\|\nabla_{\!\bm{\tau}\,}
    \partial_t\psi_m \|_{\mathbf{L}^2(\Gamma)}^2 +\chi(K)\|\partial_t\psi_m-\partial_t\varphi_m\|_{H_\Gamma}^2)
    \\
    &\quad+C\|\nabla_{\!\bm{\tau}\,} \mathcal{L}_m\|_{\mathbf{L}^2(\Gamma)}^2
    \|\mathbf{v}_m\|_{\mathbf{H}^1(\Omega)}^2,
    \\[1ex]
    K_{8}+K_{9}&  \leq C\Big(\int_\Omega|\partial_t\varphi_m|^2 \,\mathrm{d}x
    +\int_\Gamma|\partial_t\psi_m|^2 \,\mathrm{d}S\Big)
    \\
    &=C{\langle\partial_t \bm{\varphi}_m,\partial_t\bm{\varphi}_m
    \rangle_{(\mathcal{H}_{L,0}^1)', \mathcal{H}_{L,0}^1}}
    \\
	&\leq C\|\partial_t\bm{\varphi}_m\|_{\mathcal{H}_{L,0}^{-1}}\|\partial_t\bm{\varphi}_m\|_{\mathcal{H}_{L,0}^{1}}
    \\[1mm]
	&= C(\|\bm{\mu}_m-\overline{m}(\boldsymbol{\mu}_m)\boldsymbol{1}\|_{\mathcal{H}_{L,0}^{1}}
    +\|(\mathbf{v}_m,\mathbf{v}_{m,\boldsymbol{\tau}}) \|_{\boldsymbol{\mathcal{L}}^2})
    \|\partial_t\bm{\varphi}_m \|_{\mathcal{H}_{L,0}^{1}}
    \\
	&\leq \frac{1}{4}(\|\nabla\partial_t\varphi_m\|_{\mathbf{L}^2(\Omega)}^2
    +\|\nabla_{\!\bm{\tau}\,} \partial_t\psi_m\|_{\mathbf{L}^2(\Gamma)}^2 {+\chi(K)\|\partial_t\psi_m-\partial_t\varphi_m\|_{H_\Gamma}^2})
    \\
    &\quad+C\Big(1+\|\nabla\mu_m \|_{\mathbf{L}^2(\Omega)}^2
    +\|\nabla_{\!\bm{\tau}\,} \mathcal{L}_m\|_{\mathbf{L}^2(\Gamma)}^2
    +\frac{1}{L}\|\mu_m-\mathcal{L}_m\|_{H_\Gamma}^2\Big),
    \\[1ex]
    K_{10}&\leq C\|\nabla\mu_m\|_{\mathbf{L}^3(\Omega)}\|\nabla\mathbf{v}_m\|_{\mathbf{L}^2(\Omega)}\|\partial_t\mathbf{v}_m\|_{\mathbf{L}^6(\Omega)}
    \\
    &\leq C\|\nabla\mu_m\|_{\mathbf{L}^2(\Omega)}^{\frac{1}{2}}\|\nabla\mu_m\|_{\mathbf{L}^6(\Omega)}^{\frac{1}{2}}\|\partial_t\mathbf{v}_m\|_{\mathbf{H}^1(\Omega)}
    \\
    &\leq C\|\nabla\mu_m\|_{\mathbf{L}^2(\Omega)}^{\frac{1}{2}}\|\mathbb{P}\bm{\mu}_m\|_{\mathcal{H}^2}^{\frac{1}{2}}\|\partial_t\mathbf{v}_m\|_{\mathbf{H}^1(\Omega)}
    \\
    &\leq C\|\nabla\mu_m\|_{\mathbf{L}^2(\Omega)}^{\frac{1}{2}}\Big(\|\partial_t\bm{\varphi}_m\|_{\mathcal{L}^2}
    +\|\mathbf{v}_m\cdot\nabla\varphi_m\|_H
    +\|\mathbf{v}_{m,\bm{\tau}} \cdot\nabla_{\!\bm{\tau}\,}\psi_m\|_{H_\Gamma}\Big)^{\frac{1}{2}}\|\partial_t\mathbf{v}_m\|_{\mathbf{H}^1(\Omega)}
    \\
    &\leq \frac{\rho_2}{12}\|\partial_t\mathbf{v}_m\|_{\mathbf{L}^2(\Omega)}^2
    +\frac{\alpha}{4}\|\mathbb{D}\partial_t\mathbf{v}_m\|_{\mathbf{L}^2(\Omega)}^2\\
    &\quad+\frac{1}{8}\Big(\|\nabla\partial_t\varphi_m\|_{\mathbf{L}^2(\Omega)}^2
    +\|\nabla_{\!\bm{\tau}\,} \partial_t\psi_m\|_{\mathbf{L}^2(\Gamma)}^2 {+\chi(K)\|\partial_t\psi_m-\partial_t\varphi_m\|_{H_\Gamma}^2}\Big)
    \\
    &\quad+C\|\nabla\mu_m\|_{\mathbf{L}^2(\Omega)}^2
    +C(\|\mathbf{v}_m\|_{\mathbf{H}^1(\Omega)}^2 +\|\boldsymbol{\varphi}_m\|_{\mathcal{H}^2}^2),
    \\[1ex]
   K_{11}&\leq C\|\nabla\mu_m\cdot\mathbf{n}\|_{H_\Gamma}\|\mathbf{v}_{m,\bm{\tau}}\|_{\mathbf{L}^4(\Gamma)}
   \|\partial_t\mathbf{v}_{m,\bm{\tau}} \|_{\mathbf{L}^4(\Gamma)}
   \\
   &\leq C\|\nabla\mu_m\|_{\mathbf{L}^2(\Gamma)}\|\mathbf{v}_{m}\|_{\mathbf{H}^1(\Omega)}
   \|\partial_t\mathbf{v}_{m} \|_{\mathbf{H}^1(\Omega)}
   \\
   &\leq C\|\nabla\mu_m\|_{\mathbf{L}^2(\Omega)}^{\frac{1}{2}}\|\nabla\mu_m\|_{\mathbf{H}^1(\Omega)}^{\frac{1}{2}}\|\mathbf{v}_{m}\|_{\mathbf{H}^1(\Omega)}
   \|\partial_t\mathbf{v}_{m} \|_{\mathbf{H}^1(\Omega)}
   \\
  &\leq C\|\nabla\mu_m\|_{\mathbf{L}^2(\Omega)}^{\frac{1}{2}}\|\mathbb{P}\bm{\mu}_m\|_{\mathcal{H}^2}^{\frac{1}{2}}\|\partial_t\mathbf{v}_m \|_{\mathbf{H}^1(\Omega)}
  \\
    &\leq C\|\nabla\mu_m\|_{\mathbf{L}^2(\Omega)}^{\frac{1}{2}}\Big(\|\partial_t\bm{\varphi}_m\|_{\mathcal{L}^2}
    +\|\mathbf{v}_m\cdot\nabla\varphi_m\|_H
    +\|\mathbf{v}_{m,\bm{\tau}} \cdot\nabla_{\!\bm{\tau}\,}\psi_m\|_{H_\Gamma}\Big)^{\frac{1}{2}}\|\partial_t\mathbf{v}_m\|_{\mathbf{H}^1(\Omega)}
    \\
    &\leq \frac{\rho_2}{12}\|\partial_t\mathbf{v}_m\|_{\mathbf{L}^2(\Omega)}^2
    +\frac{\alpha}{4}\|\mathbb{D}\partial_t\mathbf{v}_m\|_{\mathbf{L}^2(\Omega)}^2
    \\
    &\quad+\frac{1}{8}\Big(\|\nabla\partial_t\varphi_m\|_{\mathbf{L}^2(\Omega)}^2
    +\|\nabla_{\!\bm{\tau}\,} \partial_t\psi_m\|_{\mathbf{L}^2(\Gamma)}^2 {+\chi(K)\|\partial_t\psi_m-\partial_t\varphi_m\|_{H_\Gamma}^2}\Big)
    \\
    &\quad+C\|\nabla\mu_m\|_{\mathbf{L}^2(\Omega)}^2
    +C(\|\mathbf{v}_m\|_{\mathbf{H}^1(\Omega)}^2 +\|\boldsymbol{\varphi}_m\|_{\mathcal{H}^2}^2).
\end{align*}
We note that in the estimate for $K_{8}+K_{9}$, we have used the fact that the norms $\|\cdot\|_{\mathcal{H}_{L,0}^{-1}}$ and $\|\cdot\|_{(\mathcal{H}_{L,0}^1)'} = \|\cdot\|_{(\mathcal{H}_{L}^1)'}$ on $\mathcal{H}_{L,0}^{-1}$ are equivalent.

From \eqref{0227-20} and the estimates for $K_j$ for all $1\leq j\leq {11}$, we infer that
\begin{align}	
&\frac{\mathrm{d}}{\mathrm{d}t}\mathcal{G}_m
+\frac{1}{2}\int_\Omega|\nabla\partial_t\varphi_m|^2\,\mathrm{d}x
+\frac{1}{2}\int_\Gamma|\nabla_{\!\bm{\tau}\,}\partial_t\psi_m|^2\,\mathrm{d}S
+\frac{\chi(K)}{2}\int_\Gamma|\partial_t\psi_m-\partial_t\varphi_m|^2\,\mathrm{d}S
\notag\\
&\qquad+\frac{\rho_2}{2}\int_\Omega|\partial_t \mathbf{v}_m|^2\,\mathrm{d}x
+\alpha\int_\Omega|\mathbb{D}\partial_t\mathbf{v}_m|^2\,\mathrm{d}x
\notag\\
&\quad\leq C\mathcal{G}_m+C(1+\mathcal{F}_m),
\notag
\end{align}
where $\mathcal{F}_m \coloneqq \|\mathbf{v}_m\|_{\mathbf{H}^1(\Omega)}^4+\|\boldsymbol{\varphi}_m\|_{\mathcal{H}^2}^2$. Together with \eqref{m-uni-7} and \eqref{uni-0323-phi-H2}, Gronwall's lemma implies that
\begin{align}
&\sup_{t\in[0,T]}\mathcal{G}_m(t)
+C\int_0^T\Big(\|\partial_t\bm{\varphi}_m(t)\|_{\mathcal{H}^1}^2
+\|\partial_t\mathbf{v}_m(t)\|_{\mathbf{H}^1(\Omega)}^2\Big)\,\mathrm{d}t
\notag\\
&\qquad\leq C e^{CT}\mathcal{G}_{m}(0)+Ce^{CT}\int_0^T\mathcal{F}_m(t)\,\mathrm{d}t
\leq Ce^{CT}(1+\mathcal{G}_m(0)).
\label{0227-22}
\end{align}
We point out that the lower bound \eqref{0227-21} is crucial in the derivation of \eqref{0227-22}.

It remains to control $\mathcal{G}_m(0)$.
Multiplying \eqref{eqd4} by $\partial_t \varphi_m$ and integrating over $\Omega$,
multiplying \eqref{eqd3} by $\partial_t \psi_m$ and integrating over $\Gamma$,
adding the resulting equations together, we get
\begin{align}
    &\gamma\|\partial_t\varphi_m \|_H^2+\gamma\|\partial_t\psi_m\|_{H_\Gamma}^2
    +\int_\Omega(-\Delta\varphi_m+f(\varphi_m)+k^{-1} f_\sigma(\varphi_m))\partial_t\varphi_m\,\mathrm{d}x
    \notag\\
    &\qquad+\int_\Gamma(-\Delta_{\bm{\tau}}\psi_m+g(\psi_m)
    +k^{-1}f_\sigma(\psi_m) +\partial_{\mathbf{n}}\varphi_m)\partial_t\psi_m\,\mathrm{d}S
    \notag\\
    &\quad=\int_\Omega\mu_m \partial_t\varphi_m\,\mathrm{d}x
    +\int_\Gamma\mathcal{L}_m \partial_t\psi_m\,\mathrm{d}S.
    \label{0306-Gm0-1}
\end{align}
Recalling \eqref{eqd2} and \eqref{eqd3} , we can deduce that
\begin{align}
    &\int_\Omega\mu_m\partial_t\varphi_m\,\mathrm{d}x
    +\int_\Gamma\mathcal{L}_m \partial_t\psi_m\,\mathrm{d}S
    \notag\\
    &\quad=\int_\Omega\mu_m(\Delta\mu_m-\mathrm{div}(\varphi_m\mathbf{v}_m))\,\mathrm{d}x
    +\int_\Gamma\mathcal{L}_m (\Delta_{\bm{\tau}}\mathcal{L}_m-\mathrm{div}_{\Gamma}(\psi_m\mathbf{v}_{m,\bm{\tau}})
    -\partial_{\mathbf{n}}\mu_m) \,\mathrm{d}S\notag
    \notag\\
    &\quad=-\|\nabla\mu_m\|_{\mathbf{L}^2(\Omega)}^2
    -\|\nabla_{\!\bm{\tau}\,} \mathcal{L}_m\|_{\mathbf{L}^2(\Gamma)}^2
    -\frac{1}{L}\|\mu_m-\mathcal{L}_m\|_{H_\Gamma}^2
    \notag\\
    &\qquad+\int_\Omega\varphi_m \mathbf{v}_m\cdot\nabla\mu_m\,\mathrm{d}x
    +\int_\Gamma\psi_m\mathbf{v}_{m,\bm{\tau}} \cdot\nabla_{\!\bm{\tau}\,}\mathcal{L}_m\,\mathrm{d}S
    \label{0306-Gm0-2}
\end{align}
and
\begin{align}
    &\int_\Omega(-\Delta\varphi_m+f(\varphi_m)+k^{-1} f_\sigma(\varphi_m))\partial_t\varphi_m\,\mathrm{d}x\notag\\
    &\qquad+\int_\Gamma(-\Delta_{\bm{\tau}}\psi_m +g(\psi_m)+k^{-1}f_\sigma(\psi_m)
    +\partial_{\mathbf{n}}\varphi_m) \partial_t\psi_m\,\mathrm{d}S
    \notag\\
    &\quad=\int_\Omega(-\Delta\varphi_m+f(\varphi_m)+k^{-1} f_\sigma(\varphi_m))\,
    \mathrm{div}(\nabla\mu_m-\varphi_m\mathbf{v}_m)\,\mathrm{d}x
    \notag\\
    &\qquad+\int_\Gamma(-\Delta_{\bm{\tau}}\psi_m+g(\psi_m)+k^{-1}f_\sigma(\psi_m)+\partial_{\mathbf{n}}\varphi_m)\,
    \mathrm{div}_{\Gamma}(\nabla_{\!\bm{\tau}\,}\mathcal{L}_m-\psi_m\mathbf{v}_{m,\bm{\tau}})\,\mathrm{d}S
    \notag\\
    &\qquad-\frac{1}{L}\int_\Gamma(-\Delta_{\bm{\tau}}\psi_m+g(\psi_m)+k^{-1} f_\sigma(\psi_m)
    +\partial_{\mathbf{n}}\varphi_m)(\mathcal{L}_m-\mu_m)\,\mathrm{d}S
    \notag\\
    &\quad=-\int_\Omega\nabla(-\Delta\varphi_m+f(\varphi_m)+k^{-1} f_\sigma(\varphi_m))\cdot(\nabla\mu_m-\varphi_m\mathbf{v}_m)\,\mathrm{d}x
    \notag\\
    &\qquad-\int_\Gamma\nabla_{\!\bm{\tau}\,}(-\Delta_{\bm{\tau}}\psi_m+g(\psi_m) +k^{-1}f_\sigma(\psi_m)
    +\partial_{\mathbf{n}}\varphi_m) \cdot(\nabla_{\!\bm{\tau}\,}\mathcal{L}_m-\psi_m\mathbf{v}_{m,\bm{\tau}}) \,\mathrm{d}S
    \notag\\
    &\qquad-\frac{1}{L}\int_\Gamma(-\Delta_{\bm{\tau}}\psi_m +g(\psi_m)+k^{-1} f_\sigma(\psi_m)
    +\partial_{\mathbf{n}}\varphi_m)(\mathcal{L}_m-\mu_m)\,\mathrm{d}S
    \notag\\
    &\qquad+\frac{1}{L}\int_\Gamma(-\Delta\varphi_m+f(\varphi_m)+k^{-1} f_\sigma(\varphi_m))(\mathcal{L}_m-\mu_m)\,\mathrm{d}S.
    \label{0306-Gm0-3}
\end{align}
Combining \eqref{0306-Gm0-1}, \eqref{0306-Gm0-2} and \eqref{0306-Gm0-3}, we obtain
\begin{align}
  &\gamma\|\partial_t\varphi_m \|_H^2+\gamma\|\partial_t\psi_m\|_{H_\Gamma}^2
  +\|\nabla\mu_m\|_{\mathbf{L}^2(\Omega)}^2 +\|\nabla_{\!\bm{\tau}\,}\mathcal{L}_m\|_{\mathbf{L}^2(\Gamma)}^2
  +\frac{1}{L}\|\mu_m-\mathcal{L}_m\|_{H_\Gamma}^2
  \notag\\
  &\quad=\int_\Omega\varphi_m \mathbf{v}_m\cdot\nabla\mu_m\,\mathrm{d}x
  +\int_\Gamma\psi_m\mathbf{v}_{m,\bm{\tau}} \cdot\nabla_{\!\bm{\tau}\,}\mathcal{L}_m\,\mathrm{d}S
  \notag\\
  &\qquad+\int_\Omega\nabla(-\Delta\varphi_m+f(\varphi_m)+k^{-1}f_\sigma(\varphi_m))\cdot
  (\nabla\mu_m-\varphi_m\mathbf{v}_m) \,\mathrm{d}x
  \notag\\
  &\qquad+\int_\Gamma\nabla_{\!\bm{\tau}\,}(-\Delta_{\bm{\tau}}\psi_m+
  g(\psi_m)+k^{-1}f_\sigma(\psi_m) +\partial_{\mathbf{n}}\varphi_m)\cdot
  (\nabla_{\!\bm{\tau}\,}\mathcal{L}_m -\psi_m\mathbf{v}_{m,\bm{\tau}})\,\mathrm{d}S
  \notag\\
  &\qquad+\frac{1}{L}\int_\Gamma(-\Delta_{\bm{\tau}}\psi_m+g(\psi_m)+k^{-1}f_\sigma(\psi_m)+
  \partial_{\mathbf{n}}\varphi_m) (\mathcal{L}_m-\mu_m)\,\mathrm{d}S
  \notag\\
  &\qquad-\frac{1}{L}\int_\Gamma(-\Delta\varphi_m+f(\varphi_m) +k^{-1}f_\sigma(\varphi_m))(\mathcal{L}_m-\mu_m)\,\mathrm{d}S.
  \label{0306-Gm0-4}
\end{align}
In view of \eqref{REG:APPROX}, we find
\begin{align*}
\partial_t\bm{\varphi}_m\in C([0,T];\mathcal{H}^1),\quad
\bm{\mu}_m\in C([0,T];\mathcal{H}^1).
\end{align*}
Then by comparison in \eqref{eqd4} and \eqref{eqd5}, it follows that
\begin{align*}
    -\Delta\varphi_m+f(\varphi_m) +k^{-1}f_\sigma(\varphi_m)
    &\in C([0,T];V),
    \\
    -\Delta_{\boldsymbol{\tau}}\psi_m +\partial_\mathbf{n}\varphi_m+g(\psi_m)+k^{-1}f_\sigma(\psi_m)
    &\in C([0,T];V_\Gamma).
\end{align*}
Due to these continuity properties, we can pass to the limit as $ t\to0$ in \eqref{0306-Gm0-4}
to conclude that
\begin{align}
  &\gamma\|\partial_t\varphi_m(0) \|_H^2+\gamma\|\partial_t\psi_m(0)\|_{H_\Gamma}^2
  +\|\nabla\mu_m(0)\|_{\mathbf{L}^2(\Omega)}^2 +\|\nabla_{\!\bm{\tau}\,}\mathcal{L}_m(0)\|_{\mathbf{L}^2(\Gamma)}^2
  +\frac{1}{L}\|\mu_m(0)-\mathcal{L}_m(0)\|_{H_\Gamma}^2
  \notag\\
  &\quad=\int_\Omega\varphi_{0,\sigma,k} \mathbf{v}_m(0)\cdot\nabla\mu_m(0)\,\mathrm{d}x
  +\int_\Gamma\psi_{0,\sigma,k} \mathbf{v}_{m,\bm{\tau}}(0)\cdot\nabla_{\!\bm{\tau}\,}\mathcal{L}_m(0) \,\mathrm{d}S
  \notag\\
  &\qquad+\int_\Omega\nabla(-\Delta\varphi_{0,\sigma,k}+f(\varphi_{0,\sigma,k})+k^{-1}f_\sigma(\varphi_{0,\sigma,k}))\cdot
  (\nabla\mu_m(0)-\varphi_{0,\sigma,k}\mathbf{v}_m(0)) \,\mathrm{d}x
  \notag\\
    &\qquad+\int_\Gamma\nabla_{\!\bm{\tau}\,}(-\Delta_{\bm{\tau}}\psi_{0,\sigma,k}
+g(\psi_{0,\sigma,k}) +k^{-1}f_\sigma(\psi_{0,\sigma,k})
    +\partial_{\mathbf{n}}\varphi_{0,\sigma,k}) \cdot(\nabla_{\!\bm{\tau}\,}\mathcal{L}_m(0)-\psi_{0,\sigma,k}\mathbf{v}_{m,\bm{\tau}}(0))\,\mathrm{d}S
    \notag\\
    &\qquad+\frac{1}{L}\int_\Gamma(-\Delta_{\bm{\tau}}\psi_{0,\sigma,k}+g(\psi_{0,\sigma,k})+k^{-1}f_\sigma(\psi_{0,\sigma,k})
    +\partial_{\mathbf{n}}\varphi_{0,\sigma,k})(\mathcal{L}_m(0)-\mu_m(0))\,\mathrm{d}S
    \notag\\
    &\qquad-\frac{1}{L}\int_\Gamma(-\Delta\varphi_{0,\sigma,k}+f(\varphi_{0,\sigma,k})+k^{-1}f_\sigma(\varphi_{0,\sigma,k}))
    (\mathcal{L}_m(0)-\mu_m(0))\,\mathrm{d}S
    \notag\\
    &\quad\leq\frac{1}{2}\Big(\|\nabla\mu_m(0)\|_{\mathbf{L}^2(\Omega)}^2+\|\nabla_{\!\bm{\tau}\,}\mathcal{L}_m(0)\|_{\mathbf{L}^2(\Gamma)}^2
    +\frac{1}{L}\|\mu_m(0)-\mathcal{L}_m(0)\|_{H_\Gamma}^2\Big)
    \notag\\[1mm]
    &\qquad+C\|-\Delta\varphi_{0,\sigma,k}+f(\varphi_{0,\sigma,k})+k^{-1}f_\sigma(\varphi_{0,\sigma,k})\|_{V}^2
    \notag\\
    &\qquad+C\|-\Delta_{\bm{\tau}}\psi_{0,\sigma,k}+g(\psi_{0,\sigma,k})
    +\partial_{\mathbf{n}} \varphi_{0,\sigma,k}+k^{-1}f_\sigma(\psi_{0,\sigma,k})\|_{V_\Gamma}^2
    \notag\\
    &\qquad+C\|\mathbf{v}_m(0) \|_{\mathbf{L}^2(\Omega)}^2
    +C\|\mathbf{v}_{m,\bm{\tau}}(0)\|_{\mathbf{L}^2(\Gamma)}^2.
    \label{0306-Gm0-4'}
\end{align}
Thus, applying \eqref{0315-1}, \eqref{0315-2} to the right-hand side of \eqref{0306-Gm0-4'}, we infer that
\begin{align}
\mathcal{G}_m(0)
&\leq C\Big(1+\|\widetilde{\mu}_0\|_{V}^2
+\|\widetilde{\mathcal{L}}_0\|_{V_\Gamma} ^2+\|\boldsymbol{\varphi}_{0,\sigma,k}\|_{\mathcal{H}^1}^2
+\|\mathbf{v}_m(0)\|_{\mathbf{H}^1(\Omega)}^2\Big) \leq C,
\label{0306-Gm0-5}
\end{align}
for some positive constant $C$ independent of $m$, $\gamma$, $\sigma$, and $k$.

Combining \eqref{0227-21}, \eqref{0227-22}, \eqref{0306-Gm0-5}, and recalling the definition of $\mathcal{G}_m$,
we see that \eqref{0227-23} holds.
\end{proof}

\begin{lemma}
    \label{final}
There exists a positive constant $C$, independent of $m$, $\gamma$, $\sigma$, and $k$, such that
\begin{align}
    \|f_0(\varphi_m)\|_{L^\infty(0,T;H)} +\|g_0(\psi_m)\|_{L^\infty(0,T;H_\Gamma)}
    &\leq C,\label{0306-8}
    \\
    k^{-1}\|f_\sigma(\varphi_m) \|_{L^\infty(0,T;H)}+k^{-1}\|f_\sigma(\psi_m)\|_{L^\infty(0,T;H_\Gamma)}
    &\leq C,\label{0306-8'}
    \\
    \|\bm{\varphi}_m\|_{L^\infty(0,T;\mathcal{H}^2)}
    &\leq C,\label{0306-9}
    \\
    \|\bm{\mu}_m\|_{L^2(0,T;\mathcal{H}^2)}
    &\leq C.\label{0306-9'}
\end{align}
Furthermore, the following strict separation property holds
\begin{align}
    &\varphi_m\in L^\infty(Q_T)\;\;\text{such that}\;\;|\varphi_m(x,t)|\leq (1-\sigma)(1-\widetilde{\delta})
    \ \ \text{a.e.~in }Q_T,
    \label{0306-separation-1}
    \\[1mm]
    &\psi_m\in L^\infty(\Sigma_T)\;\;\text{such that}\;\;|\psi_m(x,t)|\leq (1-\sigma)(1-\widetilde{\delta})
    \ \ \text{a.e.~on }\Sigma_T,
    \label{0306-separation-2}
\end{align}
for some $\widetilde{\delta}=\widetilde{\delta}(\sigma,k,\gamma)\in(0,1)$.
\end{lemma}
\begin{proof}
Based on Lemma~\ref{higher}, the uniform bounds \eqref{0306-8}--\eqref{0306-9}
can be shown in a similar way as in the proof of Lemmas~\ref{f-L2L2} and \ref{H2}.
Furthermore, applying the regularity theory for bulk-surface elliptic problems (see, e.g., \cite[Theorem~3.3]{KL}) to the subsystem \eqref{0321-eq-1-1}, \eqref{0321-eq-1-4}, \eqref{0321-eq-1-1}, and using Lemmas~\ref{energy-estimates}, \ref{mu-H1-phi-H2} and \ref{higher},
we can derive \eqref{0306-9'}.
Finally, combining Lemma~\ref{higher} with Theorem~\ref{A1}, we find that the strict separation property in \eqref{strong-regu} depends on $\sigma$, $\gamma$ and $k$, but is independent of $m$.
This yields \eqref{0306-separation-1} and \eqref{0306-separation-2}. The proof is complete.
\end{proof}


\subsection{Passage to the limit and existence of quasi-strong solutions}
\label{passage-to-limit}

In this subsection, we prove the existence of quasi-strong solutions in the case $(K,L)\in(0,+\infty]\times(0,+\infty)$.

\begin{proof}[Proof of Theorem~\ref{quasi-strong}: $ (K,L) \in (0,+\infty\rbrack \times (0,+\infty) $.]
We will prove the assertions by choosing $\gamma\coloneqq k^{-1}$ and
passing to the limits as $m\to+\infty$, $\sigma\to0$ and $k\to+\infty$ in sequence.

\noindent\textbf{\itshape Passage to the limit as $ m\to+\infty$.}
Based on the uniform estimates established in Lemmas~\ref{energy-estimates}, \ref{higher} and \ref{final},
we apply the Banach--Alaoglu theorem and the Aubin--Lions--Simon lemma to conclude, up to subsequence extraction, the following convergences as $m\to+\infty$:
\begin{subequations}
\label{CONV:M}
\begin{alignat}{2}
    \mathbf{v}_m
    &\to\mathbf{v}_\sigma
    &&\qquad\text{weakly star in }L^\infty(0,T;\mathbf{H}^1_{\mathrm{div}}(\Omega)),
    \\
    \partial_t\mathbf{v}_m
    &\to\partial_t\mathbf{v}_\sigma
    &&\qquad\text{weakly in }L^2(0,T;\mathbf{H}^1_{\mathrm{div}}(\Omega)),
    \\
    (\mathbf{v}_m,\mathbf{v}_{m,\boldsymbol{\tau}})
    &\to (\mathbf{v}_\sigma,\mathbf{v}_{\sigma,\boldsymbol{\tau}})
    &&\quad \ \ \ \text{ strongly in }C([0,T];\boldsymbol{\mathcal{L}}^2_{\mathrm{div}}),
    \\
    \bm{\varphi}_m
    &\to\bm{\varphi}_\sigma
    &&\qquad\text{weakly star in }L^\infty(0,T;\mathcal{H}^2)
    \text{ and strongly in }C([0,T];\mathcal{W}^{1,p}),
    \\
    \partial_t \bm{\varphi}_m
    &\to\partial_t\bm{\varphi}_\sigma
    &&\qquad\text{weakly in }L^2(0,T;\mathcal{H}^1),
    \\
    \bm{\mu}_m
    &\to\bm{\mu}_\sigma
    &&\qquad\text{weakly star in }L^\infty(0,T;\mathcal{H}^1)
    \ \text{and weakly in }L^2(0,T;\mathcal{H}^2),
\end{alignat}
\end{subequations}
for all $p\in[2,6)$. Like before, here we use the notation $\bm{\varphi}_\sigma = (\varphi_\sigma,\psi_\sigma)$ and $\bm{\mu}_\sigma = (\mu_\sigma,\mathcal{L}_\sigma)$.
As a consequence of the above convergences, we infer that
\begin{align*}
    \rho(\varphi_m)\to \rho(\varphi_\sigma)&\qquad\text{strongly in }C([0,T];{W}^{1,p}(\Omega)),\\
     \nu(\varphi_m)\to \nu(\varphi_\sigma)&\qquad\text{strongly in }C([0,T];{W}^{1,p}(\Omega)),\\
      \beta(\psi_m)\to \beta(\psi_\sigma)&\qquad\text{strongly in }C([0,T];{W}^{1,p}(\Gamma)),
\end{align*}
for all $p\in[2,6)$. Moreover, due to  Lemma~\ref{final}, we have
\begin{subequations}
\begin{align}
    \label{SEP:PHI:SIG}
    &\varphi_\sigma\in L^\infty(Q_T)\ \ \text{such that }|\varphi_\sigma(x,t)|\leq (1-\sigma)(1-\widetilde{\delta})
    \ \ \text{a.e.~in }Q_T,
    \\[1mm]
    \label{SEP:PSI:SIG}
    &\psi_\sigma\in L^\infty(\Sigma_T)\ \ \text{such that }|\psi_\sigma(x,t)|\leq (1-\sigma)(1-\widetilde{\delta})
    \ \ \text{a.e.~on }\Sigma_T,
\end{align}
\end{subequations}
with $\widetilde{\delta}=\widetilde{\delta}(\sigma,k,\gamma)\in(0,1)$.

The above properties guarantee the convergence of the nonlinear terms in \eqref{eqd1},
and of the singular nonlinearities $k^{-1} f_\sigma(\varphi_m)$ in \eqref{eqd4} and $k^{-1} f_\sigma(\psi_m)$ in \eqref{eqd5}.
This allows us to pass to the limit as $ m\to+\infty$ in \eqref{eqd1}--\eqref{approximating-solution}, which entails that the limit functions $(\mathbf{v}_\sigma,\mathbf{v}_{\sigma,\boldsymbol{\tau}},\bm{\varphi}_\sigma,\bm{\mu}_\sigma)$ satisfy the weak formulation
\begin{align}
	& \int_\Omega \rho(\varphi_\sigma)\partial_{t}\mathbf{v}_{\sigma}\cdot\widetilde{\mathbf{v}}_j\,\mathrm{d}x
    +\int_\Omega 2\alpha \mathbb{D}\partial_t\mathbf{v}_\sigma:\mathbb{D}\widetilde{\mathbf{v}}_j \,\mathrm{d}x
    +\int_\Omega 2\nu(\varphi_\sigma)\mathbb{D}\mathbf{v}_\sigma:\mathbb{D}\widetilde{\mathbf{v}}_j\,\mathrm{d}x
	+\int_\Gamma \beta(\psi_\sigma )\mathbf{v}_{\sigma,\bm{\tau}}\cdot\widetilde{\mathbf{w}}_j\,\mathrm{d}S
	\notag \\[1mm]
	& \quad =\int_\Omega  \mu_\sigma \nabla \varphi_\sigma\cdot\widetilde{\mathbf{v}}_j\,\mathrm{d}x
	-\int_\Gamma  \psi_\sigma\nabla_{\!\bm{\tau}\,}\mathcal{L}_\sigma\cdot\widetilde{\mathbf{w}}_j \,\mathrm{d}S
    -\int_\Omega \rho(\varphi_\sigma)(\mathbf{v}_\sigma\cdot\nabla)\mathbf{v}_\sigma\cdot\widetilde{\mathbf{v}}_j\,\mathrm{d}x
    \notag\\[1mm]
    &\qquad+\rho'\int_\Omega (\nabla\mu_\sigma\cdot\nabla)\mathbf{v}_\sigma\cdot\widetilde{\mathbf{v}}_j \,\mathrm{d}x
    -\frac{\rho'}{2}\int_\Gamma (\nabla\mu_\sigma\cdot\mathbf{n})\mathbf{v}_{\sigma,\bm{\tau}}\cdot\widetilde{\mathbf{w}}_j \,\mathrm{d}S
    \quad\text{for all }j\in\mathbb{N},
	\label{gamma-eqd1}
\end{align}
almost everywhere in $[0,T]$, and the equations
\begin{align}
    \left\{\,
    \begin{aligned}
	&\partial_{t}\varphi_\sigma +\mathbf{v}_\sigma\cdot\nabla\varphi_\sigma=\Delta\mu_\sigma & & \mbox{a.e.~in } Q_{T},
	\\
	& \mu_\sigma =\gamma\partial_t \varphi_\sigma-\Delta \varphi_\sigma+f(\varphi _\sigma) +k^{-1} f_\sigma(\varphi_\sigma) & & \mbox{a.e.~in } Q_{T},
	\\
    &
    \begin{cases}
        K\partial_\mathbf{n}\varphi_\sigma =\psi_\sigma-\varphi_\sigma,&K\in(0,+\infty)
        \\
        \partial_\mathbf{n}\varphi_\sigma=0, &K=+\infty
    \end{cases}
    &&\mbox{a.e.~on }
	\Sigma _{T},
    \label{sigma-eq}
    \\
	&
	L\partial_{\mathbf{n}}\mu_\sigma =\mathcal{L}_\sigma-\mu_\sigma,\quad L\in(0,+\infty)&&\mbox{a.e.~on }
	\Sigma _{T},
	\\
	&\partial_t\psi_\sigma +\mathbf{v}_{\sigma,\bm{\tau}} \cdot\nabla_{\!\bm{\tau}\,}\psi_\sigma
    =\Delta_{\bm{\tau}}\mathcal{L}_\sigma-\partial_{\mathbf{n}}\mu_\sigma & & \mbox{a.e.~on }
	\Sigma _{T},
	\\
	& \mathcal{L} _\sigma = \gamma\partial_t \psi_\sigma- \Delta _{\bm{\tau}}\psi_\sigma +\partial _{%
		\mathbf{n}}\varphi_\sigma+ g(\psi_\sigma)+k^{-1} f_\sigma(\psi_\sigma) & & \mbox{a.e.~on }
	\Sigma _{T}.
    \end{aligned}
    \right.
\end{align}
Moreover, we have
\begin{align}
\mathbf{v}_\sigma|_{t=0}=\mathbf{v}_0\quad \text{a.e. in }\Omega,\quad \ (\varphi_\sigma,\psi_\sigma)|_{t=0}=(\varphi_{0,\sigma,k},\psi_{0,\sigma,k})\;\;\text{in }\Omega\times\Gamma.\notag
\end{align}
Since $\mathrm{span}\{(\widetilde{\mathbf{v}}_j,\widetilde{\mathbf{w}}_j):j\in\mathbb{N}\}$ is dense in $\boldsymbol{\mathcal{H}}^2_{0,\mathrm{div}}$, we infer from \eqref{gamma-eqd1} that  $(\mathbf{v}_\sigma,\mathbf{v}_{\sigma,\boldsymbol{\tau}},\boldsymbol{\varphi}_\sigma,\boldsymbol{\mu}_\sigma)$ satisfy
\begin{align}
	& \int_\Omega \rho(\varphi_\sigma)\partial_{t}\mathbf{v}_{\sigma}\cdot\mathbf{w} \,\mathrm{d}x
    +\int_\Omega 2\alpha \mathbb{D}\partial_t\mathbf{v}_\sigma:\mathbb{D}\mathbf{w} \,\mathrm{d}x
    +\int_\Omega 2\nu(\varphi_\sigma)\mathbb{D}\mathbf{v}_\sigma:\mathbb{D}\mathbf{w} \,\mathrm{d}x
	+\int_\Gamma \beta(\psi_\sigma )\mathbf{v}_{\sigma,\bm{\tau}}\cdot\mathbf{w}_{\bm{\tau}} \,\mathrm{d}S
	\notag \\[1mm]
	& \quad =\int_\Omega  \mu_\sigma\nabla \varphi_\sigma\cdot\mathbf{w} \,\mathrm{d}x
	-\int_\Gamma \psi_\sigma\nabla_{\!\bm{\tau}\,}\mathcal{L}_\sigma\cdot\mathbf{w}_{\bm{\tau}} \,\mathrm{d}S
    -\int_\Omega \rho(\varphi_\sigma)(\mathbf{v}_\sigma\cdot\nabla)\mathbf{v}_\sigma\cdot\mathbf{w} \,\mathrm{d}x
    \notag\\[1mm]
    &\qquad+\rho'\int_\Omega (\nabla\mu_\sigma\cdot\nabla)\mathbf{v}_\sigma\cdot\mathbf{w} \,\mathrm{d}x
    -\frac{\rho'}{2}\int_\Gamma (\nabla\mu_\sigma\cdot\mathbf{n})\mathbf{v}_{\sigma,\bm{\tau}}\cdot\mathbf{w}_{\bm{\tau}} \,\mathrm{d}S,
	\label{'gamma-eqd1}
\end{align}
almost everywhere in $[0,T]$ for all $(\mathbf{w},\mathbf{w}_{\boldsymbol{\tau}})\in \boldsymbol{\mathcal{H}}^2_{0,\mathrm{div}}$.

\noindent\textbf{\itshape Passage to the limit as $ \sigma\to0$.}
As norms on Banach spaces are weakly (star) lower semicontinuous,
we infer from the convergences \eqref{CONV:M} and the uniform bounds established in
Lemmas~\ref{energy-estimates}, \ref{higher} and \ref{final} that
\begin{subequations}
\label{sigma-uni}
\begin{align}
    \|\mathbf{v}_\sigma\|_{L^\infty(0,T;{\mathbf{H}}^1_{\mathrm{div}}(\Omega))}
    +\|\bm{\mu}_\sigma \|_{L^\infty(0,T;\mathcal{H}^1)}
    +\|\bm{\varphi}_\sigma \|_{L^\infty(0,T;\mathcal{H}^2)}
    +\sqrt{\gamma}\|\partial_t\bm{\varphi}_\sigma \|_{L^\infty(0,T;\mathcal{L}^2)}
    &\leq C,
    \label{sigma-uni-1} \\
    \|f_0(\varphi_\sigma)\|_{L^\infty(0,T;H)}
    +\|g_0(\psi_\sigma)\|_{L^\infty(0,T;H_\Gamma)}
    &\leq C,
    \label{sigma-uni-2} \\
    k^{-1}\|f_\sigma(\varphi_\sigma) \|_{L^\infty(0,T;H)}
    +k^{-1}\|f_\sigma (\psi_\sigma)\|_{L^\infty(0,T;H_\Gamma)}
    &\leq C,
    \label{sigma-uni-2'} \\
    \|\partial_t\bm{\varphi}_\sigma \|_{L^2(0,T;\mathcal{H}^1)}
    +\|\partial_t\mathbf{v}_\sigma \|_{L^2(0,T;\mathbf{H}_{\mathrm{div}}^1(\Omega))}
    +\|\bm{\mu}_\sigma\|_{L^2(0,T;\mathcal{H}^2)}
    &\leq C,
    \label{sigma-uni-3}
\end{align}
\end{subequations}
for some positive constant $C$ independent of $\gamma$, $\sigma$ and $k$.
Therefore, arguing as above, we use the Banach--Alaoglu theorem and the Aubin--Lions--Simon lemma to deduce that the convergences
\begin{subequations}
\label{CONV:SIGMA}
\begin{alignat}{2}
    \mathbf{v}_\sigma
    &\to\mathbf{v}_k
    &&\qquad\text{weakly star in }L^\infty(0,T;\mathbf{H}^1_{\mathrm{div}}(\Omega)),
    \\
    \partial_t\mathbf{v}_\sigma
    &\to\partial_t\mathbf{v}_k
    &&\qquad\text{weakly in }L^2(0,T;\mathbf{H}^1_{\mathrm{div}}(\Omega)),
    \\
    (\mathbf{v}_\sigma,\mathbf{v}_{\sigma,\boldsymbol{\tau}})
    &\to (\mathbf{v}_k,\mathbf{v}_{k,\boldsymbol{\tau}})
    &&\quad \ \ \ \text{ strongly in }C([0,T];\boldsymbol{\mathcal{L}}^2_{\mathrm{div}}),
    \\
    \bm{\varphi}_\sigma
    &\to\bm{\varphi}_k
    &&\qquad\text{weakly star in }L^\infty(0,T;\mathcal{H}^2)
    \text{ and strongly in }C([0,T];\mathcal{W}^{1,p}),
    \\
    \partial_t \bm{\varphi}_\sigma
    &\to\partial_t\bm{\varphi}_k
    &&\qquad\text{weakly in }L^2(0,T;\mathcal{H}^1),
    \\
    \bm{\mu}_\sigma
    &\to\bm{\mu}_k
    &&\qquad\text{weakly star in }L^\infty(0,T;\mathcal{H}^1)
    \text{ and weakly in }L^2(0,T;\mathcal{H}^2),
    \\
    \rho(\varphi_\sigma)
    &\to \rho(\varphi_k)
    &&\qquad\text{strongly in }C([0,T];{W}^{1,p}(\Omega)),
    \\
    \nu(\varphi_\sigma)
    &\to \nu(\varphi_k)
    &&\qquad\text{strongly in }C([0,T];{W}^{1,p}(\Omega)),
    \\
    \beta(\psi_\sigma)
    &\to \beta(\psi_k)
    &&\qquad\text{strongly in }C([0,T];{W}^{1,p}(\Gamma)),
\end{alignat}
\end{subequations}
as $\sigma\to0$, hold for all $p\in[2,6)$ up to subsequence extraction.
Here, we use the notation $\bm{\varphi}_k = (\varphi_k,\psi_k)$ and $\bm{\mu}_k = (\mu_k,\mathcal{L}_k)$.
Thus, passing to the limit as $ \sigma\to0$ in \eqref{'gamma-eqd1}, we deduce that the functions $(\mathbf{v}_k,\mathbf{v}_{k,\boldsymbol{\tau}},\boldsymbol{\varphi}_k,\boldsymbol{\mu}_k)$ satisfy
\begin{align}
	& \int_\Omega \rho(\varphi_k)\partial_{t}\mathbf{v}_{k}\cdot\mathbf{w} \,\mathrm{d}x
    +\int_\Omega 2\alpha \mathbb{D}\partial_t\mathbf{v}_k:\mathbb{D}\mathbf{w} \,\mathrm{d}x
    +\int_\Omega 2\nu(\varphi_k)\mathbb{D}\mathbf{v}_k:\mathbb{D}\mathbf{w} \,\mathrm{d}x
	+\int_\Gamma \beta(\psi_k )\mathbf{v}_{k,\bm{\tau}}\cdot\mathbf{w}_{\bm{\tau}} \,\mathrm{d}S
	\notag \\[1mm]
	& \quad =\int_\Omega  \mu_k\nabla \varphi_k\cdot\mathbf{w} \,\mathrm{d}x
	-\int_\Gamma \psi_k\nabla_{\!\bm{\tau}\,}\mathcal{L}_k\cdot\mathbf{w}_{\bm{\tau}} \,\mathrm{d}S
    -\int_\Omega \rho(\varphi_k)(\mathbf{v}_k\cdot\nabla)\mathbf{v}_k\cdot\mathbf{w} \,\mathrm{d}x
    \notag\\[1mm]
    &\qquad+\rho'\int_\Omega(\nabla\mu_k\cdot\nabla)\mathbf{v}_k\cdot\mathbf{w} \,\mathrm{d}x
     -\frac{\rho'}{2}\int_\Gamma (\nabla\mu_k\cdot\mathbf{n})\mathbf{v}_{k,\bm{\tau}}\cdot\mathbf{w}_{\bm{\tau}} \,\mathrm{d}S
	\label{sigma-eqd1}
\end{align}
almost everywhere in $[0,T]$ for all $(\mathbf{w},\mathbf{w}_{\boldsymbol{\tau}})\in \boldsymbol{\mathcal{H}}^2_{0,\mathrm{div}}$.

Since the quantity $\widetilde{\delta}=\widetilde{\delta}(\sigma,k,\gamma)$ in the separation properties \eqref{SEP:PHI:SIG} and \eqref{SEP:PSI:SIG} is not uniform with respect to the parameter $\sigma$, we cannot directly pass to the limit as $\sigma\to0$ in nonlinear terms $f_0(\varphi_\sigma)$, $g_0(\psi_\sigma)$, $f_\sigma(\varphi_\sigma)$ and $f_\sigma(\psi_\sigma)$. Instead, recalling the maximal monotonicity of $f_0$ and $g_0$, we can apply \cite[Proposition 2.1]{Barbu} to conclude that
\begin{subequations}
\label{0813-convergence-1}
\begin{alignat}{2}
    f_0(\varphi_\sigma)
    &\to f_0(\varphi_k)
    &&\quad\text{ weakly star in }L^\infty(0,T;H),
    \\
    g_0(\psi_\sigma)
    &\to g_0(\psi_k)
    &&\quad\text{ weakly star in }L^\infty(0,T;H_\Gamma),
    \\
    f_\sigma(\varphi_\sigma) =f_0\Big(\frac{\varphi_\sigma}{1-\sigma}\Big)&\to f_0(\varphi_k)
    &&\quad\text{ weakly star in }L^\infty(0,T;H),\quad
    \\
    f_\sigma(\psi_\sigma)=f_0\Big(\frac{\psi_\sigma}{1-\sigma}\Big)&\to f_0(\psi_k)
    &&\quad\text{ weakly star in }L^\infty(0,T;H_\Gamma).
\end{alignat}
\end{subequations}
Combining \eqref{sigma-eq}, \eqref{CONV:SIGMA} and \eqref{0813-convergence-1}, we obtain
\begin{align}
    &\int_\Omega(\partial_t\varphi_k +\mathbf{v}_k\cdot\nabla\varphi_k)z\,\mathrm{d}x
    +\int_\Gamma(\partial_t\psi_k +\mathbf{v}_{k,\boldsymbol{\tau}}\cdot\nabla_{\!\boldsymbol{\tau}\,}\psi_k)z_\Gamma\,\mathrm{d}S
    \notag\\
    &\quad=-\int_\Omega\nabla\mu_k\cdot\nabla z\,\mathrm{d}x
    -\int_\Gamma \nabla_{\!\boldsymbol{\tau}\,}\mathcal{L}_k\cdot\nabla_{\!\boldsymbol{\tau}\,} z_\Gamma\,\mathrm{d}S
    -\frac{1}{L}\int_\Gamma(\mu_k-\mathcal{L}_k)(z-z_\Gamma)\,\mathrm{d}S,
    \label{weak-0813-1}
\end{align}
and
\begin{align}
&\int_\Omega(\mu_k-\gamma\partial_t\varphi_k+\Delta\varphi_k-f_1(\varphi_k))z\,\mathrm{d}x
+\int_\Gamma(\mathcal{L}_k-\gamma\partial_t\psi_k+\Delta_{\boldsymbol{\tau}}\psi_k-\partial_{\mathbf{n}}\varphi_k-g_1(\psi_k))z_\Gamma\,\mathrm{d}S
\notag\\
&\quad=\int_\Omega (f_0(\varphi_k)+k^{-1}f_0(\varphi_k))z\,\mathrm{d}x
+\int_\Gamma (g_0(\psi_k)+k^{-1}f_0(\psi_k))z_\Gamma\,\mathrm{d}S,\label{weak-0813-2}
\end{align}
for all $(z,z_\Gamma)\in \mathcal{H}^1$.
Since the test functions $(z,z_\Gamma)\in \mathcal{H}^1$ in \eqref{weak-0813-1} and \eqref{weak-0813-2} are arbitrary, this implies that the functions $(\mathbf{v}_k,\mathbf{v}_{k,\boldsymbol{\tau}},\boldsymbol{\varphi}_k,\boldsymbol{\mu}_k)$ also satisfy
\begin{align}
    \left\{
    \begin{aligned}
	&\partial_{t}\varphi_k+\mathbf{v}_k \cdot\nabla\varphi_k=\Delta\mu_k
    & & \mbox{a.e.~in } Q_{T},
	\\
	& \mu_k =\gamma\partial_t\varphi_k-\Delta \varphi _k+f(\varphi _k)+k^{-1} f_0(\varphi_k)
    & & \mbox{a.e.~in } Q_{T},
	\\
     &
    \begin{cases}
        K\partial_\mathbf{n}\varphi_k=\psi_k-\varphi_k,&K\in(0,+\infty)\\
        \partial_\mathbf{n}\varphi_k=0,&K=+\infty
    \end{cases}
    &&\mbox{a.e.~on } 	\Sigma _{T},
    \label{lambda-eq}\\
	&
	L\partial_{\mathbf{n}}\mu_k=\mathcal{L}_k-\mu_k,\quad L\in(0,+\infty)
    &&\mbox{a.e.~on }	\Sigma _{T},
	\\
	&\partial_t\psi_k+{\mathbf{v}_{k,\bm{\tau}}\cdot\nabla_{\!\bm{\tau}\,}\psi_k}
    =\Delta_{\bm{\tau}}\mathcal{L}_k-\partial_{\mathbf{n}}\mu_k
    & & \mbox{a.e.~on } \Sigma _{T},
	\\
	& \mathcal{L} _k = \gamma\partial_t\psi_k- \Delta _{\bm{\tau}}\psi_k +\partial _{%
		\mathbf{n}}\varphi_k
        + g(\psi_k)+k^{-1}f_0(\psi_k)
        & & \mbox{a.e.~on } \Sigma _{T}.
    \end{aligned}
    \right.
\end{align}
Moreover, using \eqref{initial-convergence}, we further get
\begin{align}
\mathbf{v}_k|_{t=0}=\mathbf{v}_0\quad\text{a.e. in }\Omega,\quad
(\varphi_k,\psi_k)|_{t=0}=(\varphi_{0,k},\psi_{0,k})\quad\text{in }\Omega\times\Gamma.
\end{align}


\noindent\textbf{\itshape Passage to the limit as $k\to +\infty$.}
As mentioned above, in order to pass to the limits $\gamma\to 0$ and $k\to +\infty$ simultaneously, we  choose $\gamma \coloneqq k^{-1}$ and send $k\to +\infty$. As norms on Banach spaces are weakly (star) lower semicontinuous,
we infer from the convergences \eqref{CONV:M} and the uniform bounds \eqref{sigma-uni} that
\begin{subequations}
\begin{align}
    \|\mathbf{v}_k\|_{L^\infty(0,T;{\mathbf{H}}^1_{\mathrm{div}}(\Omega))}
    +\|\bm{\mu}_k\|_{L^\infty(0,T;\mathcal{H}^1)} +\|\bm{\varphi}_k\|_{L^\infty(0,T;\mathcal{H}^2)}
    +k^{-\frac{1}{2}}\|\partial_t\bm{\varphi}_k\|_{L^\infty(0,T;\mathcal{L}^2)}
    &\leq C,\label{lambda-uni-1}
    \\
    \|f_0(\varphi_k)\|_{L^\infty(0,T;H)} +\|g_0(\psi_k)\|_{L^\infty(0,T;H_\Gamma)}
    &\leq C,\label{lambda-uni-2}
    \\
    \|\partial_t\bm{\varphi}_k \|_{L^2(0,T;\mathcal{H}^1)}
    +\|\partial_t\mathbf{v}_k \|_{L^2(0,T;\mathbf{H}_{\mathrm{div}}^1(\Omega))}
    +\|\bm{\mu}_k\|_{L^2(0,T;\mathcal{H}^2)}&\leq C,
    \label{lambda-uni-3}
\end{align}
\end{subequations}
for some positive constant $C$ independent of $k$.
Therefore, using the Banach--Alaoglu theorem and the Aubin--Lions--Simon lemma, we infer from \eqref{lambda-uni-1} and \eqref{lambda-uni-3} that
\begin{subequations}
\label{CONV:K}
\begin{alignat}{2}
    \mathbf{v}_k
    &\to\mathbf{v}
    &&\qquad\text{weakly star in }L^\infty(0,T;\mathbf{H}^1_{\mathrm{div}}(\Omega)),
    \\
    \partial_t\mathbf{v}_k
    &\to\partial_t\mathbf{v}
    &&\qquad\text{weakly in }L^2(0,T;\mathbf{H}^1_{\mathrm{div}}(\Omega)),
    \\
    (\mathbf{v}_k,\mathbf{v}_{k,\boldsymbol{\tau}})
    &\to (\mathbf{v},\mathbf{v}_{\boldsymbol{\tau}})
    &&\quad \ \ \ \text{ strongly in }C([0,T];\boldsymbol{\mathcal{L}}^2_{\mathrm{div}}),
    \\
    \bm{\varphi}_k
    &\to\bm{\varphi}
    &&\qquad\text{weakly star in }L^\infty(0,T;\mathcal{H}^2)
    \text{ and strongly in }C([0,T];\mathcal{W}^{1,p}),
    \\
    \partial_t \bm{\varphi}_k
    &\to\partial_t\bm{\varphi}
    &&\qquad\text{weakly in }L^2(0,T;\mathcal{H}^1),
    \\
    \bm{\mu}_k
    &\to\bm{\mu}
    &&\qquad\text{weakly star in }L^\infty(0,T;\mathcal{H}^1)
    \text{ and weakly in }L^2(0,T;\mathcal{H}^2),
    \\
    \rho(\varphi_k)
    &\to \rho(\varphi)
    &&\qquad\text{strongly in }C([0,T];{W}^{1,p}(\Omega)),
    \\
    \nu(\varphi_k)
    &\to \nu(\varphi)
    &&\qquad\text{strongly in }C([0,T];{W}^{1,p}(\Omega)),
    \\
    \beta(\psi_k)
    &\to \beta(\psi)
    &&\qquad\text{strongly in }C([0,T];{W}^{1,p}(\Gamma)),
\end{alignat}
\end{subequations}
for all $p\in[2,6)$ up to subsequence extraction.
Here, we use the notation $\bm{\varphi} = (\varphi,\psi)$ and $\bm{\mu} = (\mu,\mathcal{L})$.
Thus, passing to the limit as $ k\to+\infty$ in \eqref{sigma-eqd1}, we deduce that the functions $(\mathbf{v},\mathbf{v}_{\boldsymbol{\tau}},\boldsymbol{\varphi},\boldsymbol{\mu})$ satisfy \eqref{quasi-1'}.
To pass to the limit as $ k\to+\infty$ in nonlinear terms $f_0(\varphi_k)$ and $g_0(\psi_k)$, arguing as \eqref{0813-convergence-1}, we obtain
\begin{subequations}
\label{0323-convergence-2}
\begin{alignat}{2}
    f_0(\varphi_k)
    &\to f_0(\varphi)
    &&\quad\text{ weakly star in }L^\infty(0,T;H),\quad
    \\
    g_0(\psi_k)
    &\to g_0(\psi)
    &&\quad\text{ weakly star in }L^\infty(0,T;H_\Gamma).
\end{alignat}
\end{subequations}
Furthermore, we infer from \eqref{lambda-uni-2} that
\begin{subequations}
\label{CONV:F0:K}
\begin{alignat}{2}
    k^{-1} f_0(\varphi_k)
    &\to 0
    &&\quad\text{ strongly in }L^\infty(0,T;H),\quad
    \\
    k^{-1} f_0(\psi_k)
    &\to 0
    &&\quad\text{ strongly in }L^\infty(0,T;H_\Gamma).
    \label{0325-convergence-6}
\end{alignat}
\end{subequations}
Combining \eqref{lambda-eq}--\eqref{CONV:F0:K}, and arguing as \eqref{weak-0813-1} and \eqref{weak-0813-2}, we can conclude that  $(\mathbf{v}, \mathbf{v}_{\boldsymbol{\tau}}, \boldsymbol{\varphi}, \boldsymbol{\mu})$
satisfy \eqref{quasi-2'}.
Moreover, using \eqref{initial-convergence}, we find that
\begin{align}
\mathbf{v}|_{t=0}=\mathbf{v}_0\quad\text{a.e. in }\Omega,\quad
(\varphi,\psi)|_{t=0}=(\varphi_0,\psi_0)\quad\text{in }\Omega\times\Gamma.
\end{align}

\noindent\textbf{\itshape Energy equality.}
It remains to verify the energy identity \eqref{energy-equality}. Since the weak formulation \eqref{quasi-1'} holds for all test functions $(\mathbf{w},\mathbf{w}_{\boldsymbol{\tau}})\in\boldsymbol{\mathcal{H}}_{0,\mathrm{div}}^2$, we first
take $(\mathbf{w},\mathbf{w}_{\boldsymbol{\tau}})=(\mathbf{v}_m,\mathbf{v}_{m,\boldsymbol{\tau}})$ in \eqref{quasi-1'}, where $\{(\mathbf{v}_m,\mathbf{v}_{m,\boldsymbol{\tau}})\}_{m\in\mathbb{N}}$ is the sequence of approximate solutions obtained in Section \ref{SECT:EXAP}. This gives
\begin{align}
        	& \int_\Omega \rho(\varphi)\partial_{t}\mathbf{v}\cdot\mathbf{v}_m\,\mathrm{d}x
            +\int_\Omega2\alpha\, \mathbb{D}\partial_t\mathbf{v}:\mathbb{D}\mathbf{v}_m\,\mathrm{d}x
            +\int_\Omega 2\nu(\varphi )\mathbb{D}\mathbf{v}:\mathbb{D}\mathbf{v}_m\,\mathrm{d}x
        	+\int_\Gamma \beta(\psi)\mathbf{v}_{\bm{\tau}}\cdot\mathbf{v}_{m,\bm{\tau}}\,\mathrm{d}S
            \notag\\[1mm]
        	& \quad =\int_\Omega \mu \nabla \varphi\cdot\mathbf{v}_m\,\mathrm{d}x
        	-\int_\Gamma \psi\nabla_{\!\bm{\tau}\,}\mathcal{L}\cdot \mathbf{v}_{m,\bm{\tau}}\,\mathrm{d}S
            -\int_\Omega \rho(\varphi)(\mathbf{v}\cdot\nabla)\mathbf{v}\cdot\mathbf{v}_m\,\mathrm{d}x
           \notag  \\[1mm]
            &\qquad+\rho'\int_\Omega (\nabla\mu\cdot\nabla)\mathbf{v}\cdot\mathbf{v}_m\,\mathrm{d}x
            -\frac{\rho'}{2}\int_\Gamma  (\nabla \mu \cdot\mathbf{n})\mathbf{v}_{\bm{\tau}}\cdot\mathbf{v}_{m,\bm{\tau}}\,\mathrm{d}S \quad \text{for a.a.}\ t\in(0,T).
        	\label{260808-1}
\end{align}
According to the convergence results (up to a subsequence) of $\mathbf{v}_m$ (resp. $\mathbf{v}_\sigma$ and $\mathbf{v}_k$) in \eqref{CONV:M} (resp. \eqref{CONV:SIGMA} and \eqref{CONV:K}), we can pass to the limits as $m\to+\infty$, $\sigma\to0$ and $k\to+\infty$ subsequently. Proceeding similarly to the derivation of \eqref{uni-ga-1}, we obtain
\begin{align}
&\frac{1}{2}\frac{\mathrm{d}}{\mathrm{d}t}\Big(\int_\Omega \rho(\varphi)|\mathbf{v}|^2\,\mathrm{d}x
+2\alpha\int_\Omega|\mathbb{D}\mathbf{v}|^2\,\mathrm{d}x\Big)
+\int_\Omega 2\nu(\varphi)|\mathbb{D}\mathbf{v}|^2\,\mathrm{d}x
+\int_\Gamma\beta(\psi)|\mathbf{v}_{\bm{\tau}}|^2\,\mathrm{d}S
\notag\\
&\quad=-\int_\Omega\varphi\nabla\mu\cdot\mathbf{v}\,\mathrm{d}x
-\int_\Gamma \psi\nabla_{\!\bm{\tau}\,}\mathcal{L}\cdot\mathbf{v}_{\bm{\tau}}\,\mathrm{d}S.
\label{0320-energy-1}
\end{align}
Next, multiplying \eqref{quasi-1} by $\mu$ and integrating over $\Omega$,
by integration by parts, and using \eqref{quasi-2} and the chain rule for subdifferentials
(see, e.g., \cite[Chapter IV, Lemma 4.3]{R.E.S}), we obtain
\begin{align}
    &\frac{\mathrm{d}}{\mathrm{d}t}\Big(\frac{1}{2}\int_\Omega|\nabla\varphi|^2\,\mathrm{d}x
    +\int_\Omega F(\varphi)\mathrm{d}x\Big)
    +\int_\Omega|\nabla\mu|^2\,\mathrm{d}x
    -\int_\Gamma\mu\partial_{\mathbf{n}}\mu\,\mathrm{d}S\notag\\
    &\quad=\int_\Omega \varphi\mathbf{v}\cdot\nabla\mu\,\mathrm{d}x
    +\int_\Gamma \partial_t\varphi\partial_{\mathbf{n}}\varphi\,\mathrm{d}S.\label{0320-energy-2}
\end{align}
Similarly, by \eqref{quasi-4} and \eqref{quasi-5}, we have
\begin{align}
    &\frac{\mathrm{d}}{\mathrm{d}t}\Big(\frac{1}{2}\int_\Gamma|\nabla_{\!\boldsymbol{\tau}\,}\psi|^2\,\mathrm{d}S
    +\int_\Gamma G(\psi)\mathrm{d}S\Big)+\int_\Gamma|\nabla_{\!\boldsymbol{\tau}\,}\mathcal{L}|^2\,\mathrm{d}S
    +\int_\Gamma\mathcal{L}\partial_{\mathbf{n}}\mu\,\mathrm{d}S\notag\\
    &\quad=\int_\Gamma \psi\mathbf{v}_{\boldsymbol{\tau}}\cdot\nabla_{\!\boldsymbol{\tau}\,}\mathcal{L}\,\mathrm{d}S
    -\int_\Gamma \partial_t\psi\partial_{\mathbf{n}}\varphi\,\mathrm{d}S.\label{0320-energy-3}
\end{align}
Combining \eqref{0320-energy-1}, \eqref{0320-energy-2} and \eqref{0320-energy-3},
and using the boundary condition \eqref{quasi-3},
we conclude that the energy identity \eqref{energy-identity} holds. Integrating \eqref{energy-identity} with respect to time from zero to any $t\ge 0$,
we arrive at the energy equality \eqref{energy-equality}. This completes the proof of Theorem~\ref{quasi-strong} in the case $ (K,L) \in (0,+\infty\rbrack \times (0,+\infty) $.
\end{proof}

\section{Existence of Quasi-strong Solutions to the Limit Models}
\label{asy}
\setcounter{equation}{0}

To establish Theorem~\ref{quasi-strong} in the limiting cases
$(K,L) \in (0,\infty]\times\{0,+\infty\}$,
we investigate the asymptotic limits as $L\to0$ and $L\to+\infty$, respectively, when $K\in(0,+\infty]$.

\subsection{The case \texorpdfstring{$(K,L)\in(0,+\infty]\times\{0\}$}
{(K,L) ∈ (0,+∞] × \{0\} } }
\label{asymptotic-1}
In this subsection, we restrict ourselves to the matched-density case, i.e., $\rho_1=\rho_2>0$. By examining the proof of Theorem~\ref{quasi-strong}
in the case $ (K,L) \in (0,+\infty\rbrack \times (0,+\infty) $ more closely,
we observe that the estimates \eqref{260807-1}, \eqref{0306-Gm0-4'} may depend on $L\in(0,1]$.
The following two lemmas provide refined versions of the uniform estimates derived in Section~\ref{existence of quasi}, where the estimates are now independent of $L\in(0,1]$.
\begin{lemma} \label{refine1}
There exists a positive constant $C$, independent of $m$, $\gamma$, $\sigma$, $k$ and $L\in(0,1]$, such that
\begin{align*}
        \|\bm{\mu}_m\|_{L^2(0,T;\mathcal{H}^1)} +\|\bm{\varphi}_m\|_{L^2(0,T;\mathcal{H}^2)}\leq C.
\end{align*}
\end{lemma}
\begin{proof}
Proceeding as in the derivation of \eqref{mu-uni-1}, we obtain
\begin{align}
    &c_1\int_\Omega |f_0(\varphi_m)|\,\mathrm{d}x+c_1\int_\Gamma|g_0(\psi_m)|\,\mathrm{d}S
    +\frac{c_1 }{2k}\int_\Omega |f_\sigma(\varphi_m)|\,\mathrm{d}x+\frac{c_1 }{2k}\int_\Gamma |f_\sigma(\psi_m)|\,\mathrm{d}S\notag\\
    &\quad\leq \int_\Omega(\mu_m-\overline{m}(\boldsymbol{\mu}_m))(\varphi_m-\overline{m}_{0,\sigma,k})\,\mathrm{d}x
    +\int_\Gamma(\mathcal{L}_m-\overline{m}(\boldsymbol{\mu}_m))(\psi_m-\overline{m}_{0,\sigma,k})\,\mathrm{d}S\notag\\
    &\qquad-\int_\Omega(\gamma\partial_t\varphi_m+f_1(\varphi_m))(\varphi_m-\overline{m}_{0,\sigma,k})\,\mathrm{d}x
    -\int_\Gamma(\gamma\partial_t\psi_m+g_1(\psi_m))(\psi_m-\overline{m}_{0,\sigma,k})\,\mathrm{d}S\notag\\
    &\qquad+2k^{-1}(1-\sigma)c_2(|\Omega|+|\Gamma|)+2c_2(|\Omega|+|\Gamma|)\notag\\[1mm]
    &\quad\leq {\langle\boldsymbol{\varphi}_m-\overline{m}_{0,\sigma,k}\boldsymbol{1},
    \boldsymbol{\mu}_m-\overline{m}(\boldsymbol{\mu}_m)\boldsymbol{1}\rangle_{(\mathcal{H}_{L,0}^1)',\mathcal{H}_{L,0}^1}}
    +C(1+\gamma\|\partial_t\boldsymbol{\varphi}_m\|_{\mathcal{L}^2})\notag\\
    &\quad\leq \|\boldsymbol{\varphi}_m-\overline{m}_{0,\sigma,k}\boldsymbol{1}\|_{\mathcal{H}_{L,0}^{-1}}
    \|\boldsymbol{\mu}_m-\overline{m}(\boldsymbol{\mu}_m)\boldsymbol{1}\|_{\mathcal{H}_{L,0}^{1}}
    +C(1+\gamma\|\partial_t\boldsymbol{\varphi}_m\|_{\mathcal{L}^2}).\label{refine1-1}
\end{align}
Here, we have used the fact that the norms $\|\cdot\|_{\mathcal{H}_{L,0}^{-1}}$ and $\|\cdot\|_{(\mathcal{H}_{L,0}^1)'} = \|\cdot\|_{(\mathcal{H}_{L}^1)'}$ on $\mathcal{H}_{L,0}^{-1}$ are equivalent.
Recalling \cite[Lemma 5.1]{LvWuIFB}, it holds
\begin{align}
        \|\boldsymbol{\varphi}_m-\overline{m}_{0,\sigma,k}\boldsymbol{1}\|_{\mathcal{H}_{L,0}^{-1}}
        \leq \widehat{C}\|\boldsymbol{\varphi}_m-\overline{m}_{0,\sigma,k}\boldsymbol{1}\|_{\mathcal{L}^2}
        \quad\text{for all $L\in (0,1]$},
        \label{refine1-2}
\end{align}
where the positive constant $\widehat{C}$ is independent of $m$, $\gamma$, $\sigma$, $k$ and $L\in(0,1]$.
Using \eqref{refine1-1} and \eqref{refine1-2}, we obtain
\begin{align}
         &c_1\int_\Omega |f_0(\varphi_m)|\,\mathrm{d}x+c_1\int_\Gamma|g_0(\psi_m)|\,\mathrm{d}S
        +\frac{c_1 }{2k}\int_\Omega |f_\sigma(\varphi_m)|\,\mathrm{d}x+\frac{c_1 }{2k}\int_\Gamma |f_\sigma(\psi_m)|\,\mathrm{d}S
        \notag\\
         &\quad\leq \widehat{C}\|\boldsymbol{\varphi}_m-\overline{m}_{0,\sigma,k}\boldsymbol{1}\|_{\mathcal{L}^2}
         \|\boldsymbol{\mu}_m-\overline{m}(\boldsymbol{\mu}_m)\boldsymbol{1}\|_{\mathcal{H}_{L,0}^{1}}
         +C(1+\gamma\|\partial_t\boldsymbol{\varphi}_m \|_{\mathcal{L}^2}),\notag
\end{align}
which, together with Lemma \ref{energy-estimates}, implies that the constant in \eqref{mu-uni-1} is independent of $L$ as long as $L\in(0,1]$.
Based on this result, we can repeat the process of \cite[Lemma 5.9]{LvWuIFB} to complete the proof of Lemma~\ref{refine1}.
\end{proof}

\begin{lemma}
\label{refine3}
There exists a positive constant $C$, independent of $m$, $\gamma$, $\sigma$, $k$ and $L\in(0,1]$,
such that
\begin{alignat}{2}
&&\sqrt{\gamma}\|\partial_t\bm{\varphi}_m\|_{L^\infty(0,T;\mathcal{L}^2)}
+\|\bm{\mu}_m\|_{L^\infty(0,T;\mathcal{H}^1)}
\qquad
&{}
\notag\\
&&\qquad+\|\partial_t\bm{\varphi}_m\|_{L^2(0,T;\mathcal{H}^1)}
+\|\partial_t\mathbf{v}_m\|_{L^2(0,T;\mathbf{H}^1(\Omega))}
&\leq C\Big(1+\frac{1}{\sqrt{L}} \mathcal{U}_{\sigma,k}\Big),
\label{uniform-0318-3}
\\
&&\|f_0(\varphi_m)\|_{L^\infty(0,T;H)}+\|g_0(\psi_m)\|_{L^\infty(0,T;H_\Gamma)}
+\|\boldsymbol{\varphi}_m\|_{L^\infty(0,T;\mathcal{H}^2)}
&\leq C\Big(1+\frac{1}{\sqrt{L}}\mathcal{U}_{\sigma,k}\Big),
\label{uniform-0318-4}
\\
&&k^{-1}\|f_\sigma(\varphi_m)\|_{L^\infty(0,T;H)}
+k^{-1}\|f_\sigma(\psi_m)\|_{L^\infty(0,T;H_\Gamma)}
&\leq C\Big(1+\frac{1}{\sqrt{L}}\mathcal{U}_{\sigma,k}\Big),
\label{uniform-0318-4'}
\end{alignat}
where $\mathcal{U}_{\sigma,k}:=\|\boldsymbol{\varphi}_{0,\sigma,k}-\boldsymbol{\varphi}_0\|_{\mathcal{H}^1}+\|\widetilde{\mathcal{L}}_{0,k}+g_1(\psi_0)
    -\widetilde{\mu}_{0,k}-f_1(\varphi_0)\|_{H_\Gamma}$.
\end{lemma}
\begin{proof}
Based on Lemma~\ref{refine1}, we can repeat the proof of Lemmas~\ref{higher} and~\ref{final}
to verify \eqref{uniform-0318-3}, \eqref{uniform-0318-4} and \eqref{uniform-0318-4'}.
The only notable difference is that we need to refine the estimate \eqref{0306-Gm0-4'},
that is,
\begin{align*}
        &\frac{1}{L}\int_\Gamma\big(-\Delta_{\bm{\tau}}\psi_{0,\sigma,k}+g(\psi_{0,\sigma,k})
        +k^{-1}f_\sigma(\psi_{0,\sigma,k}) +\partial_{\mathbf{n}}\varphi_{0,\sigma,k}\big)(\mathcal{L}_m(0)-\mu_m(0))\,\mathrm{d}S
        \\
    &\qquad-\frac{1}{L}\int_\Gamma\big(-\Delta\varphi_{0,\sigma,k}+f(\varphi_{0,\sigma,k})
    +k^{-1}f_\sigma(\varphi_{0,\sigma,k})\big)(\mathcal{L}_m(0)-\mu_m(0))\,\mathrm{d}S
    \\
    &\quad=\frac{1}{L}\int_\Gamma\big(\widetilde{\mathcal{L}}_{0,k} +g_1(\psi_0)
    -\widetilde{\mu}_{0,k}-f_1(\varphi_0)\big) (\mathcal{L}_m(0)-\mu_m(0))\,\mathrm{d}S
    \\
    &\qquad+\frac{1}{L}\int_\Gamma \big(-f_1(\varphi_{0,\sigma,k})+f_1(\varphi_0)
    +g_1(\psi_{0,\sigma,k})-g_1(\psi_0)\big) (\mathcal{L}_m(0)-\mu_m(0))\,\mathrm{d}S
    \\
    &\quad \leq \frac{1}{4L}\int_\Gamma |\mathcal{L}_m(0)-\mu_m(0)|^2\,\mathrm{d}S
    + \frac{C}{L}\big(\|\boldsymbol{\varphi}_{0,\sigma,k}-\boldsymbol{\varphi}_0\|_{\mathcal{H}^1}^2+\|\widetilde{\mathcal{L}}_{0,k}+g_1(\psi_0)
    -\widetilde{\mu}_{0,k}-f_1(\varphi_0)\|_{H_\Gamma}^2\big).
\end{align*}
This completes the proof of Lemma~\ref{refine3}.
\end{proof}

\begin{remark}\rm
\label{no-bound}
    For the chemical potentials $\boldsymbol{\mu}_m$, the available estimates fail to provide a bound on $\|\boldsymbol{\mu}_m\|_{L^2(0,T;\mathcal{H}^2)}$ that is uniform with respect to $L\in(0,1]$. Indeed, to establish the $\mathcal{H}^2$-regularity of $\boldsymbol{\mu}_m$, we shall consider the following bulk-surface elliptic problem:
    \begin{align}
        \left\{
        \begin{aligned}
            -\Delta\mu_m
            &=-\partial_t\varphi_m-\mathbf{v}_m\cdot\nabla\varphi_m
            &\quad\text{in }\Omega,\\
            L\partial_{\mathbf{n}}\mu_m
            &=\mathcal{L}_m-\mu_m
            &\quad\text{on }\Gamma,\\
            -\Delta_{\boldsymbol{\tau}} \mathcal{L}_m+\partial_{\mathbf{n}}\mu_m
            &=-\partial_t\psi_m-\mathbf{v}_{m,\boldsymbol{\tau}}\cdot\nabla_{\!\boldsymbol{\tau}\,}\psi_m
            &\quad\text{on }\Gamma.
        \end{aligned}\notag
        \right.
    \end{align}
    By \cite[Theorem 3.3]{KL}, we have
    \begin{align*}
        \|\mathcal{L}_m\|_{H^2(\Gamma)}&\leq C\Big(\|\partial_t\psi_m\|_{H_\Gamma}+\|\mathbf{v}_{m,\boldsymbol{\tau}}\cdot\nabla_{\!\boldsymbol{\tau}\,}\psi_m\|_{H_\Gamma}+\frac{1}{L}\|\mathcal{L}_m-\mu_m\|_{H_\Gamma}\Big),\\
        \|\mu_m\|_{H^2(\Omega)}&\leq C\Big(\|\partial_t\varphi_m\|_H+\|\mathbf{v}_m\cdot\nabla\varphi_m\|_H+\|\mu_m\|_V+\frac{1}{L}\|\mathcal{L}_m-\mu_m\|_{H^{\frac{1}{2}}(\Gamma)}\Big)\\
        &\leq C\Big(\|\partial_t\varphi_m\|_H+\|\mathbf{v}_m\cdot\nabla\varphi_m\|_H+\frac{1}{L}(\|\mathcal{L}_m\|_{V_\Gamma}+\|\mu_m\|_{V})\Big).
    \end{align*}
    Here, the constant $C>0$ is independent of $m$ and $L$. The factor $1/L$ in these estimates prevents us from deriving a uniform bound for $\|\boldsymbol{\mu}_m\|_{L^2(0,T;\mathcal{H}^2)}$ with respect to $L\in(0,1]$.
    Nevertheless, in the matched-density case, i.e., $\rho_1=\rho_2>0$, the estimate of $\boldsymbol{\mu}_m$ provided in \eqref{uniform-0318-3} is sufficient to rigorously justify the limiting process as $L\to 0$.
\end{remark}

\begin{proof}[Proof of Theorem~\ref{quasi-strong}: $(K,L)\in(0,+\infty\rbrack\times\{0\}$.]
Based on Lemmas~\ref{refine1} and \ref{refine3},
we can repeat the line of argument used in the proof of Theorem~\ref{quasi-strong} in the case $(K,L)\in(0,+\infty\rbrack\times(0,+\infty)$.
In this way, we conclude that there exist functions  $(\mathbf{v}^L,\mathbf{v}_{\boldsymbol{\tau}}^L,\boldsymbol{\varphi}^L,\boldsymbol{\mu}^L)$ that satisfy the following weak formulation
\begin{align}
	& \int_\Omega  \rho(\varphi^L)\partial_{t}\mathbf{v}^L\cdot\mathbf{w} \,\mathrm{d}x
    +\int_\Omega 2\alpha \mathbb{D}\partial_t\mathbf{v}^L:\mathbb{D}\mathbf{w} \,\mathrm{d}x
    +\int_\Omega 2\nu(\varphi ^L)\mathbb{D}\mathbf{v}^L:\mathbb{D}\mathbf{w} \,\mathrm{d}x
	+\int_\Gamma \beta(\psi^L)\mathbf{v}^L_{\bm{\tau}}\cdot\mathbf{w}_{\bm{\tau}} \,\mathrm{d}S
	\notag \\[1mm]
	& \quad =\int_\Omega  \mu^L \nabla \varphi^L\cdot\mathbf{w} \,\mathrm{d}x
	-\int_\Gamma  \psi^L\nabla_{\!\bm{\tau}\,}\mathcal{L}^L\cdot\mathbf{w}_{\bm{\tau}} \,\mathrm{d}S
    -\int_\Omega \rho(\varphi^L)(\mathbf{v}^L\cdot\nabla)\mathbf{v}^L\cdot \mathbf{w} \,\mathrm{d}x
	\label{formula-0318-2}
\end{align}
almost everywhere in $[0,T]$ for all  $(\mathbf{w},\mathbf{w}_{\boldsymbol{\tau}})\in\boldsymbol{\mathcal{H}}_{0,\mathrm{div}}^2$, and the equations
\begin{align}
\left\{\,
\begin{aligned}
    &\partial_{t}\varphi^L+\mathbf{v}^L\cdot\nabla\varphi^L
    =\Delta\mu^L
    &&\quad \mbox{a.e.~in } Q_{T},
    \\[1mm]
    &\mu^L =-\Delta \varphi^L +f(\varphi^L )
    &&\quad \mbox{a.e.~in } Q_{T},
    \\[1mm]
    &\begin{cases}
    K\partial_\mathbf{n}\varphi^L=\psi^L-\varphi^L,&\text{if }K\in(0,+\infty)\\
    \partial_\mathbf{n}\varphi^L=0,&\text{if }K=+\infty
    \end{cases}
    &&\quad {\mbox{a.e.~on }\Sigma _{T},}
    \\[1mm]
    &L\partial_{\mathbf{n}}\mu^L=\mathcal{L}^L-\mu^L,\quad L\in(0,+\infty)
    &&\quad\mbox{a.e.~on }
    \Sigma _{T},
    \\[1mm]
    &\partial_t\psi^L+{\mathbf{v}^L_{\bm{\tau}}\cdot\nabla_{\!\bm{\tau}\,}\psi^L}
    =\Delta_{\bm{\tau}}\mathcal{L}^L-\partial_{\mathbf{n}}\mu^L
    &&\quad  \mbox{a.e.~on }
    \Sigma _{T},
    \\[1mm]
    &\mathcal{L} ^L = - \Delta _{\bm{\tau}}\psi ^L+\partial _{%
    \mathbf{n}}\varphi^L+ g(\psi^L)
    &&\quad \mbox{a.e.~on }
    \Sigma _{T}.
\end{aligned}
\right.\label{formula-0318-1}
\end{align}
Furthermore, using \eqref{initial-convergence}, Lebesgue's dominated convergence theorem and the compatibility condition \ref{ASS:C}, we find
\begin{align}
    \lim_{k\to\infty}\lim_{\sigma\to0}\mathcal{U}_{\sigma,k} &=\lim_{k\to\infty}\lim_{\sigma\to0}\|\boldsymbol{\varphi}_{0,\sigma,k}
    -\boldsymbol{\varphi}_0\|_{\mathcal{H}^1} +\lim_{k\to\infty}\lim_{\sigma\to0}\|\widetilde{\mathcal{L}}_{0,k}+g_1(\psi_0)
    -\widetilde{\mu}_{0,k}-f_1(\varphi_0)\|_{H_\Gamma}
    \notag\\
    &=\lim_{k\to\infty}\|\widetilde{\mathcal{L}}_{0,k} +g_1(\psi_0)
    -\widetilde{\mu}_{0,k}-f_1(\varphi_0)\|_{H_\Gamma}
    \notag\\
    &=\|\widetilde{\mathcal{L}}_{0}+g_1(\psi_0)
    -\widetilde{\mu}_{0}-f_1(\varphi_0)\|_{H_\Gamma}
    \notag\\
    &=0.\label{U}
\end{align}
Therefore, applying Lemma \ref{energy-estimates}, \eqref{uniform-0318-3}, \eqref{uniform-0318-4}, \eqref{U} and
the weak (star) lower semicontinuity of norms,
we conclude that there exists a positive constant $C$, independent of $L\in(0,1]$, such that
\begin{align}
    &\|\mathbf{v}^L\|_{L^\infty(0,T; \mathbf{H}_{\mathrm{div}}^1(\Omega))}
    +\|\boldsymbol{\varphi}^L\|_{L^\infty(0,T;\mathcal{H}^2)}
    +\|\boldsymbol{\mu}^L\|_{L^\infty(0,T;\mathcal{H}^1)}
    +\frac{1}{\sqrt{L}}\|\mu^L-\mathcal{L}^L\|_{L^2(0,T;H_\Gamma)}\notag\\
    &\quad+\|f_0(\varphi^L)\|_{L^\infty(0,T;H)} +\|g_0(\psi^L)\|_{L^\infty(0,T;H_\Gamma)}
   +\|\partial_t\mathbf{v}^L \|_{L^2(0,T;\mathbf{H}^1(\Omega))}
    +\|\partial_t\boldsymbol{\varphi}^L \|_{L^2(0,T;\mathcal{H}^1)}\leq C.
    \label{uniform-0318-5}
\end{align}
Based on \eqref{uniform-0318-5}, we employ the Banach--Alaoglu theorem and the Aubin--Lions--Simon lemma to infer that there exist limit functions $(\mathbf{v}^0,\mathbf{v}^0_{\boldsymbol{\tau}},\boldsymbol{\varphi}^0,\boldsymbol{\mu}^0)$,
such that
\begin{align*}
    \mathbf{v}^L\to\mathbf{v}^0&\qquad\text{weakly star in }L^\infty(0,T;\mathbf{H}_{\mathrm{div}}^1(\Omega)),
    \\
    \partial_t\mathbf{v}^L\to\partial_t \mathbf{v}^0&\qquad\text{weakly in }L^2(0,T;\mathbf{H}_{\mathrm{div}}^1(\Omega)),
    \\
   ( \mathbf{v}^L,\mathbf{v}^L_{\boldsymbol{\tau}}) \to (\mathbf{v}^0,\mathbf{v}^0_{\boldsymbol{\tau}})
   &\qquad\text{strongly in } C([0,T];\boldsymbol{\mathcal{L}}_{\mathrm{div}}^2),
    \\
    \boldsymbol{\varphi}^L\to\boldsymbol{\varphi}^0 &\qquad\text{weakly star in }L^\infty(0,T;\mathcal{H}^2)
    \text{ and strongly in } C([0,T];\mathcal{W}^{1,p}),
    \\
    \partial_t\boldsymbol{\varphi}^L\to \partial_t\boldsymbol{\varphi}^0 &\qquad\text{weakly in }L^2(0,T;\mathcal{H}^1),
    \\
    \boldsymbol{\mu}^L\to\boldsymbol{\mu}^0 &\qquad\text{weakly star in }L^\infty(0,T;\mathcal{H}^1),
    \\
    \mu^L-\mathcal{L}^L\to0&\qquad\text{strongly in }L^2(0,T;H_\Gamma),
\end{align*}
for all $p\in[2,6)$ up to subsequence extraction.
Here, we use the notation $\bm{\varphi}^0 = (\varphi^0,\psi^0)$ and $\bm{\mu}^0 = (\mu^0,\mathcal{L}^0)$.
As weak limits are unique, it follows that
\begin{equation}
    \label{TRREL}
    \mu^0 = \mathcal{L}^0 \quad\text{a.e.~on $\Sigma_T$.}
\end{equation}
In the case of matched densities, based on the above convergence results,
we can pass to the limit as $ L\to0$ in \eqref{formula-0318-2} to conclude that the functions $(\mathbf{v}^0,\mathbf{v}^0_{\boldsymbol{\tau}},\boldsymbol{\varphi}^0,\boldsymbol{\mu}^0)$ satisfy the weak formulation \eqref{quasi-1'}.
Proceeding similarly to the limiting process as $k\to+\infty$ in Section~\ref{existence of quasi}, we show that
\begin{align*}
    f_0(\varphi^L)\to f_0(\varphi^0)\text{ weakly star in }L^\infty(0,T;H),\quad
    g_0(\psi^L)\to g_0(\psi^0)\text{ weakly star in }L^\infty(0,T;H_\Gamma).
\end{align*}
This entails that
\begin{align*}
   &\varphi^0\in L^\infty(Q_T)\;\;\text{such that }|\varphi^0(x,t)|<1
   \;\;\text{a.e.~in }Q_T,\\[1mm]
   &\psi^0\in L^\infty(\Sigma_T)\;\;\text{such that }|\psi^0(x,t)|<1
   \;\;\text{a.e.~on }\Sigma_T.
\end{align*}
Using the above convergence results to pass to the limit in \eqref{formula-0318-1}, we deduce that
\begin{align}
\begin{aligned}
    &\mu^0 =-\Delta \varphi^0 +f(\varphi^0 )
    &&\quad \mbox{a.e.~in } Q_{T},
    \\[1mm]
    &\begin{cases}
    K\partial_\mathbf{n}\varphi^0=\psi^0-\varphi^0,&\text{if }K\in(0,+\infty)\\
    \partial_\mathbf{n}\varphi^0=0,&\text{if }K=+\infty
    \end{cases}
    &&\quad {\mbox{a.e.~on }\Sigma _{T},}
    \\[1mm]
    &\mathcal{L} ^0 = - \Delta _{\bm{\tau}}\psi ^0+\partial _{%
    \mathbf{n}}\varphi^0+ g(\psi^0)
    &&\quad \mbox{a.e.~on }
    \Sigma _{T},
\end{aligned}
\notag
\end{align}
and the weak formulation
\begin{align*}
    &\int_\Omega \nabla\mu^0\cdot\nabla z\,\mathrm{d}x+\int_\Gamma \nabla_{\boldsymbol{\tau}}\mathcal{L}^0\cdot\nabla_{\boldsymbol{\tau}}z_\Gamma\,\mathrm{d}S\\
    &\quad=-\int_\Omega(\partial_t\varphi^0+\mathbf{v}^0\cdot\nabla\varphi^0)z\,\mathrm{d}x-\int_\Gamma (\partial_t\psi^0+\mathbf{v}_{\boldsymbol{\tau}}^0\cdot\nabla_{\boldsymbol{\tau}}\psi^0)z_\Gamma\,\mathrm{d}S
\end{align*}
holds almost everywhere in $(0,T)$ for all $(z,z_\Gamma)\in\mathcal{V}^1$.
Together with \eqref{TRREL}, the latter equation implies that for almost all $t\in(0,T)$, the pair $(\mu^0(t),\mathcal{L}^0(t))$ is a weak solution to the system
\begin{align}
    \begin{cases}
        -\Delta\mu^0=-\partial_t\varphi^0-\mathbf{v}^0\cdot\nabla\varphi^0&\text{in }\Omega,\\
        \mu^0=\mathcal{L}^0&\text{on }\Gamma,
        \\
        -\Delta_{\boldsymbol{\tau}} \mathcal{L}^0+\partial_{\mathbf{n}}\mu^0=-\partial_t\psi^0-\mathbf{v}_{\boldsymbol{\tau}}^0 \cdot\nabla_{\boldsymbol{\tau}}\psi^0&\text{on }\Gamma.
    \end{cases}
    \label{p-0821-2}
\end{align}
Applying the regularity theory for bulk-surface elliptic problems (see, e.g., \cite[Theorem~3.3]{KL}), we conclude that $(\mu^0(t),\mathcal{L}^0(t))\in \mathcal{V}^2$ for almost all $t\in [0,T]$ and the equations \eqref{p-0821-2} are satisfied almost everywhere in $Q_T$. In addition, we obtain the estimate
\begin{align*}
    \|\boldsymbol{\mu}^0\|_{L^2(0,T;\mathcal{V}^2)}
    &\leq C\big(\|\boldsymbol{\mu}^0\|_{L^2(0,T;\mathcal{L}^2)}+\|\partial_t\varphi^0+\mathbf{v}^0 \cdot\nabla\varphi^0\|_{L^2(0,T;H)}
    \\
    &\qquad\quad +\|\partial_t\psi^0+\mathbf{v}_{\boldsymbol{\tau}}^0\cdot\nabla_{\boldsymbol{\tau}}\psi^0\|_{L^2(0,T;H_\Gamma)}\big)
    \\
    &\leq C.
\end{align*}
In this way, the existence of a quasi-strong solution in the case $(K,L)\in(0,+\infty\rbrack\times\{0\}$ is established. Moreover, the energy equality \eqref{energy-equality} can be verified in a similar way as in Section~\ref{existence of quasi}.
Thus, the proof for the case $(K,L)\in(0,+\infty\rbrack\times\{0\}$ is complete.
\end{proof}

\subsection{The case \texorpdfstring{$(K,L)\in(0,+\infty]\times\{+\infty\}$}
{(K,L) ∈ (0,+∞] × \{+∞\} } }

In the following lemmas, we return to the approximate solutions constructed in Section~\ref{existence of quasi} and derive uniform estimates for sufficiently large $L\ge 1$.

\begin{lemma}
\label{refine4}
There exists $L_0\geq 1$ and a positive constant $C$, independent of $m$, $\gamma$, $\sigma$, $k$ and $L\geq L_0$, such that
\begin{align}
    \|f_0(\varphi_m)\|_{L^2(0,T;L^1(\Omega))} +\|g_0(\psi_m)\|_{L^2(0,T;L^1(\Gamma))}
    &\leq C,\label{refine2-1}
    \\
    k^{-1}\|f_\sigma(\varphi_m)\|_{L^2(0,T;L^1(\Omega))}
    +k^{-1}\|f_\sigma(\psi_m)\|_{L^2(0,T;L^1(\Gamma))}
    &\leq C.\label{refine2-1'}
\end{align}
\end{lemma}
\begin{proof}
Multiplying \eqref{eqd2} by $\mathcal{N}_\Omega(\varphi_m-\langle\varphi_m\rangle_\Omega)$
and integrating over $\Omega$, and multiplying \eqref{eqd3} by $\mathcal{N}_\Gamma(\psi_m-\langle\psi_m\rangle_\Gamma)$
and integrating over $\Gamma$, adding the resultants together, we obtain
\begin{align}
    0&=\langle\partial_t\varphi_m, \mathcal{N}_\Omega(\varphi_m-\langle\varphi_m\rangle_\Omega)\rangle_{V',V}
    +\langle\partial_t\psi_m,\mathcal{N}_\Gamma(\psi_m-\langle\psi_m\rangle_\Gamma) \rangle_{V_\Gamma',V_\Gamma}
    \notag\\
    &\quad+\int_\Omega\mathbf{v}_m \cdot\nabla\varphi_m\,\mathcal{N}_\Omega(\varphi_m-\langle\varphi_m\rangle_\Omega)\,\mathrm{d}x
    +\int_\Gamma\mathbf{v}_{m,\bm{\tau}} \cdot\nabla_{\!\bm{\tau}\,}\psi_m\,\mathcal{N}_\Gamma(\psi_m-\langle\psi_m\rangle_\Gamma)\,\mathrm{d}S
    \notag\\
    &\quad+\int_\Omega \mu_m(\varphi_m-\langle\varphi_m\rangle_\Omega)\,\mathrm{d}x
    +\int_\Gamma\mathcal{L}_m(\psi_m-\langle\psi_m\rangle_\Gamma)\,\mathrm{d}S
    \notag\\
   &\quad+\frac{1}{L}\int_\Gamma(\mu_m-\mathcal{L}_m)\big[\mathcal{N}_\Omega(\varphi_m-\langle\varphi_m\rangle_\Omega)
    -\mathcal{N}_\Gamma(\psi_m-\langle\psi_m\rangle_\Gamma)\big]\,\mathrm{d}S
    \notag\\
    & \eqqcolon \sum_{j=1}^7 R_j.
    \label{refine2-2}
\end{align}
By the same calculation as in \cite[Lemma 5.8]{LvWuIFB}, we obtain
\begin{align}
        |R_1|+|R_2|&\leq C\Big(\|\partial_t\varphi_m\|_{V'}\|\nabla\varphi_m\|_{\mathbf{L}^2(\Omega)}
        +\|\partial_t\psi_m\|_{V_\Gamma'} \|\nabla_{\!\bm{\tau}\,}\psi_m\|_{\mathbf{L}^2(\Gamma)}\Big),
        \notag\\
        |R_7|&\leq \frac{C}{\sqrt{L}}\Big(\frac{1}{\sqrt{L}}\|\mu_m-\mathcal{L}_m\|_{H_\Gamma}\Big)
        (\|\varphi_m\|_H+\|\psi_m\|_{H_\Gamma}).
        \notag
\end{align}
For the terms $R_3$ and $R_4$, we have
\begin{align}
        |R_3|&=\Big|\int_\Omega\varphi_m\mathbf{v}_m \cdot\nabla\mathcal{N}_\Omega(\varphi_m-\langle\varphi_m\rangle_\Omega)\,\mathrm{d}x\Big|
        \notag\\
        &\leq \|\varphi_m\|_{L^\infty(\Omega)}\|\mathbf{v}_m\|_{\mathbf{L}^2(\Omega)}
        \|\nabla\mathcal{N}_\Omega(\varphi_m-\langle\varphi_m\rangle_\Omega)\|_{\mathbf{L}^2(\Omega)}
        \notag\\
        &\leq C\|\mathbf{v}_m\|_{\mathbf{L}^2(\Omega)}\|\varphi_m-\langle\varphi_m\rangle_\Omega\|_H,
        \notag
\end{align}
and similarly,
\begin{align}
        |R_4|&\leq C\|\mathbf{v}_{m,\bm{\tau}}\|_{\mathbf{L}^2(\Gamma)}
        \|\psi_m-\langle\psi_m\rangle_\Gamma\|_{H_\Gamma}.
        \notag
\end{align}
The estimates for $R_5$ and $R_6$ are more involved.
By \eqref{eqd4} and \eqref{eqd5}, we see that
\begin{align}
        R_5+R_6&=\int_\Omega|\nabla\varphi_m|^2\,\mathrm{d}x +\int_\Gamma|\nabla_{\!\bm{\tau}\,}\psi_m|^2\,\mathrm{d}S  {+\chi(K)\int_\Gamma|\varphi_m-\psi_m|^2\,\mathrm{d}S}
        \notag\\
        &\quad+\int_\Omega f_0(\varphi_m)(\varphi_m-\langle\varphi_m\rangle_\Omega)\,\mathrm{d}x
        +\int_\Gamma g_0(\psi_m)(\psi_m-\langle\psi_m\rangle_\Gamma)\,\mathrm{d}S
        \notag\\
        &\quad+k^{-1}\int_\Omega f_0\Big(\frac{\varphi_m}{1-\sigma}\Big)(\varphi_m-\langle\varphi_m\rangle_\Omega)\,\mathrm{d}x
        +k^{-1}\int_\Gamma f_0\Big(\frac{\psi_m}{1-\sigma}\Big)(\psi_m-\langle\psi_m\rangle_\Gamma)\,\mathrm{d}S
        \notag\\
        &\quad+\int_\Omega(\gamma\partial_t\varphi_m +f_1(\varphi_m))(\varphi_m-\langle\varphi_m\rangle_\Omega)\,\mathrm{d}x
        \notag\\
        &\quad+\int_\Gamma(\gamma\partial_t\psi_m +g_1(\psi_m))(\psi_m-\langle\psi_m\rangle_\Gamma)\,\mathrm{d}S
        +(\langle\varphi_m\rangle_\Omega-\langle\psi_m\rangle_\Gamma)\int_\Gamma\partial_{\mathbf{n}}\varphi_m\,\mathrm{d}S.
        \label{refine2-7}
\end{align}
In view of \eqref{initial-mean-value}, we can apply the method in \cite{MZ04} and an argument similar to the derivation of \eqref{uni-0323-2} and \eqref{uni-0323-3} to get
\begin{align}
        &\int_\Omega f_0(\varphi_m)(\varphi_m-\langle\varphi_m\rangle_\Omega)\,\mathrm{d}x
        +\int_\Gamma g_0(\psi_m)(\psi_m-\langle\psi_m\rangle_\Gamma)\,\mathrm{d}S
        \notag\\
        &\quad=\int_\Omega f_0(\varphi_m)(\varphi_m-\langle\varphi_{0,\sigma,k}\rangle_\Omega)\,\mathrm{d}x
        +\int_\Gamma g_0(\psi_m)(\psi_m-\langle\psi_{0,\sigma,k}\rangle_\Gamma)\,\mathrm{d}S
        \notag\\
        &\qquad+(\langle\varphi_{0,\sigma,k}\rangle_\Omega-\langle\varphi_m\rangle_\Omega)\int_\Omega f_0(\varphi_m)\,\mathrm{d}x
        +(\langle\psi_{0,\sigma,k}\rangle_\Gamma-\langle\psi_m\rangle_\Gamma)\int_\Gamma g_0(\psi_m)\,\mathrm{d}S
        \notag\\
        &\quad\geq \widetilde{c}_1\int_\Omega |f_0(\varphi_m)|\,\mathrm{d}x+\widetilde{c}_1\int_\Gamma|g_0(\psi_m)|\,\mathrm{d}S-\widetilde{c}_2(|\Omega|+|\Gamma|)
        \notag\\
        &\qquad -|\langle\varphi_{0,\sigma,k}\rangle_\Omega-\langle\varphi_m\rangle_\Omega|\int_\Omega |f_0(\varphi_m)|\,\mathrm{d}x
        -|\langle\psi_{0,\sigma,k}\rangle_\Gamma-\langle\psi_m\rangle_\Gamma|\int_\Gamma |g_0(\psi_m)|\,\mathrm{d}S
        \notag
\end{align}
and
\begin{align}
        &k^{-1}\int_\Omega f_0\Big(\frac{\varphi_m}{1-\sigma}\Big)(\varphi_m-\langle\varphi_m\rangle_\Omega)\,\mathrm{d}x
        +k^{-1}\int_\Gamma f_0\Big(\frac{\psi_m}{1-\sigma}\Big)(\psi_m-\langle\psi_m\rangle_\Gamma)\,\mathrm{d}S
        \notag\\
        &\quad=k^{-1}(1-\sigma)\int_\Omega f_0\Big(\frac{\varphi_m}{1-\sigma}\Big)\Big(\frac{\varphi_m}{1-\sigma}-\langle\varphi_m\rangle_\Omega\Big)\,\mathrm{d}x
        \notag\\
        &\qquad+k^{-1}(1-\sigma)\int_\Gamma f_0\Big(\frac{\psi_m}{1-\sigma}\Big)\Big(\frac{\psi_m}{1-\sigma}-\langle\psi_m\rangle_\Gamma\Big)\,\mathrm{d}S
        \notag\\
        &\qquad-k^{-1}\sigma\langle\varphi_m\rangle_\Omega\int_\Omega f_0\Big(\frac{\varphi_m}{1-\sigma}\Big)\,\mathrm{d}x
        -k^{-1}\sigma\langle\psi_m\rangle_\Gamma\int_\Gamma f_0\Big(\frac{\psi_m}{1-\sigma}\Big)\,\mathrm{d}S
        \notag\\
        &\quad=k^{-1}(1-\sigma)\int_\Omega f_0\Big(\frac{\varphi_m}{1-\sigma}\Big)\Big(\frac{\varphi_m}{1-\sigma}-\frac{\langle\varphi_{0,\sigma,k}\rangle_\Omega}{1-\sigma}\Big)\,\mathrm{d}x
        \notag\\
        &\qquad+k^{-1}(1-\sigma)\int_\Gamma f_0\Big(\frac{\psi_m}{1-\sigma}\Big)\Big(\frac{\psi_m}{1-\sigma}-\frac{\langle\psi_{0,\sigma,k}\rangle_\Gamma}{1-\sigma}\Big)\,\mathrm{d}S
        \notag\\
        &\qquad-k^{-1}(\langle\varphi_{m}\rangle_\Omega-\langle\varphi_{0,\sigma,k}\rangle_\Omega)\int_\Omega f_0\Big(\frac{\varphi_m}{1-\sigma}\Big)\,\mathrm{d}x
        \notag\\
        &\qquad-k^{-1}(\langle\psi_m\rangle_\Gamma-\langle\psi_{0,\sigma,k}\rangle_\Gamma)\int_\Gamma f_0\Big(\frac{\psi_m}{1-\sigma}\Big)\,\mathrm{d}S
        \notag\\
        &\quad\geq k^{-1}(1-\sigma)\widetilde{c}_1\int_\Omega \Big|f_0\Big(\frac{\varphi_m}{1-\sigma}\Big)\Big|\,\mathrm{d}x+k^{-1}(1-\sigma)\widetilde{c}_1\int_\Gamma \Big|f_0\Big(\frac{\psi_m}{1-\sigma}\Big)\Big|\,\mathrm{d}S
        \notag\\
         &\qquad -k^{-1}(1-\sigma)\widetilde{c}_2(|\Omega|+|\Gamma|)-k^{-1}|\langle\varphi_{m}\rangle_\Omega-\langle\varphi_{0,\sigma,k}\rangle_\Omega|\int_\Omega \Big|f_0\Big(\frac{\varphi_m}{1-\sigma}\Big)\Big|\,\mathrm{d}x
         \notag\\
         &\qquad -k^{-1}|\langle\psi_m\rangle_\Gamma-\langle\psi_{0,\sigma,k}\rangle_\Gamma|\int_\Gamma \Big|f_0\Big(\frac{\psi_m}{1-\sigma}\Big)\Big|\,\mathrm{d}S,
         \notag
\end{align}
where the positive constants $\widetilde{c}_1$ and $\widetilde{c}_2$ are independent of $k$ and $\sigma$.
Using Lemma~\ref{energy-estimates}, we can derive the following estimate
\begin{align}
        |\langle\varphi_{0,\sigma,k}\rangle_\Omega-\langle\varphi_m(t)\rangle_\Omega|
        &=\frac{1}{|\Omega|}\Big|\int_0^t \langle\partial_t\varphi_m(s),1\rangle_{V',V}\,\mathrm{d}s\Big|
        \notag\\
        &=\frac{1}{|\Omega|}\Big|\int_0^t\Big(- \int_\Omega \mathbf{v}_m(s)\cdot\nabla\varphi_m(s) \,\mathrm{d}x
        +\int_\Omega\Delta\mu_m(s)\,\mathrm{d}x\Big) \,\mathrm{d}s\Big|
        \notag \\
        &=\frac{1}{|\Omega|}\Big|\int_0^t\int_\Gamma\partial_{\mathbf{n}}\mu_m(s)\,\mathrm{d}S\,\mathrm{d}s\Big|
        \leq \frac{1}{|\Omega|}\frac{1}{L} \int_0^t \|\mu_m(s)-\mathcal{L}_m(s)\|_{L^1(\Gamma)}\,\mathrm{d}s
        \notag \\
        &\leq \frac{|\Gamma|^{\frac{1}{2}}|T|^{\frac{1}{2}}}{|\Omega|\sqrt{L}}
        \Big(\frac{1}{\sqrt{L}}\|\mu_m(s)-\mathcal{L}_m(s)\|_{L^2(0,T;H_\Gamma)}\Big)
        \leq  \frac{C'|\Gamma|^{\frac{1}{2}}|T|^{\frac{1}{2}}}{|\Omega|\sqrt{L}}
        \notag
\end{align}
for almost all $t\in[0,T]$, and in a similar manner, we obtain
\begin{align}
        |\langle\psi_{0,\sigma,k}\rangle_\Gamma-\langle\psi_m(t)\rangle_\Gamma|
        \leq \frac{C''|T|^{\frac{1}{2}}}{|\Gamma|^{\frac{1}{2}}\sqrt{L}}\notag
\end{align}
for almost all $t\in[0,T]$.
Let us fix $L_0\geq1$ sufficiently large such that
\begin{align*}
    \frac{C'|\Gamma|^{\frac{1}{2}}|T|^{\frac{1}{2}}}{|\Omega|\sqrt{L}}
    \leq \frac{\widetilde{c}_1}{4} \quad \text{and}
    \quad
    \frac{C''|T|^{\frac{1}{2}}}{|\Gamma|^{\frac{1}{2}}\sqrt{L}}\leq \frac{\widetilde{c}_1}{4}
    \quad\text{for all $L\geq L_0$.}
\end{align*}
Now for the last three lines on the right-hand side of \eqref{refine2-7},
by \eqref{m-uni-7} and H\"older's inequality, we get
\begin{align}
    &\int_\Omega(\gamma\partial_t\varphi_m+f_1(\varphi_m))(\varphi_m-\langle\varphi_m\rangle_\Omega)\,\mathrm{d}x +\int_\Gamma(\gamma\partial_t\psi_m+g_1(\psi_m))(\psi_m-\langle\psi_m\rangle_\Gamma)\,\mathrm{d}S
    \notag\\
    &\qquad+(\langle\varphi_m\rangle_\Omega-\langle\psi_m\rangle_\Gamma)\int_\Gamma\partial_{\mathbf{n}}\varphi_m\,\mathrm{d}S
    \notag\\
    &\quad\leq C\Big(1+\gamma\|\partial_t\boldsymbol{\varphi}_m\|_{\mathcal{L}^2}+\|\partial_{\mathbf{n}}\varphi_m\|_{H_\Gamma}\Big).
    \notag
\end{align}
Collecting the above estimates, we can conclude that
\begin{align}
    R_5+R_6 &\geq \frac{\widetilde{c}_1}{2}\int_\Omega |f_0(\varphi_m)|\,\mathrm{d}x+\frac{\widetilde{c}_1}{2}\int_\Gamma|g_0(\psi_m)|\,\mathrm{d}S-\widetilde{c}_2(|\Omega|+|\Gamma|)
    \notag\\
    &\quad +\frac{k^{-1}\widetilde{c}_1}{4}\int_\Omega \Big|f_0\Big(\frac{\varphi_m}{1-\sigma}\Big)\Big|\,\mathrm{d}x+\frac{k^{-1}\widetilde{c}_1}{4}\int_\Gamma \Big|f_0\Big(\frac{\psi_m}{1-\sigma}\Big)\Big|\,\mathrm{d}S
    \notag\\
    &\quad -k^{-1}(1-\sigma)\widetilde{c}_2(|\Omega|+|\Gamma|)
    -C\Big(1+\gamma\|\partial_t \boldsymbol{\varphi}_m\|_{\mathcal{L}^2} +\|\partial_{\mathbf{n}}\varphi_m\|_{H_\Gamma}\Big).
    \notag
\end{align}
As $K\partial_{\mathbf{n}}\varphi_m=\psi_m-\varphi_m$ a.e. on $\Sigma_T$, the conclusions \eqref{refine2-1} and \eqref{refine2-1'} follow from \eqref{refine2-2}
and the estimates above for $R_j$ with $1\leq j\leq7$.
This completes the proof of Lemma~\ref{refine4}.
\end{proof}

\begin{lemma}
    \label{refine5}
    There exists a positive constant $C$, independent of $m$, $\gamma$, $\sigma$, $k$ and $L\geq L_0$, such that
    \begin{align}
        \|\boldsymbol{\mu}_m\|_{L^2(0,T;\mathcal{H}^1)}\leq C.\label{refine5-1}
    \end{align}
\end{lemma}

\begin{proof}
   By Lemma \ref{energy-estimates} and \eqref{refine2-1}, we obtain \eqref{refine5-1} by repeating the process of \cite[Lemma 5.9]{LvWuIFB}.
\end{proof}

\begin{lemma}
    \label{refine6}
There exists a positive constant $C$, independent of $m$, $\gamma$, $\sigma$, $k$ and $L\geq L_0$, such that
\begin{align}
\sqrt{\gamma}\|\partial_t\bm{\varphi}_m\|_{L^\infty(0,T;\mathcal{L}^2)}
+\|\bm{\mu}_m\|_{L^\infty(0,T;\mathcal{H}^1)}+\|\partial_t\bm{\varphi}_m\|_{L^2(0,T;\mathcal{H}^1)}
+\|\partial_t\mathbf{v}_m\|_{L^2(0,T;\mathbf{H}^1(\Omega))}
&\leq C,\label{uniform-0318-6}
\\
\|f_0(\varphi_m)\|_{L^\infty(0,T;H)}+\|g_0(\psi_m)\|_{L^\infty(0,T;H_\Gamma)}
+\|\boldsymbol{\varphi}_m\|_{L^\infty(0,T;\mathcal{H}^2)}
+\|\boldsymbol{\mu}_m\|_{L^2(0,T;\mathcal{H}^2)}
&\leq C,\label{uniform-0318-7}
\\
k^{-1}\|f_\sigma(\varphi_m)\|_{L^\infty(0,T;H)}+k^{-1}\|f_\sigma(\psi_m)\|_{L^\infty(0,T;H_\Gamma)}
&\leq C.\label{uniform-0318-7'}
\end{align}
\end{lemma}
\begin{proof}
    Based on Lemmas~\ref{refine4} and \ref{refine5}, we can repeat the proof of Lemmas~\ref{higher} and \ref{final}
    to verify the claimed estimates.
    In particular, the constant in the estimate \eqref{0306-Gm0-4'} can now be chosen independently of  $L\in[L_0,+\infty)$. The proof is complete.
\end{proof}

In view of Lemmas~\ref{refine4}, \ref{refine5} and \ref{refine6}, we are now ready to handle the case $(K,L)\in(0,+\infty\rbrack\times\{+\infty\}$.

\begin{proof}[Proof of Theorem~\ref{quasi-strong}: $(K,L)\in(0,+\infty\rbrack\times\{+\infty\}$.]
Based on the refined estimates in Lemmas~\ref{refine4}, \ref{refine5} and \ref{refine6}, we can repeat the argument in Section~\ref{passage-to-limit} to conclude that there exist functions $(\mathbf{v}^L,\mathbf{v}^L_{\boldsymbol{\tau}},\boldsymbol{\varphi}^L,\boldsymbol{\mu}^L)$ satisfying the weak formulation \eqref{formula-0318-2} and the equations in \eqref{formula-0318-1}. As norms on Banach spaces are weakly (star) lower semicontinuous, and the constants $C$ in Lemmas~\ref{refine4}, \ref{refine5} and \ref{refine6} are independent of the parameters $m$, $\sigma$, $\gamma=k^{-1}$, and $L\in[L_0,+\infty)$, we can repeat the argument in Section~\ref{passage-to-limit} by passing to the limit as $m\to+\infty$, $\sigma\to 0$, $\gamma=k^{-1}\to 0$ subsequently, and conclude that the uniform estimate \eqref{uniform-0318-5} holds for all $L\in[L_0,+\infty)$.
Then, similar to the process for the case $L=0$, we can prove Theorem~\ref{quasi-strong} in the current case $(K,L)\in(0,+\infty\rbrack\times\{+\infty\}$ by passing to the limit as $L\to+\infty$ in \eqref{formula-0318-2} and \eqref{formula-0318-1}. The details are omitted for brevity.
\end{proof}


\appendix
\section{Analysis of a Bulk-Surface Convective Viscous Cahn--Hilliard System}
\label{appendix}
\setcounter{equation}{0}
Let $\gamma>0$, $\sigma\in (0,1/2)$, $k\in \mathbb{N}^+$. We consider the following bulk-surface convective viscous Cahn--Hilliard system
\begin{subequations}
\label{convective-viscous}
\begin{alignat}{2}
    \label{convective-viscous:1}
    &\partial_{t}\varphi+\mathbf{v} \cdot\nabla\varphi=\Delta\mu
    &&\quad  \mbox{in }\ Q_{T},
    \\[1mm]
    \label{convective-viscous:2}
    &\mu =\gamma\partial_t\varphi-\Delta \varphi +f(\varphi )+k^{-1} f_\sigma(\varphi)
    &&\quad \mbox{in }\ Q_{T},
    \\[1mm]
    \label{convective-viscous:3}
    &\begin{cases}
    K \partial_{\mathbf{n}}\varphi =\psi-\varphi, & \text{if}\ K\in(0,+\infty)\\
    \partial_{\mathbf{n}}\varphi=0,& \text{if}\ K=+\infty
    \end{cases}
    &&\quad {\mbox{on }\Sigma _{T },}
    \\[1mm]
    \label{convective-viscous:4}
    &L\partial_{\mathbf{n}}\mu=\mathcal{L}-\mu, \quad  L\in(0,+\infty)
    &&\quad\mbox{on }\ \Sigma _{T},
    \\[1mm]
    \label{convective-viscous:5}
    &\partial_t\psi+\mathbf{v}_{\bm{\tau}} \cdot\nabla_{\!\bm{\tau}\,}\psi
    =\Delta_{\bm{\tau}}\mathcal{L}
    -\partial_{\mathbf{n}}\mu
    &&\quad \mbox{on }\ \Sigma _{T},
    \\[1mm]
    \label{convective-viscous:6}
    &\mathcal{L}  =\gamma\partial_t\psi - \Delta _{\bm{\tau}}\psi +\partial _{\mathbf{n}}\varphi+g(\psi) +k^{-1} f_\sigma(\psi)
    &&\quad \mbox{on }\ \Sigma _{T},
\end{alignat}
\end{subequations}
subject to the initial conditions
\begin{align}
    \label{convective-viscous-initial}
	\varphi(0)=\varphi_0\;\;\text{ in }\Omega,\qquad
    \psi(0)=\psi_0\;\;\text{ on }\Gamma.
\end{align}
Here, $\mathbf{v}$ is a prescribed velocity field and $f_\sigma$ is the nonlinear function introduced in \eqref{f_sigma}.

In the following, we establish the well-posedness and regularity results for problem \eqref{convective-viscous}--\eqref{convective-viscous-initial}.
The analysis extends the previous work \cite{Gior22} to the viscous approximation of a bulk-surface convective Cahn--Hilliard system.

\begin{theorem}
\label{A1}
Let $\Omega\subset\mathbb R^d$ with $d\in\{2,3\}$ be a bounded domain with smooth boundary $\Gamma$. Assume that \ref{ASS:A2} and \ref{ASS:A3} hold, $\gamma>0$, $\sigma\in (0,1/2)$, $k\in \mathbb{N}^+$. Moreover, we assume that $\mathbf{v}$ is a velocity field satisfying $(\mathbf{v},\mathbf{v}_{\boldsymbol{\tau}})\in L^\infty(0,T;\bm{\mathcal{L}}_{\mathrm{div}}^2\cap \bm{\mathcal{L}}^3)$,
and the initial data $\boldsymbol{\varphi}_0 = (\varphi_0,\psi_0)\in \mathcal{H}^1\cap \mathcal{L}^\infty$
satisfies  $\|\boldsymbol{\varphi}_0\|_{\mathcal{L}^\infty}\leq 1-\sigma$ and $|\overline{m}(\boldsymbol{\varphi}_0)|<1-\sigma$.
Then, problem \eqref{convective-viscous}--\eqref{convective-viscous-initial} admits a unique weak solution $(\boldsymbol{\varphi},\boldsymbol{\mu})$ on $[0,T]$ such that
\begin{align}
    \label{REG:A1}
    \left\{
    \begin{aligned}
	&\boldsymbol{\varphi} = (\varphi,\psi)\in {L^\infty(0,T;\mathcal{H}^1)}\cap L^2(0,T;\mathcal{H}^2),\quad
    (\partial_t\varphi,\partial_t\psi)\in L^2(0,T;\mathcal{L}^2),\\
	&|\varphi(x,t)|<1-\sigma\;\; \text{a.e.~in }Q_T,\quad
    |\psi(x,t)|<1-\sigma\;\; \text{a.e.~on }\Sigma_T,\\
	&\boldsymbol{\mu}=(\mu,\mathcal{L})\in L^2(0,T;\mathcal{H}^2),\quad
    \big(f_0(\varphi),g_0(\psi)\big)\in L^2(0,T;\mathcal{L}^2),
    \end{aligned}
    \right.
\end{align}
which satisfies the equations in \eqref{convective-viscous} almost everywhere in $Q_T$ (resp. $\Sigma_T$) and the initial conditions \eqref{convective-viscous-initial} almost everywhere in $\Omega$ (resp. $\Gamma$).
In addition, the weak solution has the following regularity properties:
\begin{enumerate}[label=\textnormal{\bfseries (R\arabic*)}, leftmargin=*, topsep=0.5ex]
	\item \label{R1}
    Suppose that $(\partial_t\mathbf{v},\partial_t\mathbf{v}_{\boldsymbol{\tau}})\in L^\frac{4}{3}(0,T;\bm{\mathcal{L}}^1)$ and $(\mu_0^\ast,\mathcal{L}_0^\ast)\in\mathcal{L}^2$,
    where
    \begin{align*}
    \mu_0^\ast \coloneqq -\Delta\varphi_0+f_0(\varphi_0)+k^{-1} f_\sigma(\varphi_0),\quad\mathcal{L}_0^\ast \coloneqq
    -\Delta_{\bm{\tau}}\psi_0+\partial_{\mathbf{n}}\varphi_0+g_0(\psi_0)+k^{-1} f_\sigma(\psi_0).
    \end{align*}
    Furthermore, the following compatibility condition holds
    \begin{align}
    \begin{cases}
    K \partial_{\mathbf{n}}\varphi_0 =\psi_0 -\varphi_0 , & \text{if}\ K\in(0,+\infty)\\
    \partial_{\mathbf{n}}\varphi_0 =0,& \text{if}\ K=+\infty
    \end{cases}
    \qquad {\mbox{a.e.~on } \Sigma _{T }}.
    \label{appendix-compatibility}
    \end{align}
    Then, we have
	\begin{align*}
	&(\varphi,\psi)\in L^\infty(0,T;\mathcal{H}^2),\quad
    (\mu,\mathcal{L})\in L^\infty(0,T;\mathcal{H}^2),\\
	&(\partial_t\varphi,\partial_t\psi)\in L^\infty(0,T;\mathcal{L}^2)
    \cap {L^2(0,T;\mathcal{H}^1)}.
	\end{align*}
	\item \label{R2}
    Suppose that the assumptions of \ref{R1} hold.
    Moreover, assume that
    \begin{align}
         \|\varphi_0\|_{L^\infty(\Omega)}\leq (1-\sigma)(1-\delta_0),\quad
    \|\psi_0\|_{L^\infty(\Gamma)}\leq (1-\sigma)(1-\delta_0)
    \label{A-separation}
    \end{align}
	 for some $\delta_0\in(0,1)$. Then, there exists $\delta\in(0,1)$ such that
	\begin{align}
		\max_{(x,t)\in\Omega\times[0,T]} |\varphi(x,t)|\leq (1-\sigma)(1-\delta),\quad
        \max_{(x,t)\in\Gamma\times[0,T]} |\psi(x,t)|\leq (1-\sigma)(1-\delta),
        \label{viscous-separation}
	\end{align}
	and $(\varphi,\psi)\in L^2(0,T;\mathcal{H}^3)$.
	\item \label{R3}
    Suppose that the assumptions of \ref{R2} hold.
    Moreover, assume that $(\partial_t\mathbf{v},\partial_t\mathbf{v}_{\boldsymbol{\tau}})\in L^2(0,T;\boldsymbol{\mathcal{L}}^2)$ and $(\mu_0^\ast,\mathcal{L}_0^\ast)\in\mathcal{H}^1$. Then, we have
	\begin{alignat*}{2}
		&(\varphi,\psi) \in L^\infty(0,T;\mathcal{H}^3),
        &\quad (\partial_t\varphi,\partial_t\psi) &\in L^\infty(0,T;\mathcal{H}^1)\cap L^2(0,T;\mathcal{H}^2),
        \\
		&(\partial_{t}^2\varphi,\partial_{t}^2\psi) \in L^2(0,T;\mathcal{L}^2),
        &\quad
        (\partial_t\mu,\partial_t\mathcal{L}) &\in L^2(0,T;\mathcal{H}^1).
	\end{alignat*}
\end{enumerate}
\end{theorem}

\begin{proof}
The proof is divided into several steps.
When $K=+\infty$, the boundary condition of $\varphi$ reduces to the homogeneous Neumann boundary condition
$\partial_\mathbf{n}\varphi=0$ on $\Sigma_T$.
This case can be treated by a limiting procedure as $K\to +\infty$, so we focus on the case $K\in(0,+\infty)$ below.
\medskip

\noindent\textbf{\itshape Existence.}
The existence of a global weak solution on $[0,T]$ with the regularity \eqref{REG:A1} can be shown in a standard way:
first for regular potentials through a Faedo--Galerkin method (cf.~\cite{KS2024}),
and afterwards for singular potentials by means of a Yosida approximation (cf.~\cite{KSJDE}).

\medskip

\noindent\textbf{\itshape Uniqueness.}
{For $i=1,2$, let $\boldsymbol{\varphi}_{0,i} = (\varphi_{0,i}, \psi_{0,i})$ be two initial data, which are admissible in the sense of Theorem~\ref{A1} and satisfy $\overline{m}(\boldsymbol{\varphi}_{0,1})=\overline{m}(\boldsymbol{\varphi}_{0,2})$. Furthermore, let
$(\boldsymbol{\varphi}_i,\boldsymbol{\mu}_i)$ with $\boldsymbol{\varphi}_i=(\varphi_i,\psi_i)$ and $\boldsymbol{\mu}_i=(\mu_i,\mathcal{L}_i)$ be the  corresponding weak solutions of problem
\eqref{convective-viscous}--\eqref{convective-viscous-initial}.
For brevity, we use the notation
\begin{align*}
\widetilde{\boldsymbol{\varphi}}
= (\widetilde \varphi,\widetilde\psi)
\coloneqq \boldsymbol{\varphi}_1-\boldsymbol{\varphi}_2,
\quad
\widetilde{\boldsymbol{\mu}}
= (\widetilde \mu,\widetilde{\mathcal{L}})
\coloneqq \boldsymbol{\mu}_1-\boldsymbol{\mu}_2,
\quad
\widetilde{\boldsymbol{\varphi}}_0
= (\widetilde \varphi_0,\widetilde\psi_0)
\coloneqq \boldsymbol{\varphi}_{0,1}-\boldsymbol{\varphi}_{0,2}.
\end{align*}
Then the pairs $\widetilde{\boldsymbol{\varphi}}$ and $\widetilde{\boldsymbol{\mu}}$ solve}
\begin{align}
        \begin{cases}
            \partial_t\widetilde{\varphi}+\mathbf{v}\cdot\nabla\widetilde{\varphi}=\Delta\widetilde{\mu}&\text{a.e.~in }Q_T,\\
           L\partial_{\mathbf{n}}\widetilde{\mu}=\widetilde{\mathcal{L}}-\widetilde{\mu}&\text{a.e.~on }\Sigma_T,\\
            \partial_t\widetilde{\psi}+\mathbf{v}_{\bm{\tau}}\cdot\nabla_{\!\bm{\tau}\,}\widetilde{\psi}
            =\Delta_{\bm{\tau}}\widetilde{\mathcal{L}}-\partial_{\mathbf{n}}\widetilde{\mu}&\text{a.e.~on }\Sigma_T,
        \end{cases}\label{0305-7}
\end{align}
and
\begin{align}
        \begin{cases}
            \widetilde{\mu}=\gamma\partial_t\widetilde{\varphi}-\Delta\widetilde{\varphi}
            +f(\varphi_1)-f(\varphi_2)+k^{-1}f_\sigma(\varphi_1)-k^{-1} f_\sigma(\varphi_2)&\text{a.e.~in }Q_T,\\
		{K\partial_{\mathbf{n}}\widetilde{\varphi}=\widetilde{\psi}-\widetilde{\varphi},\quad K\in(0,+\infty)}
	&\mbox{a.e.~on } \Sigma _{T},
	 \\
	 \widetilde{\mathcal{L}  }=\gamma\partial_t\widetilde{\psi }- \Delta _{\bm{\tau}}\widetilde{\psi}
     +\partial _{\mathbf{n}}\widetilde{\varphi}+g(\psi_1)-g(\psi_2)
     +k^{-1} f_\sigma(\psi_1)-k^{-1} f_\sigma(\psi_2)  & \mbox{a.e.~on }
	\Sigma _{T}.
        \end{cases}\notag
\end{align}
Since $L\in(0,+\infty)$, we use the definition of $\mathfrak{S}^L$
to rewrite \eqref{0305-7} as
\begin{align}
        \mathbb{P}\big(\widetilde{\mu}, \widetilde{\mathcal{L}}\big)
        =-\mathfrak{S}^{L}\big(\partial_t\widetilde{\varphi},\partial_t\widetilde{\psi}\big)
        -\mathfrak{S}^L(\mathbf{v} \cdot\nabla\widetilde{\varphi},\mathbf{v}_{\bm{\tau}}\cdot\nabla_{\!\bm{\tau}\,}\widetilde{\psi})
        \qquad\text{in }\mathcal{H}_{L,0}^1,
        \label{0305-9}
\end{align}
where $\mathbb P$ is the projection introduced in \eqref{DEF:PRO:L0}.
It follows directly from \eqref{0305-7} that $\overline{m}(\widetilde{\boldsymbol{\varphi}}(t))=0$ for all $t\in[0,T]$. Hence, multiplying
\eqref{0305-9} by $\mathfrak{S}^L(\widetilde{\varphi},\widetilde{\psi})$ in $\mathcal{H}_{L,0}^1$,
and recalling \ref{ASS:A2}, \eqref{f_lambda_p1}, we find that
\begin{align}
        \frac{1}{2}\frac{\mathrm{d}}{\mathrm{d}t}\|\bm{\widetilde{\varphi}}\|_{\mathcal{H}_{L,0}^{-1}}^2
        &=-\int_\Omega \mathbf{v}\cdot\nabla\widetilde{\varphi}\,\mathfrak{S}^L_\Omega (\bm{\widetilde{\varphi}})
        \,\mathrm{d}x
        -\int_\Gamma\mathbf{v}_{\bm{\tau}} \cdot\nabla_{\!\bm{\tau}\,}\widetilde{\psi}
        \,\mathfrak{S}^L_\Gamma (\bm{\widetilde{\varphi}})\,\mathrm{d}S
        -\int_\Omega\widetilde{\mu} \widetilde{\varphi}\,\mathrm{d}x
        -\int_\Gamma\widetilde{\mathcal{L}} \widetilde{\psi}\,\mathrm{d}S
        \notag\\
        &=-\int_\Omega \mathbf{v}\cdot\nabla\widetilde{\varphi}\,\mathfrak{S}^L_\Omega (\bm{\widetilde{\varphi}})
        \,\mathrm{d}x
        -\int_\Gamma\mathbf{v}_{\bm{\tau}} \cdot\nabla_{\!\bm{\tau}\,}\widetilde{\psi}\,\mathfrak{S}^L_\Gamma(\bm{\widetilde{\varphi}})\,\mathrm{d}S
        \notag\\
        &\quad- \frac{\gamma}{2}\frac{\mathrm{d}}{\mathrm{d}t}\|\bm{\widetilde{\varphi}}\|_{\mathcal{L}^2}^2
        -\|\nabla\widetilde{\varphi} \|_{\mathbf{L}^2(\Omega)}^2
        -\|\nabla_{\!\bm{\tau}\,} \widetilde{\psi}\|_{\mathbf{L}^2(\Gamma)}^2 {-\frac{1}{K}\int_\Gamma|\widetilde{\psi}-\widetilde{\varphi}|^2\,\mathrm{d}S}
        \notag\\
        &\quad-k^{-1}\int_\Omega(f_\sigma(\varphi_1)-f_\sigma(\varphi_2)) \widetilde{\varphi}\,\mathrm{d}x
        -\int_\Omega(f_0(\varphi_1)-f_0(\varphi_2))\widetilde{\varphi} \,\mathrm{d}x
        \notag\\
        &\quad-k^{-1}\int_\Gamma(f_\sigma(\psi_1)-f_\sigma(\psi_2))\widetilde{\psi} \,\mathrm{d}S
        -\int_\Gamma(g_0(\psi_1)-g_0(\psi_2))\widetilde{\psi}\,\mathrm{d}S
        \notag\\
        &\quad-\int_\Omega(f_1(\varphi_1)-f_1(\varphi_2))\widetilde{\varphi} \,\mathrm{d}x
        -\int_\Gamma(g_1(\psi_1)-g_1(\psi_2))\widetilde{\psi} \,\mathrm{d}S
        \notag\\
        &\leq -\int_\Omega \mathbf{v}\cdot\nabla\widetilde{\varphi}\,\mathfrak{S}^L_\Omega (\bm{\widetilde{\varphi}})
        \,\mathrm{d}x
        -\int_\Gamma\mathbf{v}_{\bm{\tau}} \cdot\nabla_{\!\bm{\tau}\,}\widetilde{\psi}
        \,\mathfrak{S}^L_\Gamma (\bm{\widetilde{\varphi}})\,\mathrm{d}S
        \notag\\
        &\quad- \frac{\gamma}{2}\frac{\mathrm{d}}{\mathrm{d}t}\|\bm{\widetilde{\varphi}} \|_{\mathcal{L}^2}^2
        -\|\nabla\widetilde{\varphi} \|_{\mathbf{L}^2(\Omega)}^2
        -\|\nabla_{\!\bm{\tau}\,} \widetilde{\psi}\|_{\mathbf{L}^2(\Gamma)}^2 {-\frac{1}{K}\int_\Gamma|\widetilde{\psi}-\widetilde{\varphi}|^2\,\mathrm{d}S}
        +C_{\text{Lip}}\|\bm{\widetilde{\varphi}} \|_{\mathcal{L}^2}^2.
        \notag
\end{align}
Using the Sobolev embedding theorem and the regularity estimate \eqref{Hk-regularity} with $k=0$ to bound $\mathfrak{S}^L(\bm{\widetilde{\varphi}})$, we get
\begin{align}
    &\frac{1}{2}\frac{\mathrm{d}}{\mathrm{d}t}\Big(\|\bm{\widetilde{\varphi}}\|_{\mathcal{H}_{L,0}^{-1}}^2
    +\gamma\|\bm{\widetilde{\varphi}} \|_{\mathcal{L}^2}^2\Big)
    +\|\nabla\widetilde{\varphi} \|_{\mathbf{L}^2(\Omega)}^2
    +\|\nabla_{\!\bm{\tau}\,}\widetilde{\psi} \|_{\mathbf{L}^2(\Gamma)}^2 {+\frac{1}{K}\int_\Gamma|\widetilde{\psi}-\widetilde{\varphi}|^2\,\mathrm{d}S}
    \notag\\
    &\quad\leq    -\int_\Omega \mathbf{v}\cdot\nabla\widetilde{\varphi} \,\mathfrak{S}^L_\Omega(\bm{\widetilde{\varphi}})\,\mathrm{d}x
    -\int_\Gamma\mathbf{v}_{\bm{\tau}} \cdot\nabla_{\!\bm{\tau}\,}\widetilde{\psi}\,\mathfrak{S}^L_\Gamma(\bm{\widetilde{\varphi}})\,\mathrm{d}S
    +C_{\text{Lip}}\|\bm{\widetilde{\varphi}} \|_{\mathcal{L}^2}^2
    \notag\\
    &\quad= \int_\Omega \widetilde{\varphi}\mathbf{v}\cdot\nabla\mathfrak{S}^L_\Omega(\bm{\widetilde{\varphi}})\,\mathrm{d}x
    +\int_\Gamma\widetilde{\psi} \mathbf{v}_{\bm{\tau}}\cdot\nabla_{\!\bm{\tau}\,}\mathfrak{S}^L_\Gamma(\bm{\widetilde{\varphi}})\,\mathrm{d}S
    +C_{\text{Lip}}\|\bm{\widetilde{\varphi}} \|_{\mathcal{L}^2}^2
    \notag\\
    &\quad\leq \|\mathbf{v}\|_{\mathbf{L}^3(\Omega)}\|\widetilde{\varphi}\|_{H}\|\nabla\mathfrak{S}^L_\Omega(\bm{\widetilde{\varphi}})\|_{\mathbf{L}^6(\Omega)}
    +\|\mathbf{v}_{\bm{\tau}}\|_{\mathbf{L}^3 (\Gamma)}\|\widetilde{\psi}\|_{H_\Gamma}\|\nabla_{\!\bm{\tau}\,}\mathfrak{S}^L_\Gamma(\bm{\widetilde{\varphi}})\|_{\mathbf{L}^6(\Gamma)}
    +C_{\text{Lip}}\|\bm{\widetilde{\varphi}} \|_{\mathcal{L}^2}^2
    \notag\\[1mm]
    &\quad \leq C(1+\|\mathbf{v}\|_{\mathbf{L}^3(\Omega)}
    +\|\mathbf{v}_{\bm{\tau}} \|_{\mathbf{L}^3(\Gamma)})
    \|\bm{\widetilde{\varphi}}\|_{\mathcal{L}^2}^2.
    \notag
\end{align}
Finally, an application of Gronwall's lemma yields
\begin{align*}
&\|\bm{\widetilde{\varphi}}(t)\|_{\mathcal{H}_{L,0}^{-1}}^2
+\gamma\|\bm{\widetilde{\varphi}}(t)\|_{\mathcal{L}^2}^2
\\
&\quad\leq
\big( \|\bm{\widetilde{\varphi}}_0\|_{\mathcal{H}_{L,0}^{-1}}^2
    +\gamma\|\bm{\widetilde{\varphi}}_0 \|_{\mathcal{L}^2}^2\big)
\exp\left( C\int_0^t(1+\|\mathbf{v}(s)\|_{\mathbf{L}^3(\Omega)}+\|\mathbf{v}_{\bm{\tau}}(s)\|_{\mathbf{L}^3(\Gamma)})\,\mathrm{d}s \right)\notag
\end{align*}
for all $t\in[0,T]$.
This continuous dependence estimate easily implies the uniqueness of weak solutions.

\medskip
\noindent\textbf{\itshape Regularity \ref{R1}.} We proceed to derive further regularity properties of the weak solution. The subsequent estimates may depend on the parameters $\gamma, \sigma, k$.

Let $h\in \big(0, T/2 \big)$. For any Banach space $X$, any $t\in (0,T-h)$ and any function ${f:(0,T-h)\to X}$, we denote by $\partial_t^h f(t) = \frac 1h \big[ f(t+h) - f(t) \big] \in X$ the forward difference quotient of $f$.
In the following, the generic constants denoted by $C$ will always be independent of $h$.
Applying the difference quotient on the equations of \eqref{convective-viscous}, we obtain
\begin{subequations}
\label{difference-quotient}
\begin{alignat}{2}
    \label{difference-quotient:1}
    &\partial_{t}\partial_t^h \varphi+\partial_t^h(\mathbf{v}\cdot\nabla\varphi)
    =\Delta\partial_t^h\mu
    &&\quad  \mbox{a.e.~in } Q_{T-h},
    \\
    \label{difference-quotient:2}
    &\partial_t^h\mu =\gamma\partial_t\partial_t^h\varphi-\Delta\partial_t^h \varphi
    +\partial_t^hf(\varphi )+k^{-1}\partial_t^h f_\sigma(\varphi)
    &&\quad \mbox{a.e.~in }  Q_{T-h},
    \\
    \label{difference-quotient:3}
    & K\partial_\mathbf{n}\partial_t^h\varphi=\partial_t^h\psi-\partial_t^h\varphi,\quad K\in(0,+\infty)
    &&\quad\mbox{a.e.~on } \Sigma_{T-h},
    \\
    \label{difference-quotient:4}
    &L\partial_{\mathbf{n}}\partial_t^h\mu =\partial_t^h\mathcal{L}-\partial_t^h\mu,\quad L\in(0,+\infty)
    &&\quad\mbox{a.e.~on }
    \Sigma_{T-h},
    \\
    \label{difference-quotient:5}
    &\partial_t\partial_t^h\psi +\partial_t^h(\mathbf{v}_{\bm{\tau}}\cdot\nabla_{\!\bm{\tau}\,}\psi)
    =\Delta_{\bm{\tau}}\partial_t^h\mathcal{L}-\partial_{\mathbf{n}}\partial_t^h\mu
    &&\quad \mbox{a.e.~on }
    \Sigma_{T-h},
    \\
    \label{difference-quotient:6}
    &\partial_t^h\mathcal{L}  =\gamma\partial_t\partial_t^h\psi - \Delta _{\bm{\tau}}\partial_t^h\psi
    +\partial _{\mathbf{n}}\partial_t^h\varphi+\partial_t^hg(\psi)+k^{-1}\partial_t^h f_\sigma(\psi)
    &&\quad \mbox{a.e.~on }
    \Sigma_{T-h}.
\end{alignat}
\end{subequations}
It follows from \eqref{difference-quotient:1} and \eqref{difference-quotient:5} that
$\overline{m}(\partial_t^h\boldsymbol{\varphi})=0$. Thus, we obtain
\begin{align*}
\mathbb{P}(\partial_t^h\mu,\partial_t^h\mathcal{L})
=-\mathfrak{S}^L(\partial_{t}\partial_t^h \varphi+\partial_t^h(\mathbf{v}\cdot\nabla\varphi),
\partial_t\partial_t^h\psi+\partial_t^h(\mathbf{v}_{\bm{\tau}}\cdot\nabla_{\!\bm{\tau}\,}\psi))
\quad\text{in }\mathcal{H}_{L,0}^1,
\end{align*}
where $\mathbb P$ is the projection introduced in \eqref{DEF:PRO:L0}.
Arguing as in the proof of uniqueness, we have
\begin{align}
    &\frac{1}{2}\frac{\mathrm{d}}{\mathrm{d}t}\Big(\|\partial_t^h\bm{\varphi}\|_{\mathcal{H}_{L,0}^{-1}}^2
    +\gamma\|\partial_t^h\bm{\varphi} \|_{\mathcal{L}^2}^2\Big)+\|\nabla\partial_t^h\varphi\|_{\mathbf{L}^2(\Omega)}^2
    +\|\nabla_{\!\bm{\tau}\,}\partial_t^h\psi \|_{\mathbf{L}^2(\Gamma)}^2 {+\frac{1}{K}\int_\Gamma|\partial_t^h\psi-\partial_t^h\varphi|^2\,\mathrm{d}S}
    \notag\\
    &\quad=-\int_\Omega \partial_t^h(\mathbf{v}\cdot\nabla\varphi)\mathfrak{S}^L_\Omega(\partial_t^h\boldsymbol{\varphi})\,\mathrm{d}x
    -\int_\Gamma\partial_t^h (\mathbf{v}_{\boldsymbol{\tau}}\cdot\nabla_{\!\boldsymbol{\tau}\,}\psi)\mathfrak{S}_\Gamma^L(\partial_t^h\boldsymbol{\varphi})\,\mathrm{d}S
    \notag\\
     &\qquad-k^{-1}\int_\Omega \partial_t^h f_\sigma(\varphi)\partial_t^h \varphi\,\mathrm{d}x
     -k^{-1}\int_\Gamma\partial_t^h f_\sigma(\psi)\partial_t^h\psi\,\mathrm{d}S
     \notag\\
     &\qquad-\int_\Omega \partial_t^h f_0(\varphi)\partial_t^h\varphi\,\mathrm{d}x
     -\int_\Gamma \partial_t^h g_0(\psi)\partial_t^h\psi\,\mathrm{d}S
     \notag\\
    &\qquad-\int_\Omega \partial_t^h f_1(\varphi)\partial_t^h\varphi\,\mathrm{d}x
    -\int_\Gamma \partial_t^h g_1(\psi)\partial_t^h\psi\,\mathrm{d}S
    \notag\\
    &\quad\leq  - \int_\Omega\partial_t^h(\mathbf{v}\cdot\nabla\varphi)\mathfrak{S}^L_\Omega(\partial_t^h\boldsymbol{\varphi})\,\mathrm{d}x
     - \int_\Gamma\partial_t^h (\mathbf{v}_{\boldsymbol{\tau}}\cdot\nabla_{\!\boldsymbol{\tau}\,}\psi)\mathfrak{S}_\Gamma^L(\partial_t^h\boldsymbol{\varphi})\,\mathrm{d}S
    \notag\\
    &\qquad-\int_\Omega \partial_t^h f_1(\varphi)\partial_t^h\varphi\,\mathrm{d}x
    -\int_\Gamma \partial_t^h g_1(\psi)\partial_t^h\psi\,\mathrm{d}S.
    \label{0307-A-1}
\end{align}
Using the Sobolev embedding theorem, Agmon's inequality, the regularity estimate \eqref{Hk-regularity} for $k=0,1$, and Young's inequality, we estimate the first two integrals on the right-hand side of \eqref{0307-A-1} as follows:
\begin{align}
    &\int_\Omega\partial_t^h(\mathbf{v} \cdot\nabla\varphi)
    \mathfrak{S}^L_\Omega(\partial_t^h \boldsymbol{\varphi})\,\mathrm{d}x
   +\int_\Gamma\partial_t^h (\mathbf{v}_{\boldsymbol{\tau}}\cdot\nabla_{\!\boldsymbol{\tau}\,}\psi)
   \mathfrak{S}_\Gamma^L (\partial_t^h\boldsymbol{\varphi})\,\mathrm{d}S
   \notag\\
   &\quad=\int_\Omega \mathrm{div}(\partial_t^h(\varphi\mathbf{v})) \,\mathfrak{S}^L_\Omega(\partial_t^h\bm{\varphi})\,\mathrm{d}x
   +\int_\Gamma\mathrm{div}_{\Gamma}(\partial_t^h(\psi\mathbf{v}_{\bm{\tau}}))
   \,\mathfrak{S}^L_\Gamma (\partial_t^h\bm{\varphi})\,\mathrm{d}S
   \notag\\
   &\quad=-\int_\Omega\partial_t^h(\varphi\mathbf{v})
   \cdot\nabla\mathfrak{S}^L_\Omega (\partial_t^h\bm{\varphi})\,\mathrm{d}x
   -\int_\Gamma\partial_t^h (\psi\mathbf{v}_{\bm{\tau}})
   \cdot\nabla_{\!\bm{\tau}\,} \mathfrak{S}^L_\Gamma(\partial_t^h\bm{\varphi})\,\mathrm{d}S
   \notag\\
   &\quad=-\int_\Omega\varphi(\cdot+h) \partial_t^h\mathbf{v}
   \cdot\nabla\mathfrak{S}^L_\Omega (\partial_t^h\bm{\varphi})\,\mathrm{d}x
   -\int_\Omega\partial_t^h\varphi\,\mathbf{v}
   \cdot\nabla\mathfrak{S}^L_\Omega (\partial_t^h\bm{\varphi})\,\mathrm{d}x
   \notag\\
   &\qquad-\int_\Gamma\psi(\cdot+h)\partial_t^h \mathbf{v}_{\bm{\tau}}\cdot\nabla_{\!\bm{\tau}\,}
   \mathfrak{S}^L_\Gamma (\partial_t^h\bm{\varphi})\,\mathrm{d}S
   -\int_\Gamma\partial_t^h\psi\,\mathbf{v}_{\bm{\tau}}\cdot\nabla_{\!\bm{\tau}\,}
   \mathfrak{S}^L_\Gamma (\partial_t^h\bm{\varphi})\,\mathrm{d}S\notag\\
   &\quad\leq \|\varphi(\cdot+h)\|_{L^\infty(\Omega)} \|\partial_t^h\mathbf{v}\|_{\mathbf{L}^1(\Omega)}
   \|\nabla\mathfrak{S}^L_\Omega (\partial_t^h\bm{\varphi})\|_{\mathbf{L}^\infty(\Omega)}
   \notag\\
   &\qquad+\|\psi(\cdot+h)\|_{L^\infty(\Gamma)}
   \|\partial_t^h\mathbf{v}_{\bm{\tau}} \|_{\mathbf{L}^1(\Gamma)}
   \|\nabla_{\!\bm{\tau}\,} \mathfrak{S}^L_\Gamma(\partial_t^h\bm{\varphi})
   \|_{\mathbf{L}^\infty(\Gamma)}
   \notag\\
   &\qquad +\|\partial_t^h\varphi\|_H \|\mathbf{v}\|_{\mathbf{L}^3(\Omega)}
   \|\nabla\mathfrak{S}^L_\Omega (\partial_t^h\bm{\varphi})
   \|_{\mathbf{L}^6(\Omega)} +\|\partial_t^h\psi\|_{H_\Gamma}
   \|\mathbf{v}_{\bm{\tau}}\|_{\mathbf{L}^3(\Gamma)}
   \|\nabla_{\!\bm{\tau}\,} \mathfrak{S}^L_\Gamma(\partial_t^h\bm{\varphi})
   \|_{\mathbf{L}^6(\Gamma)}
   \notag\\
   &\quad\leq \|\partial_t^h\mathbf{v}\|_{\mathbf{L}^1(\Omega)}
   \|\nabla\mathfrak{S}^L_\Omega (\partial_t^h\bm{\varphi})\|_{\mathbf{H}^1(\Omega)}^{\frac{1}{2}}
   \|\nabla\mathfrak{S}^L_\Omega (\partial_t^h\bm{\varphi})\|_{\mathbf{H}^2(\Omega)}^{\frac{1}{2}}
   \notag\\
   &\qquad+\|\partial_t^h \mathbf{v}_{\bm{\tau}}\|_{\mathbf{L}^1(\Gamma)}
   \|\nabla_{\!\bm{\tau}\,} \mathfrak{S}^L_\Gamma(\partial_t^h\bm{\varphi})
   \|_{\mathbf{H}^1(\Gamma)}^\frac{1}{2}
   \|\nabla_{\!\bm{\tau}\,} \mathfrak{S}^L_\Gamma(\partial_t^h\bm{\varphi})
   \|_{\mathbf{H}^2(\Gamma)}^\frac{1}{2}
   \notag\\
   &\qquad+\|\partial_t^h\varphi\|_H \|\mathbf{v}\|_{\mathbf{L}^3(\Omega)}
   \|\nabla\mathfrak{S}^L_\Omega (\partial_t^h\bm{\varphi})\|_{\mathbf{L}^6(\Omega)}
   +\|\partial_t^h\psi\|_{H_\Gamma} \|\mathbf{v}_{\bm{\tau}}\|_{\mathbf{L}^3(\Gamma)}
   \|\nabla_{\!\bm{\tau}\,} \mathfrak{S}^L_\Gamma(\partial_t^h\bm{\varphi})
   \|_{\mathbf{L}^6(\Gamma)}
   \notag\\
   &\quad\leq C \|(\partial_t^h\mathbf{v},\partial_t^h \mathbf{v}_{\boldsymbol{\tau}})\|_{\boldsymbol{\mathcal{L}}^1}
   \|\partial_t^h\boldsymbol{\varphi} \|_{\mathcal{L}^2}^\frac{1}{2}
   \|\partial_t^h\boldsymbol{\varphi} \|_{\mathcal{H}^1}^\frac{1}{2}
   +C \|(\mathbf{v},\mathbf{v}_{\boldsymbol{\tau}})\|_{\boldsymbol{\mathcal{L}}^3}
   \|\partial_t^h\boldsymbol{\varphi} \|_{\mathcal{L}^2}^2
   \notag\\
   &\quad\leq \frac{1}{2}\Big(\|\nabla\partial_t^h\varphi\|_{\mathbf{L}^2(\Omega)}^2
   +\|\nabla_{\!\bm{\tau}\,} \partial_t^h\psi\|_{\mathbf{L}^2(\Gamma)}^2 +\frac{1}{K}\|\partial_t^h\psi-\partial_t^h\varphi\|_{H_\Gamma}^2\Big)
   \notag\\
   &\qquad+C \|(\partial_t^h\mathbf{v},\partial_t^h\mathbf{v}_{\boldsymbol{\tau}})\|_{\boldsymbol{\mathcal{L}}^1}^\frac{4}{3}\Big(1+\|\partial_t^h\bm{\varphi} \|_{\mathcal{L}^2}^2\Big)
   +C \|(\mathbf{v},\mathbf{v}_{\boldsymbol{\tau}})\|_{\boldsymbol{\mathcal{{L}}}^3}
   \|\partial_t^h\bm{\varphi}\|_{\mathcal{L}^2}^2.
   \notag
\end{align}
Besides, the last two integrals on the right-hand side of \eqref{0307-A-1} can simply be estimated by
\begin{align*}
  -\int_\Omega \partial_t^h f_1(\varphi)\partial_t^h\varphi\,\mathrm{d}x
    -\int_\Gamma \partial_t^h g_1(\psi)\partial_t^h\psi\,\mathrm{d}S
  \leq C_{\text{Lip}} \|\partial_t^h\bm{\varphi}\|_{\mathcal{L}^2}^2.
\end{align*}
With the aid of the above estimates, we apply Gronwall's lemma to \eqref{0307-A-1} and obtain
\begin{align}
    &\|\partial_t^h\bm{\varphi}(t)\|_{\mathcal{H}_{L,0}^{-1}}^2
    +\gamma\|\partial_t^h\bm{\varphi}(t)\|_{\mathcal{L}^2}^2
    \notag\\
    &\qquad+\int_0^t\Big(\|\nabla\partial_t^h \varphi(s)\|_{\mathbf{L}^2(\Omega)}^2
    +\|\nabla_{\!\bm{\tau}\,}\partial_t^h \psi(s)\|_{\mathbf{L}^2(\Gamma)}^2
    {+\frac{1}{K}\int_\Gamma|\partial_t^h\psi(s)-\partial_t^h\varphi(s)|^2\,\mathrm{d}S}\Big) \,\mathrm{d}s
    \notag\\
    &\quad\leq \left(\|\partial_t^h\bm{\varphi}(0)\|_{\mathcal{H}_{L,0}^{-1}}^2
    +\gamma\|\partial_t^h\bm{\varphi}(0)\|_{\mathcal{L}^2}^2
    +C\int_0^t\|(\partial_t^h\mathbf{v}(s),\partial_t^h\mathbf{v}_{\boldsymbol{\tau}}(s))\|_{\boldsymbol{\mathcal{L}}^1}^\frac{4}{3}\,\mathrm{d}s\right)
    \notag\\
    &\qquad \times
    \exp\left( C_\gamma\int_0^t\Big(\|(\partial_t^h\mathbf{v}(s),\partial_t^h\mathbf{v}_{\boldsymbol{\tau}}(s))\|_{\boldsymbol{\mathcal{L}}^1}^\frac{4}{3}
    +\|(\mathbf{v}(s),\mathbf{v}_{\boldsymbol{\tau}}(s))\|_{\boldsymbol{\mathcal{{L}}}^3}\Big)\,\mathrm{d}s \right)
    \label{0307-Gronwall}
\end{align}
for all $t\in[0,T]$ and some constant $C_\gamma>0$. It remains to control $\boldsymbol{\varphi}$-dependent terms on the right-hand side of \eqref{0307-Gronwall}. To this end, we use the compatibility condition \eqref{appendix-compatibility}, the properties of $\mathfrak{S}^L$, and the regularity estimate \eqref{Hk-regularity} for $k=0$ to get
\begin{align*}
    &\frac{1}{2}\frac{\mathrm{d}}{\mathrm{d}t}\Big(\|\bm{\varphi}-\bm{\varphi}_0\|_{\mathcal{H}_{L,0}^{-1}}^2
    +\gamma\|\bm{\varphi}-\bm{\varphi}_0\|_{\mathcal{L}^2}^2\Big)
    \\[1ex]
    &\quad=(\partial_t\bm{\varphi},\bm{\varphi}-\bm{\varphi}_0)_{\mathcal{H}_{L,0}^{-1}}
    +\gamma(\partial_t\bm{\varphi},\bm{\varphi}-\bm{\varphi}_0)_{\mathcal{L}^2}
    \\
    &\quad=\int_\Omega(\gamma\partial_t\varphi-\mu)(\varphi-\varphi_0)\,\mathrm{d}x
    +\int_\Gamma(\gamma\partial_t\psi-\mathcal{L})(\psi-\psi_0)\,\mathrm{d}S
    \\
    &\qquad-\int_\Omega(\nabla\varphi\cdot\mathbf{v})\mathfrak{S}_\Omega^L(\bm{\varphi}-\bm{\varphi}_0)\,\mathrm{d}x
    -\int_\Gamma(\nabla_{\!\bm{\tau}\,}\psi\cdot\mathbf{v}_{\bm{\tau}})\mathfrak{S}_\Gamma^L(\bm{\varphi}-\bm{\varphi}_0)\,\mathrm{d}S
    \\
    &\quad=\int_\Omega\big(\Delta\varphi-f_0(\varphi)-k^{-1}f_\sigma(\varphi)\big)(\varphi-\varphi_0)\,\mathrm{d}x
    \notag\\
    &\qquad+\int_\Gamma\big(\Delta_{\bm{\tau}}\psi-\partial_{\mathbf{n}}\varphi-g_0(\psi)-k^{-1}f_\sigma(\psi)\big)(\psi-\psi_0)\,\mathrm{d}S\\
    &\qquad-\int_\Omega f_1(\varphi)(\varphi-\varphi_0)\,\mathrm{d}x-\int_\Gamma g_1(\psi)(\psi-\psi_0)\,\mathrm{d}S
    \\
    &\qquad-\int_\Omega\mathrm{div}(\varphi\mathbf{v})\mathfrak{S}_\Omega^L(\bm{\varphi}-\bm{\varphi}_0)\,\mathrm{d}x
    -\int_\Gamma\mathrm{div}_{\Gamma}(\psi\mathbf{v}_{\bm{\tau}})\mathfrak{S}_\Gamma^L(\bm{\varphi}-\bm{\varphi}_0)\,\mathrm{d}S
    \\
    &\quad=\int_\Omega\big(\Delta(\varphi-\varphi_0)-(f_0(\varphi)-f_0(\varphi_0)\big)
    (\varphi-\varphi_0)\,\mathrm{d}x
    \\
    &\qquad+\int_\Gamma\big(\Delta_{\bm{\tau}}(\psi-\psi_0)-\partial_{\mathbf{n}}(\varphi-\varphi_0)
    -(g_0(\psi)-g_0(\psi_0))\big)(\psi-\psi_0)\,\mathrm{d}S
    \\
     &\qquad-k^{-1}\int_\Omega(f_\sigma(\varphi)-f_\sigma(\varphi_0))(\varphi-\varphi_0)\,\mathrm{d}x
     -k^{-1}\int_\Gamma (f_\sigma(\psi)-f_\sigma(\psi_0))(\psi-\psi_0)\,\mathrm{d}S
    \notag\\
    &\qquad+\int_\Omega \big(\Delta\varphi_0-f_0(\varphi_0)-k^{-1}f_\sigma(\varphi_0)\big)
    (\varphi-\varphi_0)\,\mathrm{d}x
    \notag\\
    &\qquad+\int_\Gamma \big(\Delta_{\bm{\tau}}\psi_0-\partial_{\mathbf{n}}\varphi_0
    -g_0(\psi_0)-k^{-1}f_\sigma(\psi_0)\big)
    (\psi-\psi_0)\,\mathrm{d}S
    \\
    &\qquad-\int_\Omega f_1(\varphi)(\varphi-\varphi_0)\,\mathrm{d}x
    -\int_\Gamma g_1(\psi)(\psi-\psi_0)\,\mathrm{d}S
    \\
    &\qquad+\int_\Omega\varphi\mathbf{v} \cdot\nabla\mathfrak{S}_\Omega^L(\bm{\varphi}-\bm{\varphi}_0)\,\mathrm{d}x
    +\int_\Gamma\psi\mathbf{v}_{\bm{\tau}}
    \cdot\nabla_{\!\bm{\tau}\,} \mathfrak{S}_\Gamma^L(\bm{\varphi}-\bm{\varphi}_0)\,\mathrm{d}S
    \\[1ex]
    &\quad\le - \| \nabla (\varphi - \varphi_0) \|_{\mathbf{L}^2(\Omega)}^2
        - \| \nabla_{\!\boldsymbol{\tau}\,}(\psi - \psi_0) \|_{\mathbf{L}^2(\Gamma)}^2
        -\frac{1}{K}\|(\varphi-\varphi_0)
        -(\psi-\psi_0)\|_{H_\Gamma}^2
    \\
    &\qquad - \int_\Omega \mu_0^*\,(\varphi-\varphi_0)\,\mathrm{d}x
    -\int_\Gamma \mathcal{L}_0^*\, (\psi-\psi_0)\,\mathrm{d}S
    \\
    &\qquad + C \|\bm{\varphi}-\bm{\varphi}_0\|_{\mathcal{L}^2}
        + \|(\mathbf{v},\mathbf{v}_{\boldsymbol{\tau}}) \|_{\boldsymbol{\mathcal{L}}^2}
            \|\mathfrak{S}^L(\bm{\varphi}-\bm{\varphi}_0)\|_{{\mathcal{H}}^1}
    \\[1ex]
    &\quad\leq C\Big(1+\|\mu_0^\ast\|_{H}+\|\mathcal{L}_0^\ast\|_{H_\Gamma}
    + \|(\mathbf{v},\mathbf{v}_{\boldsymbol{\tau}}) \|_{\boldsymbol{\mathcal{L}}^2}\Big)\|\bm{\varphi}-\bm{\varphi}_0\|_{\mathcal{L}^2}.
\end{align*}
Applying Lemma~\ref{LEM:SGW}, we find that
\begin{align*}
    &\|\bm{\varphi}{(t)}-\bm{\varphi}_0\|_{\mathcal{H}_{L,0}^{-1}}^2
    +\gamma\|\bm{\varphi}{(t)}-\bm{\varphi}_0\|_{\mathcal{L}^2}^2
    \\
    &\quad\leq C_\gamma(1+\|\mu_0^\ast\|_{H}+\|\mathcal{L}_0^\ast\|_{H_\Gamma}) \, t +C_\gamma\int_0^t\|(\mathbf{v}(s),\mathbf{v}_{\boldsymbol{\tau}}(s))\|_{\boldsymbol{\mathcal{L}}^2}\mathrm{d}s
\end{align*}
for all $t\in[0,T]$. Choosing $t=h$, we obtain
\begin{align}
    &\|\partial_t^h\bm{\varphi}(0)\|_{\mathcal{H}_{L,0}^{-1}}^2+\gamma\|\partial_t^h\bm{\varphi}(0)\|_{\mathcal{L}^2}^2
    \leq  C_\gamma\Big(1+\|\mu_0^\ast\|_{H}+\|\mathcal{L}_0^\ast\|_{H_\Gamma}
    +\|(\mathbf{v}, \mathbf{v}_{\boldsymbol{\tau}})\|_{L^\infty(0,T;\boldsymbol{\mathcal{L}}^2)}^2\Big).
    \label{0307-d-1}
\end{align}
Combining \eqref{0307-Gronwall} and \eqref{0307-d-1}, we can deduce that
\begin{align}
    &
    \gamma\|\partial_t^h\bm{\varphi}(t)\|_{\mathcal{L}^2}^2
    +\int_0^t\Big(\|\nabla\partial_t^h \varphi(s)\|_{\mathbf{L}^2(\Omega)}^2
    +\|\nabla_{\!\bm{\tau}\,} \partial_t^h\psi(s)\|_{\mathbf{L}^2(\Gamma)}^2+ \frac{1}{K}\|\partial_t^h\psi(s)-\partial_t^h\varphi(s)\|_{H_\Gamma}^2
    \Big)\,\mathrm{d}s
    \notag\\
    &\quad\leq \left( C_\gamma\big(1+\|\mu_0^\ast\|_{H}+\|\mathcal{L}_0^\ast\|_{H_\Gamma}
    +\| (\mathbf{v},\mathbf{v}_{\boldsymbol{\tau}})\|_{L^\infty(0,T;\boldsymbol{\mathcal{L}}^2)}^2\big)
    + C \int_0^t\| (\partial_t^h\mathbf{v}(s),\partial_t^h\mathbf{v}_{\boldsymbol{\tau}}(s))\|_{\boldsymbol{\mathcal{L}}^1}^\frac{4}{3}\,\mathrm{d}s
    \right)
    \notag\\
    &\qquad \times
    \exp\left( C_\gamma\int_0^t\Big(\| (\partial_t\mathbf{v}(s),\partial_t\mathbf{v}_{\boldsymbol{\tau}}(s))\|_{\boldsymbol{\mathcal{L}}^1}^\frac{4}{3}
    +\|(\mathbf{v}(s),\mathbf{v}_{\boldsymbol{\tau}}(s))\|_{\boldsymbol{\mathcal{{L}}}^3}\Big)\,\mathrm{d}s \right).
    \label{0307-Gronwall*}
\end{align}
Passing to the limit $h\to0$, we conclude that
\begin{align}
    \|\partial_t\bm{\varphi} \|_{L^\infty(0,T;\mathcal{L}^2)}
    +\|\partial_t\bm{\varphi} \|_{L^2(0,T;\mathcal{H}^1)}
    \leq C(\gamma,T,\|\mu_0^\ast\|_{H},\|\mathcal{L}_0^\ast\|_{H_\Gamma},\|(\mathbf{v},\mathbf{v}_{\boldsymbol{\tau}})\|_{X_T}),
    \label{0307-d-2}
\end{align}
where $X_T \coloneqq L^\infty(0,T;\boldsymbol{\mathcal{L}}^3)\cap W^{1,\frac{4}{3}}(0,T;\boldsymbol{\mathcal{L}}^1)$.

Next, we derive regularity properties on $\boldsymbol{\varphi}$ and $\boldsymbol{\mu}$.
By \eqref{convective-viscous}, the incompressibility properties of $\mathbf{v}$ and $\mathbf{v}_{\boldsymbol{\tau}}$, and Lemma \ref{generalized-Poin}, we obtain
\begin{align*}
    \|\mathbb{P}\boldsymbol{\mu} \|_{\mathcal{H}_{L,0}^1}^2&=\|\nabla\mu\|_{\mathbf{L}^2(\Omega)}^2+\|\nabla_{\!\boldsymbol{\tau}\,}\mathcal{L}\|_{\mathbf{L}^2(\Gamma)}^2+\chi(L)\|\mathcal{L}-\mu\|_{H_\Gamma}^2
    \\
    &=-\int_\Omega\partial_t\varphi \mu\,\mathrm{d}x-\int_\Omega(\mathbf{v}\cdot\nabla\varphi)\mu\,\mathrm{d}x-\int_\Gamma\partial_t\psi \mathcal{L}\,\mathrm{d}S-\int_\Gamma (\mathbf{v}_{\boldsymbol{\tau}}\cdot\nabla_{\!\boldsymbol{\tau}\,}\psi)\mathcal{L}\,\mathrm{d}S
    \\
    &=-(\partial_t\boldsymbol{\varphi},\boldsymbol{\mu}-\overline{m}(\boldsymbol{\mu})\boldsymbol{1})_{\mathcal{L}^2}+\int_\Omega (\mathbf{v}\cdot\nabla\mu)\varphi\,\mathrm{d}x+\int_\Gamma (\mathbf{v}_{\boldsymbol{\tau}}\cdot\nabla_{\!\boldsymbol{\tau}\,}\mathcal{L})\psi\,\mathrm{d}S
    \\
    &\leq \|\partial_t\boldsymbol{\varphi}\|_{\mathcal{L}^2}\|\boldsymbol{\mu}-\overline{m}(\boldsymbol{\mu})\boldsymbol{1}\|_{\mathcal{L}^2}+\|\mathbf{v}\|_{\mathbf{L}^2(\Omega)}\|\nabla\mu\|_{\mathbf{L}^2(\Omega)}\|\varphi\|_{L^\infty(\Omega)}+\|\mathbf{v}_{\boldsymbol{\tau}}\|_{\mathbf{L}^2(\Gamma)}\|\nabla_{\!\boldsymbol{\tau}\,}\mathcal{L}\|_{\mathbf{L}^2(\Gamma)}\|\psi\|_{L^\infty(\Gamma)}
    \\
    &\leq C\|\partial_t\boldsymbol{\varphi}\|_{\mathcal{L}^2}\|\mathbb{P}\boldsymbol{\mu}\|_{\mathcal{H}_{L,0}^1}+\|\mathbf{v}\|_{\mathbf{L}^2(\Omega)}\|\nabla\mu\|_{\mathbf{L}^2(\Omega)}+\|\mathbf{v}_{\boldsymbol{\tau}}\|_{\mathbf{L}^2(\Gamma)}\|\nabla_{\!\boldsymbol{\tau}\,}\mathcal{L}\|_{\mathbf{L}^2(\Gamma)}
    \\
    &\leq \frac{1}{2}\|\mathbb{P}\boldsymbol{\mu}\|_{\mathcal{H}_{L,0}^1}^2+C\big(\|\partial_t\boldsymbol{\varphi}\|_{\mathcal{L}^2}^2+\|(\mathbf{v},\mathbf{v}_{\boldsymbol{\tau}})\|_{\boldsymbol{\mathcal{L}}^2}^2\big),
\end{align*}
which implies that
\begin{align}
    \|\mathbb{P}(\mu,\mathcal{L})\|_{\mathcal{H}_{L,0}^1}
    \leq C(\|\partial_t\boldsymbol{\varphi}\|_{\mathcal{L}^2}+\|(\mathbf{v},\mathbf{v}_{\boldsymbol{\tau}})\|_{\boldsymbol{\mathcal{L}}^2}).
    \label{A-Pmu}
\end{align}
Similar to \eqref{260807-1} and \eqref{260807-2}, it holds
\begin{align}
    |\overline{m}(\mu,\mathcal{L})|\leq C(1+\|\mathbb{P}(\mu,\mathcal{L})\|_{\mathcal{H}_{L,0}^1}
    +\|\partial_t\boldsymbol{\varphi} \|_{\mathcal{L}^2}).
    \label{A-mu-mean}
\end{align}
Then, recalling the definition of the projection $\mathbb P$ (see \eqref{DEF:PRO:L0}), we infer from \eqref{0307-d-2}, \eqref{A-Pmu} and \eqref{A-mu-mean} that
\begin{align}
    \label{EST:MU:LINF}
    \|\boldsymbol{\mu}\|_{L^\infty(0,T; \mathcal{H}^1)}
    \leq C(\gamma,T,\|\mu_0^\ast\|_{H},\|\mathcal{L}_0^\ast\|_{H_\Gamma},\|(\mathbf{v},\mathbf{v}_{\boldsymbol{\tau}})\|_{X_T}).
\end{align}
For $K\in (0,+\infty)$, we reformulate the subsystem \eqref{convective-viscous:2}, \eqref{convective-viscous:3}, \eqref{convective-viscous:6} as
\begin{subequations}
\begin{alignat}{2}
    \label{ellip:1}
    -\Delta \varphi +k^{-1} f_\sigma(\varphi)
    &=\mu - \gamma\partial_t\varphi - f(\varphi )
    &&\quad \mbox{in }\ Q_{T},
    \\[1mm]
    \label{ellip:2}
    - \Delta _{\bm{\tau}}\psi +\partial_{\mathbf{n}}\varphi +k^{-1} f_\sigma(\psi)
    &=\mathcal{L} - \gamma\partial_t\psi - g(\psi)
    &&\quad \mbox{on }\ \Sigma _{T},
    \\[1mm]
    \label{ellip:3}
    K \partial_{\mathbf{n}}\varphi &=\psi-\varphi
    &&\quad {\mbox{on }\Sigma _{T }.}
\end{alignat}
\end{subequations}
This allows us to apply a regularity result for bulk-surface elliptic systems with singular nonlinearities (see \cite[Propositions~5.1 and 5.2]{GKS}) to infer that $\boldsymbol{\varphi}\in L^\infty(0,T;\mathcal{H}^2)$ with
\begin{align}
    \label{EST:PHI:LINFH2*}
    \|\boldsymbol{\varphi} \|_{L^\infty(0,T;\mathcal{H}^2)}
    &\le C\big( 1
        + \|\boldsymbol{\mu} \|_{L^\infty(0,T;\mathcal{L}^2)}
        + \gamma\, \| \partial_t \boldsymbol{\varphi}\|_{L^\infty(0,T;\mathcal{L}^2)}
        + \|(f(\varphi),g(\psi)) \|_{L^\infty(0,T;\mathcal{L}^2)}
    \big).
\end{align}
Recalling that $|\varphi|<1-\sigma$ a.e.~in $Q_T$ and $|\psi|<1-\sigma$ a.e.~on $\Sigma_T$, we have
\begin{align}
    \label{EST:FG:LINF}
    \big| \big( f(\varphi) , g(\psi) \big) \big|
    \le \underset{r\in [-1+\sigma,1-\sigma]}{\max}\,\big| \big( f(r) , g(r) \big) \big|
    \qquad\text{a.e.~on $\Omega\times\Gamma$.}
\end{align}
In combination with \eqref{0307-d-2}, \eqref{EST:MU:LINF} and \eqref{EST:PHI:LINFH2*}, we conclude that
\begin{align*}
    \|\boldsymbol{\varphi} \|_{L^\infty(0,T;\mathcal{H}^2)}
    &\le C(\sigma,\gamma,T,\|\mu_0^\ast\|_{H},\|\mathcal{L}_0^\ast\|_{H_\Gamma},\|(\mathbf{v},\mathbf{v}_{\boldsymbol{\tau}})\|_{X_T}).
\end{align*}
Next, we reformulate the subsystem \eqref{convective-viscous:1}, \eqref{convective-viscous:4}, \eqref{convective-viscous:5} as
\begin{subequations}
\label{ellip*}
\begin{alignat}{2}
    \label{ellip*:1}
    - \Delta\mu &= - \partial_{t}\varphi - \mathbf{v}\cdot\nabla\varphi
    &&\quad  \mbox{in }\ Q_{T},
    \\[1mm]
    \label{ellip*:2}
    - \Delta_{\bm{\tau}}\mathcal{L} + \partial_{\mathbf{n}}\mu
    &= - \partial_t\psi - \mathbf{v}_{\bm{\tau}}\cdot\nabla_{\!\bm{\tau}\,}\psi
    &&\quad \mbox{on }\ \Sigma _{T},
    \\[1mm]
    \label{ellip*:3}
    L\partial_{\mathbf{n}}\mu
    &=\mathcal{L}-\mu 
    &&\quad\mbox{on }\ \Sigma _{T}.
\end{alignat}
\end{subequations}
Applying the elliptic regularity theory for bulk-surface elliptic systems (see, e.g., \cite[Theorem~3.3]{KL}), we deduce that $\boldsymbol{\mu} \in L^\infty(0,T;\mathcal{H}^2)$ with
\begin{align}
    \label{EST:MU:LINFH2*}
    \|\boldsymbol{\mu} \|_{L^\infty(0,T;\mathcal{H}^2)}
    &\le C\big( 1
        + \|\partial_t\boldsymbol{\varphi} \|_{L^\infty(0,T;\mathcal{L}^2)}
        + \|(\mathbf{v},\mathbf{v}_{\boldsymbol{\tau}})\|_{L^\infty(0,T;\boldsymbol{\mathcal{L}}^3)}
        \|\boldsymbol{\varphi} \|_{L^\infty(0,T;\mathcal{L}^6)}
    \big).
\end{align}
Using \eqref{0307-d-2}, we can conclude that
\begin{align}
    \label{EST:MU:LINFH2}
    \|\boldsymbol{\mu}\|_{L^\infty(0,T;\mathcal{H}^2)}
    &\leq C(\sigma,\gamma,T,\|\mu_0^\ast\|_{H},\|\mathcal{L}_0^\ast\|_{H_\Gamma},\|(\mathbf{v},\mathbf{v}_{\boldsymbol{\tau}})\|_{X_T}).
\end{align}

\smallskip
\noindent\textbf{\itshape Regularity \ref{R2}.}
We employ a strategy that has been applied to the viscous (bulk) Cahn--Hilliard equation with classical Neumann boundary conditions in \cite[Proposition A.3]{MZ04} and \cite[Theorem A.1]{Gior22}.
To this end, we rewrite the subsystem \eqref{convective-viscous:2}, \eqref{convective-viscous:3}, \eqref{convective-viscous:6} as
\begin{subequations}
\label{re-1}
\begin{alignat}{2}
    \gamma\partial_t\varphi-\Delta\varphi+k^{-1} f_\sigma(\varphi)
    &=h
    &&\quad\text{a.e.~in }Q_T,
    \\
    K\partial_\mathbf{n}\varphi
    &=\psi-\varphi
    &&\quad\text{a.e.~on }\Sigma_T,
    \\
    \gamma\partial_t \psi-\Delta_{\bm{\tau}}\psi+\partial_{\mathbf{n}}\varphi+k^{-1} f_\sigma(\psi)
    &=h_\Gamma
    &&\quad\text{a.e.~on }\Sigma_T,
\end{alignat}
\end{subequations}
where
\begin{align*}
h\coloneqq \mu-f(\varphi),\quad h_\Gamma\coloneqq \mathcal{L}-g(\psi).
\end{align*}
{Due to \eqref{EST:FG:LINF}, \eqref{EST:MU:LINFH2} and the continuous embedding $\mathcal{H}^2\hookrightarrow \mathcal{L}^\infty$, we know} that $(h,h_\Gamma)\in L^\infty(0,T;\mathcal{L}^\infty)$.
Next, consider the ODE problems
\begin{align}
    \begin{cases}
        \gamma\partial_t U+k^{-1} f_\sigma(U)= H ,\\
        U(0)=(1-\sigma)(1-\delta_0),
    \end{cases}\quad\quad	
    \begin{cases}
    \gamma\partial_t V+k^{-1}f_\sigma(V)= -H,\\
    V(0)=(1-\sigma)(-1+\delta_0),
    \end{cases}\label{ODE}
\end{align}
where
\begin{align*}
    H\coloneqq \max\{\|h\|_{L^\infty(\Omega)},\|h_\Gamma\|_{L^\infty(\Gamma)}\}.
\end{align*}
We note that due to \eqref{EST:FG:LINF} and \eqref{EST:MU:LINFH2}, $H$ may depend on $\gamma$ and $\sigma$. It is straightforward to check that there exist two unique solutions $U,V\in C([0,T])$
with $\partial_t U$, $\partial_t V\in L^\infty(0,T)$ to the two ODE systems in \eqref{ODE}, respectively.
Define
\begin{align*}
    \delta &\coloneqq 1 - \max\big\{ 1-\delta_0,f_0^{-1}(k\|H\|_{L^\infty(0,T)}),-f_0^{-1}(-k\|H\|_{L^\infty(0,T)})\big\},
    \\
    M &\coloneqq (1-\sigma) (1-\delta).
\end{align*}
Since $\delta_0>0$ and $f_0^{-1}(\mathbb R) = (-1,1)$, it follows that $\delta>0$.
It is easy to check that $-M$ is a subsolution for the ODE involving $V$, and $M$ is a supersolution for the ODE involving $U$. Moreover, since $0<\delta \le \delta_0$, we further have
$-M \le V(0) \le U(0) \le M$.
Hence, a simple comparison argument entails that
\begin{align*}
-M \leq  V(t)\leq U(t)\leq M
\quad\text{for all $t\in[0,T]$}.
\end{align*}
Next, we show that
\begin{align*}
    V(t)\leq \varphi(x,t)\leq U(t)\quad \text{for all $(x,t)\in Q_T$},
    \qquad
    V(t)\leq \psi(x,t)\leq U(t)\quad \text{for all $(x,t)\in \Sigma_T$}.
\end{align*}
Let us define the pair $(w,w_\Gamma) \coloneqq (\varphi-U,\psi-U)$ that satisfies
\begin{subequations}
\label{0303-12}
\begin{alignat}{2}
    \label{0303-12:1}
    &\gamma\partial_t w-\Delta\varphi+k^{-1} f_\sigma(\varphi)-k^{-1} f_\sigma(U)
    =h-H
    &&\quad \text{a.e.~in }Q_T,
    \\
    \label{0303-12:2}
    &K\partial_\mathbf{n}w
    =w_\Gamma-w,\quad K\in(0,+\infty)
    &&\quad \text{a.e.~on }\Sigma_T,
    \\
    \label{0303-12:3}
    &\gamma\partial_t w_\Gamma-\Delta_{\bm{\tau}}\psi+\partial_{\mathbf{n}}\varphi+k^{-1} f_\sigma(\psi)-k^{-1} f_\sigma(U)
    =h_\Gamma-H
    &&\quad \text{a.e.~on }\Sigma_T,
    \\
    \label{0303-12:4}
    &w(0)=\varphi_0-(1-\sigma)(1-\delta_0)
    &&\quad \text{a.e.~in }\Omega,
    \\
     & w_\Gamma(0)
    =\psi_0-(1-\sigma)(1-\delta_0)
    &&\quad \text{a.e.~on }\Gamma.
\end{alignat}
\end{subequations}
Multiplying \eqref{0303-12:1} by $w^+=\max\{\varphi-U,0\}$ and integrating over $\Omega$,
multiplying \eqref{0303-12:3} by $w_\Gamma^+=\max\{\psi-U,0\}$ and integrating over $\Gamma$,
using the following facts
\begin{alignat*}{2}
\nabla\varphi
&=\nabla w^+
&&\quad\text{a.e.~in }\{x\in\Omega:\, \varphi\geq U\} \times (0,T),
\\
\nabla_{\!\bm{\tau}\,}\psi
&=\nabla_{\!\bm{\tau}\,}w_\Gamma^+
&&\quad\text{a.e.~on }\{x\in\Gamma:\,\psi\geq U\} \times (0,T),
\\
{\partial_\mathbf{n} \varphi}
&{= \tfrac 1K (\psi - \varphi) = \tfrac 1K (w_\Gamma - w) }
&&\quad\text{a.e.~on }\Sigma_T,
\end{alignat*}
we find that
\begin{align}
    &\frac{\gamma}{2}\frac{\mathrm{d}}{\mathrm{d}t}\Big(\|w^+\|_{H}^2+\|w_\Gamma^+\|_{H_\Gamma}^2\Big)
    +\|\nabla w^+\|_{\mathbf{L}^2(\Omega)}^2+\|\nabla_{\!\bm{\tau}\,}w_\Gamma^+\|_{\mathbf{L}^2(\Gamma)}^2
    + \underbrace{\int_\Gamma \partial_{\mathbf{n}}\varphi(w^+_\Gamma - w^+)\,\mathrm{d}S}_{\geq 0}
    \notag\\
    &\qquad+k^{-1}\int_\Omega(f_\sigma(\varphi)-f_\sigma(U))w^+\,\mathrm{d}x
    +k^{-1}\int_\Gamma (f_\sigma(\psi)-f_\sigma(U))w_\Gamma^+\,\mathrm{d}S
    \notag\\[1ex]
    &\quad \le \int_\Omega(h-H)w^+\,\mathrm{d}x
    +\int_\Gamma(h_\Gamma-H)w_\Gamma^+\,\mathrm{d}S.\label{0303-13}
\end{align}
By the definition of $H$, we have $h-H \le 0$ a.e.~in $Q_T$ and $h_\Gamma - H \le 0$ a.e.~on $\Sigma_T$.
Thus, due to the monotonicity of $f_\sigma$, we get
\begin{align*}
\frac{\mathrm{d}}{\mathrm{d}t}\Big(\|w^+\|_{H}^2+\|w_\Gamma^+\|_{H_\Gamma}^2\Big)\leq 0 .
\end{align*}
This implies that
\begin{align*}
	\|w^+(t)\|_{H}^2+\|w_\Gamma^+(t)\|_{H_\Gamma}^2
    \leq \|w^+(0)\|_{H}^2+\|w_\Gamma^+(0)\|_{H_\Gamma}^2=0
    \quad\text{for all $t\in[0,T]$}.
\end{align*}
Hence, we obtain
\begin{align*}
\varphi(x,t)\leq U(t) \leq M\;\; \text{in }Q_T,\quad
\psi(x,t)\leq U(t) \leq M\;\; \text{on }\Sigma_T.
\end{align*}
Similarly, we can conclude
\begin{align*}
\varphi(x,t)\geq V(t) \geq -M \;\; \text{in }Q_T,\quad
\psi(x,t)\geq V(t) \geq -M \;\; \text{on }\Sigma_T.
\end{align*}
This verifies the strict separation property \eqref{viscous-separation}.
As a consequence, we have
\begin{align*}
(k^{-1}f_\sigma(\varphi),k^{-1} f_\sigma(\psi))\in L^\infty(0,T;\mathcal{H}^1).
\end{align*}
Applying the elliptic regularity theory for bulk-surface elliptic systems (see, e.g., \cite[Theorem~3.3]{KL}) to system \eqref{ellip*} (or \eqref{re-1}, respectively), we finally arrive at $(\varphi,\psi)\in L^2(0,T;\mathcal{H}^3)$.

\smallskip
	
\noindent\textbf{\itshape Regularity \ref{R3}.}

We use the time difference method again. Testing \eqref{difference-quotient:1} by $\partial_t^h\mu$
and \eqref{difference-quotient:5} by $\partial_t^h\mathcal{L}$, we get
\begin{align}
	&\|\nabla\partial_t^h\mu \|_{\mathbf{L}^2(\Omega)}^2
    +\|\nabla_{\!\bm{\tau}\,} \partial_t^h\mathcal{L}\|_{\mathbf{L}^2(\Gamma)}^2
    +\frac{1}{L}\|\partial_t^h\mu-\partial_t^h\mathcal{L} \|_{H_\Gamma}^2
    \notag\\
	&\quad=-\int_\Omega \partial_t\partial_t^h\varphi\, \partial_t^h\mu\,\mathrm{d}x
    -\int_\Gamma \partial_t\partial_t^h\psi\,\partial_t^h\mathcal{L}\,\mathrm{d}S
    -\int_\Omega\partial_t^h (\mathbf{v}\cdot\nabla\varphi)\,\partial_t^h\mu\,\mathrm{d}x
    -\int_\Gamma \partial_t^h(\mathbf{v}_{\bm{\tau}}\cdot\nabla_{\!\bm{\tau}\,}\psi)\,\partial_t^h\mathcal{L} \,\mathrm{d}S
    \notag\\
	&\quad=-\gamma\|\partial_t\partial_t^h \varphi\|_{H}^2
    -\gamma\|\partial_t\partial_t^h \psi\|_{H_\Gamma}^2
    -\frac{1}{2}\frac{\mathrm{d}}{\mathrm{d}t}\Big(\|\nabla\partial_t^h \varphi\|_{\mathbf{L}^2(\Omega)}^2
    +\|\nabla_{\!\bm{\tau}\,} \partial_t^h\psi\|_{\mathbf{L}^2(\Gamma)}^2
    +\frac{1}{K}\int_\Gamma|\partial_t^h\psi-\partial_t^h\varphi|^2\,\mathrm{d}S\Big)
    \notag\\
    &\qquad-k^{-1}\int_\Omega \partial_t^h f_\sigma(\varphi)\partial_t\partial_t^h \varphi\,\mathrm{d}x
    -k^{-1}\int_\Gamma \partial_t^h f_\sigma(\psi)\partial_t\partial_t^h \psi\,\mathrm{d}S
    \notag\\
	&\qquad
    -\int_\Omega\partial_t^h f(\varphi)\partial_t\partial_t^h \varphi\,\mathrm{d}x
    -\int_\Gamma\partial_t^h g(\psi)\partial_t\partial_t^h\psi \,\mathrm{d}S
    \notag\\
	&\qquad-\int_\Omega\partial_t^h (\mathbf{v}\cdot\nabla\varphi)\,\partial_t^h\mu\,\mathrm{d}x
    -\int_\Gamma \partial_t^h(\mathbf{v}_{\bm{\tau}} \cdot\nabla_{\!\bm{\tau}\,}\psi)\,\partial_t^h\mathcal{L}\,\mathrm{d}S.
    \notag
\end{align}
This implies that
\begin{align}
		&\frac{1}{2}\frac{\mathrm{d}}{\mathrm{d}t}\Big(\|\nabla\partial_t^h\varphi\|_{\mathbf{L}^2(\Omega)}^2
        +\|\nabla_{\!\bm{\tau}\,} \partial_t^h\psi\|_{\mathbf{L}^2(\Gamma)}^2
        {+\frac{1}{K}\int_\Gamma|\partial_t^h\psi-\partial_t^h\varphi|^2\,\mathrm{d}S}\Big)
        \notag\\ &\qquad+\gamma\|\partial_t\partial_t^h \varphi\|_{H}^2+\gamma\|\partial_t\partial_t^h\psi\|_{H_\Gamma}^2
        +\|\nabla\partial_t^h\mu \|_{\mathbf{L}^2(\Omega)}^2+\|\nabla_{\!\bm{\tau}\,}\partial_t^h\mathcal{L}\|_{\mathbf{L}^2(\Gamma)}^2
        +\frac{1}{L}\|\partial_t^h\mu-\partial_t^h\mathcal{L} \|_{H_\Gamma}^2
        \notag\\
		&\quad=-k^{-1}\int_\Omega \partial_t^h f_\sigma(\varphi)\partial_t\partial_t^h\varphi\,\mathrm{d}x
        -k^{-1}\int_\Gamma \partial_t^h f_\sigma(\psi)\partial_t\partial_t^h\psi\,\mathrm{d}S
        \notag\\
        &\qquad-\int_\Omega\partial_t^h f(\varphi)\partial_t\partial_t^h\varphi\,\mathrm{d}x
        -\int_\Gamma\partial_t^h g(\psi)\partial_t\partial_t^h\psi\,\mathrm{d}S
        \notag\\
		&\qquad-\int_\Omega\partial_t^h(\mathbf{v}\cdot\nabla\varphi)\,\partial_t^h\mu\,\mathrm{d}x
        -\int_\Gamma \partial_t^h(\mathbf{v}_{\bm{\tau}}\cdot\nabla_{\!\bm{\tau}\,}\psi)\,\partial_t^h\mathcal{L}\,\mathrm{d}S
        \notag\\
        &\quad \eqqcolon \sum_{j=1}^{6}Q_j.
        \label{0303-2}
\end{align}
Testing \eqref{difference-quotient:2} by $\Delta\partial_t^h\varphi$
and \eqref{difference-quotient:6} by $\Delta_{\bm{\tau}}\partial_t^h\psi$, we obtain
\begin{align}
		&\frac{\gamma}{2}\frac{\mathrm{d}}{\mathrm{d}t}\Big(\|\nabla\partial_t^h\varphi\|_{\mathbf{L}^2(\Omega)}^2
        +\|\nabla_{\!\bm{\tau}\,} \partial_t^h\psi\|_{\mathbf{L}^2(\Gamma)}^2+\frac{1}{K}\|\partial_t^h\psi-\partial_t^h\varphi\|_{H_\Gamma}^2\Big)
        +\|\Delta\partial_t^h\varphi\|_{H}^2 +\|\Delta_{\bm{\tau}}\partial_t^h\psi\|_{H_\Gamma}^2
        \notag\\
		&\quad=\int_\Omega \nabla\partial_t^h\mu\cdot\nabla\partial_t^h\varphi\,\mathrm{d}x
        +\int_\Gamma \nabla_{\!\bm{\tau}\,}\partial_t^h\mathcal{L}\cdot\nabla_{\!\bm{\tau}\,}\partial_t^h\psi\,\mathrm{d}S
        +k^{-1}\int_\Omega \partial_t^h f_\sigma(\varphi)\Delta\partial_t^h\varphi\,\mathrm{d}x
        \notag\\
        &\qquad+k^{-1}\int_\Gamma \partial_t^h f_\sigma(\psi)\Delta_{\boldsymbol{\tau}}\partial_t^h\psi\,\mathrm{d}S
        +\int_\Omega \partial_t^h f(\varphi)\Delta\partial_t^h \varphi\,\mathrm{d}x
        +\int_\Gamma\partial_t^h g(\psi)\Delta_{\bm{\tau}}\partial_t^h\psi\,\mathrm{d}S
        \notag\\
        &\qquad-\int_\Gamma \partial_{\mathbf{n}}\partial_t^h\varphi\,\partial_t^h\mu\,\mathrm{d}S
        +\int_\Gamma \partial_{\mathbf{n}}\partial_t^h \varphi\Delta_{\boldsymbol{\tau}}\partial_t^h\psi\,\mathrm{d}S
        +\gamma\int_\Gamma\partial_\mathbf{n} \partial_t^h\varphi\partial_t\partial_t^h\psi\,\mathrm{d}S
        \notag\\
        &\quad {\eqqcolon} \sum_{j=7}^{15}Q_j.
        \label{0303-3}
\end{align}
Combining \eqref{0303-2} and \eqref{0303-3}, we deduce that
\begin{align}
    &\frac{1+\gamma}{2}\frac{\mathrm{d}}{\mathrm{d}t}\Big(\|\nabla\partial_t^h\varphi\|_{\mathbf{L}^2(\Omega)}^2
    +\|\nabla_{\!\bm{\tau}\,}\partial_t^h\psi\|_{\mathbf{L}^2(\Gamma)}^2
    {+\frac{1}{K}\int_\Gamma|\partial_t^h\psi -\partial_t^h\varphi|^2\,\mathrm{d}S}\Big)
    \notag\\[1ex]
    &\quad+\gamma\|\partial_t\partial_t^h \varphi\|_{H}^2
    +\gamma\|\partial_t\partial_t^h \psi\|_{H_\Gamma}^2
    +\|\Delta\partial_t^h\varphi\|_{H}^2
    +\|\Delta_{\bm{\tau}}\partial_t^h \psi\|_{H_\Gamma}^2
    \notag\\
    &\quad+\|\nabla\partial_t^h\mu \|_{\mathbf{L}^2(\Omega)}^2
    +\|\nabla_{\!\bm{\tau}\,} \partial_t^h\mathcal{L}\|_{\mathbf{L}^2(\Gamma)}^2
    +\frac{1}{L}\|\partial_t^h\mu-\partial_t^h\mathcal{L}\|_{H_\Gamma}^2
    =\sum_{j=1}^{15}Q_j.
    \label{0303-4}
\end{align}
It remains to bound the terms $Q_j$ for $1\leq j\leq 15$.
For the first twelve terms we can use H\"older's inequality to derive the following estimates
\begin{align*}
        \sum_{j=1}^4 Q_j&\leq \frac{\gamma}{4}\|\partial_t\partial_t^h\boldsymbol{\varphi}\|_{\mathcal{L}^2}^2
        +\frac{4}{\gamma}\Big(\|\partial_t^h f(\varphi)\|_H^2+\|\partial_t^h g(\psi)\|_{H_\Gamma}^2
        +k^{-2}\|\partial_t^h f_\sigma(\varphi)\|_H^2+k^{-2}\|\partial_t^h f_\sigma(\psi)\|_{H_\Gamma}^2\Big),
        \\
        \sum_{j=7}^8 Q_j&\leq \frac{1}{8}\|\nabla\partial_t^h\mu\|_{\mathbf{L}^2(\Omega)}^2
        +\frac{1}{8}\|\nabla_{\!\bm{\tau}\,}\partial_t^h\mathcal{L}\|_{\mathbf{L}^2(\Gamma)}^2
        +2\|\nabla\partial_t^h\varphi \|_{\mathbf{L}^2(\Omega)}^2
        +2\|\nabla_{\!\bm{\tau}\,} \partial_t^h\psi\|_{\mathbf{L}^2(\Gamma)}^2,\\
        \sum_{j=9}^{12}Q_j&\leq \frac{1}{4}\|\Delta\partial_t^h\varphi\|_{H}^2
        +\frac{1}{4}\|\Delta_{\boldsymbol{\tau}}\partial_t\psi\|_{H_\Gamma}^2
        \\
        &\quad+2\Big(\|\partial_t ^hf(\varphi)\|_{H}^2
        +\|\partial_t^h g(\psi)\|_{H_\Gamma}^2
        +k^{-2}\|\partial_t^h f_\sigma(\varphi)\|_{H}^2
        +k^{-2}\|\partial_t^h f_\sigma(\psi)\|_{H_\Gamma}^2\Big).
\end{align*}
For the remaining terms, we apply H\"older's inequality and the Sobolev's embedding theorem to obtain
\begin{align*}
		Q_5&=\int_\Omega \partial_t^h(\mathbf{v}\varphi)\cdot\nabla\partial_t^h\mu\,\mathrm{d}x
        \\
		&\leq \frac{1}{8}\|\nabla\partial_t^h\mu\|_{\mathbf{L}^2(\Omega)}^2
        +2\|\partial_t^h(\mathbf{v}\varphi) \|_{\mathbf{L}^2(\Omega)}^2
        \\
		&\leq\frac{1}{8}\|\nabla\partial_t^h\mu\|_{\mathbf{L}^2(\Omega)}^2
        +4\|\mathbf{v}(\cdot+h)\partial_t^h \varphi\|_{\mathbf{L}^2(\Omega)}^2
        +4\|\partial_t^h\mathbf{v}\, \varphi\|_{\mathbf{L}^2(\Omega)}^2
        \\
		&\leq \frac{1}{8}\|\nabla\partial_t^h\mu\|_{\mathbf{L}^2(\Omega)}^2
        +4\|\mathbf{v}(\cdot+h)\|_{\mathbf{L}^3(\Omega)}^2 \|\partial_t^h\varphi\|_{L^6(\Omega)}^2
        +4\|\partial_t^h\mathbf{v} \|_{\mathbf{L}^2(\Omega)}^2\|\varphi\|_{L^\infty(\Omega)}^2,
        \\[1ex]
		Q_6&=\int_\Gamma \partial_t^h(\mathbf{v}_{\bm{\tau}}\psi)\cdot\nabla_{\!\bm{\tau}\,}\partial_t^h\mathcal{L}\,\mathrm{d}S
        \\
		&\leq \frac{1}{8}\|\nabla\partial_t^h\mathcal{L}\|_{\mathbf{L}^2(\Gamma)}^2
        +2\|\partial_t^h(\mathbf{v}_{\bm{\tau}} \psi)\|_{\mathbf{L}^2(\Gamma)}^2
        \\
		&\leq\frac{1}{8}\|\nabla\partial_t^h\mathcal{L}\|_{\mathbf{L}^2(\Gamma)}^2
        +4\|\mathbf{v}_{\bm{\tau}}(\cdot+h)\partial_t^h\psi\|_{\mathbf{L}^2(\Gamma)}^2
        +4\|\partial_t^h\mathbf{v}_{\bm{\tau}} \,\psi\|_{\mathbf{L}^2(\Gamma)}^2
        \\
		&\leq \frac{1}{8}\|\nabla\partial_t^h\mathcal{L}\|_{\mathbf{L}^2(\Gamma)}^2
        +4\|\mathbf{v}_{\bm{\tau}}(\cdot+h)\|_{\mathbf{L}^3(\Gamma)}^2 \|\partial_t^h\psi\|_{L^6(\Gamma)}^2
        +4\|\partial_t^h\mathbf{v}_{\bm{\tau}} \|_{\mathbf{L}^2(\Gamma)}^2\|\psi\|_{L^\infty(\Gamma)}^2,
        \\[1ex]
		Q_{13}&\leq\|\partial_{\mathbf{n}} \partial_t^h\varphi\|_{H_\Gamma}\|\partial_t^h\mu\|_{H_\Gamma}
        \leq C\|\partial_t^h\varphi\|_{H^\frac{7}{4}(\Omega)}\|\partial_t^h\mu\|_{V}\\
		&\leq \frac{1}{8}\Big(|\overline{m}(\partial_t^h\boldsymbol{\mu})|^2 +\|\nabla\partial_t^h\mu\|_{\mathbf{L}^2(\Omega)}^2
        +\|\nabla_{\!\bm{\tau}\,} \partial_t^h\mathcal{L}\|_{\mathbf{L}^2(\Gamma)}^2
        +\frac{1}{L}\|\partial_t^h \mu-\partial_t^h\mathcal{L}\|_{H_\Gamma}^2\Big) +
		C\|\partial_t^h\varphi\|_{H^\frac{7}{4}(\Omega)}^2
        \\
		&\leq \frac{1}{8}\Big(|\overline{m}(\partial_t^h\boldsymbol{\mu})|^2+\|\nabla\partial_t^h\mu \|_{\mathbf{L}^2(\Omega)}^2
        +\|\nabla_{\!\bm{\tau}\,}\partial_t^h \mathcal{L}\|_{\mathbf{L}^2(\Gamma)}^2
        +\frac{1}{L}\|\partial_t^h\mu-\partial_t^h\mathcal{L}\|_{H_\Gamma}^2\Big)
        \\
		&\quad+\frac{1}{16}(\|\Delta\partial_t^h\varphi\|_{H}^2+\|\Delta_{\bm{\tau}}\partial_t^h\psi \|_{H_\Gamma}^2)+C(\|\partial_t^h\varphi\|_{H}^2+\|\partial_t^h\psi\|_{H_\Gamma}^2),
        \\[1ex]
		Q_{14}&\leq \frac{1}{16}\|\Delta_{\bm{\tau}}\partial_t^h\psi\|_{H_\Gamma}^2
        +\|\partial_{\mathbf{n}}\partial_t^h \varphi\|_{H_\Gamma}^2
        \\
		&\leq \frac{1}{8}(\|\Delta\partial_t^h\varphi\|_{H}^2+\|\Delta_{\bm{\tau}}\partial_t^h\psi \|_{H_\Gamma}^2)
        +C(\|\partial_t^h\varphi\|_{H}^2
        +\|\partial_t^h\psi\|_{H_\Gamma}^2),
        \\[1ex]
        Q_{15}&=\gamma\int_\Gamma \partial_\mathbf{n}\partial_t^h\varphi\partial_t\partial_t^h\psi\,\mathrm{d}S
        \leq \frac{\gamma}{8}\|\partial_t\partial_t^h \psi\|_{H_\Gamma}^2
        +C\|\partial_t^h \varphi\|_{H^\frac{7}{4}(\Omega)}^2
        \\
        &\leq  \frac{\gamma}{8}\|\partial_t\partial_t^h\psi\|_{H_\Gamma}^2
        +\frac{1}{16}(\|\Delta\partial_t^h\varphi\|_{H}^2+\|\Delta_{\bm{\tau}}\partial_t^h\psi \|_{H_\Gamma}^2)+C(\|\partial_t^h\varphi\|_{H}^2+\|\partial_t^h\psi\|_{H_\Gamma}^2).
\end{align*}
Combining the above estimates, from \eqref{0303-4}, we infer that
\begin{align}
&\frac{\mathrm{d}}{\mathrm{d}t}\Big((1+\gamma)\|\nabla\partial_t^h\varphi\|_{\mathbf{L}^2(\Omega)}^2
        +(1+\gamma)\|\nabla_{\!\bm{\tau}\,} \partial_t^h\psi\|_{\mathbf{L}^2(\Gamma)}^2
        {+\frac{1+\gamma}{K}\int_\Gamma|\partial_t^h\psi-\partial_t^h\varphi|^2\,\mathrm{d}S}\Big)
        \notag\\
		&\qquad+\gamma\|\partial_t\partial_t^h \varphi\|_{H}^2+\gamma\|\partial_t\partial_t^h\psi\|_{H_\Gamma}^2
        +\|\Delta\partial_t^h\varphi\|_{H}^2 +\|\Delta_{\bm{\tau}}\partial_t^h\psi\|_{H_\Gamma}^2
        \notag\\[1mm]
		&\qquad+\|\nabla\partial_t^h \mu\|_{\mathbf{L}^2(\Omega)}^2
        +\|\nabla_{\!\bm{\tau}\,}\partial_t^h \mathcal{L}\|_{\mathbf{L}^2(\Gamma)}^2
        +\frac{1}{L}\|\partial_t^h\mu-\partial_t^h\mathcal{L}\|_{H_\Gamma}^2
        \notag\\[1mm]
		&\quad\leq C\Big(1+\|(\mathbf{v}(\cdot+h),\mathbf{v}_{\boldsymbol{\tau}}(\cdot+h))\|_{\boldsymbol{\mathcal{L}}^3}^2
       \Big)
        \Big(\|\nabla\partial_t^h \varphi\|_{\mathbf{L}^2(\Omega)}^2
        +\|\nabla_{\!\bm{\tau}\,}\partial_t^h \psi\|_{\mathbf{L}^2(\Gamma)}^2+K^{-1}\|\partial_t^h\psi-\partial_t^h\varphi\|_{H_\Gamma}^2\Big)
        \notag\\
		&\qquad+C\Big(1+\|\partial_t^h f(\varphi)\|_H^2+\|\partial_t^h g(\psi)\|_{H_\Gamma}^2
        +k^{-2}\|\partial_t^h f_\sigma(\varphi)\|_H^2 +k^{-2}\|\partial_t^h f_\sigma (\psi)\|_{H_\Gamma}^2\Big)
        \notag\\
        &\qquad+C\Big(\|\partial_t^h \varphi\|_{H}^2+\|\partial_t^h\psi\|_{H_\Gamma}^2
        +\|\partial_t^h\mathbf{v} \|_{\mathbf{L}^2(\Omega)}^2
        +\|\partial_t^h\mathbf{v}_{\bm{\tau}} \|_{\mathbf{L}^2(\Gamma)}^2\Big).
        \label{0303-5}
\end{align}
Define
\begin{align*}
    \mathcal{F}& \coloneqq (1+\gamma)\|\nabla\partial_t^h\varphi\|_{\mathbf{L}^2(\Omega)}^2
    +(1+\gamma)\|\nabla_{\!\bm{\tau}\,} \partial_t^h\psi\|_{\mathbf{L}^2(\Gamma)}^2
    {+\frac{1+\gamma}{K}\int_\Gamma|\partial_t^h\psi-\partial_t^h\varphi|^2\,\mathrm{d}S},
    \\
    \mathcal{G}& \coloneqq 1+\|(\mathbf{v}(\cdot+h),\mathbf{v}_{\boldsymbol{\tau}}(\cdot+h))\|_{\boldsymbol{\mathcal{L}}^3}^2,
    \\[1ex]
    \mathcal{K}& \coloneqq 1+\|\partial_t^h f(\varphi)\|_H^2+\|\partial_t^h g(\psi)\|_{H_\Gamma}^2
    +k^{-2}\|\partial_t^h f_\sigma(\varphi)\|_H^2+k^{-2}\|\partial_t^h f_\sigma (\psi)\|_{H_\Gamma}^2
    \\
    &\qquad+\|\partial_t^h\varphi\|_{H}^2 +\|\partial_t^h\psi\|_{H_\Gamma}^2
    +\|\partial_t^h\mathbf{v} \|_{\mathbf{L}^2(\Omega)}^2
    +\|\partial_t^h\mathbf{v}_{\bm{\tau}} \|_{\mathbf{L}^2(\Gamma)}^2.
\end{align*}
It follows from \eqref{0303-5} that
\begin{align}
    \frac{\mathrm{d}}{\mathrm{d}t}\mathcal{F}(t)\leq C\mathcal{G}(t)\mathcal{F}(t)+C\mathcal{K}(t),
    \label{0303-6}
\end{align}
where the positive constant $C$ is independent of $h$.
Due to the assumed regularity of $(\mathbf{v},\mathbf{v}_{\boldsymbol{\tau}})$ (in particular, $(\partial_t\mathbf{v},\partial_t\mathbf{v}_{\boldsymbol{\tau}})\in L^2(0,T;\boldsymbol{\mathcal{L}}^2)$) and the regularity properties obtained in \ref{R1} and \ref{R2}, we can check that
\begin{align}
    \label{EST:GK}
    0\le \int_0^{t-h} \mathcal{G}(s) \,\mathrm ds \le C
    \quad\text{and} \quad
    0\le \int_0^{t-h} \mathcal{K}(s) \,\mathrm ds \le C
    \quad\text{for all $t\in (0,T-h)$.}
\end{align}
An application of Gronwall's lemma to \eqref{0303-6} yields that
\begin{align}
    \mathcal{F}(t)\leq e^{C\int_0^t \mathcal{G}(s)\,\mathrm{d}s}\mathcal{F}(0)
    +C\int_0^t e^{C\int_s^t \mathcal{G}(\tau)\,\mathrm{d}\tau}\mathcal{K}(s)\,\mathrm{d}s
    \leq C(1+\mathcal{F}(0))
    \label{0303-7}
\end{align}
for all $t\in(0,T-h)$.
Therefore, it remains to control the initial value $\mathcal{F}(0)$. Since $(\mu_0^\ast,\mathcal{L}_0^\ast)\in\mathcal{H}^1$
and the initial data are strictly separated from pure states $\pm 1$ (see~\eqref{A-separation}), it is easy to check that $\boldsymbol{\varphi}_0\in \mathcal{H}^3$. Using \eqref{convective-viscous:2} and \eqref{convective-viscous:6},
we can derive the following estimate
\begin{align}
		&\frac{1}{2}\frac{\mathrm{d}}{\mathrm{d}t}\Big(\gamma\|\nabla(\varphi-\varphi_0)\|_{\mathbf{L}^2(\Omega)}^2
        +\gamma\|\nabla_{\!\bm{\tau}\,}(\psi-\psi_0)\|_{\mathbf{L}^2(\Gamma)}^2
        {+\frac{\gamma}{K}\|(\psi-\psi_0)-(\varphi-\varphi_0)\|_{H_\Gamma}^2}\Big)
        \notag\\
		&\quad=\gamma\int_\Omega\nabla \partial_t\varphi\cdot\nabla(\varphi-\varphi_0)\,\mathrm{d}x
        +\gamma\int_\Gamma\nabla_{\!\bm{\tau}\,} \partial_t\psi\cdot\nabla_{\!\bm{\tau}\,}(\psi-\psi_0)\,\mathrm{d}S
        \notag\\
		&\qquad {+\frac{\gamma}{K}\int_\Gamma (\partial_t\psi-\partial_t\varphi)(\psi-\psi_0-\varphi+\varphi_0)\,\mathrm{d}S}
        \notag\\
        &\quad=-\gamma\int_\Omega\partial_t\varphi\Delta(\varphi-\varphi_0)\,\mathrm{d}x
        +\gamma\int_\Gamma \partial_t\varphi\partial_\mathbf{n}(\varphi-\varphi_0)\,\mathrm{d}S
        -\gamma\int_\Gamma\partial_t \psi\Delta_{\bm{\tau}}(\psi-\psi_0)\,\mathrm{d}S
        \notag\\
        &\qquad {+\gamma\int_\Gamma (\partial_t\psi-\partial_t\varphi)\partial_\mathbf{n}(\varphi-\varphi_0)\,\mathrm{d}S}
        \notag\\
        &\quad=-\gamma\int_\Omega\partial_t\varphi\Delta(\varphi-\varphi_0)\,\mathrm{d}x
        -\gamma\int_\Gamma \partial_t\psi(\Delta_{\boldsymbol{\tau}}(\psi-\psi_0)-\partial_\mathbf{n}(\varphi-\varphi_0))\,\mathrm{d}S
        \notag\\
        &\quad=\int_\Omega \Delta(\varphi-\varphi_0)(-\Delta\varphi+f(\varphi) +k^{-1}f_\sigma(\varphi)-\mu)\,\mathrm{d}x \notag\\
        &\qquad+\int_\Gamma (\Delta_{\boldsymbol{\tau}}(\psi-\psi_0)-\partial_\mathbf{n}(\varphi-\varphi_0))
        (-\Delta_{\boldsymbol{\tau}}\psi +\partial_\mathbf{n}\varphi+g(\psi)+k^{-1}f_\sigma(\psi)-\mathcal{L})\,\mathrm{d}S
        \notag\\
        &\quad=-\int_\Omega |\Delta(\varphi-\varphi_0)|^2\,\mathrm{d}x
        -\int_\Omega \Delta(\varphi-\varphi_0)\Delta\varphi_0\,\mathrm{d}x
        -\int_\Omega \nabla(\varphi-\varphi_0)\cdot\nabla(f(\varphi)+k^{-1}f_\sigma(\varphi)-\mu)\,\mathrm{d}x
        \notag\\
        &\qquad+\int_\Gamma \partial_\mathbf{n}(\varphi-\varphi_0)\big(f(\varphi)+k^{-1}f_\sigma(\varphi)-\mu
        +\Delta_{\boldsymbol{\tau}}\psi-\partial_\mathbf{n}\varphi-g(\psi)-k^{-1}f_\sigma(\psi) +\mathcal{L}\big) \,\mathrm{d}S
        \notag\\
        &\qquad-\int_\Gamma |\Delta_{\boldsymbol{\tau}}(\psi-\psi_0)|^2\,\mathrm{d}S
        -\int_\Gamma \Delta_{\boldsymbol{\tau}}(\psi-\psi_0)\Delta_{\boldsymbol{\tau}} \psi_0\,\mathrm{d}S
        \notag\\
        &\qquad-\int_\Gamma \nabla_{\!\boldsymbol{\tau}\,}(\psi-\psi_0)\cdot\nabla_{\!\boldsymbol{\tau}\,}
        (\partial_\mathbf{n}\varphi +g(\psi)+k^{-1}f_\sigma(\psi)-\mathcal{L})\,\mathrm{d}S
        \notag\\
        &\quad\leq \int_\Omega \nabla(\varphi-\varphi_0)\cdot\nabla\Delta\varphi_0\,\mathrm{d}x
        -\int_\Gamma \partial_\mathbf{n} (\varphi-\varphi_0)\Delta\varphi_0\,\mathrm{d}S
        +\int_\Gamma\nabla_{\!\boldsymbol{\tau}\,}(\psi-\psi_0)\cdot\nabla_{\!\boldsymbol{\tau}\,}\Delta_{\boldsymbol{\tau}\,}\psi_0 \,\mathrm{d}S
        \notag\\
        &\qquad+\|\nabla(\varphi-\varphi_0)\|_{\mathbf{L}^2(\Omega)}\|\nabla(f(\varphi)+k^{-1}f_\sigma(\varphi)-\mu)\|_{\mathbf{L}^2(\Omega)}
        \notag\\
        &\qquad+\|\nabla_{\!\boldsymbol{\tau}\,}(\psi-\psi_0)\|_{\mathbf{L}^2(\Gamma)}
        \|\nabla_{\!\boldsymbol{\tau}\,}(\partial_\mathbf{n}\varphi+g(\psi)+k^{-1}f_\sigma(\psi)-\mathcal{L})\|_{\mathbf{L}^2(\Gamma)}
        \notag\\
        &\qquad+\|\partial_\mathbf{n}(\varphi-\varphi_0)\|_{H_\Gamma}\|f(\varphi)+k^{-1}f_\sigma(\varphi)-\mu
        +\Delta_{\boldsymbol{\tau}}\psi-\partial_\mathbf{n}\varphi-g(\psi)-k^{-1}f_\sigma(\psi)
        +\mathcal{L}\|_{H_\Gamma}
        \notag\\
		&\quad\leq C\Big(\gamma\|\nabla(\varphi-\varphi_0)\|_{\mathbf{L}^2(\Omega)}^2
        +\gamma\|\nabla_{\!\bm{\tau}\,} (\psi-\psi_0)\|_{\mathbf{L}^2(\Gamma)}^2
     +\frac{\gamma}{K}\|(\psi-\psi_0) -(\varphi-\varphi_0)\|_{H_\Gamma}^2\Big)^{\frac{1}{2}}.
     \label{0303-8}
\end{align}
Now, applying Lemma~\ref{LEM:SGW} and choosing $t=h$ in the resulting estimate, we infer that
\begin{align}
	&\gamma\|\nabla\partial_t^h\varphi(0)\|_{\mathbf{L}^2(\Omega)}^2
    +\gamma\|\nabla_{\!\boldsymbol{\tau}\,} \partial_t^h\psi(0)\|_{\mathbf{L}^2(\Gamma)}^2
    {+\frac{\gamma}{K}\|\partial_t^h\psi(0)-\partial_t^h\varphi(0)\|_{H_\Gamma}^2}\leq C_\gamma.\notag
\end{align}
This entails that $\mathcal{F}(0)$ is uniformly bounded with respect to {$h\in \big(0,T/2\big)$}. 
In particular, in view of \eqref{EST:GK}, we can deduce from \eqref{0303-7} that
\begin{align}
    \label{EST:F}
    \mathcal{F}(t) \le C \quad\text{for all $t\in (0,T-h)$.}
\end{align}
Integrating \eqref{0303-5} with respect to time over $(0,T-h)$ and using the estimates \eqref{EST:GK} and \eqref{EST:F} to bound the right-hand side of the resulting estimate, after passing to the limit as $h\to 0$, we obtain
\begin{align}
    & (\partial_t\varphi,\partial_t\psi)\in L^\infty(0,T;\mathcal{H}^1)\cap L^2(0,T;\mathcal{H}^2),
    \notag \\
    &(\partial_{t}^2\varphi,\partial_{t}^2\psi)\in L^2(0,T;\mathcal{L}^2),\quad
    (\partial_t\mu,\partial_t\mathcal{L})\in L^2(0,T;\mathcal{H}^1).
    \notag
\end{align}
Finally, applying the elliptic regularity theory for bulk-surface elliptic systems (see, e.g., \cite[Theorem~3.3]{KL}) to the system \eqref{ellip*} (resp. \eqref{re-1}), we find
$ (\varphi,\psi)\in L^\infty(0,T;\mathcal{H}^3)$.
Thus, the proof of Theorem~\ref{A1} is complete.
\end{proof}

\section*{Declarations}
\noindent \textbf{Acknowledgments.}
HW is a member of the Key Laboratory of Mathematics
for Nonlinear Sciences (Fudan University), Ministry of Education of China.
\smallskip

\noindent \textbf{Funding.}
PK was supported by the Deutsche Forschungsgemeinschaft (DFG, German Research Foundation): on the one hand by the DFG-project 524694286, and on the other hand by the RTG 2339 ``Interfaces, Complex Structures, and Singular Limits''. HW was partially supported by the Natural Science Foundation of Shanghai 25ZR1401023.
The support is gratefully acknowledged.
\smallskip

\noindent \textbf{Competing interests.}
The authors declare that they have no known competing financial interests or personal relationships that could have appeared to influence the work reported in this paper.
\smallskip

\noindent \textbf{Data Availability Statement.}
Data sharing not applicable to this article as no datasets were generated or analyzed during the current study.

\smallskip


%
\end{document}